\documentclass[11pt]{article}

\usepackage{amsthm,amsmath,amssymb,bbm}
\usepackage{natbib}
\usepackage{multirow}
\usepackage{subfigure}
\usepackage{makecell}
\usepackage{booktabs}
\usepackage{array}
\usepackage{tabularx}
\usepackage{tabulary}
\usepackage{caption}
\usepackage{booktabs}
\usepackage{url}
\usepackage{algorithm}
\usepackage{algorithmic}
\usepackage{bm}
\usepackage{wrapfig}
\usepackage{lipsum}
\usepackage{mathrsfs}
\usepackage{dsfont}
\usepackage{titling}
\usepackage{txfonts}
\usepackage{tikz}
\usepackage{graphicx}
\usepackage{etoolbox}

\newtoggle{vaos}
\togglefalse{vaos}
\newcommand{\aosversion}[2]{\iftoggle{vaos}{#1}{#2}}

\usepackage[colorlinks,
linkcolor=red,
anchorcolor=blue,
citecolor=blue
]{hyperref}

\newcommand{\sgn}{\mathop{\mathrm{sign}}}

\usepackage{smile}

\newcommand{\nn}{\nonumber}

\def\##1\#{\begin{align}#1\end{align}}
\def\$#1\${\begin{align*}#1\end{align*}}

\newcommand{\blue}[1]{\textcolor{blue}{#1}}

\usepackage{relsize}

\newcommand{\bfsym}[1]{\ensuremath{\boldsymbol{#1}}}
\def \balpha   {\bfsym{\alpha}}       \def \bbeta    {\bfsym{\beta}}

\newcommand{\Rom}[1]{\text{\uppercase\expandafter{\romannumeral #1\relax}}}

\usepackage{geometry}
\usepackage[font={footnotesize,it}]{caption}

\usepackage{enumitem}

\definecolor{revisecolor}{HTML}{000000}
\definecolor{myred}{HTML}{ae1908}
\definecolor{myblue}{HTML}{05348b}
\definecolor{myorange}{HTML}{ec813b}
\definecolor{mylightblue}{HTML}{9acdc4}
\newcommand{\myred}[1]{{\color{myred} #1}}
\newcommand{\myblue}[1]{{\color{myblue} #1}}

\newcommand{\myorange}[1]{{\color{myorange} #1}}

\numberwithin{equation}{section}

\begin{document}

\title{How do noise tails impact on deep ReLU networks?%\\ for Nonparametric Regression?
}
%Adaptive Huber Nonparameteric Regression with ReLU Deep Neural Networks}

%\author{Jianqing Fan, Yihong Gu, Wen-Xin Zhou}
\author{Jianqing Fan\thanks{Department of Operations Research and Financial Engineering, Princeton University, Princeton, NJ 08544. E-mail: jqfan@princeton.edu. Fan's research was supported by ONR grant N00014-19-1-2120 and the NSF grants DMS-1662139, DMS-1712591, DMS-2052926, and DMS-2053832.},~~~Yihong Gu\thanks{Department of Operations Research and Financial Engineering, Princeton University, Princeton, NJ 08544.  E-mail: yihongg@princeton.edu.}~~\mbox{ and }~Wen-Xin Zhou\thanks{Department of Mathematical Sciences,
University of California, San Diego, La Jolla, CA 92093. E-mail: wez243@ucsd.edu. Zhou's research was supported by the NSF grant DMS-2113409.} }

\date{Princeton University and University of California, San Diego}

%\date{}
\maketitle

\vspace{-0.25in}

\def\todo#1{\textcolor{purple}{[TODO] #1}}
\def\r#1{\textcolor{red}{#1}}
\def\b#1{\textcolor{blue}{#1}}

\newcommand{\revise}[1]{\textcolor{revisecolor}{#1}}

\begin{abstract}
This paper investigates the stability of deep ReLU neural networks for nonparametric regression under the assumption that the noise has only a finite $p$-th moment. We unveil how the optimal rate of convergence depends on $p$, the degree of smoothness and the intrinsic dimension in a class of nonparametric regression functions with hierarchical composition structure when deep ReLU neural networks and an adaptively chosen Huber loss are used.  This optimal rate of convergence cannot be obtained by the ordinary least squares but can be achieved by the Huber loss with a properly chosen parameter that adapts to the sample size, smoothness, and moment parameters. A concentration inequality for the adaptive Huber ReLU neural network estimators with allowable optimization errors is also derived.  To establish a matching lower bound within the class of neural network estimators using the Huber loss, we employ a different strategy from the traditional route:  constructing a deep ReLU network estimator that has a smaller empirical loss than the true function and the difference between these two functions furnishes a low bound.  This step is related to the Huberization bias, yet more critically to the approximability of  deep ReLU networks. As a result, we also contribute some new results on the approximation theory of deep ReLU neural networks.
\end{abstract}

\noindent{\bf Keywords}: Robustness, Heavy Tails, Optimal Rates, Approximability of ReLU networks, Composition of functions.

%%%%%%%%%%%%%%%%%%%%%%%%%%%%%%%%%%%%%%%%%%%
%%%%%%%%%%%%%%%%%%%%%%%%%%%%%%%%%%%%%%%%%%%
% Introduction
%%%%%%%%%%%%%%%%%%%%%%%%%%%%%%%%%%%%%%%%%%%
%%%%%%%%%%%%%%%%%%%%%%%%%%%%%%%%%%%%%%%%%%%
\section{Introduction}
\label{sec:1} 

Consider a nonparametric regression model $Y = f_0(X) + \varepsilon$, where $Y$ is the response variable, $X$ is a $d$-dimensional covariate vector,  $f_0: \RR^d \to \RR$ is an unknown function of interest, and the observation noise $\varepsilon$ satisfies $\mathbb{E}[\varepsilon|X]=0$. A fundamental statistical problem is to estimate the unknown regression function $f_0$ based on a sample of $n$ independent observations $\{(X_i, Y_i)\}_{i=1}^n$ from the above model.
From the early 1960s to the late 1990s, there has been a rich development in nonparametric regression methods, including kernel and local polynomial regressions \citep{nadaraya1964estimating,watson1964smooth,haberman1977maximum,cleveland1979robust,fan1992design,fan1993local}, spline-based methods \citep{friedman1991multivariate}, tree-based methods \citep{breiman2017classification}, regression in reproducing kernel Hilbert  spaces (RKHS) \citep{wahba1990spline} and (shallow) neural network regression \citep{barron1994approximation, mccaffrey1994convergence}, to name a few. We refer to \cite{fan1996local}, \cite{gyorfi2002distribution} and \cite{tsybakov2009nonparametric} for a comprehensive exposition on nonparametric regression. 
%There are various methods to estimate $f_0$ such as kernel estimator and multivariate spline estimator, see a monograph \cite{gyorfi2002distribution}. 

With the development of computational capability that makes training deep neural networks possible and efficient, deep neural networks have shown superior performance over classical methods in many machine learning tasks \citep{lecun2015deep}. From a statistical view, 
a key reason for the great success of neural networks is their approximation ability in the sense that many complex nonlinear functions (of several variables) can be effectively approximated by neural networks.
The well-known {\it universal approximation theorem} \citep{cybenko1989approximation, hornik1991approximation, barron1993universal} shows that a neural network with one hidden layer can approximate any continuous function up to any specified precision.
To analyze the statistical properties of neural network
estimators, it is necessary to derive nonasymptotic approximation error bounds for neural networks. For fully-connected deep neural networks with ReLU activation function (ReLU-DNN), \cite{telgarsky2016benefits} demonstrated the great benefits of using deep neural networks. As an illustrative example, a tooth function with $\mathcal{O}(2^k)$ oscillations can be realized as a ReLU-DNN with depth $\mathcal{O}(k)$ and width $\mathcal{O}(1)$, and a one-hidden-layer ReLU-DNN but with $\Omega(2^k)$ nodes. 
Since smooth functions can be well approximated by their Taylor series, based on the idea of point fitting and approximating polynomials via tooth function, \cite{yarotsky2017error} proved a near-optimal approximation error (in terms of the number of parameters) for Lipschitz functions.
This idea is widely used in deriving nonasymptotic approximation error bounds for various classes of functions \citep{approxshen2019deep, approxshen2022optimal, approxsuzuki2018adaptivity, approxyarotsky2018optimal, lu2020deep}. Via more delicate point fitting techniques, \cite{lu2020deep} established a nonasymptotic $L_\infty$ approximation error bound for smooth functions using ReLU-DNN with arbitrary depth and width, which is optimal (in terms of width and depth) up to logarithmic factors.

% Nonparametric regression using neural network.  
Another important reason for the success of neural networks is their ability to be adaptive to unknown low-dimensional structures.
By employing the compositional nature of the deep neural network and the aforementioned approximation results for smooth functions,  
\cite{bauer2019deep}, \cite{schmidt2020nonparametric} and \cite{kohler2021rate} showed that neural networks could circumvent the curse of dimensionality if the intrinsic dimension $d^*$ of the regression function $f_0$ is much smaller than the input dimension $d$. Specifically, suppose $f_0$ can be represented as a hierarchical composition of several smooth functions, with either a high degree of smoothness condition or low input dimension,  the neural network can automatically adapt to the intrinsic low-dimensional structure without knowing the composition structure explicitly. \cite{schmidt2020nonparametric} also showed that neural network estimates achieve the minimax-optimal rate of convergence when the regression function has such a structure. Furthermore, the deep neural network can also be used to estimate the nonlinear component of a semi-parametric model, which circumvents the curse of dimensionality and facilitates statistical inference on the linear component \citep{farrell2021deep, zhong2022deep}.

% Robust estimation
The existing results for the least squares ReLU-DNN regression estimates rely on a sub-Gaussian moment condition on the regression error \citep{schmidt2020nonparametric, kohler2021rate}. A natural question is:
\begin{align*}
    \mbox{\emph{Whether the least squares ReLU-DNN estimator achieves the same convergence rate when the}} \\
    \mbox{\emph{error distribution is heavy-tailed, and if not, whether there exists a robust alternative that can.}}
\end{align*}
%whether the least squares ReLU-DNN estimator achieves the same convergence rate when the error distribution is heavy-tailed, and if not, whether there exists a robust alternative that can. 

\subsection{Related Works}

\noindent \textbf{Nonparametric least squares with heavy-tailed errors.} 
When regressing directly over the nonparametric function class to which $f_0$ belongs,  some recent works discuss the effect of heavy-tailed errors on the convergence rate of (constrained) least squares estimator \citep{han2018robustness, han2019convergence, kuchibhotla2019least}. They argue that the convergence rate of least squares estimators may depend on both the complexity of the function class and the order of moments of the regression error. Specifically, for a uniformly bounded function class satisfying a standard ``entropy condition'' with exponent $\alpha \in (0, 2)$, \cite{han2019convergence} showed that the corresponding least squares estimator converges at a rate $\mathcal{O}_{\mathbb{P}}(n^{-1/(2+\alpha)} \vee n^{-\frac{1}{2}(1-1/p)})$ (in $L_2$ error) when $\varepsilon$ has bounded $(p+\epsilon)$-th  ($p\geq 1$) moment and is independent of $X$.
Therefore, for certain function class that is not sufficiently complex, the heavy-tailedness of the errors is the main cause for the least squares estimator to converge at a slower rate (when $p< 1+ 2/\alpha$). Moreover, \cite{han2019convergence} also proved the sharpness of this rate by constructing some non-smooth function classes that witness the worst case rate $\mathcal{O}_{\mathbb{P}}(  n^{-\frac{1}{2}(1-1/p)})$. \cite{kuchibhotla2019least} proved similar results when $\varepsilon$ may also depend on $X$. Specifically, they provided a detailed characterization of the convergence rate when $\mathcal{F}$ is a uniform VC-type function class (indexed by $\alpha \geq 0$) and $p=2$. Note that a ReLU-DNN with fixed depth and width belongs to a parametric function class with finite VC dimension, which corresponds to the case of $\alpha=0$.
In this case, \cite{han2019convergence} showed that the convergence rate is of order $\mathcal{O}_{\PP}(n^{-1/4})$, achieved by a highly non-smooth function class.  \cite{kuchibhotla2019least} claimed that the degree of smoothness for the function class, measured via local envelope function, determines the convergence rate which is $\mathcal{O}_{\PP}(n^{-1/2})$ for the most smooth function class and $\mathcal{O}_{\PP}(n^{-1/4})$ for the most non-smooth class. However, due to the more complex nature of neural network classes, the impact of the tails of the noise on the least squares estimator remains unclear.
\medskip

%Though the above results for neural networks are astonishing, they are not applicable when the noise is heavy-tailed. 
\noindent \textbf{Robust loss for heavy-tailed errors.} 
To robustify least squares estimates, several robust loss functions have been widely used, including but not limited to the $L_1$ loss, Huber's loss \citep{huber1973robust}, the Cauchy loss and Tukey's biweight loss \citep{beaton1974fitting}.
Originally these robust methods were introduced to guard against outliers in the observations, say under Huber's contamination model. When there is no contamination but the underlying distribution itself is heavy-tailed and skewed,  
\cite{fan2017estimation} and \cite{sun2020adaptive} revisited the Huber regression method and advocated the use of an adaptive robustification parameter $\tau$ for the bias-robustness tradeoff. Via a deviation study, \cite{sun2020adaptive} showed that the adaptive Huber (linear) regression estimator  satisfies sub-Gaussian-type concentration bounds even when the error only has low-order moments. It should be noted that for linear models, both the least squares estimator and its robust alternative admit the same rate of convergence as long as the errors have finite variance, while the advantage of the latter is that it achieves exponential-type deviation bounds even when the error variable does not have exponentially thin tails. The main reason for this is that linear functions of the form $f(x)=\beta^\top x$ not only have simple structures but also are sufficiently smooth if $\|\beta\|_2$ is bounded. For nonparametric models, it is unclear whether a robust regression estimator can achieve a faster convergence rate than least squares estimators when $p\ge 2$.

% Robust estimation using neural network
\medskip
\noindent \textbf{Robust methods for ReLU neural network.} 
The shortcomings of the nonparametric least squares estimators, specifically the lack of robustness, have motivated the development of robust methods {when ReLU-DNN is used} \citep{shen2021deep, shen2021robust, hernan2020quantile, lederer2020risk}. Using deep neural networks, these papers studied nonparametric robust regression with a $\lambda_L$-Lipschitz continuous loss $\rho$, typified by the Huber loss and the check loss \citep{shen2021deep, hernan2020quantile}, and established upper bounds on the excess risk $\mathbb{E}[\rho(\hat{f}(X)-Y)-\rho(f^*(X)-Y)]$ where $f^*$ is the population risk minimizer. \cite{shen2021robust} showed that when the observation noise has bounded $p$-th moment, the empirical risk minimizer $\hat{f}$ satisfies the excess risk bound
\begin{align*}
        \mathbb{E} \big[\rho(Y-\hat{f}(X)) - \rho(Y-f^*(X)) \big] \lesssim \frac{\lambda_L (NL)^2}{n^{1-1/p}} +  \omega^2_{f^*}\big((NL)^{-2/d}\big) 
\end{align*} up to logarithmic factors, where $\omega_{f}(\cdot)$ is the modulus of continuity of function $f$, i.e., $\omega_{f}(\delta) = \sup_{\|x- y\|\le \delta} |f(x) - f(y)|$. These results provide a first glance at the impact of noise tails on regression with ReLU-DNN, but still leave several loopholes as follows: (i) the convergence rate cannot take advantage of the low-dimensional structure of $f_0$ since $f^*$ and $f_0$ are in general not the same; (ii) the convergence rate (under $L_2$ loss) for estimating $f_0$, i.e., $ \mathbb{E}_X|\hat{f}(X)-f_0(X)|^2 $, is still unclear; (iii) the theoretical benefit of using a robust ReLU-DNN estimator is ambiguous from the above result because the obtained convergence rate turns out to be slower than that of the least squares counterpart when $f_0=f^*$ is $(\beta, C)$-smooth; see the discussions in Section~\aosversion{\blue{D.1}}{\ref{subsec:upper-bound-discussion}}.

\subsection{Our contributions}

In this paper, we attempt to address the aforementioned questions by comprehensively analyzing the impact of heavy-tailed noise on the convergence rate of fully-connected ReLU-DNN estimators. Inspired by \cite{fan2017estimation}, we focus on the Huber-type ReLU-DNN estimator, defined as the empirical Huber loss minimizer over the ReLU-DNN function class with robustification parameter $\tau$. {\color{revisecolor} When the noise $\varepsilon$ is heavy-tailed and has uniformly bounded $p$-th moment, we unveil how the $L_2$ error $\|\hat{f}_{n} - f_0\|_2=  \{\EE_X |\hat{f}_{n}(X)- f_0(X)|^2\}^{1/2}$ depends on the smoothness of $f_0$, moment index $p$, and a combination of hyper-parameters, including the network depth $\bar{L}$, width $\bar{N}$ and robustification parameter $\tau$. This further demystifies how the regression function class and the degree of heavy-tailedness jointly impact the convergence rates of both adaptive Huber and least squares ReLU-DNN estimators.

\medskip

\noindent \textbf{A generic upper bound.} We start by establishing a non-asymptotic bound on the $L_2$ error of the Huber ReLU-DNN estimator for any $\bar{N}$, $\bar{L}\geq 3$ and sufficiently large $\tau$. Specifically, we will show in Theorem \ref{thm:generic-bound} that if the noise $\varepsilon$ has bounded $p$-th moment ($p\geq 2$), any approximate (within a given order of optimization error) empirical (Huber) risk minimizer satisfies, up to logarithmic factors, that, for large enough $n$ and $D$,
\begin{align}
    \label{eq:intro-upper}
    \begin{split}
    \mathbb{P}\Bigg[\| \hat{f}_n - f_0 \|_2 \ge D \Bigg\{ \inf_{f\in \mathcal{F}_n} \|f-f_0\|_2 + \frac{1}{\tau^{p-1}} & + \frac{\bar{N}\bar{L}\sqrt{\tau}}{\sqrt{n}} \land \Bigg(\frac{\bar{N}\bar{L}}{\sqrt{n}}\Bigg)^{1-1/p}\Bigg\} \Bigg]  \\
    & \lesssim e^{-(\bar{N}\bar{L})^2 D^2} +\frac{1\{ (\frac{\sqrt{n}}{\bar{N}\bar{L}})^{\frac{2}{p}} = o(\tau)\}}{D^{2p}}.
   	\end{split}
\end{align}
%For two non-negative sequences $\{a_n\}_{n\geq 1}, \{ b_n \}_{n\geq 1}$,  $a_n \ggg b_n$ means $\limsup_{n\to\infty} a_n/ b_n  = \infty$. 
This result applies to the least squares ReLU-DNN estimator by taking $\tau=\infty$. Compared to the oracle-type $L_2$ error bound under sub-Gaussian noise, which takes the form $\delta_{\mathtt{a}} + \delta_{\mathtt{s,g}}$ with approximation error $\delta_{\mathtt{a}} = \inf_{f\in \mathcal{F}_n} \|f-f_0\|_2$ and stochastic error $\delta_{\mathtt{s,g}} = n^{-1/2}\bar{N}\bar{L}$, our result \eqref{eq:intro-upper} depicts how the heavy-tailed noise impacts the $L_2$ error: it first introduces a Huberization bias term $\delta_{\mathtt{b}}=\tau^{1-p}$ when the noise is asymmetric, and then inflates the stochastic error term, leading to $\delta_{\mathtt{s}} = \delta_{\mathtt{s,g}} (\sqrt{\tau}\land \delta_{\mathtt{s,g}}^{-1/p})$. Moreover, the error bound and tail probability go through two phases according to the choice of $\tau$. If $\tau = \mathcal{O}(\delta_{\mathtt{s,g}}^{-2/p})$, the $L_2$ error admits an exponential-type deviation bound and is of order $\delta_{\mathtt{s,g}}\sqrt{\tau}$. On the other hand, if $ \delta_{\mathtt{s,g}}^{-2/p}=o(\tau)$, the error admits a polynomial-type deviation bound with a dominating term $\delta_{\mathtt{s,g}}^{1-1/p}$ that is independent of $\tau$. This is the same as the error bound for the least squares estimator ($\tau=\infty$).

\medskip
\noindent \textbf{Adapting to the low-dimensional structure under heavy-tailed noise.} The above oracle-type inequality indicates that one needs to carefully balance the Huberization bias $\delta_{\mathtt{b}}$, the ReLU-DNN approximation error $\delta_{\mathtt{a}}$ and the stochastic error $\delta_{\mathtt{s}}$ to reach an optimal statistical rate of convergence. As an application of \eqref{eq:intro-upper}, we show that with properly tuned hyper-parameters ($\tau$, $\bar{N}$, $\bar{L}$) depending on $p$ and $\gamma^*$, the intrinsic dimension-adjusted smoothness of the regression function class $\mathcal{H}$, it holds for any $D\geq 1$ that
\begin{align}
\label{eq:intro:error-bound}
    \sup_{f_0\in \mathcal{H}} \mathbb{P}\big(  \|\hat{f}_n - f_0 \|_2 \ge D \delta_n \big) \lesssim \exp  (- n^{c} D^2  ) \qquad  \text{with} \qquad \delta_n \asymp n^{-\frac{ \nu^* \gamma^*}{2\gamma^* + \nu^* } } ,
\end{align} 
where $\nu^*=1-1/(2p-1)$, and $c$ is a positive constant independent of $n$ and $D$. Our results reveal the following two advantages of the \emph{adaptive Huber estimator}.
\begin{itemize}
    \item[(a).] With properly chosen hyper-parameters, the adaptive Huber ReLU-DNN estimator circumvents the curse of dimensionality in the heavy-tailed setting the same way as the least squares estimator does with sub-Gaussian errors: the convergence rate depends only on the intrinsic dimension. 
    \item[(b).] Applying \eqref{eq:intro-upper} also yields an error bound for the least squares ReLU-DNN estimator. In particular, the least squares estimator converges at the rate $\mathcal{O}_{\mathbb{P}} ( n^{-  \nu^{\dagger}  \gamma^* / (2\gamma^* +\nu^{\dagger}) } )$ with $\nu^\dagger = 1- 1/p < \nu^*$. Although both estimators adapt to the low-dimensional structure of $f_0$, the adaptive Huber estimator achieves a faster convergence rate than the least squares estimator under the $p$-th ($p\geq 2$) moment condition. In addition, from a nonasymptotic perspective, the adaptive Huber estimator admits exponential-type deviation bounds whereas the tail probability for the least squares estimator decays polynomially. 
\end{itemize}

Moreover, in the special case where the heavy-tailed error is symmetric, we further show in Theorem \ref{thm:convergence-rate-nn-symmetric} that the Huber ReLU-DNN estimator with a fixed robustification parameter achieves the optimal convergence rate $\mathcal{O}_{\mathbb{P}} ( n^{- \gamma^* / ( 2\gamma^*+1 )} )$, attainable by the least squares estimator only when $\varepsilon$ is sub-Gaussian.

\begin{figure}[t!]
    \centering
    \includegraphics{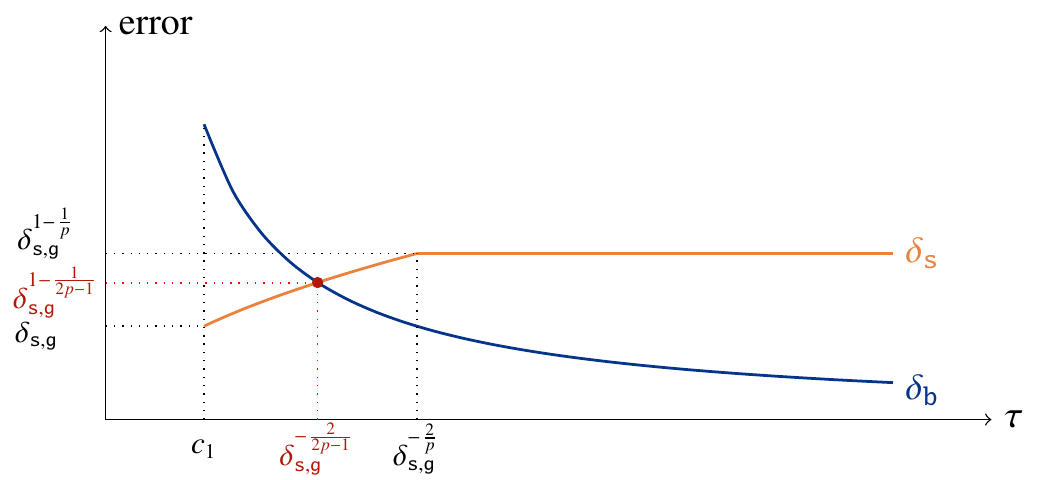}
\caption{A visualization of the bias-robustness trade-off for fixed $\bar{N}$, $\bar{L}$, $n$ and varying $\tau$, provided that it is sufficiently large. 
The bias $\delta_{\mathtt{b}}$ is the dominating term when $\tau\in [0,\delta_{\mathtt{s,g}}^{-2/(2p-1)}]$, and the stochastic error $\delta_{\mathtt{s}}$ dominates when $\tau \geq \delta_{\mathtt{s,g}}^{-2/(2p-1)}$. In particular when $\tau \geq \delta_{\mathtt{s,g}}^{-2/p}$, $\delta_{\mathtt{s}}$ coincides with $\delta_{\mathtt{s},\mathtt{LS}}$, the stochastic error term that corresponds to the least squares estimator.
}
\label{fig:effects-tau}
\end{figure}

\medskip
\noindent \textbf{The tightness of the upper bound.} A natural question is whether the obtained upper bound is sharp. We further show in Theorem \ref{thm:generic-lower-bound} that if the minimax-optimal $L_2$ risk over the function class $\mathcal{F}_0$ is lower bounded by $n^{-\frac{\alpha}{2\alpha+1}}$ , then for any $\bar{N}, \bar{L}, \tau \ge C$, 
\begin{align}
\label{eq:intro-lower}
    \sup_{\substack{f_0 \in \mathcal{F}_0, X\sim \mathrm{Unif} \\ \mathbb{E}[|\varepsilon|^p|X=x]\le 1}} \mathbb{P}\Bigg[ \|\hat{f} - f_0\|_2 \ge C^{-1} \Bigg\{ (\bar{N}\bar{L})^{-2\alpha} + \frac{1}{\tau^{p-1}}+ \delta_{\mathtt{s,g}} (\sqrt{\tau} \land \delta_{\mathtt{s,g}}^{-1/p})\Bigg\} \Bigg] = 1 - o(1) ,
\end{align}
ignoring logarithmic factors. This confirms the tightness of \eqref{eq:intro-upper}. Moreover, combining \eqref{eq:intro-lower} with the minimax-optimal rate associated with class $\mathcal{H}$  indicates that the convergence rates for the adaptive Huber estimator \eqref{eq:intro:error-bound} and the least squares estimator are both sharp. The former is intrinsically more robust to heavy-tailedness as expected.

\medskip
\noindent \textbf{An exact characterization of the stochastic error.} A key ingredient in proving \eqref{eq:intro-upper} and \eqref{eq:intro-lower} is an exact characterization of the stochastic error $\delta_{\mathtt{s}}$ under heavy-tailed noise.  
The upper bound analysis involves a combination of peeling and truncation arguments and strengthens the previous results \citep{shen2021robust, kuchibhotla2019least}. The analysis of the lower bound consists of two novel parts. The first one is an argument that relates the slow stochastic error rate to a specific approximation ability. The second one is the development of ReLU-DNN's approximation ability for fitting outliers in a highly non-smooth manner (Theorem \ref{thm:approx-noise-dim-d}). Such an approximation result plays a key role in attaining the lower bound and is of independent interest. It also uncovers a less favorable aspect of deep ReLU neural networks' superior approximation ability, which makes ReLU-DNN regression vulnerable to outliers and thus less robust.

%On one side, their flexible approximation ability allows them to efficiently approximate a large family of functions as much of neural network approximation literature does. On the other side, such flexibility also makes the neural networks vulnerable to outliers and thus less robust.

Moreover, the above sharp characterization indicates that adaptive Huberization always helps in the presence of heavy-tailed noise. As shown in \eqref{eq:intro-upper} and \eqref{eq:intro-lower}, the heavy-tailedness contributes to two terms, $\delta_{\mathtt{b}}$ and $\delta_{\mathtt{s}}$, in the $L_2$ error bound. Compared to the least squares estimator whose stochastic error is $\delta_{\mathtt{s},\mathtt{LS}} = (\delta_{\mathtt{s,g}})^{1-1/p}$, the stochastic error of the Huber estimator with optimally tuned $\tau$ that balances $\delta_{\mathtt{b}}$ and $\delta_{\mathtt{s}}$ is
\begin{align*}
    \delta_{\mathtt{s},\mathtt{AH}} \asymp \delta_{\mathtt{s,g}}^{1- 1/(2p-1)} = o(\delta_{\mathtt{s}, \mathtt{LS}}).
\end{align*} 
See Figure \ref{fig:effects-tau} for an illustration of the impact of $\tau$ on $L_2$ error. Though stated specifically for ReLU-DNN, our analysis of $\delta_{\mathtt{s}}$ and the higher degree of robustness of adaptive Huberization (over least squares) also apply to more general classes that share the same non-smooth approximation ability as ReLU-DNN.
}\subsection{Organization} 
The rest of the paper is as follows. Section \ref{sec:method} describes the Huber ReLU-DNN estimator, which is the main objective of this work. In Section \ref{sec3}, we first provide an upper bound on the Huberization bias, and then establish a generic nonasymptotic $L_2$ error bound for the Huber ReLU-DNN estimator with any hyper-parameters. Based on this generic bound, we present convergence rates for the adaptive Huber estimator (with an optimally tuned $\tau$) and the least squares estimator ($\tau=\infty$) under both symmetric and asymmetric error settings. To examine the tightness of the obtained upper bounds, in Section \ref{sec:lb} we further establish a generic lower bound on the $L_2$ error of Huber ReLU-DNN estimators with any sufficiently large robustification parameter. All the technical proofs are provided in the supplemental material.

\subsection{Notations} 
The following notations will be used throughout this paper. We use $c_1,c_2,\ldots$ to denote the global constants that appear in the statement of any theorem, proposition, corollary, and lemma. We use $C_1, C_2,\ldots$ to denote the local intermediate constants in the proof. Hence all the  $c_1,c_2,\cdots$ have unique referred numbers, while all the $C_1, C_2,\ldots$ will have different referred numbers in respective proofs. We use $a\lesssim b$ if there exists some universal constant $C$ such that $a \le Cb$, we use $a\gtrsim b$ if there exists some universal constant $C>0$ such that $a \ge Cb$, we use $a \asymp b$ if $a\lesssim b$ and $a \gtrsim b$. 
%%%%%%%%%%%%%%%%%%%%%%%%%%%%%%%%%%%%
%%%%%%%%%%%%%%%%%%%%%%%%%%%%%%%%%%%%
% Method
%%%%%%%%%%%%%%%%%%%%%%%%%%%%%%%%%%%%
%%%%%%%%%%%%%%%%%%%%%%%%%%%%%%%%%%%%

\section{Setup and Methodology}
\label{sec:method}

Consider a nonparametric regression model 
\begin{align}
    Y = f_0(X) + \varepsilon, \label{np.model}
\end{align} 
where $X\in [0,1]^d$ is the $d$-dimensional covariate vector, and $\varepsilon$ is the noise variable satisfying
\begin{align}
    \mathbb{E}[\varepsilon|X=x]=0 ~~\mbox{ and }~~ \mathbb{E}[|\varepsilon|^p|X=x]\le v_p < \infty \text{ for all } x\in [0,1]^d. \label{moment.cond}
\end{align}
{\color{revisecolor} 
Note that the bounded $p$-th moment assumption is related to Huber's contamination model in a specific way.  It contains errors of the form $\varepsilon \sim (1-\epsilon)  N(0, 1) + \epsilon  N(0, \epsilon^{-2/p})$, or more generally that the distribution of the main component (inlier distribution) has bounded $p$-th moment and the distribution of the contaminated component (outlier distribution) has $p$-th moment bounded by $\epsilon^{-1}$,  which admits outliers of magnitude $n^{1/p}$ among $n$ data points. 
}

Let $\{(X_i, Y_i)\}_{i=1}^n$ be i.i.d. observations from model \eqref{np.model}. Our goal is to estimate the unknown regression function $f_0: [0, 1]^d \to \mathbb{R}$. Within a suitably chosen function class $\mathcal{F}_n$, the nonparametric least squares method aims to find some $\hat{f}_n$ that minimizes the $L_2$ loss
\begin{align}
    \hat{\mathcal{R}}(f) = \frac{1}{n} \sum_{i=1}^n \{ f(X_i) - Y_i \}^2.
\end{align}

The accuracy of the estimator $\hat{f}_n$ can be evaluated through the mean squared error $\mathbb{E}_X |\hat f_n(X) - f_0(X)|^2$,  which is the \emph{excess risk} %which is defined as the difference between the $L_2$-risks of $\hat{f}_n$ and $f_0$:
\begin{align}
    \mathcal{R}(\hat{f}_n) - \mathcal{R}(f_0) 
    %= \mathbb{E}[|\hat{f}_n(X)-Y|^2] - \mathbb{E}[|f_0(X)-Y|^2] 
    = \EE_X [|\hat{f}_n(X)-f_0(X)|^2],
\end{align}
where $\mathcal{R}(f) = \mathbb{E}_{(X,Y)}[|Y - f(X)|^2]$ is the {\it $L_2$ risk} of $f$. The statistical rate of convergence of $\hat f_n$ depends, among several other factors, on the class that $f_0$ lies in. We first revisit the class of $(\beta, C)$-smooth functions as follows.

\begin{definition}[$(\beta,C)$-smooth function]
    Let $\beta = r+s$ for some nonnegative integer $r$ and $0<s\le 1$, and $C>0$. A $d$-variate function $f$ is called $(\beta, C)$-smooth if for every sequence $\{\alpha_j\}_{j=1}^d$ of nonnegative integers such that $\sum_{j=1}^d \alpha_j = r$, the partial derivative $(\partial f)/(\partial x_1^{\alpha_1}\cdots \partial x_d^{\alpha_d})$ exists and satisfies for any $x, z \in \RR^d$ that
    \begin{align}
        \left|\frac{\partial^r f}{ \partial x_1^{\alpha_1}\cdots \partial x_d^{\alpha_d}}(x) - \frac{\partial^r f}{\partial x_1^{\alpha_1}\cdots \partial x_d^{\alpha_d}}(z) \right| \le C \|x-z\|_2^s.
    \end{align} 
\end{definition}
It is well-known that the minimax rate of convergence over the $(\beta, C)$-smooth function class is of order $n^{-\frac{\beta}{2\beta+d}}$ \citep{gyorfi2002distribution}. This is often referred to as the \emph{curse of dimensionality} in nonparametric regression as the rate is substantially slower when $d$ is moderately large.   \revise{As our problem involves different dimensions and degrees of smoothness, we will refer to $\gamma = \beta/d$ as dimension-adjusted degree of smoothness and expression the convergence rate as $n^{-\gamma/(2\gamma+1)}$.}  \revise{To alleviate the curse of dimensionality in an algorithmic manner, that is, efficiently estimate a regression function $f_0$ when it satisfies an inherently low-dimensional structure, \cite{bauer2019deep} and \cite{kohler2021rate} introduced the following hierarchical composition model to characterize such a structure of $f_0$.}

{
\color{revisecolor}
\begin{definition}[Hierarchical composition model]
\label{hcm} 
Given positive integers $d, l \in \mathbb{N}^+$ and a subset of $[1, \infty) \times \mathbb{N}^+$, denoted by $\mathcal{P}$, satisfying $\sup_{(\beta, t) \in \mathcal{P}} \max\{ \beta, t\} <\infty$, the hierarchical composition model $\mathcal{H}(d, l, \mathcal{P})$ is defined recursively as follows. For $l=1$, 
\begin{align*}
\mathcal{H}(d, 1,\mathcal{P}) = \big\{  &h: \mathbb{R}^d \to \mathbb{R}: h(x) = g(x_{\pi(1)},...,x_{\pi(t)})\text{, where } \pi: [t]\to [d] \text{ and} \\
&~~~~~~~~g:\mathbb{R}^t \to \mathbb{R}\text{ is }(\beta, C)\text{-smooth for some } (\beta,t)\in \mathcal{P}, C>0 \big\};
\end{align*} 
and for $l>1$,
%In other words, $\mathcal{H}(d, 1,\mathcal{P})$ consists of all $t$-variate $(\beta, C)$-smooth functions for some positive constant $C$.
%For $l>1$, $\mathcal{H}(d, l,\mathcal{P})$ is defined recursively as
\begin{align*}
\mathcal{H}(d, l,\mathcal{P}) = \big\{&h: \mathbb{R}^d \to \mathbb{R}: h(x) = g(f_1(x),...,f_t(x))\text{, where } f_i \in \mathcal{H}(d, l-1,\mathcal{P}) \text{ and}\\
&~~~~~g:\mathbb{R}^t \to \mathbb{R}\text{ is }(\beta, C)\text{-smooth for some } (\beta,t)\in \mathcal{P}, C>0  \big\}.
\end{align*}
\end{definition}

When $f_0\in \mathcal{H}(d, l, \mathcal{P})$ and the noise variable $\varepsilon \in \RR$ is sub-Gaussian, \cite{kohler2021rate} showed that the least squares ReLU-DNN estimator achieves the convergence rate $n^{-\frac{\gamma^*}{2\gamma^*+1}}$ (up to some logarithmic factor) with properly tuned network width $\bar{N}$ and depth $\bar{L}$ that depend on $n$ and $\gamma^*$, where
\begin{align}
\label{eq:def-gamma-star}
    \gamma^* = \frac{\beta^*}{d^*}~~~~\text{with}~~~~ (\beta^*,d^*) = \argmin_{(\beta, t)\in \mathcal{P}} \frac{\beta}{t}
\end{align} characterizes the dimension-adjusted smoothness of the least smooth (after dimension adjustment) component in the compositions. A similar result under sparsely connected deep ReLU networks is obtained by \cite{schmidt2020nonparametric}.
}

%Under a hierarchical composition assumption on $f_0$, \cite{schmidt2020nonparametric} showed that the estimator obtained from learning a deep neural network with ReLU activation and sparse connection achieves a faster convergence rate $n^{-\frac{2\beta^*}{2\beta^*+d^*}}$, where $d^*$ is the intrinsic dimension that can be much smaller than the input dimension $d$, and $\beta^*$ is corresponding smoothness parameter.
%if $f_0$ has some kind of structure information, such that its intrinsic dimension $d^*$ is much less than the input dimension, then the convergence rate of its excess risk can be improved to be something similar to $n^{-\frac{2\beta}{2\beta+d^*}}$ if we use deep neural network with ReLU activation and sparse connection as non-parameteric estimator. 
%Further, by using advanced neural network approximation results, \cite{kohler2021rate} proved that the sparse connection requirement can be relaxed. More specifically, they showed that if $f_0$ satisfies an hierarchical composition model with smoothness and order parameters $(\beta^*, d^*)$, the convergence rate is $n^{-\frac{2\beta^*}{2\beta^*+d^*}}$ (up to some logarithmic factor).

Notably, most of the existing results on estimation error rates are established under the assumption that the noise variable $\varepsilon$, or equivalently, the response variable $Y$, is sub-Gaussian. Such an assumption would raise legitimate concerns when heavy-tailed data is observed. 
A natural question is how well would deep neural networks work in the context of nonparametric regression with heavy-tailed errors, or how critical this sub-Gaussian condition is so as to achieve a faster convergence rate via DNN.
To approach this question, we start with the Huber loss \citep{huber1964}, which robustifies the $L_2$-loss through a truncation parameter $\tau>0$.

\begin{definition}[Huber Loss]
Given some parameter $\tau \in (0, \infty]$, the Huber loss $\ell_{\tau}(\cdot)$ is defined as
\begin{align}
    \ell_\tau(x) = \begin{cases}
     \frac{1}{2} x^2 & \qquad |x| \le \tau \\
     \tau |x| - \frac{1}{2}\tau^2 & \qquad |x| > \tau
    \end{cases} .
\end{align} 
Note that the Huber loss is continuously differentiable with score function $\ell_\tau'(x) = \min \{ \max( -\tau, x ), \tau \}$. \revise{In particular, the Huber loss with $\tau=\infty$ coincides with the squared loss. }
\end{definition}

%The Huber loss is quadratic for small $x$ and is linear for large $x$. 
Given a robustification parameter $\tau =\tau_n>0$, consider the empirical Huber loss
\begin{align}
    \hat{\mathcal{R}}_\tau(f) = \frac{1}{n} \sum_{i=1}^n \ell_{\tau}(Y_i - f(X_i)) .
\end{align} 
%the subscript $n$ means that the sample size is $n$ and the robustification parameter also depends on $n$. 
%Denote by $\mathcal{F}_n$ a set of deep ReLU neural networks whose depth and width may also depend on the sample size $n$.
The corresponding nonparametric Huber estimator is defined as
\begin{align}
    \hat{f}_n  \in  \text{argmin}_{f\in \mathcal{F}_n(d, L, N, M)} \hat{\mathcal{R}}_\tau(f) ,  \label{huber-DNN-est}
\end{align} 
where $\mathcal{F}_n(d, L, N, M)$ denotes the space of truncated ReLU neural networks with width $N$ (number of neurons per hidden layer), depth $L$ (number of layers), input dimension $d$ and a truncation parameter $M$. A multilayer feedforward neural network with network architecture $(L, N)$ and the ReLU activation function can be written as
\begin{align}
    f(x) = \mathcal{L}_{L+1} \circ \sigma \circ \mathcal{L}_{L} \circ \sigma \circ \mathcal{L}_{L-1} \circ \sigma \circ \cdots \circ \mathcal{L}_2 \circ \sigma \circ \mathcal{L}_1 (x), \label{def:dnn}
\end{align} 
where $\mathcal{L}_i(x)=W_i x + b_i$ is a linear transformation with $W_i\in \mathbb{R}^{d_i\times d_{i-1}}$, $b_i\in \mathbb{R}^{d_i}$ and $(d_0, d_1, \cdots, d_L, d_{L+1})$ $= (d, N, \cdots, N, 1)$, and $\sigma:\mathbb{R}^{d_i}\to \mathbb{R}^{d_i}$ applies the ReLU function $\sigma(x)=\max\{0,x\}$ to each entry of an $\mathbb{R}^{d_i}$-valued vector. We refer to this type of networks as \emph{deep ReLU network with width $N$ and depth $L$}, and $\{(W_i, b_i)\}_{i=1}^{L+1}$ are the \emph{network weights} or parameters. 
%Let $\|f\|_\infty$ be the supremum norm of $f(\cdot)$ over the unit cube $[0,1]^d$.
Now we are ready to define the following two classes of network functions:
\begin{align*}
    \mathcal{F}_n(d, L, N) = \big\{&f : \RR^d \to \RR \text{ is of the form \eqref{def:dnn} with width $N$ and depth $L$}  \big\}
\end{align*} 
and 
\begin{align*}
    \mathcal{F}_n(d, L, N, M) = T_M  \mathcal{F}_n(d, L, N) =  \big\{f=T_M g : g \in \mathcal{F}_n(d, L, N) \big\} ,
\end{align*} 
where $T_M$ is the truncation operator at level $M>0$, defined as $T_M u  = \text{sgn}(u) (|u|\land M)$.

%%%%%%%%%%%%%%%%%%%%%%%%%%%%%%%%%%%%%%%%%%%%%%%%%%%%%%%%%%%%%%%%
%%%%%%%%%%%%%%%%%%%%  Statistical Analysis %%%%%%%%%%%%%%%%%%%%%
%%%%%%%%%%%%%%%%%%%%%%%%%%%%%%%%%%%%%%%%%%%%%%%%%%%%%%%%%%%%%%%%

\section{Statistical Analysis}
\label{sec3}
{
We first impose the following minimal assumptions on the regression model \eqref{np.model}.
\color{revisecolor}
\begin{condition}[Boundedness]
\label{cond1}
The random covariate vector $X\in \RR^d$ follows some distribution $\mathcal{P}_X$ over the unit cube $[0,1]^d$. The regression function $f_0: [0, 1]^d \to \RR$ is uniformly bounded, i.e., $\|f_0\|_\infty := \sup_{x\in [0, 1]^d } |f_0(x)| \le M$ for some $M \ge 1$. 
\end{condition}
}

\begin{condition}[Moment conditions]
\label{cond2}
The noise variable  $\varepsilon$ has zero mean and uniformly bounded (conditional) $p$-th moments for some $p\ge 1$, that is,
\begin{align} 
    \mathbb{E}[ \varepsilon|X=x  ] =0 ~\mbox{ and }~  \mathbb{E} [  | \varepsilon|^p|X=x ] \le v_p < \infty ~ \text{ for all } x\in [0,1]^d . \label{moment.cond}
\end{align}
\end{condition}

For any $\tau \in (0, \infty]$, define the population risk under the Huber loss
\begin{align}
\label{eq:rn-loss}
    \mathcal{R}_\tau(f) = \mathbb{E}_{X,Y} \big\{  \ell_{\tau} ( Y-f(X)  ) \big\}.
\end{align} 
Our goal is to derive the rate of convergence for $\hat f_n$ \eqref{huber-DNN-est} under the $|| \cdot ||_2= \| \cdot \|_{L_2(\mathcal{P}_X)}$-norm, defined as $\| f \|_2 = \sqrt{ \EE_{X \sim \mathcal{P}_X} | f(X) |^2 }$. 

For the Huber loss with $\tau\in(0, \infty)$,  let $\psi_\tau(x) = \ell_\tau'(x) =  \text{sgn}(x) (|x| \land \tau)$ be the corresponding score function, which is Lipschitz continuous and has a derivative almost everywhere, that is,
\begin{align} \label{def:score}
  \psi'_\tau(x) = 1\{|x| \le \tau\}.
\end{align}
The parameter $\tau$ plays an important role in robustness-bias tradeoff \citep{sun2020adaptive}, and depends on the scale of $\varepsilon$ (e.g., standard deviation).
If the distribution of $\varepsilon$ is symmetric around zero, Huberization will not introduce bias because the underlying regression function $f_0$ is also the population Huber risk minimizer for any $\tau>0$. In this case, $\tau$ will only depend on the noise scale. Under a bounded second moment condition, we will show that the Huber regression estimator based on deep ReLU networks achieves the same convergence rate as its least squares counterpart under sub-Gaussian noise.
In the case of asymmetric noise, the Huberization bias can no longer be disregarded, and needs to be balanced with the statistical error.
Therefore, the robustification parameter $\tau$ should adapt to the sample size $n$ in a suitable way so as to achieve a bias-robustness trade-off. The amount of bias that should be traded for robustness depends on the moment order $p$. As we shall see, the final convergence rate is slower than but infinitely close to the minimax rate as $p\to \infty$.

As a complement to the moment Condition \ref{cond2}, the following symmetry assumption is of independent interest, although it deemphasizes the impact of heavy-tailedness.

\begin{condition}[Symmetric noise]
\label{cond3}
For each $x\in [0,1]^d$, the conditional distribution of $\varepsilon|X=x$ is  symmetric around 0. %symmetric, i.e., the distribution of $\varepsilon|X=x$ is same with respect to $-\varepsilon|X=x$
\revise{Moreover,  there exists some constant $v_1>0$ such that}
\begin{align}
    \revise{\mathbb{E}  \big[  |\varepsilon| | X= x \big] \le v_1 ~~\text{for all}~~ x\in [0,1]^d.}
\end{align}
\end{condition}
\subsection{Lower bound on excess risk and upper bound on Huberization bias}

We first examine the population Huber loss and quantify the bias induced by Huberization. Denote by $\bTheta$ the set of all measurable functions $f: [0, 1]^d \to \mathbb{R}$ satisfying $\| f \|_\infty \leq M$ for the same $M\geq 1$ as in Condition~\ref{cond1}.  For any $r>0$, define the local set $\Theta_0(r) = \{f\in \Theta: \|f-f_0\|_2\le r\}$ and its complement $\Theta_0^{{\rm c}}(r)=\{f\in\Theta: \|f-f_0\|_2 > r\}$. For every $f \in \Theta$, we write $\Delta_f(x) = f_0(x) - f(x)$ such that the population Huber risk at $f$ can be written as 
%$\Theta=\{f: [0,1]^d \to \RR \,| \, \|f\|_\infty \le M, \|f\|_2 < \infty\}$ the global function space.  For any $r>0$, define $\Theta_0(r) = \{f\in \Theta: \|f-f_0\|_2\le r\}$ and $\Theta_0^{{\rm c}}(r)=\{f\in\Theta: \|f-f_0\|_2 > r\}$.  For each function $f$, write $\Delta_f(x) = f_0(x) - f(x)$, so that
\begin{align*}
    \mathcal{R}_\tau(f) = \mathbb{E}_{(X,Y)} \big[ \ell_\tau (Y-f(X)  )\big] =\mathbb{E}_{(X,Y)} \big[  \ell_\tau ( \varepsilon + \Delta_f(X) )\big] .
\end{align*}  

\begin{proposition}
\label{prop:strong-convex}
    \revise{Assume that $f_0 \in \Theta$ and let $\tau \ge c_1 = 2\max\{ 2M ,   (2v_p)^{1/p} \}$. Then, Condition \ref{cond2} ensures that }
    \begin{align}
    \label{eq:str-convex}
        \mathcal{R}_\tau(f) - \mathcal{R}_\tau(f_0) \ge \frac{1}{8} \|f-f_0\|_2^2 ~~\text{ for all }~ f\in \Theta_0^{{\rm c}}(8v_p \tau^{1-p}).
    \end{align} 
    \revise{Under Condition \ref{cond3}, we have}
    \begin{align}
    \label{eq:str-convex-bias}
        \mathcal{R}_\tau(f) - \mathcal{R}_\tau(f_0) \ge \frac{1}{4} \|f-f_0\|_2^2 ~~\text{ for all }~ f\in \Theta.
    \end{align}
\end{proposition}

Proposition~\ref{prop:strong-convex} provides lower bounds for the population excess risk over some subset of the function space.  Let $f_{0, \tau}$ be the global minimizer of the population Huber risk, i.e.,
\begin{align}  \label{eq3.8}
    f_{0, \tau} \in  \text{argmin}_{f\in \Theta} \mathcal{R}_{\tau}(f). 
\end{align} 
If the distribution of $\varepsilon$ is asymmetric,  $f_{0, \tau}$ generally differs from $f_0$ with $r_\tau =\|f_{0, \tau}-f_0\|_2 > 0$. By the optimality of $f_{0, \tau}$,  $\cR_\tau(f_{0, \tau}) - \cR_\tau(f_0) \leq 0$.  This explains why the lower bound \eqref{eq:str-convex} holds only outside a local neighborhood of $f_0$. 
On the other hand,  if the (conditional) distribution of $\varepsilon$ is symmetric,  it is easy to see that $f_{0, \tau} =f_0$.  The lower bound \eqref{eq:str-convex} can be viewed as a form of the restricted (outside the local neighborhood) strong convexity,  provided that the robustification parameter $\tau$ is sufficiently large.

%the population loss is strongly convex in the whole space if we allow $\tau$ to be larger than a constant $c_1$.

%It should be noted that we can obtain \emph{global strongly convexity} except a ball near $f_0$ in general. And this ball will become smaller and smaller if $\tau$ increases. The intuition behind it is as follows. 

%Now if $f_\tau^*= f_0$ doesn't $\nu$-a.e. hold, which is the usual case if the noise $\varepsilon$ is not symmetric, then let's consider what will happen on the ball $\Theta(r_*)$ for $r_*=\|f_\tau^*-f_0\|_2 > 0$. By optimality of $f_\tau^*$, we have the left hand side of \eqref{eq:str-convex} is less than or equal to $0$, while the right hand side of \eqref{eq:str-convex} is a number greater than 0. This means at least the strong convexity property doesn't hold in the ball $\Theta(r_*)$. 

The following proposition provides an upper bound for the Huberization bias $\| f_{0, \tau} -f_0\|_2$.

%Moreover, we also have the following result showing that in general, the bias, i.e., the $L_2$ difference between $f^*$ and $f^*_\tau$ can be upper-bounded by a function of $\tau$:

\begin{proposition}
\label{prop:bias}
Assume Condition \ref{cond2} holds,  and let $\tau \ge c_1= 2\max\{ 2M ,   (2v_p)^{1/p} \}$.  Then, the global minimizer  $f_{0, \tau}$ of the population Huber loss satisfies
 %and let $f_\tau^*$ be the population risk minimizer defined above, then we have
    \begin{align}
    \| f_{0, \tau}  - f_0\|_2 \le 4 v_p  \tau^{1-p} .
    \end{align}
In addition, assume there exists some constant $\sigma  >0$ such that
\begin{align}
	\PP\big( |\varepsilon| \geq   t | X=x\big) \leq 2 e^{-\frac{t^2}{2\sigma^2 }}~\mbox{ for all }~  t \geq 0 ~\mbox{ and }~ x \in [0, 1]^d .  \label{subgaussian.cond}
\end{align}
Then
\begin{align}
 \| f_{0, \tau}  - f_0\|_2 \le    2. 75 (\tau + \sigma^2/\tau)  e^{- \frac{\tau^2}{2 \sigma^2}} . \label{bias.exponential.decay}
\end{align}
\end{proposition}

From the above result we see that the Huberization bias, at least an upper bound of it, depends on both the tuning parameter $\tau$ and the moment index $p\geq 2$.  If  $\varepsilon$ only has bounded moments up to order $p$,  the bias decays at polynomial rates; if $\varepsilon$ is sub-Gaussian as assumed in \cite{kohler2021rate},  the bias decays exponentially fast as a function of $\tau$. In this case, \eqref{bias.exponential.decay} with $\tau = \sigma \sqrt{2 \log n}$ implies $
 \| f_{0, \tau}  - f_0\|_2 \lesssim \sigma n^{-1} (\log n)^{1/2}$. The robustification bias is thus negligible compared to the statistical error.  In the case of heavy-tailed noise,  say $p=2$, the bias will play a bigger role,  and may result in a slower convergence rate.

{
\color{revisecolor}
\subsection{A generic upper bound}

We first present a generic upper bound for a Huber ReLU-DNN estimator with arbitrary network architecture hyper-parameters $(\bar{L}$, $\bar{N})$ and robustification hyper-parameter $\tau \in [c_1, \infty]$. 
Here we use $\bar L$ and $\bar N$ to denote the depth and width, respectively.

\begin{theorem}[High probability bounds for Huber ReLU-DNN estimation]
\label{thm:generic-bound}
Assume Conditions~\ref{cond1} and \ref{cond2} hold with $p \ge 2$. Write $\mathcal{F}_n = \mathcal{F}_n(d, \bar L , \bar N, M)$ with  $\bar{N}, \bar{L}, n \in \{3,4,\ldots,\}$, and let $\tau \in [c_1,\infty]$.  For any $\omega, D \ge 1$, define $\delta_{n,\tau} = \delta_{\mathtt{b}} \lor \delta_{\mathtt{a}} \lor \delta_{\mathtt{s}}$, where
\begin{align}
\label{eq:generic-upper-bound-three-terms}
    \delta_{\mathtt{b}} = \frac{v_p}{\tau^{p-1}}, ~~~~~~~~ \delta_{\mathtt{a}} = \inf_{f\in \mathcal{F}_n } \|f - f_0\|_2, ~~~~~~~~ \delta_{\mathtt{s}} = \sqrt{V_n  \{v_2 + \min \{ \tau ,  \omega v_p^{1/p} V_n^{-1/p}  \} \}} 
\end{align} 
and $V_n = n^{-1} (\bar{N} \bar{L})^2 \log (\bar{N}\bar{L}) \log n$. Let $\mathcal{S}_{n,\tau}(\delta)$ be the set of approximate empirical risk minimizers with optimization error $\delta$, that is, $\mathcal{S}_{n,\tau}(\delta) =   \{f \in \mathcal{F}_n : \hat{\mathcal{R}}_\tau(f) \le \inf_{g\in \mathcal{F}_n} \hat{\mathcal{R}}_\tau(g) + \delta^2  \}$. Then, there exists some universal constant $c_2>0$ independent of $(\bar{N}, \bar{L}, n, \tau, \omega, D, v_p, v_2)$, $f_0$ and the distribution of $(X,\varepsilon)$ such that, for any $\delta_{\mathtt{opt}} > 0$,
\begin{align}
\label{eq:high-prob-generic-error-upper-bound}
    \mathbb{P}\Bigg\{  \sup_{f\in \mathcal{S}_{n,\tau}(\delta_{\mathtt{opt}})} &\|f - f_0\|_2 \ge c_2 (D\delta_{n,\tau} + \delta_{\mathtt{opt}})\Bigg\} \nn \\
    &\lesssim \begin{cases}
        e^{- nV_n D^2 / c_2} & \text{if}~\tau \le \omega D^2 (v_p/V_n)^{1/p} \\
        e^{- nV_n / c_2} + \omega^{1-p} D^{-2p} & \text{if}~\tau > \omega D^2 (v_p/V_n)^{1/p} 
    \end{cases}.
\end{align}
\end{theorem}

Theorem \ref{thm:generic-bound} provides a general high probability bound on the $L_2$ error of any Huber ReLU-DNN estimator with $\tau \in [c_1, \infty]$. The total estimation error, which depends explicitly on all the hyper-parameters, is composed of four terms: the optimization error $\delta_{\mathtt{opt}}$, the neural network approximation error $\delta_{\mathtt{a}}$ to the regression function $f_0$, the bias $\delta_{\mathtt{b}}$ induced by the Huber loss due to asymmetric noise tails, and the stochastic error $\delta_{\mathtt{s}}$. From this, one can derive a specific error bound by choosing a robustification parameter $\tau \in [c_1, \infty]$, neural network hyper-parameters $\bar{N}$, $\bar{L}$ and a function class that includes $f_0$. As we shall see, $\bar{N} \bar{L}$ determines the complexity of the network class, so that a larger $\bar{N} \bar{L}$ corresponds to a smaller approximation error $\delta_{\mathtt{a}}$ but an increased stochastic error $\delta_{\mathtt{s}}$. On the other hand, the magnitude of $\tau$ controls the degree of robustness of the estimator against heavy-tailed errors. A smaller value of $\tau$ helps improve robustness, resulting in a smaller stochastic error $\delta_{\mathtt{s}}$, at the cost of a larger bias $\delta_{\mathtt{b}}$ in the presence of asymmetric errors.

%Table~\ref{table:convergence-rate-upper} lists a series of upper bounds that will be derived from Theorem~\ref{thm:generic-bound}. Logarithmic factors are omitted for a clearer presentation. 

\begin{remark}[High probability bound for least squares ReLU-DNN estimation]
\label{coro:generic-bound-lse}
Taking $\tau = \infty$ and $\omega=1$ in Theorem~\ref{thm:generic-bound} immediately yields a high probability bound for the least squares ReLU-DNN estimator under heavy-tailed errors, which is of independent interest. 
Write $\mathcal{F}_n = \mathcal{F}_n(d, \bar L , \bar N, M)$ with  $\bar{N}, \bar{L}, n \in \{3,4,\ldots,\}$, and define 
\begin{align*}
    \delta_n = \inf_{f\in \mathcal{F}_n } \|f - f_0\|_2 + v_p^{\frac{1}{2p}} V_n^{\frac{1}{2}(1-1/p)} + \sqrt{v_2 V_n} ~~ \text{with} ~ V_n = n^{-1} (\bar{N} \bar{L})^2 \log (\bar{N}\bar{L}) \log n.
\end{align*} 
With the same universal constant as in Theorem \ref{thm:generic-bound}, we have for any $D \geq 1$ that
\begin{align*}
    \mathbb{P}\Bigg\{  \sup_{f\in \mathcal{S}_{n,\infty}(\delta_{\mathtt{opt}})} \|f - f_0\|_2 \ge c_2 (D\delta_n + \delta_{\mathtt{opt}})\Bigg\}  \lesssim \exp\big\{  - n^{1/p} (\bar{N}\bar{L})^{2(1-1/p)} / c_2 \big\} + \frac{1}{D^{2p}}
\end{align*}
This result improves Theorem~5.1 of \cite{kuchibhotla2019least} in the case of $\alpha = \beta=0$.
\end{remark}

In the proof of Theorem~\ref{thm:generic-bound}, the stochastic error term $\delta_{\text{s}}$ is shown to be of the form
\begin{align}
\label{eq:stochastic-error-term}
    \delta_{\text{s}} = \sqrt{\frac{\mathrm{Pdim}(\mathcal{F}_n) \log n}{n}}  \underbrace{  \min \Bigg\{\sqrt{\tau}, \,  \bigg(\frac{n}{\mathrm{Pdim}(\mathcal{F}_n) \log n}\bigg)^{1/(2p)} \Bigg\} }_{=: \, \lambda(\tau, n, \mathcal{F}_n ) }  ,
\end{align} 
where $\mathrm{Pdim}(\mathcal{F}_n)$ denotes the pseudo-dimension of the ReLU-DNN class $\mathcal{F}_n$, satisfying  $\mathrm{Pdim}(\mathcal{F}_n) \asymp (\bar{N} \bar{L})^2 \log (\bar{N}\bar{L})$ \citep{BHLM2019}; see Definition~\aosversion{\blue{A.2}}{\ref{def-pdim}} for a precise definition of pseudo-dimension. In contrast to the sub-Gaussian error case, $\lambda(\tau, n, \mathcal{F}_n)$ quantifies the joint impact of the Huber loss and the lower-order moments of $\varepsilon$. In the heavy-tailed case that $\EE[|\varepsilon|^p | X]$ is uniformly bounded, the Huber and least squares ReLU-DNN estimators achieve the same convergence rate provided that $\tau \gtrsim \big[ n/\{ \mathrm{Pdim}(\mathcal{F}_n) \log n \} \big]^{1/p}$. 

Compared to previous works, one of the key contributions of our analysis is to provide a tight upper bound on the stochastic error term $ \delta_{\mathtt{s}}$ in \eqref{eq:generic-upper-bound-three-terms} for the Huber ReLU-DNN estimator when the noise is heavy-tailed. The tightness of the above stochastic error term will be affirmed by a matching lower bound in Proposition \ref{prop:lowerbound-variance}. It also strengthens a result concerning the least squares estimator in \cite{kuchibhotla2019least} under heavy-tailed errors; see Theorem~5.1 therein when $\alpha=\beta=0$. We refer to Section~\aosversion{\blue{D.1}}{\ref{subsec:upper-bound-discussion}} for a summary of upper bounds results we derived and a detailed comparison of our upper bound results with those from \cite{farrell2021deep} and \cite{shen2021robust}.

Although stated specifically for ReLU-DNN estimation, Theorem \ref{thm:generic-bound} and Corollary \ref{coro:generic-bound-lse} also apply to other nonparametric regression estimators (e.g., spline-based methods) as long as the pseudo-dimension of the function class is bounded.

\begin{remark}
We adopt a neural network assumption similar to those considered in \cite{farrell2021deep}, \cite{kohler2021rate}, and \cite{shen2021robust}, which does not require the network weights to be uniformly bounded. Such a relaxation of the uniform boundedness constraint on network weights not only facilitates practical implementation but also strengthens neural network approximation power to some extent. 
Also, it is worth noticing that our lower bound analysis remains valid even when the network weights are uniformly bounded as long as $\bar L \lesssim n^\epsilon$ for any $\epsilon>0$. See Remark~\ref{remark:bounded-approximation} for a detailed discussion.

\end{remark}

\begin{remark}
Given the high probability bound \eqref{eq:high-prob-generic-error-upper-bound} on the population $L_2$ error, one can further establish a similar bound on the empirical $L_2$ error $\|f - f_0\|_n$, defined as $\|f - f_0\|_n^2 =  (1/n) \sum_{i=1}^n \{f(X_i) - f_0(X_i)\}^2$, using a uniform concentration property of the empirical $L_2$ risk around the population $L_2$ risk. More specifically, Lemma 3 in \cite{fan2022factor}  implies that the event
\begin{align*}
   \mathcal{E}_t = \bigg\{  \big| \|f - f_0\|_n^2 - \|f - f_0\|_2^2 \big| \leq \frac{1}{2} \|f - f_0\|_2^2 + C\bigg(V_n + \frac{t}{n}\bigg) , ~~   \forall f\in \mathcal{F}_n(d, \bar{L}, \bar{N}, M)   \bigg\}
\end{align*} occurs with probability at least $1-e^{-t/C}$ for any $t\ge 1$. Therefore, further conditioned on $\mathcal{E}_{nV_nD^2}$, we obtain
\begin{align*}
  \forall f \in \mathcal{S}_{n,\tau}(D\delta_{n,\tau}), ~~~\|f - f_0\|_2 \lor \|f - f_0\|_n \lesssim D\sqrt{V_n} + D\delta_{n,\tau} \lesssim D\delta_{n,\tau}.
\end{align*}
\end{remark}
}

\subsection{Convergence analysis under general heavy-tailed noise}
\label{subsec:roc}

We first present the general convergence results in the absence of symmetry -- Condition \ref{cond3}.  The following neural network approximation result provides the key to establishing the \revise{$L_2$ error rate}.
{
\color{revisecolor}
\begin{proposition}[An upper bound on neural network approximation error for $\mathcal{H}(d,l,\mathcal{P})$]
\label{prop:nn-approx}
There exist universal constants $c_3$--$c_5$ that depend only on $l, \sup_{(\beta,t)\in \mathcal{P}} \max\{ \beta, t\}$ and $C$ from Definition \ref{hcm} such that for any $L, N\geq 3$, 
\begin{align*}
    \sup_{f_0 \in \mathcal{H}(d,l,\mathcal{P})} \inf_{ f_n \in \mathcal{F}_n(d, \bar L, \bar N,M)} \| f_n  -f_0\|_\infty \le c_5  ( N L )^{-2\gamma^*} ,
\end{align*}
where $\bar L = c_3 \lceil L\log L\rceil$ and $\bar N = c_4 \lceil N\log N\rceil$.
\end{proposition}
}
Compared to the approximation results in \cite{kohler2021rate},  the above result applies to a broader range of neural network architectures. To be specific, to obtain similar approximation errors, \cite{kohler2021rate} suggested using neural networks with a ``special shape'', which is either thin and deep  or wide and shallow. In contrast, Proposition \ref{prop:nn-approx} allows one to tune the width and depth of a neural network more flexibly,  thus leading to weakened conditions and new insights.

{
\color{revisecolor}
Based on the above neural network approximation result, the following theorem establishes the convergence rate for the adaptive Huber ReLU-DNN estimator after a delicate trade-off among the bias $\delta_{\text{b}}$, approximation error $\delta_{\text{a}}$ and stochastic error $\delta_{\text{s}}$.

\begin{theorem}[Optimal rate for adaptive Huber estimator under $\mathcal{H}(d,l,\mathcal{P})$ and asymmetric noise]
\label{thm:rate-adaptive-huber-upper}
    Assume Conditions \ref{cond1} holds, $v_2 \asymp 1$ and $v_p^{\frac{1}{(2p-1)p}} \lesssim (\log n)^{C'}$ for some constant $C'>0$. Let $\beta^*$, $d^*$ and $\gamma^*$ be as in \eqref{eq:def-gamma-star}, and $L, N \geq 3$ be such that 
\begin{align}
\label{eq:depth.width.order.n}
        L N \asymp  \Bigg(\frac{n}{\log^6 n}\Bigg)^{\frac{\nu^*}{2(2\gamma^* + \nu^*)}} ~~\mbox{ with }~~\nu^*=1-\frac{1}{2p-1}.% \in [2/3 ,  1) .
\end{align} 
Consider the neural network class $\mathcal{F}_n=\mathcal{F}_n(d,  \bar{L}, \bar{N}, M)$ with depth and width
    \begin{align}
      \bar{L} = c_3 \lceil L \log L \rceil ~~\mbox{ and }~~  \bar{N} = c_4 \lceil N \log N \rceil,  \label{depth.width.order}
    \end{align} 
    where $c_3$--$c_4$ are the positive constants from Proposition~\ref{prop:nn-approx}. Moreover, let $\delta_{n,\mathtt{AH}}$ and $\tau_n$ be such that 
    \begin{align}
        \delta_{n,\mathtt{AH}} \asymp \Bigg(\frac{\log^6 n}{n}\Bigg)^{\frac{\gamma^* \nu^*}{2\gamma^* + \nu^*}} v_p^{\frac{1}{2p-1}} ~~~~~~\text{and}~~~~~~ \tau_n \asymp  \Bigg(\frac{n}{\log^6 n}\Bigg)^{\frac{2\gamma^*(1-\nu^*)}{2\gamma^* + \nu^*}} v_p^{\frac{2}{2p-1}} \label{tau.delta.order}.
    \end{align} 
    Provided that $n$ is sufficiently large, we have for any $D\geq 1$ that
    \begin{align*}
        \sup_{\substack{f_0\in \mathcal{H}(d,l,\mathcal{P}), \\ \mathbb{E}[\varepsilon|X]=0, \mathbb{E}[|\varepsilon|^p|X] \le v_p}} \mathbb{P} \Bigg\{  \sup_{f\in \mathcal{S}_{n,\tau_n}(\delta_{n,\mathtt{AH}})} \|f-f_0\|_2 &  \ge c_6 D \delta_{n,\mathtt{AH}} \Bigg\}  \\
        &   \lesssim \exp\Big\{ - D^2 n^{\frac{\nu^*}{2\gamma^*+\nu^*}} (\log n)^{\frac{8\gamma^*-2\nu^*}{2\gamma^*+\nu^*}} / c_6\Big\}  ,
    \end{align*}
    where $c_6>0$ is a universal constant independent of $(n, D, p, v_p)$.
    %Then, there exist positive constants $c_6$ and $c_7$ independent of $(n, t)$ such that for any $t>c_7$, any approximate empirical risk minimizer $\hat{f}_n$ with optimization error of order $\delta_n^2$, more specifically, 
    %\begin{align*}
    %	\mathbb{P} \Big\{\hat{\mathcal{R}}_\tau(\hat{f}_n) \ge \inf_{f\in \mathcal{F}_n}\hat{\mathcal{R}}_\tau(f) + 2^{-8}\delta_n^2 t^2\Big\}  \lesssim \exp\Big\{ -c_6 n^{\frac{\nu^*}{2\gamma^*+\nu^*}}(\log n)^{6} t^2\Big\}
    %\end{align*} with
    %\begin{align}
    %    \tau  \asymp  \big\{  n^{\frac{\gamma^*}{2\gamma^*+\nu^*}} (\log n)^{-3} \big\}^{2(1-\nu^*) }~~\mbox{ and }~~ \delta_n \asymp   \big\{  n^{- \frac{\gamma^*}{2\gamma^*+\nu^*}} (\log n)^{3}  \big\}^{\nu^*}  \label{tau.delta.order} ,
    %\end{align} 
  %has statistical error
    %\begin{align}
    %    \mathbb{P}\Big(\|\hat{f}_n - f_0\| \ge t \delta_n\Big) \lesssim \exp\Big\{ -c_6 n^{\frac{\nu^*}{2\gamma^*+\nu^*}}(\log n)^{6} t^2\Big\}.
    %\end{align}
\end{theorem}
}

If the noise variable $\varepsilon$ has a bounded (conditional) $p$-th ($p\geq 2$) moment, Theorem~\ref{thm:rate-adaptive-huber-upper} shows that the adaptive Huber ReLU-DNN estimator  with a suitably chosen robustification parameter admits a convergence rate
\begin{align}
\label{eq:asymmetric-convergence-rate-adaptive-huber}
    \|\hat{f}_n - f_0\|_2 = O_\mathbb{P}\Big\{ n^{-\frac{ \beta^* \nu^*}{2\beta^*+d^*\nu^*}} (\log n)^{\frac{6\gamma^*\nu^*}{2\gamma^*+\nu^*} } \Big\}, ~~\mbox{ where }~ \nu^* =  \frac{2 p-2}{2p-1} \in [ 2/3 , 1).
\end{align}  
Compared to the least squares ReLU-DNN estimator that achieves a convergence rate of $ n^{-\frac{ \beta^* }{2\beta^*+d^* }}$ (up to some logarithmic factor) when $\varepsilon$ is sub-Gaussian, there is a statistical price to be paid by allowing for heavy-tailed errors that only have bounded moments of low order.

\begin{remark}
When the noise variable $\varepsilon$ satisfies the sub-Gaussian tail assumption \eqref{subgaussian.cond},  by Proposition~\eqref{prop:bias} we may choose $\tau = \tau_n \asymp \sigma \sqrt{\log n}$, which is much smaller than that in \eqref{tau.delta.order}, so that the Huberization bias is negligible.  Following the proof of Theorem~\ref{thm:generic-bound}, it can be shown that  the resulting nonparametric Huber estimator satisfies
\begin{align*}
	\| \hat f_n - f_0 \|_2 = \cO_{\PP}\Big\{ n^{-\frac{\beta^*}{2\beta^*+d^*}} (\log n)^{\frac{6\gamma^*}{2\gamma^*+1}}\Big\}.
\end{align*}
\end{remark}
\begin{remark}
	If $p=p_n$ grows with $n$ and satisfies $p \asymp  \log n$, it follows that $\nu^* = 1 - (2p-1)^{-1} = 1 - \cO((\log n)^{-1})$. Then, the convergence rate in Theorem~\ref{thm:rate-adaptive-huber-upper} coincides with that under the sub-Gaussian tail assumption by noting that
	\begin{align*}
			n^{-\frac{\beta^* \nu^*}{2\beta^*+d^* \nu^*}} 
			\lesssim %n^{-\frac{\beta^* \nu^*}{2\beta^*+d^*}} =
			n^{-\frac{\beta^*}{2\beta^*+d^*}\big(1-\frac{1}{2p-1}\big)} \lesssim n^{-\frac{\beta^*}{2\beta^*+d^*}} n^{\frac{\beta^*}{2\beta^*+d^*}\frac{1}{\log n}} \lesssim n^{-\frac{\beta^*}{2\beta^*+d^*}} ,
	\end{align*} 
	where the last step follows from the fact that $n^{\alpha/\log n} = e^{\alpha\log n/\log n} = e^{\alpha}$ for any constant $\alpha$.
\end{remark}

\begin{remark}
   Adaptive Huber ReLU-DNN regression is easy-to-implement using the \texttt{Python} library \texttt{TensorFlow}. Specifically, we can use the function \texttt{tf.keras.losses.Huber} instead of mean-square loss \texttt{tf.keras.losses.MeanSquaredError}.
\end{remark}

{
\color{revisecolor}
Combining Proposition \ref{prop:nn-approx} and Corollary \ref{coro:generic-bound-lse}, we further obtain the following optimal error rate for the least squares ReLU-DNN estimator.
\begin{theorem}[Optimal rate for least squares estimator under $\mathcal{H}(d,l,\mathcal{P})$]
\label{thm:rate-lse-upper}
    Assume Condition \ref{cond1} holds and $v_2 \asymp 1$. Let $\beta^*, d^*$ and $\gamma^*$ be as in \eqref{eq:def-gamma-star}, and $N, L \ge 3$ satisfying $(NL) \asymp \{n/  \log^6(n) \}^{\nu^\dagger/(4\gamma^*+2\nu^\dagger)}$ for $\nu^\dagger=1-1/p$. Consider the approximate least squares estimates with optimization error bounded by $\delta_{n,\mathtt{LS}}^2$, where $\delta_{n,\mathtt{LS}} \asymp \{ \log^6(n)/n\}^{ \gamma^*\nu^\dagger/(2\gamma^*+\nu^\dagger)} v_p^{1/(2p)}$ and the depth and width of the network class $\mathcal{F}_n = \mathcal{F}_n(d, \bar{L}, \bar{N}, M)$ satisfy \eqref{depth.width.order}. Then, for all sufficiently large $n$, the following bound
    \begin{align*}
        \sup_{\substack{f_0\in \mathcal{H}(d,l,\mathcal{P}) \\ \mathbb{E}[\varepsilon|X]=0, \mathbb{E}[|\varepsilon|^p|X] \le v_p}} \mathbb{P} \Bigg\{ \sup_{f\in \mathcal{S}_{n,\infty}(\delta_{n,\mathtt{LS}})}\|f - f_0\|_2 \ge c_7 D \delta_{n, \mathtt{LS}} \Bigg\} \le c_7 D^{-2p}
    \end{align*}
    holds for any $D \geq 1$, where $c_7$ is a universal constant independent of $(n, D, p, v_p)$.
\end{theorem}

From Theorems \ref{thm:rate-adaptive-huber-upper} and \ref{thm:rate-lse-upper} we see that the convergence rate $\delta_{n,\mathtt{LS}}$ for the least squares ReLU-DNN estimator is slower than that for the adaptive Huber estimator since $\nu^\dagger < \nu^*$. In Section \ref{sec:lb} we will show that the two upper bounds, $\delta_{n,\mathtt{LS}}$ and $\delta_{n,\mathtt{AH}}$, are both \emph{sharp} up to logarithmic factors.

It is also worth noting that the convergence rate of the least squares estimator (under the bounded $p$-th moment condition) is regardless of the symmetry/asymmetry of $\varepsilon$ provided that the conditional mean of $\varepsilon$ is zero.
This indicates that, unlike adaptive Huber regression, the least squares estimator does not benefit from the blessing of symmetry in the presence of heavy-tailed errors. We will provide a detailed discussion in Remark \ref{remark:stochastic-error-symmetric}.}

%\r{Q: May not right comment for LS-estimation.  The class still has the Huberization bias.}

\subsection{Faster rate under symmetric noise}

In the robust regression literature,  the case of symmetric noise is often of independent interest \citep{hampel1986robust, ronchetti2009robust}.
The following result shows that the Huber estimator will benefit from the blessing of symmetry although the tails are still heavy: with a robustification parameter of constant level,  it achieves the same rate of convergence as its least squares counterpart when the noise is sub-Gaussian.

%Different from the general case where the noise $\varepsilon$ doesn't need to be symmetric, our next result shows that if the noise $\varepsilon$ is symmetric, we can use a fixed robustification parameter $\tau_n$ and the convergence rate can be improved.
{
\color{revisecolor}
\begin{theorem}[Optimal rate under $\mathcal{H}(d,l,\mathcal{P})$ and symmetric noise]
\label{thm:convergence-rate-nn-symmetric}
    Assume Conditions \ref{cond1} and \ref{cond3} hold. Consider the function class $\mathcal{F}_n=\mathcal{F}_n(d, \bar{L}, \bar{N}, M)$, where the depth $\bar{L}$ and width $\bar{N}$ satisfies \eqref{depth.width.order}
    with $L, N \geq 3$ satisfying $LN \asymp \{n/ \log^6(n)\}^{1/(4\gamma^*+2)}$. Furthermore, let $c_1 \le \tau \lesssim 1$ and $\delta_{n,\mathtt{H}} \asymp  \{(\log n)^6/n\}^{\frac{\gamma^*}{2\gamma^*+1}}$. Then, for all sufficiently enough $n$ and arbitrary $D\ge 1$, it holds
    \begin{align*}
        \sup_{f_0\in \mathcal{H}(l,d,\mathcal{P})} \mathbb{P}\Bigg\{ \sup_{f\in \mathcal{S}(\delta_{n,\mathtt{H}})} \|f - f_0\|_2 \ge c_8 D \delta_{n,\mathtt{H}} \Bigg\}  \lesssim \exp\Big\{ - n^{\frac{1}{2\gamma^*+1}}(\log n)^{\frac{12\gamma^*}{2\gamma^*+1}} D^2 / c_8 \Big\} ,
    \end{align*} where $c_8$ is a universal constant independent of $n, D$.
\end{theorem}
}

\begin{remark}
    It should be noted that the adaptive Huber estimator is not the only estimator that achieves the $n^{-\frac{ \beta^* \nu^*}{2\beta^*+d^*\nu^*}}$ rate of convergence under asymmetric and heavy-tailed noise. 
    For example, one may also use the robust loss considered in \cite{Catoni2012} or a pseudo-Huber loss that is twice continuously differentiable everywhere. The theoretical analysis of these estimators will follow the same argument.
    On the other hand, a simpler robustification strategy is to 
    apply adaptive truncation on $Y$ \citep{fan2021shrinkage}, resulting in the empirical risk $\hat{\mathcal{R}}_{T,\tau}(f) =  n^{-1} \sum_{i=1}^n \{T_\tau(Y_i) - f(X_i)\}^2$. This corresponds to the least squares estimator with truncated response responses and facilitates neural network training.
    It is possible to obtain a result that is comparable to Theorem \ref{thm:generic-bound} using a similar argument. 
    The first key step is to show a lower bound of the population excess risk as in Proposition \ref{prop:strong-convex}, i.e.,
    \begin{align*}
        \mathcal{R}_{T,\tau}(f) - \mathcal{R}_{T,\tau}(f_0) \gtrsim \|f - f_0\|_2^2 \qquad \text{ as long as } \qquad \|f - f_0\|_2 \gtrsim \frac{1}{\tau^{p-1}},
    \end{align*} 
    where $\mathcal{R}_{T,\tau}(f) = \mathbb{E}_{(X,Y)} \{T_\tau(Y) - f(X)\}^2$ is the population risk. The proof of this is similar to that for the Huber risk by using the fact that $ | \mathbb{E} \{ T_\tau(f_0(X) + \varepsilon) | X \} - f_0(X)  | \lesssim \tau^{1-p}$ under Condition \ref{cond2}. The second key step is to show that the variance term grows linearly with $\tau$ as we see from Lemma~\aosversion{\blue{A.5}}{\ref{lemma:empirical-process}} for the Huber estimator. Combining these with a modified version of Lemma~\aosversion{\blue{A.1}}{\ref{lemma:convergence-rate}} yields an upper bound on the convergence rate. However, when the noise is symmetric, Theorem \ref{thm:convergence-rate-nn-symmetric} implies that the least squares estimator with truncated response cannot achieve the rate $n^{-\frac{ \beta^* }{2\beta^*+d^*}}$. The main reason is that the (population) excess risk lower bound \eqref{eq:str-convex-bias} may not hold because $\mathbb{E} \{ T_\tau(f_0(X) + \varepsilon) | X\}$ is not necessarily equal to $f_0(X)$. 
\end{remark}

%\r{Q: move this remark to the appendix if it is too long}
%\subsection{Summary of upper bounds on convergence rates}

\section{Lower Bound under Heavy-tailed Noise}
\label{sec:lb}
This section provides lower bounds for the Huber ReLU-DNN regression estimator for any given network structure under a bounded (conditional) $p$-th moment condition. In Section \ref{subsec:main-result}, we present the setting and a generic lower bound on the $L_2$ error of Huber ReLU-DNN estimators (Theorem \ref{thm:generic-lower-bound}). We also present a generic lower bound on the $L_2$ error of the least squares ReLU-DNN estimator (Theorem \ref{thm:generic-lower-bound-lse}) as a special case of Theorem \ref{thm:generic-lower-bound} when $\tau=\infty$. Then we apply the above results to obtain lower bounds on the $L_2$ errors for both the adaptive Huber and least squares ReLU-DNN estimators when the regression function lies in (1) hierarchical composition model $\mathcal{H}(d,l,\mathcal{P})$ and (2) $d$-variate $(\beta, C)$-smooth function class $\mathcal{C}(d,\beta)$. Subsequently, we develop the tools, including an approximation theory of the ReLU-DNN network, to prove the result.  In Section~\ref{subsec:lowerbound}, we first provide some intuitive explanations of the three terms in the lower bound, namely, the bias induced by Huber loss, the approximation error of ReLU-DNN to potential function class, and the statistical error. Then we outline the key ideas behind the proof of Theorem \ref{thm:generic-lower-bound}.   In Section~\ref{sec:lb:revisit}, we furnish a new ReLU-DNN approximation result that is related to ReLU-DNN's non-robustness nature. This result is the key to analyzing the statistical error term in Section \ref{subsec:lowerbound}. We provide insights into why the ReLU-DNN approximation result leads to a slower error rate under heavy-tailed errors. The above discussion focuses only on the best convergence rate for Huber-type ReLU-DNN estimators under heavy-tailed errors. A natural question is what  the best possible rate a ReLU-DNN estimator can achieve (under the same scenario) is. We summarize our results and give some preliminary answers to this question in Section \ref{sec:nn-meet-htr}. The relationships among all the results in this section are depicted in Figure \ref{fig:lowerbound-roadmap}.

\begin{figure}[!t]
\centering
\includegraphics[scale=0.3]{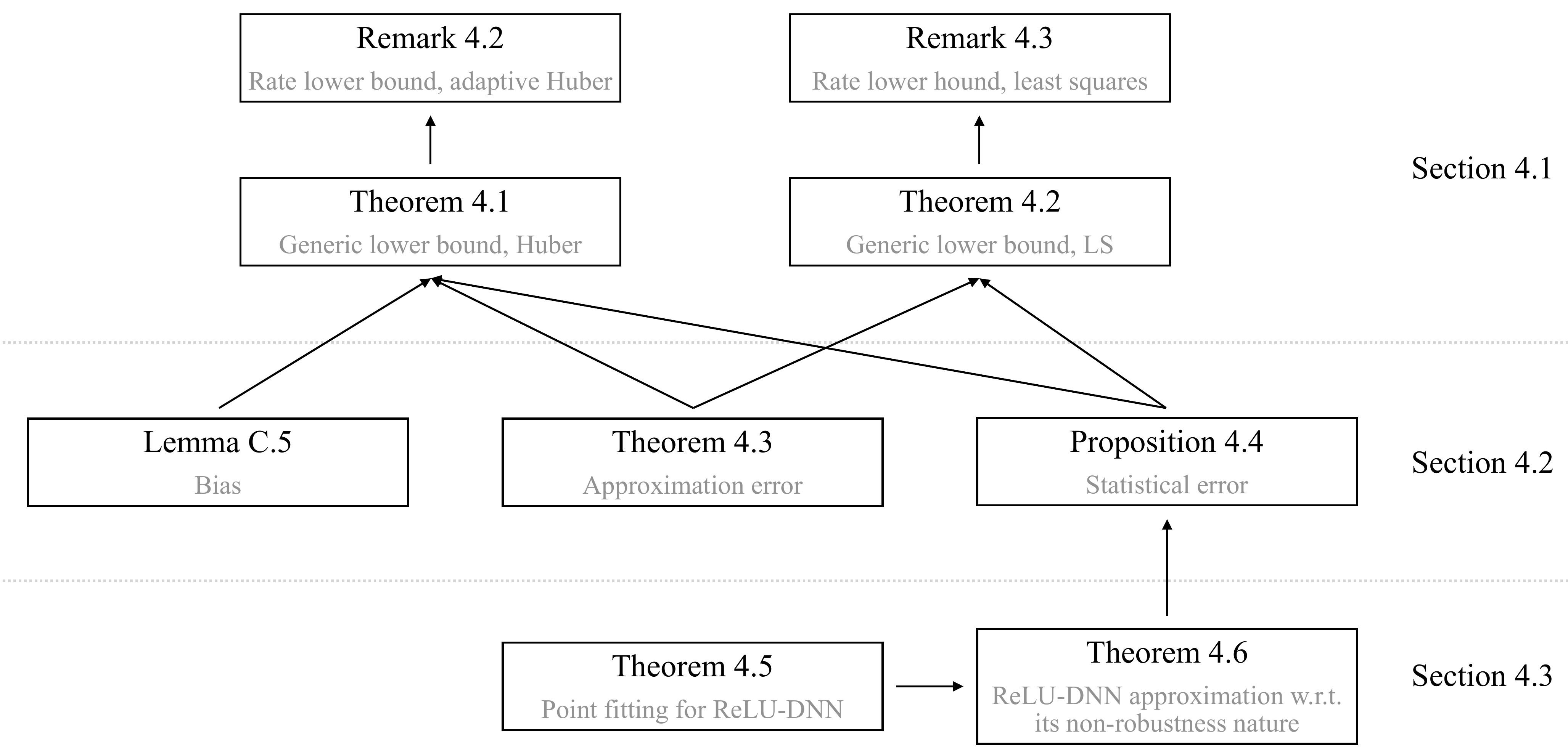}
\caption{The connections of all the results presented in Section \ref{sec:lb}. The arrow from claim A to claim B means we will use claim A to prove claim B.  The main result is presented in Theorem 4.1, and subsequently, we develop the tools, including an approximation theory of ReLU-DNN, to prove it.}
\label{fig:lowerbound-roadmap}
\end{figure}

\subsection{Main results on lower bound}
\label{subsec:main-result}

\revise{Recall that Theorem \ref{thm:generic-bound} and Remark \ref{coro:generic-bound-lse} establish upper bounds on the $L_2$ error of the Huber ReLU-DNN estimator and least squares ReLU-DNN estimator, respectively, under a bounded (conditional) $p$-th moment condition.} A natural question arises: are these upper bounds tight under the assumed moment condition? To answer this question, in this section we provide several lower bounds for both (adaptive) Huber and least squares ReLU-DNN estimators under the same moment conditions. 

Without loss of generality, in this section we assume $\| f_0 \|_\infty \leq 1$, $X$ follows a uniform distribution on $[0,1]^d$ and $ \mathbb{E} [|\varepsilon|^p|X=x] \leq 1$ (almost surely). \revise{Moreover, we assume that the regression function $f_0$ belongs to a class $\mathcal{F}$ with intrinsic dimension-adjusted smoothness bounded by $\alpha$ in the following sense.}

{
\color{revisecolor}
\begin{definition}[Intrinsic dimension-adjusted smoothness] \label{def:4.1}
We say a class $\mathcal{F}$ of functions $\{f: \mathbb{R}^d\to \mathbb{R}\}$ has an {intrinsic dimension-adjusted smoothness upper bounded by $\alpha$} if the minimax $L_2$ risk over this function class is lower bounded by $n^{-\frac{2\alpha}{2\alpha+1}}$ up to a constant, that is,
\begin{align*}
    \liminf_{n\to\infty} \inf_{\hat{f}_n} \sup_{f_0 \in \mathcal{F}} n^{\frac{2\alpha}{2\alpha+1}} \mathbb{E} [  \|\hat{f}_n - f_0\|_2^2 ]  > 0,
\end{align*} where the infimum is taken over all estimators constructed from the i.i.d. sample $\{ (X_i, Y_i)  \}_{i=1}^n$ satisfying $Y_i = f_0(X_i) + \varepsilon_i$ with $\varepsilon_i \sim \mathcal{N}(0,1)$ and $X_i\sim \mathrm{Uniform}([0,1]^d)$.
\end{definition}

For example, the $d$-variate $(\beta, 1)$-smooth function class $\mathcal{C}(d,\beta)$, formally defined as
\begin{align}
\label{eq:smooth-function-class}
		\mathcal{C}(d,\beta) = \big\{f: [0,1]^d\to \mathbb{R}  :   f ~\text{is}~(\beta, 1)~\text{smooth} \big\},
\end{align}
is a function class with intrinsic dimension-adjusted smoothness upper bounded by $\beta/d$. For the hierarchical composition model $\mathcal{H}(d,l,\mathcal{P})$ specified in Definition \ref{hcm}, the following lemma shows that its intrinsic dimension-adjusted smoothness is upper bounded by $\gamma^*$.
\begin{lemma}
\label{lemma:minimax-hcm}
    Assume that $d\ge d^*$, we have \begin{align*}
    \liminf_{n\to\infty} \inf_{\hat{f}_n} \sup_{f_0 \in \mathcal{H}(d,l,\mathcal{P})} n^{\frac{2\gamma^*}{2\gamma^*+1}} \mathbb{E}[\|\hat{f}_n - f_0\|_2^2]   > 0 .
\end{align*}
\end{lemma}

Under the above settings, we first present a generic lower bound on the $L_2$ error of the class of Huber ReLU-DNN estimators.

\begin{theorem}[A generic lower bound on the $L_2$ error]
\label{thm:generic-lower-bound}
	Let $\mathcal{F}_0 \subseteq \{f: \mathbb{R}^d\to [-1, 1]\}$ be a function class with intrinsic dimension-adjusted smoothness upper bounded by some $\alpha >0$. Suppose $0 \in \mathcal{F}_0$. Define the family of data generating processes as
	\begin{align}
	\label{eq:lb-dist-family}
	    \mathcal{U}(d, p, \mathcal{F}) = \big\{ (X,f_0, \varepsilon): X\sim \mathrm{Uniform}([0,1]^d), f_0 \in \mathcal{F},  \mathbb{E}[\varepsilon|X]=0, \mathbb{E}[ |\varepsilon|^p | X] \le 1 \big\}.
	\end{align} 
	Let $\mathcal{S}^{\mathtt{HN}}_{n,\tau}(\delta)$ be the set of all approximate Huber ReLU-DNN estimates with given robustification parameter $\tau$, depth $\bar{L}$ and width $\bar{N}$, i.e.,  
	\begin{align}
	\label{eq:thm:approx-adaptive-huber}
	\begin{split}
		\mathcal{S}^{\mathtt{HN}}_{n,\tau}(\delta) = \Bigg\{\tilde{f}_n \in \mathcal{F}_n(d, \bar{L}, \bar{N}, 1): ~& \hat{\mathcal{R}}_\tau(\tilde{f}_n) \le \min \Big \{\hat{\mathcal{R}}_\tau(f_{0,\tau}), \inf_{f\in \mathcal{F}_n(d, \bar{L}, \bar{N}, 1)} \hat{\mathcal{R}}_\tau(f) + c_{9}\delta^2\Big\} \\
		& \text{or}~~ \hat{\mathcal{R}}_\tau(\tilde{f}_n) \le \inf_{f\in \mathcal{F}_n(d, \bar{L}, \bar{N}, 1)} \hat{\mathcal{R}}_\tau(f) + n^{-100} \Bigg\}
	\end{split}
	\end{align} for some universal constant $c_{9}>0$ (independent of $\bar{N}$, $\bar{L}$, $n$ and $\tau$).
	Then, there exists universal positive constants $c_{10}$--$c_{12}$ independent of $n, \bar{L}, \bar{N}$, $\tau$ such that the following statements hold.
	\begin{itemize}
	    \item[(1)] For any $n \ge 3$, $\bar{N}, \bar{L} \ge c_{11}$  and $\tau \ge c_{12}$,
	\begin{align*}
		\sup_{(X,f_0,\varepsilon) \in \mathcal{U}(d,p,\mathcal{F}_0)}\mathbb{P}\big\{ \exists \hat{f}_n \in \mathcal{S}^{\mathtt{HN}}_{n,\tau}(\delta_n) ~{\rm s.t.}~ \|\hat{f}_n - f_0\|_2 \ge \delta_n\big\}  \ge 1-\frac{c_{10}}{\log n},
	\end{align*} where
	\begin{align}
	\label{eq:lb-terms}
	    \delta_n \asymp \frac{1}{(\bar{N}\bar{L})^{2\alpha} \log^{5\alpha} (\bar{N}\bar{L})} \bigvee \Bigg[  1\bigwedge \frac{\bar{N}\bar{L}}{\sqrt{n} \log n}  \Bigg\{ \sqrt{\tau} \land \Bigg(\frac{\sqrt{n} \log n}{\bar{N} \bar{L}}\Bigg)^{1/p}\Bigg\} \Bigg]  \bigvee \frac{1}{\tau^{p-1} \log^2(n)}.
	\end{align} 
        \item[(2)] There exists $\delta_{n,*} \asymp n^{-\frac{\alpha \nu^*}{2\alpha+\nu^*}} (\log n)^{-\frac{\alpha(3\nu^*+4)}{2\alpha+\nu^*}}$ such that
        \begin{align*}
            \liminf_{n\to\infty} \inf_{\bar{N}, \bar{L} \ge c_{11}, \tau \ge c_{12}} \sup_{(X,f_0,\varepsilon) \in \mathcal{U}(d,p,\mathcal{F}_0)} \mathbb{P} \big\{  \exists \hat{f}_n \in \mathcal{S}_{n,\tau}^{\mathtt{HN}}(\delta_{n,*}) ~{\rm s.t.}~  \|\hat{f}_n - f_0\|_2 \ge \delta_{n,*}\big\}  = 1.
        \end{align*}
    \end{itemize}
\end{theorem}
}

% \r{Q: why lower bound is now so ugly?}

\begin{remark}
Note that our lower bound result holds for a sufficiently large $\tau$, i.e., $\tau \ge c_{12}$. When $\tau<c_{12}$, by following the similar proof strategy, we can find some $f_0$ and asymmetric noise distribution such that $\|f_0 - f_{0,\tau}\|_2 \gtrsim \frac{1}{\tau^{p-1}} \gtrsim 1$, while $\|\hat{f}_n - f_{0,\tau}\|_2 \to 0$. But in this case, we will rely on a different convergence analysis of $\|\hat{f}_n - f_{0,\tau}\|_2$ because Proposition \ref{prop:strong-convex} does not necessarily hold. This will further complicate the proofs. We thus omit this scenario for ease of presentation.
\end{remark}

{\color{revisecolor} 
\begin{remark}
\label{remark:adaptive-huber-lb}
Theorem \ref{thm:generic-lower-bound} provides a generic lower bound on the $L_2$ error for the class of Huber ReLU-DNN estimators.  Plugging in minimax optimal convergence rates over $\mathcal{H}(d,l,\mathcal{P})$ with $d\ge d^*$ and $\mathcal{C}(d,\beta)$ into Theorem \ref{thm:generic-lower-bound}, (2) shows that the lower bounds on the convergence rates of adaptive Huber estimators over these two function classes are
\begin{align*}
	\delta_{n,*}^{\mathcal{H}(d,l,\mathcal{P})} \asymp  n^{-\frac{\gamma^* \nu^*}{2\gamma^*+\nu^*}} (\log n)^{-\frac{\gamma^*(3\nu^*+4)}{2\gamma^*+\nu^*}} \qquad \text{and} \qquad \delta_{n,*}^{\mathcal{C}(d,\beta)} \asymp n^{-\frac{\beta \nu^*}{2\beta+d\nu^*}} (\log n)^{-\frac{\beta(3\nu^*+4)}{2\beta+d\nu^*}},
\end{align*}
respectively.  This confirms that the obtained upper bound for adaptive Huber estimator (Theorem \ref{thm:rate-adaptive-huber-upper}) is sharp up to a logarithmic factor of $n$.

\end{remark}
Meanwhile, since the Huber ReLU-DNN estimator coincides with its least squares counterpart when $\tau = \infty$, we obtain the following generic lower bound on the $L_2$ error of least squares ReLU-DNN estimators.

\begin{theorem}[A generic lower bound on the $L_2$ error of least squares estimator]
\label{thm:generic-lower-bound-lse}
	Let $\mathcal{F}_0$, $\mathcal{U}(d, p,\mathcal{F})$ and $\mathcal{S}_{n,\tau}^{\mathtt{HN}}(\delta)$ be as in Theorem \ref{thm:generic-lower-bound}. The following two statements hold.
	\begin{itemize}
	    \item[(1)] For any $n \ge 3$, $\bar{N}, \bar{L} \ge c_{11}$,
        	\begin{align*}
        		\sup_{(X,f_0,\varepsilon) \in \mathcal{U}(d,p,\mathcal{F}_0)}\mathbb{P}\big\{ \exists \hat{f}_n \in \mathcal{S}_{n,\infty}^{\mathtt{HN}}(\delta_n) ~{\rm s.t.}~ \|\hat{f}_n - f_0\|_2 \ge \delta_n\big\} \ge 1-\frac{c_{10}}{\log n},
        	\end{align*} 
         where $\delta_n \asymp \frac{1}{(\bar{N}\bar{L})^{2\alpha} \log^{5\alpha} (\bar{N}\bar{L})} \lor \big\{(\tfrac{ \bar{N}\bar{L} }{\sqrt{n}\log n})^{1-1/p} \land 1\big\}$.
        \item[(2)] There exists $\delta_{n,*} \asymp n^{-\frac{\alpha \nu^\dagger}{2\alpha+\nu^\dagger}} (\log n)^{-\frac{7\alpha\nu^\dagger}{2\alpha+\nu^\dagger}}$ such that
        \begin{align}
        \label{eq:lower-bound-lse-alpha}
            \liminf_{n\to\infty} \inf_{\bar{N}, \bar{L} \ge c_{11}} \sup_{(X,f_0,\varepsilon) \in \mathcal{U}(d,p,\mathcal{F}_0)} \mathbb{P} \big\{  \exists \hat{f}_n \in \mathcal{S}_{n,\infty}^{\mathtt{HN}}(\delta_{n,*}) ~{\rm s.t.}~ \|\hat{f}_n - f_0\|_2 \ge \delta_{n,*}\big\}  = 1,
        \end{align}
    \end{itemize}
\end{theorem}}

{
\color{revisecolor}
\begin{remark}
	Similar to Remark \ref{remark:adaptive-huber-lb}	 for adaptive Huber ReLU-DNN estimators, plugging $\alpha=\beta/d$ for $\mathcal{C}(d,\beta)$ and $\alpha=\gamma^*$ for $\mathcal{H}(d,l,\mathcal{P})$ with $d\ge d^*$ into Theorem \ref{thm:generic-lower-bound-lse}, (2) gives  lower bounds on the convergence rates of least squares estimators over these two function classes, which are 
	\begin{align*}
		\delta_{n,*}^{\mathcal{H}(d,l,\mathcal{P})} \asymp  n^{-\frac{\gamma^* \nu^\dagger}{2\gamma^*+\nu^\dagger}} (\log n)^{-\frac{7\gamma^*\nu^\dagger}{2\gamma^*+\nu^\dagger}}  \qquad \text{and} \qquad \delta_{n,*}^{\mathcal{C}(d,\beta)} \asymp n^{-\frac{\beta \nu^\dagger}{2\beta+d\nu^\dagger}} (\log n)^{-\frac{7\beta \nu^\dagger}{2\beta+d\nu^\dagger}} ,
	\end{align*} 
	respectively.  This affirms the tightness of Theorem \ref{thm:rate-lse-upper} as well as the superiority of adaptive Huber ReLU-DNN regression over its least squares counterpart under heavy-tailed errors.

\end{remark}
}

It is worth noticing that the above lower bound is tailored to the Huber regression estimator trained on ReLU-DNN to reveal the impact of the tails of noise on ReLU-DNN estimators. 
For the the H\"older class $\mathcal{C}(\beta, d)$, \cite{kuchibhotla2019least} showed that the $\mathcal{C}(\beta, d)$-constrained least squares estimators achieves the optimal rate $n^{-\frac{\beta}{2\beta+d}}$ when $p \ge 2+d/\beta$.
In Section \ref{sec:nn-meet-htr}, we will discuss whether ReLU-DNN-based estimators can achieve such an optimal rate.

\subsection{Insights from the lower bound analysis}
\label{subsec:lowerbound}

From the previous upper bound analysis, we see that the overall convergence rate is determined by a trade-off among the three terms in \eqref{eq:generic-upper-bound-three-terms}: bias introduced by the Huber loss, neural network approximation error, and statistical error. 

To derive a lower bound on the convergence rate, we will show that when $\mathcal{F}_n(d, \bar{L}, \bar{N}, 1)$ is used as the function class, for each one of the above three terms, there exists \revise{some data generating process $(X,f_0,\varepsilon) \in \mathcal{U}(d,p,\mathcal{F}_0)$} such that, up to logarithmic factors,
\begin{align*}
	\revise{\delta_{\mathtt{b}}}(\tau) \gtrsim \frac{1}{\tau^{p-1}}, \quad \revise{\delta_{\mathtt{a}}}(\bar{N}, \bar{L}) \gtrsim \big( \bar{N}  \bar{L}  \big)^{-2 \revise{\alpha}}, \quad \revise{\delta_{\mathtt{s}}(n,\bar{N}, \bar{L}, \tau)} \gtrsim  \bar{N} \bar{L} \sqrt{\frac{\tau \land   \{  n / (\bar{N} \bar{L})^2  \}^{1/p}}{n} } \bigwedge 1 
\end{align*}
and the $L_2$ error is bounded from below by each of these terms.
This gives rise to the lower bound	\eqref{eq:lb-terms} and forms the strategy of our proof.  We now further explain each of the above three terms, starting with the second and third terms.

\medskip
\noindent \textbf{Approximation error.} We begin with a lower bound on the approximation error term $\revise{\delta_{\mathtt{a}}}(\bar{N}, \bar{L})$. Since $X$ is assumed to follow the uniform distribution on $[0,1]^d$,  $\|\cdot\|_2$ coincides with the $L_2$ norm on $[0,1]^d$ equipped with the Lebesgue measure. The next theorem provides a lower bound on the approximation error. 

{\color{revisecolor}
\begin{theorem}[Neural network approximation error, lower bound]
\label{thm:neural-network-approx-lower-bound}
	 Suppose $\bar{L}, \bar{N} \geq 3$ are arbitrarily given integers. Let $\mathcal{F}_0 \subseteq \{f: \mathbb{R}^d \to [-1, 1]\}$ be a function class with intrinsic dimension-adjusted smoothness upper bounded by $\alpha$. Let $\mathcal{F}_n=\mathcal{F}_n(d, \bar{L}, \bar{N}, 1)$ be the neural network class of interest. Then, there exists some constant $c_{13}>0$ independent of $\bar{N}$ and $\bar{L}$ such that
	\begin{align*}
		\sup_{f_0 \in \mathcal{F}_0} \inf_{f\in \mathcal{F}_n} \|f-f_0\|_2 \ge c_{13} \big\{ (\bar{N}\bar{L})^2\log^5(\bar{N}\bar{L})\big\}^{-\alpha},
	\end{align*} where $\|\cdot\|_2=\| \cdot \|_{L_2(\mathcal{P}_X)}$ with $\mathcal{P}_X$ denoting the uniform distribution on $[0,1]^d$. 
\end{theorem}}

The result in Theorem~\ref{thm:neural-network-approx-lower-bound} furnishes a lower bound on the $L_2$ error for any ReLU-DNN estimator\revise{, that is,  given any $\bar{N}, \bar{L} \geq 3$, we have}
\begin{align*}
    \revise{\inf_{\hat{f}_n \in \mathcal{F}(d,\bar{L}, \bar{N}, 1)} \sup_{f_0\in \mathcal{F}_0} \|\hat{f}_n - f_0\|_2 \ge \sup_{f_0 \in \mathcal{F}_0} \inf_{f\in \mathcal{F}_n} \|f-f_0\|_2 \gtrsim \big\{ (\bar{N}\bar{L})^2\log^5(\bar{N}\bar{L})\big\}^{-\alpha}.}
\end{align*}
\begin{remark}
    \revise{A direct application of Theorem \ref{thm:neural-network-approx-lower-bound} in the case of $\mathcal{F}_0=\mathcal{C}(d,\beta)$ yields
    \begin{align*}
        \sup_{f_0 \in \mathcal{C}(d,\beta)} \inf_{f\in \mathcal{F}_n} \|f-f_0\|_2 \gtrsim \big\{ (\bar{N}\bar{L})^2\log^5(\bar{N}\bar{L})\big\}^{-\beta/d}.
    \end{align*}}Using VC dimension-based techniques, \cite{yarotsky2017error} and \cite{lu2020deep} derived lower bounds on ReLU-DNN approximation errors under $L_\infty$ norm for the H\"older class $\mathcal{C}(d,\beta)$, with the latter being $\sup_{f_0 \in \mathcal{C}(d,\beta)} \inf_{f\in \mathcal{F}_n} \|f-f_0\|_\infty \gtrsim \big\{ (\bar{N}\bar{L})^2\log(\bar{N}\bar{L}) \big\}^{-\beta/d}$. The VC dimension-based technique, however, cannot be directly applied to control the $L_2$ approximation error. We thus follow a different route by combining an upper bound of the least squares ReLU-DNN estimator \citep{kohler2021rate} and a lower bound for nonparametric estimators \revise{over the function class $\mathcal{F}$ of interest. This strategy can be applied to a wide range of function classes as long as the corresponding minimax risks are known.}
\end{remark}

\noindent \textbf{Statistical error.} The next proposition establishes a lower bound on the statistical error term \revise{$\delta_{\mathtt{s}}(n, \bar{N}, \bar{L}, \tau)$}. To this end, we let $X$ be uniformly distributed on $[0,1]^d$ and $f_0=0$. The following proposition provides a lower bound on $\|\hat{f}_n - f_0\|_2$, and serves as a supporting lemma for the proof of Theorem~\ref{thm:generic-lower-bound}; see Figure~\ref{fig:lowerbound-roadmap}.

\begin{proposition}[Lower bound of convergence rate in the null case]
\label{prop:lowerbound-variance}
	Let $f_0=0$, $X$ be uniformly distributed on $[0,1]^d$, and $M=1$. Moreover, let  $n\ge\sqrt{2(d+1)}$, $p\ge 2$, $\tau \geq c_{12}$ be arbitrary, and let $N, L$ be positive integers satisfying $(NL)^2 \ge c_{14}$.
	Then, there exists some symmetric distribution of $\varepsilon$ with $\mathbb{E}[|\varepsilon|^p|X=x]\le 1$ such that with probability at least $ 1- \frac{c_{15}}{n\land (NL)^2}$, there exists some $\tilde{f}_n \in \cF_n$  satisfying $ \hat{\mathcal{R}}_\tau(\tilde{f}_n) \le \min \big\{\hat{\mathcal{R}}_\tau(f_0), \inf_{ \|f\|_\infty \le 1} \hat{\mathcal{R}}_\tau(f) + c_{16}\delta_n^2 \big\}$ and the lower bound $ \|\tilde{f}_n - f_0\|_2  = \|\tilde{f}_n\|_2   \ge \delta_n$, where $\cF_n$ is either $\mathcal{F}_n(d, c_{17}L \log n, c_{18} N, 1)$ or $\mathcal{F}_n(d, c_{19}L, c_{20} N \log n, 1)$, and 
		\begin{align}
		\label{eq:prop-lb-delta-n}
	 	\delta_n \asymp \frac{N L}{\sqrt{n}} \sqrt{  \tau \land \big\{ n/(NL)^2 \big\}^{1/p} }  \bigwedge 1  .
	\end{align}
	Here $c_{12}$ is the constant in Theorem \ref{thm:generic-lower-bound}, $c_{14}$--$c_{20}$ are positive constants independent of $N$, $L$, $\tau$, $n$.
\end{proposition}

\begin{remark}
Note that $\inf_{ \|f\|_\infty \le 1} \hat{\mathcal{R}}_\tau(f)  \le  \inf_{ f\in \mathcal{F}_n} \hat{\mathcal{R}}_\tau(f)$.  Thus,
in Proposition~\ref{prop:lowerbound-variance} we show the existence of an approximate empirical risk minimizer $\tilde{f}_n$, instead of the exact minimizer, satisfying the desired lower bound. We believe this is an artifact of the proof technique.
%We believe this is due to some technical issues. 
The obtained lower bound still matches the upper bound stated in Theorem \ref{thm:generic-bound} because $\hat{f}_n$ therein only needs to be an approximate minimizer satisfying $\hat{\mathcal{R}}_\tau(\hat{f}_n) - \inf_{f\in \mathcal{F}_n} \hat{\mathcal{R}}_\tau(f) \lesssim \delta_n^2$ with high probability.
\end{remark}

\begin{remark}
 \label{remark:stochastic-error-symmetric}
\revise{Recall from Theorem \ref{thm:convergence-rate-nn-symmetric} that when $\varepsilon$ is symmetric and has bounded $p$-th moment ($p\geq 1$), the Huber estimator achieves the near-optimal rate $\tilde{\mathcal{O}}(n^{-\gamma^*/(2\gamma^*+1)})$ as if $\varepsilon$ is sub-Gaussian. {In the proof of Proposition \ref{prop:lowerbound-variance}, the noise $\varepsilon$ is constructed to be symmetric.} Therefore, one can show a stronger version of claim \eqref{eq:lower-bound-lse-alpha} that with $\delta_{n,*} \asymp n^{- \alpha \nu^\dagger/(2\alpha+\nu^\dagger)}$ up to logarithmic factors,
\begin{align*}
    \liminf_{n\to\infty} \inf_{\bar{N}, \bar{L} \ge c_{11}} \sup_{\substack{(X,f_0,\varepsilon) \in \mathcal{U}(d,p,\mathcal{F}_0) \\ \varepsilon ~{\rm is~symmetric}}} \mathbb{P} \big\{  \exists \hat{f}_n \in \mathcal{S}_{n,\infty}^{\mathtt{HN}}(\delta_{n,*}) ~{\rm s.t.}~ \|\hat{f}_n - f_0\|_2 \ge \delta_{n,*}\big\} = 1.
\end{align*} 
Applying this to the function class $\mathcal{H}(d,l,\mathcal{P})$ indicates that the least squares estimator cannot achieve the optimal rate as the Huber estimator does under heavy-tailed symmetric noise.} 
\end{remark}

\begin{figure*}[t!]
    \centering
\begin{tikzpicture}
\definecolor{myred}{HTML}{ae1908}
\definecolor{myblue}{HTML}{05348b}
\definecolor{myorange}{HTML}{ec813b}
\draw[->] (0,-3.4) -- (0,3.4);
\draw[->] (0,0) -- (11,0);
\draw (0.5pt,-1cm) -- (-0.5pt,-1cm) node[anchor=east] {$-1$};
\draw (0.5pt,1cm) -- (-0.5pt,1cm) node[anchor=east] {$1$};
\draw (0.5pt,0.25cm) -- (-0.5pt,0.25cm) node[anchor=east] {$u$};
\draw (0.5pt, 2.8cm) -- (-0.5pt, 2.8cm) node[anchor=east] {$(\frac{n}{S})^{1/p}$};
\draw (0.5pt, -2.8cm) -- (-0.5pt, -2.8cm) node[anchor=east] {$-(\frac{n}{S})^{1/p}$};

\draw (11, -0.3) node{$x$};
\draw[myblue, thick] (0,0) -- (10, 0);
\draw[myblue] (10.1, -0.3) node{$f_0$};
\draw (1.5,2.8) circle[radius=0.1];
\draw (3.5,-2.8) circle[radius=0.1];
\draw (6,2.8) circle[radius=0.1];
\draw (9,-2.8) circle[radius=0.1];
\draw (1,0) circle[radius=0.1];
\draw (2.5,0) circle[radius=0.1];
\draw (4.5,0) circle[radius=0.1];
\draw (5,0) circle[radius=0.1];
\draw (7.2,0) circle[radius=0.1];
\draw (8.1,0) circle[radius=0.1];
\draw (9.7,0) circle[radius=0.1];

\draw[myred, thick] (0, 0.25) -- (1.4, 0.25);
\draw[myred, thick] (1.4, 0.25) -- (1.5, 1);
\draw[myred, thick] (1.5, 1) -- (1.6, 0.25);
\draw[myred, thick] (1.6, 0.25) -- (3.4, 0.25);
\draw[myred, thick] (3.4, 0.25) -- (3.5, -1);
\draw[myred, thick] (3.5, -1) -- (3.6, 0.25);
\draw[myred, thick] (3.6, 0.25) -- (5.9, 0.25);
\draw[myred, thick] (5.9, 0.25) -- (6, 1);
\draw[myred, thick] (6, 1) -- (6.1, 0.25);
\draw[myred, thick] (6.1, 0.25) -- (8.9, 0.25);
\draw[myred, thick] (8.9, 0.25) -- (9.0, -1);
\draw[myred, thick] (9.0, -1) -- (9.1, 0.25);
\draw[myred, thick] (9.1, 0.25) -- (10.0, 0.25);
\draw[myred] (10.1, 0.5) node{$\tilde{f}$};
\draw[gray, dotted] (0, 2.8) -- (11, 2.8);
\draw[gray, dotted] (0, 1) -- (11, 1);
\draw[gray, dotted] (0, -1) -- (11, -1);
\draw[gray, dotted] (0, -2.8) -- (11, -2.8);

\end{tikzpicture}
    \caption{An illustration on the construction of $f_0$, which is zero (blue), and noise $\varepsilon$ that follows a Trinomial distribution such that $y=\varepsilon$ (black circle), and $\tilde{f}$ (red function), which is capped at 1 due to the constraint $\|\tilde{f}\|_\infty \leq1$.  The red function fits better than the true function (blue) when $u$ is sufficiently small. This is related to the approximation ability of a ReLU neural network with a given depth and width.}
    \label{fig:proof-sketch}
\end{figure*}
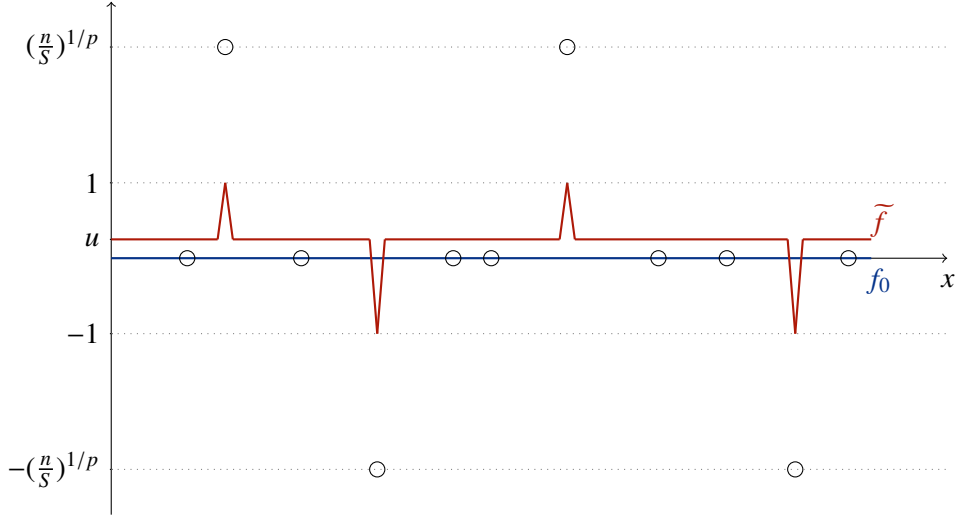

To gain insights into how such a lower bound is established, we consider the nonparametric Huber regression estimator over some uniformly bounded function class $\mathcal{F}$ in the simplest case of $d=1$. Here $\mathcal{F}$ is a generic function class that is capable of fitting $S$ points in a highly non-smooth manner. Its structure will be specified in the proof sketch below and Section \ref{sec:lb:revisit}. The proof sketch reveals that the approximation ability of the ReLU neural network is essential in our construction of the lower bound and demystifies how this is related to the statistical rate of convergence for heavy-tailed noises.  To complete the proof of Proposition \ref{prop:lowerbound-variance} in more general cases,  Theorem \ref{thm:approx-noise-dim-d} also plays an important role, indicating that a ReLU neural network with depth $\lesssim \bar{L}$ and width $\lesssim \bar{N}$ can approximate $S\asymp (\bar{N}\bar{L})^2$ points arbitrarily in a highly non-smooth manner. 

\begin{proof}[Proof Sketch of Proposition~\ref{prop:lowerbound-variance}]
Our target is to find a distribution of $\varepsilon$ satisfying $\mathbb{E}[|\varepsilon|^p|X=x] \le 1$ such that there exists some $\tilde{f} \in \mathcal{F}$ satisfying (with $f_0=0$)
\begin{align}
\label{eq:lb-stochastic-error-req}
    \| \tilde{f} - f_0 \|_2 \gtrsim \delta_n \qquad \text{ and } \qquad \hat{\mathcal{R}}_\tau(\tilde{f}) \le \hat{\mathcal{R}}_\tau(f_0) \le \inf_{\|f\|_\infty \le 1} \hat{\mathcal{R}}_\tau(f) + \delta_n^2 .
\end{align}
Let the random noise
\begin{align*}
    \varepsilon = \begin{cases}
        \left(\frac{n}{S}\right)^{1/p} \qquad &\text{ with probability } \frac{S}{2n} \\
        -\left(\frac{n}{S}\right)^{1/p} \qquad &\text{ with probability } \frac{S}{2n} \\
        0 \qquad &\text{ with probability } 1-\frac{S}{n}
    \end{cases} 
\end{align*}
be independent of $X \sim \text{Uniform}[0,1]$. The observed data are $\{(X_i,Y_i)\}_{i=1}^n$ with $Y_i = \varepsilon_i$ satisfying $\mathbb{E}(|\varepsilon_i|^p|X_i) =1$. 

If  $S=o(n)$, then by a concentration result, we have that there are approximate $k \asymp S$ samples with non-zero $\varepsilon_i$ while the rest $n-k \asymp n-S$ samples all have zero $\varepsilon_i$. Without loss of generality, let $|\varepsilon_1| = |\varepsilon_2| = \cdots = |\varepsilon_k| = (n/S)^{1/p}$ and $\varepsilon_{k+1}=\varepsilon_{k+2}=\cdots = \varepsilon_n = 0$. Now, we need to find some $\tilde{f}$ satisfying \eqref{eq:lb-stochastic-error-req}.  The key idea is to construct a red function as shown in Figure~\ref{fig:proof-sketch}
in which $n=10$ and $k=S=4$. Given the large outliers, fitting those by $\sgn(\varepsilon_i)$ (the maximum magnitude allowed in our function class) has much smaller losses than by zero (true value), and they can even compensate the losses elsewhere by fitting a non-zero constant $u$. As a result, the red function has a better fit to the data than the true function (blue), yet $\| \tilde{f} - f_0 \|_2 \geq u$, satisfying \eqref{eq:lb-stochastic-error-req}.

Let us now formally implement the above idea.
For an arbitrary $\Delta>0$, suppose we can take $\tilde{f} \in \mathcal{F}$ such that  $\tilde{f}(X_i)=\text{sgn}(\varepsilon_i)$ for $i\in \{1,\ldots, k\}$, and $\tilde{f}(x)\equiv u$ in $[0,1]\setminus \bigcup_{i=1}^k (X_i-\Delta, X_i+\Delta)$.
%A potential $\tilde{f}$ satisfies \eqref{eq:lb-stochastic-error-req} with $\delta_n \asymp \sqrt{\frac{S}{n}\left((n/S)^{1/p} \land \tau \right)}$ is realized by a piecewise linear function.
Then, we have $\|\tilde{f}-f_0\|_2 \gtrsim u$ if $\Delta$ is small. If we further require that $\Delta < \min_{i\neq j} |X_i-X_j|$ with probability tending to $1$, then $\tilde{f}(X_i) = u$ for all $i\in \{k+1,\ldots , n\}$. By a second-order Taylor expansion, we have
\begin{align*}
    \hat{\mathcal{R}}_\tau(\tilde{f}) - \hat{\mathcal{R}}_\tau(f_0) &\le  \frac{1}{n} \left\{\sum_{i=1}^k \left(\psi_\tau(\varepsilon_i) \tilde{f}(X_i) + \frac{1}{2} \tilde{f}(X_i)^2 \right) +  \sum_{i=k+1}^{n} \frac{1}{2} \tilde{f}(X_i)^2 \right\} \\
    &\le -\frac{k}{n} \big\{(n/S)^{1/p} \land \tau \big\} + \frac{k}{2n} + \frac{1}{2} u^2 \lesssim u^2 - \frac{S}{n} \big\{(n/S)^{1/p} \land \tau \big\},
\end{align*} as long as $\tau \ge 1$. Thus,  $\hat{\mathcal{R}}_\tau(\tilde{f}) \le \hat{\mathcal{R}}_\tau(f_0)$ if $u \le \sqrt{\frac{S}{n}\left\{(n/S)^{1/p} \land \tau \right\}}$. Now by choosing $\delta_n \asymp u \asymp \sqrt{\frac{S}{n}\left\{(n/S)^{1/p} \land \tau \right\}}$, we also have
\begin{align*}
    \hat{\mathcal{R}}_\tau(f_0) - \inf_{\|f\|_\infty\le 1}\hat{\mathcal{R}}_\tau(f) =   \frac{2 k}{n} \big\{  (n/S)^{1/p} \land \tau \big\}  - \frac{k}{n} \le \delta_n^2.
\end{align*}
Therefore,  Claim \eqref{eq:lb-stochastic-error-req} holds with $\delta_n \asymp \sqrt{\frac{S}{n} \{ (n/S)^{1/p} \land \tau \} } \land 1$.
\end{proof}

In Section \ref{sec:lb:revisit}, we will show that under multivariate settings ($d \ge 2$), the class of ReLU neural networks with depth $\bar{L}$ and width $\bar{N}$ shares a similar approximation ability as $\mathcal{F}$ when $S\asymp (\bar{L} \bar{N})^2$. Proposition \ref{prop:lowerbound-variance} can then be proved following a similar argument.

\medskip
\noindent \textbf{Robustification bias.} Suppose $X$ and $\varepsilon$ are independent. It is easy to construct an asymmetric error $\varepsilon$ such that $\|f_0 - f_{0,\tau}\|_2 \gtrsim \tau^{1-p}$. The difficulty here is to show that $\|\hat{f}_n - f_0\|_2 \gtrsim \|f_0 - f_{0,\tau}\|_2 \gtrsim \tau^{1-p}$. The key idea of the bias analysis is that if we choose $f_0=0$ and let $\varepsilon$ be independent of $X$, then $f_{0,\tau}(x)\equiv C_\tau$ and $\|\hat{f}_n - f_{0,\tau}\|_2$ can be upper bounded by \revise{$\frac{\bar{N}\bar{L}}{\sqrt{n}} \{\sqrt{\tau} \land (\frac{\sqrt{n}}{\bar{N}\bar{L}})^{1/p}\}$} up to logarithmic factors. This result is formally stated in Lemma C.5. 
We now divide the discussion into two cases.

(1) $\hat{f}_n$ converges to $f_0$ faster than $f_{0,\tau}$ does, i.e., $\revise{\frac{\bar{N}\bar{L}}{\sqrt{n}} \{\sqrt{\tau} \land (\frac{\sqrt{n}}{\bar{N}\bar{L}})^{1/p}\}} \lesssim \tau^{1-p}$. In this case, based on our construction of $\varepsilon$ and the triangle inequality, the $L_2$ risk satisfies $\|\hat{f}_n - f_0\|_2 \ge \|f_{0,\tau} - f_0\|_2 - \|f_{0,\tau} - \hat{f}_n\|_2 \gtrsim \tau^{1-p}$.

(2) $\hat{f}_n$ converges to $f_0$ slower than $f_{0,\tau}$ does, i.e., $\revise{\frac{\bar{N}\bar{L}}{\sqrt{n}} \{\sqrt{\tau} \land (\frac{\sqrt{n}}{\bar{N}\bar{L}})^{1/p}\}} \gtrsim \tau^{1-p}$. In this case the bias is no longer the dominating term. Taking $f_0$ and $\varepsilon$ as in Proposition \ref{prop:lowerbound-variance}, we show that $\|\hat{f}_n - f_0\|_2 \gtrsim \revise{\delta_{\mathtt{s}}(n, \bar{N}, \bar{L},\tau)} \gtrsim \revise{\delta_{\mathtt{b}}}(\tau)$.

\subsection{Revisiting neural network approximation}
\label{sec:lb:revisit}

The proof sketch in Section \ref{subsec:lowerbound} connects the lower bound on the convergence rate to a specific approximation ability of the function class $\mathcal{F}$ as shown in Figure~\ref{fig:proof-sketch}. To be specific, suppose the function class $\mathcal{F}$ has the approximation ability that it can approximate any $S=o(n)$ points $\{(x_i, y_i)\}_{i=1}^S$, with $y_i$ being arbitrarily chosen from $\{-1,+1\}$, in a way that it stays as a constant $u \asymp \sqrt{\frac{S}{n} \{ (n/S)^{1/p} \land \tau \} } \asymp \delta_n$ except in the region $\bigcup_{i=1}^S B_\infty(x_i, \Delta)$ for a sufficiently small $\Delta>0$, where $B_\infty(x, r)=\{y: \|y-x\|_\infty \le r\}$ is the hypercube with the half-width $r$. Then the convergence rate will be lower bounded by $\delta_n$.

It is, however, unclear whether the aforementioned approximation ability holds. 
Yet, we can show a weak version of the approximation ability by relaxing the requirements in two aspects: the $S$ points $\{x_i\}$ are located in $S$ distinct sub-cubes in $[0,1]^d$ (see Figure~\ref{fig:thm-approx}(a)) and  either the width or the depth can depend logarithmically on $\Delta$. It turns out that such a weaker approximation ability is sufficient to prove Proposition \ref{prop:lowerbound-variance} based on a similar strategy as outlined in Section \ref{subsec:lowerbound}.

The next theorem claims that a ReLU neural network with depth $\bar{L}$ and width $\bar{N}$ can fit any piecewise constant function in $  (\bar{N}\bar{L})^2$ sub-hypercubes. This lays the foundation to prove Theorem \ref{thm:approx-noise-dim-d}, which in turn provides the approximation ability required in Proposition \ref{prop:lowerbound-variance}.

\begin{theorem}
\label{thm:approx-n2l2-dim-d}
	For any given $N, L\in \mathbb{N}^+$, let $K=\lfloor N^{1/d}\rfloor^2 \lfloor L^{1/d}\rfloor^2$, and $\{y_\alpha\}_{\alpha\in \mathcal{A}} \subseteq [0,1]$ be an arbitrary set of values indexed by $\mathcal{A}=\{1, \ldots, K\}^d$. For any tolerance parameter $\Delta \in (0, 1/(3K)]$, and precision parameter $\epsilon \in (0,1)$, let 
	\begin{align}
	\label{eq:good-region-q-alpha}
		Q_\alpha(\Delta) = \Big\{x = (x_1,\cdots, x_d):  (\alpha_i-1)/K \le x_i \le \alpha_i / K - 1_{\{\alpha_i < K\}}\Delta \Big\}.
	\end{align} 
	Then, there exist a deep ReLU neural network $f^\dagger_1$ with depth $(5L+7)(\lceil \log_2 (1/\epsilon) \rceil+2)$ and width $(4N+3)d \lor (8N+10)$, and a deep ReLU neural network $f_2^\dagger$ with depth  $9L+12$ and width $(4N+3)d\lor (8N+6)(\lceil \log_2 (1/\epsilon) \rceil +1)$ such that 
	\begin{align}
	\label{eq:approx-n2l2-dim-d-res}
		|f_s^\dagger(x) - y_\alpha| \le \epsilon ~\text{ for all }~ x\in Q_\alpha(\Delta), \qquad s=1,2.
	\end{align} Moreover, if $y_\alpha = \sum_{i=0}^{r} 2^{-i} \theta_i$ for $(\theta_1,\ldots,\theta_r)\in \{0,1\}^r$ with $r \le \lceil \log_2 (1/\epsilon)\rceil$, we have $f_s^\dagger(x)=y_\alpha$ instead of $|f_s^\dagger(x)-y_\alpha| \le \epsilon$ in \eqref{eq:approx-n2l2-dim-d-res}.  In this case, the term $\lceil \log_2 (1/\epsilon) \rceil +1)$ in the width and depth can further be reduced to $r$ if all the $y_\alpha$ can be written as the above form.
\end{theorem}

Theorem \ref{thm:approx-n2l2-dim-d} indicates that for any integers $N, L \geq 1$, if we divide the unit cube $[0,1]^d$ uniformly into $K^d$ sub-cubes with length $\frac{1}{K}$ for $K\asymp (N L )^{2/d}$, then it is possible to find a deep ReLU neural network $f^\dagger$ with depth $\lesssim L$ and width $\lesssim N$ such that $f^\dagger$ is approximately equal to an arbitrary pre-specified value for each sub-cube except in a small ``bad'' region, i.e., $[0,1]^d \setminus \bigcup_{\alpha \in \mathcal{A}} Q_\alpha(\Delta)$. Moreover, such a bad region can be arbitrarily small, and to achieve  $\epsilon$-accuracy, we only need to multiply a factor of $\log(1/\epsilon)$ to either the width or the depth of the neural network. 

Figure~\ref{fig:thm-approx}(a) illustrates the result of Theorem \ref{thm:approx-n2l2-dim-d}. In the case where $d=2$, $(NL)^2=16$, we divide $[0,1]^2$ into $K^2$ sub-squares of equal size with $K=4$ so that each sub-square has length $1/K=0.25$. %We do not need to care about 
The values of $f^\dagger$ in the ``bad'' region (green) can be arbitrary. However, in the blue regions with index $\alpha\in \{1,..., 4\}^2$, we can find some $f^\dagger$ such that it  approximates the pre-specified $y_\alpha$ in each sub-square.

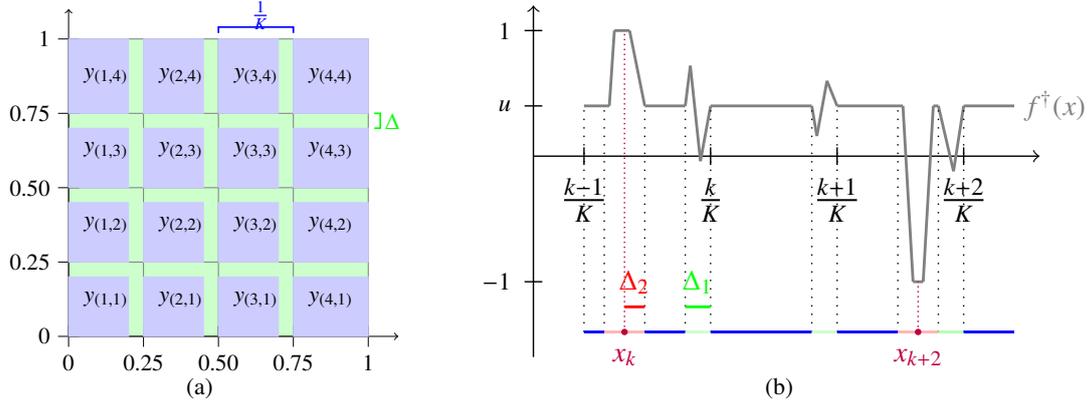
\begin{figure*}[t!]
\begin{tabular}{cc}
\subfigure[]{
\resizebox{0.36\textwidth}{!}{
\begin{tikzpicture}[scale=5]
\filldraw[green!20] (0,0) rectangle +(1,1);
\draw[step=.25cm,gray,very thin] (0,0) grid (1,1);
\draw[->] (0,0) -- (1.1,0);
\draw[->] (0,0) -- (0,1.1);
\foreach \x in {0, 0.25, 0.50, 0.75, 1}
\draw (\x cm,1pt) -- (\x cm,-1pt) node[anchor=north] {$\x$};
\foreach \y in {0, 0.25, 0.50, 0.75, 1}
\draw (1pt,\y cm) -- (-1pt,\y cm) node[anchor=east] {$\y$};

\foreach \x in {1, 2, 3}
    \foreach \y in {1, 2, 3}
    {
        \filldraw[blue!20] (\x*0.25-0.25,\y*0.25-0.25) rectangle +(0.20, 0.20);
        \draw (\x*0.25-0.125,\y*0.25-0.125) node{$y_{(\x,\y)}$};
    }

\foreach \y in {1, 2, 3}
{
	\filldraw[blue!20] (4*0.25-0.25,\y*0.25-0.25) rectangle +(0.25, 0.20);
    \draw (4*0.25-0.125,\y*0.25-0.125) node{$y_{(4,\y)}$};
}
\foreach \x in {1, 2, 3}
{
	\filldraw[blue!20] (\x*0.25-0.25,4*0.25-0.25) rectangle +(0.20, 0.25);
    \draw (\x*0.25-0.125,4*0.25-0.125) node{$y_{(\x, 4)}$};
}
\filldraw[blue!20] (4*0.25-0.25,4*0.25-0.25) rectangle +(0.25, 0.25);
\draw (4*0.25-0.125,4*0.25-0.125) node{$y_{(4, 4)}$};

\draw[blue, thick] (0.50, 1.02) -- (0.50, 1.04) -- (0.750, 1.04) -- (0.750, 1.02);
\draw[blue] (0.64, 1.08) node {$\frac{1}{K}$};

\draw[green, thick] (1.02, 0.75) -- (1.04, 0.75) -- (1.04, 0.70) -- (1.02, 0.70);
\draw[green] (1.08, 0.725) node{$\Delta$};

\end{tikzpicture}
}
} &
\subfigure[]{
\resizebox{0.56\textwidth}{!}{

\begin{tikzpicture}[scale=5]
\draw[->] (0,-0.4) -- (0,0.3);
\draw[->] (0,0) -- (1,0);

\draw (0.5pt,-0.25cm) -- (-0.5pt,-0.25cm) node[anchor=east] {\small $\scriptstyle -1$};
\draw (0.5pt,0.25cm) -- (-0.5pt,0.25cm) node[anchor=east] {\small $\scriptstyle 1$};
\draw (0.5pt,0.1cm) -- (-0.5pt,0.1cm) node[anchor=east] {\small $\scriptstyle u$};

\draw (0.10cm,0.5pt) -- (0.10 cm,-0.5pt) node[anchor=north] {\small $\frac{k-1}{K}$};

\draw (0.35cm,0.5pt) -- (0.35 cm,-0.5pt) node[anchor=north] {\small $\frac{k}{K}$};

\draw (0.60cm,0.5pt) -- (0.60 cm,-0.5pt) node[anchor=north] {\small $\frac{k+1}{K}$};

\draw (0.85cm,0.5pt) -- (0.85 cm,-0.5pt) node[anchor=north] {\small $\frac{k+2}{K}$};

\draw[gray, thick] (0.10, 0.1) -- (0.15, 0.1);
\draw[gray, thick] (0.15, 0.1) -- (0.16, 0.25);
\draw[gray, thick] (0.16, 0.25) -- (0.19, 0.25);
\draw[gray, thick] (0.19, 0.25) -- (0.22, 0.1);
\draw[gray, thick] (0.22, 0.1) -- (0.30, 0.1);
\draw[gray, thick] (0.30, 0.1) -- (0.31, 0.18);
\draw[gray, thick] (0.31, 0.18) -- (0.33, -0.01);
\draw[gray, thick] (0.33, -0.01) -- (0.35, 0.1);
\draw[gray, thick] (0.35, 0.1) -- (0.55, 0.1);
\draw[gray, thick] (0.55, 0.1) -- (0.56, 0.04);
\draw[gray, thick] (0.56, 0.04) -- (0.58, 0.15);
\draw[gray, thick] (0.58, 0.15) -- (0.60, 0.1);
\draw[gray, thick] (0.60, 0.1) -- (0.73, 0.1);
\draw[gray, thick] (0.73, 0.1) -- (0.75, -0.25);
\draw[gray, thick] (0.75, -0.25) -- (0.77, -0.25);
\draw[gray, thick] (0.77, -0.25) -- (0.79, 0.1);
\draw[gray, thick] (0.79, 0.1) -- (0.80, 0.1);
\draw[gray, thick] (0.80, 0.1) -- (0.83, -0.03);
\draw[gray, thick] (0.83, -0.03) -- (0.85, 0.1);
\draw[gray, thick] (0.85, 0.1) -- (0.95, 0.1);

\draw[dotted] (0.10, -0.35) -- (0.10, 0.1);
\draw[dotted] (0.14, -0.35) -- (0.14, 0.1);
\draw[densely dotted, purple] (0.18, -0.35) -- (0.18, 0.25);
\draw[dotted] (0.22, -0.35) -- (0.22, 0.1);
\draw[dotted] (0.30, -0.35) -- (0.30, 0.1);
\draw[dotted] (0.35, -0.35) -- (0.35, 0.1);
\draw[dotted] (0.55, -0.35) -- (0.55, 0.1);
\draw[dotted] (0.60, -0.35) -- (0.60, 0.1);
\draw[dotted] (0.72, -0.35) -- (0.72, 0.1);
\draw[densely dotted, purple] (0.76, -0.35) -- (0.76, -0.25);
\draw[dotted] (0.80, -0.35) -- (0.80, 0.1);
\draw[dotted] (0.85, -0.35) -- (0.85, 0.1);

\draw[blue, thick] (0.10, -0.35) -- (0.14, -0.35);
\draw[red!30, thick] (0.14, -0.35) -- (0.22, -0.35); 
\draw[blue, thick] (0.22, -0.35) -- (0.30, -0.35);
\draw[green!20, thick] (0.30, -0.35) -- (0.35, -0.35);
\draw[blue, thick] (0.35, -0.35) -- (0.55, -0.35);
\draw[green!20, thick] (0.55, -0.35) -- (0.60, -0.35);
\draw[blue, thick] (0.60, -0.35) -- (0.72, -0.35);
\draw[red!30, thick] (0.72, -0.35) -- (0.80, -0.35);
\draw[green!30, thick] (0.80, -0.35) -- (0.85, -0.35);
\draw[blue, thick] (0.85, -0.35) -- (0.95, -0.35);

\filldraw[purple] (0.18, -0.35) circle [radius=0.005cm];
\filldraw[purple] (0.76, -0.35) circle [radius=0.005cm];

\draw[purple] (0.18, -0.40) node{$\scriptstyle x_k$};
\draw[purple] (0.76, -0.40) node{$\scriptstyle x_{k+2}$};

\draw[green, thick] (0.30, -0.30) -- (0.35, -0.30);
\draw[green] (0.325, -0.25) node{ $\scriptstyle \Delta_1$};
\draw[red, thick] (0.18, -0.30) -- (0.22, -0.30);
\draw[red] (0.20, -0.25) node{ $\scriptstyle \Delta_2$};
\draw[gray] (1.03, 0.1) node{ $\scriptstyle f^\dagger(x)$};
\end{tikzpicture}
}
}
\end{tabular}
\caption{Illustrative explanations of Theorems \ref{thm:approx-n2l2-dim-d} and  \ref{thm:approx-noise-dim-d}.  The left panel shows that there exists a deep ReLU neural network with depth $L$ and width $N$ that is approximately piecewise constant in $N^2L^2 \asymp K^d$ blue regions with pre-specified values except for green regions. The right panel indicates that there exists a deep ReLU neural network that takes value 1 at $x_k$ and value $-1$ at $x_{k+2}$ and constant value $u$ in other regions except in small regions around the points $x_k$ and $x_{k+2}$ with radius $\Delta_2$ (red) and regions  $(j/K-\Delta_1, j/K]$ $(j=k, k+1, k+2)$ (green). The second part of the exception region is similar to the green region in the left panel for 1-dimension.    }
\label{fig:thm-approx}
\end{figure*}

Based on Theorem \ref{thm:approx-n2l2-dim-d}, we are ready to give a formal statement on the approximation ability of ReLU neural networks. 

\begin{theorem}
\label{thm:approx-noise-dim-d}
	Given any integers $N, L \geq 1$, let $K=\lfloor N^{1/d}\rfloor^2 \lfloor L^{1/d}\rfloor^2$. For any $\Delta_1 \in (0, 1/(3K)]$, $\Delta_2 >0$, suppose $(x_\alpha, y_\alpha)_{\alpha \in \tilde{\mathcal{A}}}$ is a set of arbitrary points indexed by $\tilde{\mathcal{A}} \subseteq \{1,\cdots, K\}^d$. Each element $(x_\alpha,y_\alpha)$ satisfies 
		$x_\alpha \in Q_\alpha(\Delta_1)$ defined by \eqref{eq:good-region-q-alpha} and $y_i \in \{-1,1\}$. Then there exist some constants $c_{22}$--$c_{25}$ independent of $N, L, \Delta_1, \Delta_2$ such that for any $u\in [-1,1]$, we can find a deep ReLU neural network $f_1^\dagger$ with depth $c_{22} L\log_2 (1/\Delta_2)$ and width $c_{23} N$ and a deep ReLU neural network $f_2^\dagger$  with depth $c_{24} L$, $c_{25} N \log_2 (1/\Delta_2)$ satisfying 
	\begin{align*}
		f^\dagger_s(x_\alpha) = y_\alpha ~\text{ for all }~ \alpha \in \tilde{\mathcal{A}}, \qquad s=1,2,
	\end{align*} and \begin{align*}
 		f^\dagger_s(x) = u ~\text{ if }~ x\in Q \text{, and }\|x - x_\alpha\|_\infty \ge \Delta_2 ~\text{ for all }~ \alpha \in \tilde{\mathcal{A}},\qquad s=1,2 , 
 \end{align*} 
 where $Q=\bigcup_{\alpha \in \{1,\cdots, K\}^d} Q_\alpha(\Delta_1)$.
\end{theorem}

Theorem \ref{thm:approx-noise-dim-d} shows that if the points $\{x_i\}_{i=1}^S$ with $S\lesssim K^d\asymp L^2N^2$ are located on distinguished regions of $K^d$ sub-cubes with length $1/K$ in $[0,1]^d$, then we can find some $f^\dagger \in \mathcal{F}(d, N, L, 1)$ such that its value at each $x_i$ ($i=1,\ldots, S$) equals a pre-specified constant in $\{-1,+1\}$, while stays at a constant outside the cubes (with length $2\Delta_2$) of these $(NL)^2$ points and a ``bad'' region $[0,1]^d\setminus Q$. 
Note that $\Delta_1 \in (0, 1/(3K)]$ can be arbitrarily chosen, and the network depth or width only depends on $\Delta_2$ logarithmically.
Figure \ref{fig:thm-approx}(b) provides an example of this approximation ability when $d=1$. 

{
\color{revisecolor}
\begin{remark}
\label{remark:bounded-approximation} We use the point fitting module in \cite{lu2020deep} to construct our target function $f^\dagger$. More specifically, they use a $\tilde{\mathcal{O}}(L)$ depth $\tilde{\mathcal{O}}(N)$ width ReLU neural network, whose weights explicitly scale with $e^{N+L}$, to approximate $(NL)^2$ uniformly located points. We claim that Theorem \ref{thm:approx-noise-dim-d} also holds for uniformly bounded weights under a special scenario, in which we can use $\mathcal{O}(\log N)$ depth and $\mathcal{O}(N)$ width ReLU neural network with weights explicitly bounded by $1$ to approximate $N^2$ uniformly located points; see Lemma 14 and proof of Theorem 4 in \cite{fan2022factor}. This demonstrates that the approximability of ReLU-DNN described in Theorem \ref{thm:approx-n2l2-dim-d} and \ref{thm:approx-noise-dim-d} is universal and is not related to the unbounded weights.
\end{remark}
}

\subsection{Neural network meets heavy-tailed error: a comprehensive picture.}
\label{sec:nn-meet-htr}

In this section, we summarize the main results obtained in this paper and leave an interesting open question. 
\revise{Assume that the $d$-variate regression function $f_0$ is bounded in magnitude by $1$}, and the noise variable satisfies $\mathbb{E}[|\varepsilon|^p | X=x] \le 1$ for some $p\ge 2$. 
Table~\ref{table:convergence-rate} summarizes our main results, and for comparison purposes, also includes an upper bound result on the least squares estimator for H\"older classes \citep{kuchibhotla2019least}. \revise{We provide proofs for the lower and upper bounds for all the ReLU-DNN estimators in the table.} Theoretically, the adaptive Huber estimator outperforms its least squares counterpart under the same moment conditions. In the case of $f_0\in \mathcal{C}(d,\beta)$, Figure \ref{fig:rate-wrt-p} depicts how the optimal rate changes with the moment index $p$ for each estimator.

\begin{table}[htb!]
    \centering \small
    \begin{tabular}{llllll}
      \hline
      \hline
      $f_0$ & Function class & Loss & Upper bound & Lower bound \\
      \hline
      $\mathcal{C}(d,\beta)$ & H\"older class & Square & $n^{-\frac{\beta}{2\beta + d}}$   ($p\ge 2+d/\beta$) & $n^{-\frac{\beta}{2\beta + d}}$   ($p\ge 2$) \\
      $\mathcal{C}(d,\beta)$ & ReLU-DNN & Square & $n^{-\frac{(1-1/p) \beta}{2\beta + d(1-1/p)}}$ ($p \geq 2$) & $n^{-\frac{(1-1/p) \beta}{2\beta + d(1-1/p)}}$ ($p \geq 2$) \\
      $\mathcal{C}(d,\beta)$ & ReLU-DNN & Huber & $n^{-\frac{(1-\frac{1}{2p-1}) \beta}{2\beta + d(1-\frac{1}{2p-1})}}$ ($p \geq 2$)& $n^{-\frac{(1-\frac{1}{2p-1}) \beta}{2\beta + d(1-\frac{1}{2p-1})}}$  ($p \geq 2$) \\
      \hline
      $\mathcal{H}(d,l,\mathcal{P})$ & ReLU-DNN & Square & $n^{-\frac{(1-1/p) \gamma^*}{2\gamma^* + (1-1/p)}}$ ($p \geq 2$) & $n^{-\frac{(1-1/p) \gamma^*}{2\gamma^* + (1-1/p)}}$ ($p \geq 2$) \\
      $\mathcal{H}(d,l,\mathcal{P})$ & ReLU-DNN & Huber & $n^{-\frac{(1-\frac{1}{2p-1}) \gamma^*}{2\gamma^* + (1-\frac{1}{2p-1})}}$ ($p \geq 2$)& $n^{-\frac{(1-\frac{1}{2p-1}) \gamma^*}{2\gamma^* + (1-\frac{1}{2p-1})}}$  ($p \geq 2$) \\      
      \hline
      \hline
    \end{tabular}
	\caption{A summary of convergence rates for nonparametric estimators  when $\EE (|\varepsilon|^p|X)$ is bounded.}
    \label{table:convergence-rate}
\end{table}

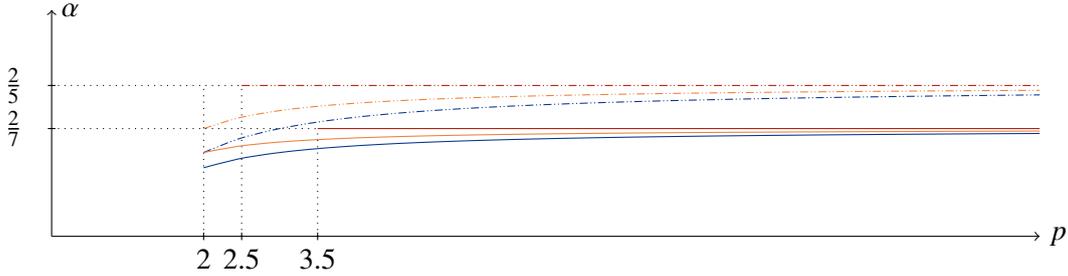
\begin{figure*}[t!]
    \centering
\begin{tikzpicture}

\draw[->] (0, 0) -- (0, 3) node[right] {$\alpha$};
\draw[->] (0, 0) -- (13, 0) node[right] {$p$};

\draw[myred, densely dashdotdotted] (2.5, 0.4*5) -- (13, 0.4*5);
%\draw[red] (1.2, 2.8) node{H\"older-LS};
\draw[scale=1, domain=2:13, smooth, variable=\p, myorange, densely dashdotdotted] plot ({\p}, {5 * (2 * (1-1/(2*\p-1))) / (4 + 1 * (1-1/(2*\p-1)))});
%\draw[blue] (1.2, 2.2) node{DNN-AH};
\draw[scale=1, domain=2:13, smooth, variable=\p, myblue, densely dashdotdotted] plot ({\p}, {5 * (2 * (1-1/\p)) / (4 + 1 * (1-1/\p))});
%\draw[green] (1.2, 1.8) node{DNN-LS};
\draw[dotted] (0, 2) -- (2.5, 2);
\draw[black] (-0.05, 2) -- (0.05, 2);
\draw[black] (-0.5, 2) node{$\frac{2}{5}$};
\draw[dotted] (2, 0) -- (2, 2);
\draw[black] (2, -0.05) -- (2, 0.05);
\draw[black] (2, -0.3) node{$2$};
\draw[dotted] (2.5, 0) -- (2.5, 2);
\draw[black] (2.5, -0.05) -- (2.5, 0.05);
\draw[black] (2.5, -0.3) node{$2.5$};

\draw[myred] (3.5, 5*2/7) -- (13, 5*2/7);
\draw[scale=1, domain=2:13, smooth, variable=\p, myorange] plot ({\p}, {5 * (2 * (1-1/(2*\p-1))) / (4 + 3 * (1-1/(2*\p-1)))});
\draw[scale=1, domain=2:13, smooth, variable=\p, myblue] plot ({\p}, {5 * (2 * (1-1/\p)) / (4 + 3 * (1-1/\p))});

\draw[dotted] (3.5, 0) -- (3.5, 5*2/7);
\draw[dotted] (0, 5*2/7) -- (3.5, 5*2/7);
\draw[black] (-0.05, 5*2/7) -- (0.05, 5*2/7);
\draw[black] (-0.5, 5*2/7) node{$\frac{2}{7}$};

\draw[black] (3.5, -0.05) -- (3.5, 0.05);
\draw[black] (3.5, -0.3) node{$3.5$};
\end{tikzpicture}
\caption{Plots of convergence rate $\alpha(p,d, \beta)$ (see Table~\ref{table:convergence-rate}) versus $p$ for the least squares estimator in H\"older function class (\myred{red}), the adaptive Huber ReLU-DNN estimator (\myorange{orange}), and the least squares ReLU-DNN estimator (\myblue{blue}).  Here, we take $\beta =2$, $d=1$ (dash) and $3$ (solid). }
\label{fig:rate-wrt-p}

\end{figure*}

\revise{When $f_0 \in \mathcal{C}(d,\beta)$,} a natural question is whether we can construct a robust neural network regression estimator that achieves the optimal rate $n^{-\beta/(2 \beta + d)}$ under heavy-tailed errors.
The answer to this question is related to how well neural networks can approximate a function with additional smoothness control.
For example, if we can approximate regression function $f_0 \in \mathcal{C}(d,\beta)$ using $\mathcal{F}_n^{(C)} = \{f\in \mathcal{F}_n:|f(x)-f(y)| \le C\|x-y\|\}$ with some constant $C$ independent of $\bar{L}$ and $\bar{N}$, while maintaining a similar approximation error, i.e., 
\begin{align}
    \sup_{f_0 \in \mathcal{C}(d, \beta)} \inf_{f\in \mathcal{F}_n^{(C)}} \|f_0 - f\|_2 \lesssim \sup_{f_0 \in \mathcal{C}(d,\beta) } \inf_{f\in \mathcal{F}_n} \|f_0 - f\|_2,  \label{eq4.10}
\end{align} then we only need to consider the estimator 
\begin{align}
\label{smooth-estimator}
    \hat{f}_n \in  \text{argmin}_{f\in \mathcal{F}_n^{(C)}} \frac{1}{n} \sum_{i=1}^n \{Y_i - f(X_i)\}^2.
\end{align} 
Note that functions in class $\mathcal{F}_n^{(C)}$ are $C$-Lipschitz. Thus the corresponding least squares estimator is expected to achieve a similar convergence rate as the least squares estimator for H\"older class discussed in \cite{kuchibhotla2019least}. If \eqref{eq4.10} holds, using a similar argument it can be shown that the constrained least squares estimator defined in \eqref{smooth-estimator} achieves the rate $n^{-\frac{\beta}{2\beta+d}}$ when $p \geq 2+d$.
Numerically, however, efficient algorithms for solving the constrained minimization problem in \eqref{smooth-estimator} are lacking.

\section{Conclusion}

In this paper, we have studied how the heavy-tailed errors impact on the rate of convergence of nonparametric regression estimators fitted on deep ReLU neural networks. We consider the adaptive Huber estimator and establish non-asymptotic error bounds on the $L_2$ risk. By presenting a matching lower bound, we further show that this is the best possible convergence rate the adaptive Huber ReLU-DNN estimator can obtain, which is faster than that of the least squares counterpart.
%faster than that of the least-squares estimator by presenting a matching lower bound on the rate of convergence for the adaptive Huber estimator. 
This provides a comprehensive picture of the stability of deep ReLU neural networks under polynomial-tail errors.

%Several interesting questions are worth exploring further. The first one is whether 
An interesting yet challenging open question is whether it is possible to obtain a ReLU-DNN estimator that achieves the minimax-optimal rate over $\cH(d, l, \cP)$ 
%the class of Lipschitz-continuous functions that have a hierarchical composition structure  
when the noise distribution is heavy-tailed.
%using deep ReLU neural networks for regression function with a hierarchical composition structure with Lipschitz continuity. 
%To answer this question, we point out a promising direction that can be divided into two steps. 
The key step is to develop a neural network approximation result \eqref{eq4.10}, which shows that constraining deep ReLU neural networks with a bounded Lipschitz constant does not reduce much the approximation error. If so, then for heavy-tailed noise distribution with $p\ge 2+d$, the estimator in \eqref{smooth-estimator} can achieve the same rate of convergence as if the noise is sub-Gaussian. 
%The second step is to develop a computationally feasible method to compute such an estimator \eqref{smooth-estimator}. 
Note that the above estimator only attains the minimax-optimal rate when $p\ge 2+d$. The case of $2\leq p<2+d$ remains unclear.
%It is still unclear what is 
%Even the above program is successful, the optimal rate for regression functions  when $2\leq p < 2+d$ remains unknown.  
In addition, regardless of the validity of \eqref{eq4.10}, it is still interesting to understand the extent to which the `regularized' estimator can provide stability.

\bibliographystyle{apalike2}
\bibliography{arxiv_main}

\newpage
\appendix
%\tableofcontents

\section{Proof for Section~\ref{sec3}}

This section contains the proofs of all the theoretical results in Section~\ref{sec3}.

\subsection{Proof of Proposition \ref{prop:strong-convex}}

%The key idea of the loss landscape analysis is the second order Taylor expansion of loss around $\ell_\tau(\cdot)$. By doing second order Taylor expansion, we have
To begin with, we derive from the fundamental theorem of calculus that for every $v, w \in \RR$,
\begin{align}
    \ell_\tau(v+w) - \ell_\tau(v) = \psi_\tau(v) w + \int_{0}^w \psi_\tau'(v+t) (w-t) {\rm d}t ,
\end{align} 
where $\psi_\tau(\cdot)$ and $\psi'_\tau(\cdot)$ are given in \eqref{def:score}.  
%the first and second derivative of the Huber loss $\ell_\tau(\cdot)$ respectively.
For any $\tau>0$ and function $f:[0, 1]^d \to \RR$, it follows that
%Now we try to expand $\mathcal{R}^\tau(f) - \mathcal{R}^\tau(f_0)$. For fixed $\tau$, we have
\begin{align}
    \mathcal{R}_\tau(f) - \mathcal{R}_\tau(f_0) &= \mathbb{E}\big\{ \ell_\tau \big(  \varepsilon + \Delta_f(X) \big) \big\} - \mathbb{E}\big\{ \ell_\tau(\varepsilon) \big\}  \nn \\
    &= \mathbb{E} \big\{ \psi_\tau(\varepsilon) \Delta_f(X) \big\} + \mathbb{E}\Bigg[\int_{0}^{\Delta_f(X)} 1( |\varepsilon+t| \le \tau ) \{ \Delta_f(X)-t \}  {\rm d} t \Bigg] . \label{pop.loss.diff.decomposition}
\end{align}
In the following, we bound the two terms on the right-hand side of \eqref{pop.loss.diff.decomposition} separately.

For the former,   note that  $ \psi_\tau(\varepsilon) = 1\{|\varepsilon| \le \tau \} \varepsilon + 1\{\varepsilon > \tau\} \tau - 1\{\varepsilon < -\tau\} \tau$.  Recall the assumption that $\mathbb{E}( \varepsilon|X=x )=0$,  we have $ \mathbb{E}\big\{ 1( |\varepsilon| \le \tau )  \varepsilon \, | \, X=x \big\} = -\mathbb{E} \big\{ 1( |\varepsilon| > \tau )  \varepsilon  \, |  \, X=x \big\}$. Taking the conditional expectation of $\psi_\tau(\varepsilon)$ given $X=x$ yields 
\begin{align}
    \big|  \mathbb{E}\big\{  \psi_\tau(\varepsilon)|X=x \big\} \big| &=    \big|\mathbb{E} \big\{ -1( |\varepsilon| > \tau) \varepsilon + 1( \varepsilon > \tau ) \tau - 1( \varepsilon < -\tau) \tau  \,| \, X=x  \big\}  \big|   \nn \\
    &\le \mathbb{E} \big\{ (|\varepsilon|-\tau) 1( |\varepsilon| > \tau ) \,|\, X=x\big\}   \nn \\
    &\le \mathbb{E} \big\{  |\varepsilon|  (|\varepsilon| / \tau)^{p-1}   | X=x\big\}   \nn   \\
    & \leq v_p \tau^{1-p}. \label{bias.control}
\end{align}
Combining this with the Cauchy-Schwarz inequality,  we conclude that
\#
 \mathbb{E} \big\{ \psi_\tau(\varepsilon) \Delta_f(X) \big\} \geq - v_p \tau^{1-p} \EE | \Delta_f(X) | \ge - v_p \tau^{1-p}  \|  f - f_0 \|_2 .  \label{bias.bound}
\#

Turning to the second term on the right-hand side of \eqref{pop.loss.diff.decomposition},  for every $f\in \Theta$ we have
\begin{align*}
     & \mathbb{E}\Bigg[ \int_{0}^{\Delta_f(x)} 1( | \varepsilon +t| \le \tau ) \{ \Delta_f(x)-t \} {\rm  d} t \bigg| X=x  \Bigg]   \\
   &  =    \mathbb{E}\Bigg[  \int_{0}^{\Delta_f(x)} \big\{ 1-1( |\varepsilon+t| > \tau) \big\}  \{ \Delta_f(x)-t \} {\rm  d} t \Big| X=x\Bigg] \\
   & \ge  \frac{1}{2} |\Delta_f(x)|^2 - \mathbb{E}\Bigg\{  \int_{0}^{\Delta_f(x)} \big[ 1( |\varepsilon| > \tau/2 )  + 1\{|\Delta_f(x)|>\tau/2\}\big]  \{ \Delta_f(x)-t \} {\rm d} t \Big| X=x\Bigg\}   \\
   &  = \frac{1}{2} |\Delta_f(x)|^2 \big\{   1 - \mathbb{P}\big(|\varepsilon|> \tau / 2 |X=x \big)  \big\} ,
\end{align*} 
where the last step follows from the fact that $|\Delta_f(x)|=|f(x)-f_0(x)| \leq  2M \le \tau/2$ provided $\tau\geq 4M$. %and Fubini Theorem by exchanging expectation and integral. 
By Markov's inequality and Condition \ref{cond2},  
\begin{align*}
   \mathbb{P}\big(|\varepsilon|> \tau / 2 |X=x \big)  \le \frac{\mathbb{E}( |\varepsilon|^p|X=x ) }{(\tau/2)^p } \leq  v_p (2/\tau )^{ p} \le \frac{1}{2}~\mbox{ for all } x
\end{align*} 
as long as $\tau \ge 2(2v_p)^{1/p}$.  Taking the expectation with respect to $X \sim \cP_X$ gives
\#
    \mathbb{E}\Bigg[\int_{0}^{\Delta_f(x)} 1( |\varepsilon +t| \le \tau ) \{ \Delta_f(x)-t\}  {\rm d} t  \Bigg] \ge  \frac{1}{4} \EE \{ \Delta^2_f(X) \}  = \frac{1}{4} \| f - f_0 \|_2^2 .  \label{curvature.bound}
\#

Together, \eqref{pop.loss.diff.decomposition}, \eqref{bias.bound} and \eqref{curvature.bound} imply that as long as $\tau \geq  2\max\{  2M,  (2v_p)^{1/p} \}$, 
\begin{align*}
    \mathcal{R}_\tau(f) - \mathcal{R}_\tau(f_0) & \geq  -v_p \tau^{1-p} \|f-f_0\|_2 + \frac{1}{4} \|f-f_0\|_2^2 
\end{align*} 
holds for all $f\in \Theta$.  In particular,  for $f\in  \Theta \setminus \Theta_0(8 v_p \tau^{1-p})$ we have $  v_p \tau^{1-p} \|f - f_0\|_2 \le   \|f - f_0\|_2^2 / 8$, and hence
\begin{align*}
    \mathcal{R}_\tau(f) - \mathcal{R}_\tau(f_0) \ge \frac{1}{8} \|f-f_0\|_2^2 .
\end{align*}

If the distribution of $\varepsilon | X=x$ is symmetric,  
and since $\psi_\tau(\cdot)$ is also symmetric (around zero), we have
\$
& \mathbb{E} \{ \psi_\tau(\varepsilon)|X=x \} = \Bigg\{ \int_{-\infty}^0 + \int_0^\infty  \Bigg\} \psi_\tau(t) \, {\rm d}  F_{\varepsilon | X=x} (t) \\
& =   \int_0^\infty  \psi_\tau(t) \, {\rm d} F_{\varepsilon | X=x} (-t)  +  \int_0^\infty \psi_\tau(t) \, {\rm d}  F_{\varepsilon | X=x} (t) \\
& =   \int_0^\infty  \psi_\tau(t) \, {\rm d} \{ 1 - F_{\varepsilon | X=x} ( t)\}  +  \int_0^\infty \psi_\tau(t) \, {\rm d}  F_{\varepsilon | X=x} (t)  = 0, 
\$
implying that $\mathbb{E}\{ \psi_\tau(\varepsilon) \Delta_f(X)\}=0$.  Consequently, 
\begin{align*}
    \mathcal{R}_\tau(f) - \mathcal{R}_\tau(f_0) = \mathbb{E}\Bigg[\int_{0}^{\Delta_f(X)} 1( |\varepsilon+t| \le \tau ) \{ \Delta_f(X)-t \}  {\rm d} t \Bigg]   \ge \frac{1}{4} \|f-f_0\|_2^2 
\end{align*} 
for all $f\in \Theta$ provided that $\tau \ge  2\max\{ 2M,  (2v_p)^{1/p} \}$.  This completes the proof of Proposition~\ref{prop:strong-convex}. \qed

%%%%%%%%%%%%%%%%%%%%%%%%%%%%%%%%%%%%%%%%%%%%%%%%%%%%%%%%%%%%%%
%%%%%%%%%%%%%%%%%% Proof of Proposition 3.2 %%%%%%%%%%%%%%%%%% 
%%%%%%%%%%%%%%%%%%%%%%%%%%%%%%%%%%%%%%%%%%%%%%%%%%%%%%%%%%%%%%

\subsection{Proof of Proposition \ref{prop:bias}}

Recall from the proof of Proposition \ref{prop:strong-convex} that as long as $\tau \geq 2\max\{ 2M,  (2v_p)^{1/p} \}$, 
    \begin{align*}
        \mathcal{R}_\tau(f) - \mathcal{R}_\tau( f_0) 
        \ge -v_p \tau^{1-p} \|f-f_0\|_2 + \frac{1}{4} \|f-f_0\|_2^2 ~\mbox{ for all }  f\in \Theta. 
    \end{align*} 
Taking $f=f_{0, \tau}$,  the claimed result follows immediately from the fact that $  \mathcal{R}_\tau( f_{0, \tau}) - \mathcal{R}_\tau( f_0) \le 0$.  

Next, assume that the noise variable $\varepsilon$ is sub-Gaussian satisfying \eqref{subgaussian.cond}.  Similarly to \eqref{bias.control}, now we have for any $x\in [0,1]^d$ that
$\PP\big( |\varepsilon| > \tau/2 | X=x\big)  \leq 2 e^{- \tau^2/ (8 \sigma^2) }$ and
\$
	&  \big|  \mathbb{E}\big\{  \psi_\tau(\varepsilon)|X=x \big\} \big|   \leq \EE \big\{ |\varepsilon| 1(|\varepsilon| > \tau ) | X=x \big\} \\
	 & = \sigma \EE  \int_0^\infty   1 \big(|\varepsilon/\sigma| > \tau/ \sigma \big)  1 \big(|\varepsilon/\sigma| >t  \big) {\rm d}t \\
	 & = \tau \PP \big(|\varepsilon/\sigma| > \tau/ \sigma \big)   + \sigma \int_{\tau/\sigma}^\infty \PP \big(|\varepsilon/\sigma| >t  \big) {\rm d}t \\
	 & \leq 2 \tau e^{-\frac{\tau^2}{2\sigma^2 }} + 2 \sigma \int_{\tau/\sigma}^\infty e^{-t^2/2} {\rm d} t  \leq  2 (  \tau + \sigma^2/\tau ) e^{-\frac{\tau^2}{2\sigma^2 }}  .
\$
Keeping the remaining arguments the same,  we obtain that as long as $\tau \geq  4\max\{ M, \sigma \}$, 
\$
  \mathcal{R}_\tau(f) - \mathcal{R}_\tau( f_0) 
        \ge -  2 (  \tau + \sigma^2/\tau ) e^{-\frac{\tau^2}{2\sigma^2 }}    \|f-f_0\|_2 +(1 - 2 e^{-2}) \|f-f_0\|_2^2 ~\mbox{ for all }  f\in \Theta. 
\$
This proves \eqref{bias.exponential.decay} by taking $f=f_{0, \tau}$.  	 \qed

\subsection{Proof of Theorem \ref{thm:generic-bound}}

%We first revisit Talagrand's inequality for suprema of empirical processes, and then provide a tail probability inequality for Huber estimator.

%To begin with a general statement about the tail probability for huber estimator, we need to use the Talagrand inequality for suprema of empirical process as follows.

In order to prove our main Theorem \ref{thm:generic-bound}, we first introduce some notations. For any given $f^* \in \Theta$, Define
\begin{align*}
    \Theta_*(r) = \{f \in \Theta: \|f-f^*\|_2 \le r\} ~~~~~~ \text{and} ~~~~~~ \Theta_*^c(r) = \{f \in \Theta: \|f-f^*\|_2 > r\}.
\end{align*}
Moreover, define the difference of Huber loss of $f$ and $f^*$ at $(X,\varepsilon)$ as
\begin{align*}
    h_{f,f^*}(X,\varepsilon) &= \ell_{\tau}(Y-f(X)) - \ell_{\tau}(Y-f^*(X)) \\
    &=\ell_{\tau}(\varepsilon+f_0(X) - f(X)) - \ell_{\tau}(\varepsilon+f_0(X) - f^*(X)).
\end{align*} 

Denote $\Delta_{f,*}(X) := f^*(X) - f(X)$, it follows easily from the Taylor expansion that
\begin{align}
\label{eq:hf_def}
\begin{split}
    h_{f,f^*}(X,\varepsilon) &= \ell_\tau'(\varepsilon+f_0(X) - f^*(X)) \Delta_{f,*}(X) \\
    &~~~~~~~~~~~~ + \int_{0}^{\Delta_{f,*}(X)} \ell_\tau''(t + \varepsilon+f_0(X)-f^*(X)) (\Delta_{f,*}(X)-t)\mathrm{d}t \\
    &= T_\tau(\varepsilon + f_0(X) - f^*(X)) \Delta_{f,*}(X) \\
    &~~~~~~~~~~~~ + \int_{0}^{\Delta_{f,*}(X)} 1\{|t + \varepsilon + f_0(X)-f^*(X)|\le \tau\} (\Delta_{f,*}(X)-t)\mathrm{d}t.
\end{split}
\end{align} We also denote
\begin{align}
\label{eq:hfb_def}
\begin{split}
    h^B_{f,f^*}(X,\varepsilon) &= T_{\tau \land B}(\varepsilon + f_0(X) - f^*(X)) \Delta_{f,*}(X) \\
    &~~~~~~~~~~~~ + \int_{0}^{\Delta_{f,*}(X)} 1\{|t + \varepsilon + f_0(X)-f^*(X)|\le \tau\} (\Delta_{f,*}(X)-t)\mathrm{d}t.
\end{split}
\end{align} 
We will abbreviate the above two quantities to be $h_f(X,\varepsilon)$ and $h^B_f(X,\varepsilon)$ respectively if $f^*$ is clear from the context. We also define
\begin{align}
    C_v =  v_2 \land \tau^2  .
\end{align}

\begin{lemma}[A tail probability inequality for Huber estimator]
\label{lemma:convergence-rate}
    Assume Condition \ref{cond1} holds, $C_v < \infty$. Let $\mathcal{R}_\tau$ and $\hat{\mathcal{R}}_\tau$ be the population and empirical Huber risks, respectively, parameterized by $\tau \ge 1$. Given some positive real number $D\ge 1$, $\delta_*$, $\delta_{\mathtt{opt}}$ and a sequence of real positive number $\{B_k\}_{k=1}^\infty$,  assume that
  \begin{enumerate}
  \item[(1)] there exists a function $\tilde{f} \in \mathcal{F}_n$ satisfying $\|\tilde{f} - f^*\|_2 \le 3\delta_*$ and 
  \begin{align}
  	\mathcal{R}_\tau(\tilde{f}) - \mathcal{R}_\tau(f^*) \le 9(\delta_*)^2 ;	\label{eq:tail-prob-approx-error}
  \end{align}

\item[(2)] there exits some constant $c_{26}>0$ such that for any $\delta \geq {\delta}_*$,
    \begin{align}
        \cR_\tau(f) - \cR_\tau(f^*) \ge  c_{26} \delta^2~\mbox{ for all } f \in  \Theta_*^c(\delta/2); \label{part2}
    \end{align}  
    
\item[(3)] there exists a function $\phi_n: \RR^+ \times \mathbb{R}^+ \to \RR^+$ such that for any $\delta \ge \delta_*$ and $B>0$, 
 \begin{align}
        \mathbb{E} \Bigg\{ \sup_{f\in \mathcal{F}_n \cap \Theta_*(\delta)}  |\Delta_{n,B}(f)| \Bigg\}  \le  \phi_n(\delta, B) ~~~~\text{with}~~~~  \Delta_{n,B}(f) = \frac{1}{n} \sum_{i=1}^n h_f^B(X_i, \varepsilon_i) - \mathbb{E} [h_f^B(X,\varepsilon)],\label{part3}
    \end{align} and the function $\phi_n(\cdot, \cdot)$ satisfies $\phi_n(\delta \alpha, B) \le \alpha \phi_n(\delta, B)$ for any $\delta \ge \delta_*$, $B > 0$ and $\alpha \ge 1$. 
\item[(4)] let ${\delta}^\dagger = (D\delta_* + \delta_{\mathtt{opt}})$, for any $k \in \mathbb{N}^+$, one has
    \begin{align*}
        \phi_n(2^k {\delta}^\dagger, B_k) \le c_{27} 2^{2k} ({\delta}^\dagger)^2
    \end{align*} for some universal constant $c_{27}>0$ independent of $k$.

%\item[(5)] $\delta_n$ and $\tau$ are such that {$n \delta_n^2 \geq c_{29} \tau$} for some constant {$ c_{29}>0$}.  %\r{Q: $\hat{f}_n$ has not been defined}
\end{enumerate}

Then, let 
\begin{align*}
    c_{28} = \sqrt{\frac{9}{2 c_{26}}} \lor 48 \frac{c_{27}}{c_{26}}, ~~~~~~~~ c_{29} = \left(\frac{c_{26}^2}{20 \times 24^2} \land \frac{c_{26}}{128} \right) c_{28}^2, ~~~~~~~~ c_{30} = \frac{32}{c_{26} c_{28}^2}
\end{align*}
we have
\begin{align}
\begin{split}
    &\mathbb{P}\left[\exists f \in \mathcal{S}_{n,\tau}(\delta_{\mathtt{opt}}) ~~\text{s.t.}~~ \|f-f^*\|_2 \ge c_{28} {\delta}^\dagger \right] \\
    &~~~~~~~~~~~~~~~~~~~~~~~ \le \sum_{k=1}^\infty 2\exp\Bigg(-\frac{c_{29} n({\delta}^\dagger)^2}{ C_v + M^2 + M(B_k \land \tau)} 2^{2k}\Bigg) \\
    &~~~~~~~~~~~~~~~~~~~~~~~~~~~~~~~~~~~~~  +  \frac{c_{30} 1\{B_k < \tau\} \mathbb{E} \left[|\psi_\tau(\varepsilon + f_0(X) - f^*(X))|^p\right]}{B_k^{p-1} 2^{2k} (\delta^\dagger)^2}
\end{split}
\end{align}
\end{lemma}

Lemma~\ref{lemma:convergence-rate} provides a general, high-level result for establishing the convergence rate of the nonparametric Huber estimator $\hat f_n$ to some function $f^*$. 
Now we focus on the case where $f^*=f_0$. Recall from Proposition~\ref{prop:strong-convex} that a lower bound for the excess risk holds outside some local region, from which \eqref{part2} follows.  The main difficulty is then to validate condition (3), that is,  inequality \eqref{part3}.  To this end, we define
%Since we have obtain global strong convexity outside a small ball for the population loss $\mathcal{R}_n$ whose robustification parameter $\tau_n$ is adaptive to sample size $n$. Now in order to use Lemma \ref{lemma:convergence-rate} to prove convergence rate, the key step is to validate the second condition of Lemma \ref{lemma:convergence-rate}, i.e., to calculate 
%\begin{align}
%    \mathbb{E} [\sup_{f\in \mathcal{F}_n \cap \Theta(\delta)} \sqrt{n} |\Delta_n(f)|] \le c_{11} \phi_n(\delta)
%\end{align} for all the $n$ and $\delta>\tilde{\delta}_n$.
\begin{align} \label{def:g}
    g_f(x ,  \epsilon)  % = \ell_{\tau}(Y-f(X)) - \ell_{\tau}(Y-f_0(X))  
     = h_{f,f_0}^B(X,\varepsilon) = \psi_{\tau \land B}(\varepsilon) \Delta_f(X) + \int_{0}^{\Delta_f(X)} 1\{|\varepsilon+t| \le \tau\} (\Delta_f(X)-t) \mathrm{d}t,
\end{align} 
and let $\mathcal{G}_n = \{g=g_f : [0, 1]^d \times \RR \to \RR \,| \, f\in \mathcal{F}_n\}$. Moreover, let $\mathcal{G}_n(\delta) = \{g=g_f : [0, 1]^d \times \RR \to \RR \,| \, f\in \mathcal{F}_n \cap \Theta_0(\delta)\}$.  
The next two lemmas characterize the properties of function $g$ defined in \eqref{def:g} and the envelop  $G$ of $\cG_n$,  satisfying
\begin{align}
    \sup_{g\in \mathcal{G}_n} |g(x,\epsilon)| \le G(x,\epsilon) ~\mbox{ for all } x\in [0, 1]^d , \, \epsilon \in \RR.
\end{align}

\begin{lemma} 
\label{lemma:prop-g}
Let $g: [0, 1]^d \times \RR \to \RR$ be as in \eqref{def:g}.  
\begin{itemize}
\item[(1)] If Condition \ref{cond2} holds, we have $\mathbb{E} \{ g(X,\varepsilon)\}  \le \|f-f_0\|^2_2$ for all $f\in \Theta_0^{{\rm c}}(2v_p (\tau\land B)^{1-p})$.  If Condition \ref{cond3} holds,  we have $\mathbb{E} \{ g(X,\varepsilon)\}  \le \|f-f_0\|^2_2$ for all $f\in \Theta$.

\item[(2)] We have $\mathbb{E}\{  g^2(X,\varepsilon) \} \le 2(M^2+C_v) \|f-f_0\|_2^2$  for all $f\in \Theta$.
\end{itemize}
\end{lemma}

\begin{lemma}
\label{lemma:prop-envelop}
Let $G(x, \epsilon) = 2 M^2 + 2 M|\psi_{\tau \land B}(\epsilon)|$ be an envelop function for the class $\mathcal{G}_n$. Then
\begin{align*}
    \mathbb{E} \{ G^2(X,\varepsilon) \}  \le c_{14}:= 8 M^2( M^2 + C_v ) ~~\mbox{ and }~~   \| G \|_\infty \leq 2M ( M + \tau \land B) .
\end{align*}
\end{lemma}

To apply Lemma~\ref{lemma:convergence-rate}, the key is an upper bound on
\begin{align}
\label{eq:empirical-process-g}
    \mathbb{E} \Bigg\{ \sup_{g\in \mathcal{G}_n(\delta)}  \Bigg|\frac{1}{n} \sum_{i=1}^n g(X_i, \varepsilon_i) - \mathbb{E} g(X, \varepsilon) \Bigg|  \Bigg\} 
\end{align}
for a suitably chosen neural network $\cF_n$.   Specifically,  let $\mathcal{F}_n=\mathcal{F}_n(d,  \bar{L}, \bar{N}, M)$ be the function class realized by ReLU neural network with depth $\bar{L}$  and width $\bar{N}$, respectively. %, for some $L, N \geq 1 $ satisfying $L N \asymp    n^{1/2 - \alpha^*}$.   
	The next result characterizes the complexity  of the function class $\mathcal{F}_n$ via an upper bound on the uniform covering number, defined as follows.
	
\begin{definition}[Uniform covering number]
 Let $m\in \mathbb{N}^+$, and $\mathcal{F}=\{f:\mathcal{X}\to \mathbb{R}\}$ be a function class. We define the uniform covering number under $L_\infty$-norm for the function class $\mathcal{F}$ as
	  \begin{align*}
  	\mathcal{N}_\infty(\epsilon, \mathcal{F}, m) = \sup_{X=(x_1,\cdots,x_m) \in \mathcal{X}^m}  \mathcal{N}(\epsilon, \mathcal{F}|_X, \|\cdot \|_\infty) ,
  \end{align*} 
  where $\mathcal{F}|_X  := \{(f(x_1),\cdots, f(x_m)): f\in \mathcal{F}\}$.
\end{definition}

\begin{definition}[Pseudo-dimension \citep{AB1999}]
\label{def-pdim}
For a class $\cF$ of functions $\cX \to \RR$, its pseudo dimension, denoted by $\mathrm{Pdim}(\cF)$, is defined as the largest integer $m$ for which there exist $(x_1,\ldots, x_m, y_1, \ldots, y_m) \in \cX^m \times \RR^m$ such that for any $(b_1, \ldots , b_m) \in \{ 0, 1 \}$, there exists some $f \in \cF$ satisfying that $f(x_i) > y_i \Longleftrightarrow b_1 =1 $ for all $i = 1, \ldots ,m$.
\end{definition}

%We also need one lemma to characterize the structure of hypothesis $\mathcal{F}_n$ and another technical lemma to help us calculate the supremum of empirical process:

\begin{lemma}[Uniform covering number bound for $\cF_n$]
\label{lemma:cover-number-nn}
  For any $\epsilon >0$, 
    \begin{align}
       \log \mathcal{N}_\infty(\epsilon, \mathcal{F}_n, m)  \lesssim  \log(e mM /\epsilon) {\log(\bar L \cdot  \bar N)} (\bar L \cdot  \bar N)^2. 
    \end{align}
\end{lemma}

With the above preparations, we are now ready to bound the expected value of the supremum \eqref{eq:empirical-process-g}.  

\begin{lemma}
\label{lemma:empirical-process}
Assume $C_v < \infty$ and $M \ge 1$. %Let $\mathcal{F}_n=\mathcal{F}_n(d,  \bar{L}, \bar{N}, M)$ be the deep ReLU neural network class. 
Define
\begin{align}
\label{eq:variance-with-truncation}
    V_{n,B} = \frac{(\bar{N}\bar{L})^2 \log (\bar{N} \bar{L}) \log \{2n (B\land \tau)\}}{n}.
\end{align} Then
    \begin{align}
    \label{eq:empircal-process}
       \mathbb{E} \Bigg\{ \sup_{g\in \mathcal{G}_n(\delta)}  \Bigg|\frac{1}{
       n} \sum_{i=1}^n g(X_i, \varepsilon_i) - \mathbb{E} g(X, \varepsilon ) \Bigg|  \Bigg\}     \le c_{33} \Bigg\{\delta \sqrt{(1+C_v)V_{n,B}} +  (\tau \land B)V_{n,B}\Bigg\}
    \end{align} 
    for all $\delta \ge 1/n$ and $\tau \land B\ge 1$, where $c_{33}>0$ is constant independent of $C_v$, $\delta$, $n$, $\bar{N}$, $\bar{L}$, $\tau$ and $B$.
\end{lemma}

We are ready to prove Theorem \ref{thm:generic-bound}.

\begin{proof}[Proof of Theorem \ref{thm:generic-bound}]

The proof is based on the high-level result stated in Lemma \ref{lemma:convergence-rate}. To this end, we first choose appropriate $\delta_*$, $\tilde{f}$, $D$, $\{B_k\}_{k=1}^\infty$ based on $\bar{N}, \bar{L}, n, \tau$ such that conditions (1)-(4) in Lemma \ref{lemma:convergence-rate} is satisfied. To be specific, we let $D$ be that in our main statement, and
\begin{align*}
    \delta_* = \delta_{\mathtt{a}} \lor 16\delta_{\mathtt{b}} \lor 2\delta_{\mathtt{s}} ~~~~~~ \text{and} ~~~~~~ B_k = \omega D^2 v_p^{1/p} V_n^{-1/p}.
\end{align*}

It suffices to consider the case where $D^2 \lor \omega \le n$, otherwise the bound is trivial since $\delta_* \ge 2\delta_{\mathtt{s}} \gtrsim 1$. Through the definition of $\delta_{\mathtt{a}}$, there exists some $\tilde{f}\in \mathcal{F}_{n}(d,\bar{L}, \bar{N}, M)$ such that
\begin{align*}
    \|\tilde{f} - f_0\|_2 \le 2 \delta_\mathtt{a} \le 3\delta_*.
\end{align*} then, we have
\begin{align*}
    \mathcal{R}_\tau(f) - \mathcal{R}_\tau(f_0) &= \mathbb{E} \Bigg[\psi_\tau(\varepsilon) (f_0(X) - \tilde{f}(X)) \\
    &~~~~~~~~~~~~~~~~~~~ + \int_{0}^{f_0(X)-\tilde{f}(X)} 1\{|\varepsilon + t| \le \tau\} (f_0(X) - \tilde{f}(X) - t) \mathrm{d} t\Bigg] \\
    &\overset{(a)}{\le} \frac{1}{2} \sup_{x} \big|\mathbb{E} [\psi_\tau(\varepsilon)|X=x] \big|^2 + \frac{1}{2} \|\tilde{f}-f_0\|^2_2 + \frac{1}{2} \|\tilde{f} - f_0\|^2 \\
    &\overset{(b)}{\le} \frac{1}{2} \left(\frac{v_p}{\tau^{p-1}}\right)^2 + \|\tilde{f} - f_0\|_2^2 \le 4 \delta_{\mathtt{a}}^2 + \frac{1}{2} \delta_{\mathtt{b}}^2 \le 9(\delta_*)^2
\end{align*} where $(a)$ follows from the fact that $xy \le \frac{1}{2} (x^2 + y^2)$, $(b)$ follows from \eqref{bias.control}. This validates condition (1).

For condition (2), applying Proposition~\ref{prop:strong-convex} we see that 
\begin{align*}
    \mathcal{R}_\tau(f) - \mathcal{R}_\tau(f_0) \geq  \frac{1}{8} \|f - f_0\|_2^2  ~~\mbox{ for any }~ f \in \Theta_0^{{\rm c}}(  \delta_*/2) .
\end{align*} 
More generally,  for $f\in \Theta_0^{{\rm c}}(\delta/2)$ with $\delta \geq \delta_*$,  we have
\$
  \mathcal{R}_\tau(f) - \mathcal{R}_\tau(f_0) \geq \frac{1}{32} \delta^2 , 
\$
which certifies condition (2) of Lemma~\ref{lemma:convergence-rate} with $c_{26}=1/32$.
Turning to condition (3),  it follows from Lemma~\ref{lemma:empirical-process} that for all $\delta\ge \delta_n \ge 1/n$ and $B>1$, one has
\begin{align}
	\EE \left[ \sup_{f\in \cF_n \cap \Theta_0(\delta) } | \Delta_{n,B}(f) | \right] \leq C' \Big\{ \delta \sqrt{(1+v_2)V_{n,B}} + (B\land \tau) V_{n,B}\Big\} = \phi_{n}(\delta,B)\label{c6},
\end{align} for $V_{n,B}$ defined in \eqref{eq:variance-with-truncation}, and $C'=c_{33}$. 
It is easy to see that $\phi_n(\alpha \delta, B) \le \alpha \phi_n(\delta, B)$ for all $\alpha \ge 1$ because $\phi_n(\cdot, B)$ is linear function and satisfies $\phi_n(0, B)\ge 0$.

For condition (4), our choice of $B_k$ satisfies $V_{n,B_k} \le 4V_n$, this implies
\begin{align*}
    (B_k \land \tau) V_{n,B_k} \le 4(B_k \land \tau) V_n \le 4 \{(v_p^{1/p} V_n^{1-1/p} \omega D^2) \land (V_n \tau) \}\le 4 D^2\delta_{\mathtt{s}}^2 \le D^2 \delta_*^2
\end{align*} provided $D\ge 1$, together with the fact that $V_n \le D^2\delta_{\mathtt{s}}^2 \le D^2\delta_*^2$, we find
\begin{align*}
    \phi_n( 2^{k} \delta^\dagger, B_k) &\le C' \left(2^k \delta^\dagger \sqrt{(1+v_2) V_{n,B_k}} + (B_k \land \tau) V_{n,B_k}\right) \\
    &\le C' \left(2^k \delta^\dagger D\delta_* + D^2 \delta_*^2 \right) \le C' 2^{2k} (\delta^\dagger)^2.
\end{align*}

Therefore, Lemma \ref{lemma:convergence-rate} implies the bound
\begin{align*}
    &\mathbb{P} \Big[ \exists f\in \mathcal{S}_{n,\tau}(\delta_{\mathtt{opt}}) ~~ \text{s.t.} ~~ \|f-f_0\|_2 \ge C_1 \left(D \delta_* + \delta_{\mathtt{opt}}\right)\Big] \\
    &~~~~~~~~~~\le \sum_{k=1}^\infty 2\exp\left(-\frac{C_2 n (D \delta_* + \delta_{\mathtt{opt}})^2 }{v_2 + B_k \land \tau} 2^{2k}\right) + \frac{C_3 1\{B_k < \tau\} \mathbb{E} \left[|\psi_\tau(\varepsilon)|^p\right]}{B_k^{p-1} 2^{2k} (\delta^\dagger)^2} \\
    &~~~~~~~~~~:= \mathsf{T}_1 + \mathsf{T_2}
\end{align*} for some universal constant $C_1$, $C_2$, $C_3$. 

It follows from the definition of $\delta_{\mathtt{s}}$ that
\begin{align}
\label{eq:tail-critic-raduis}
    \frac{n(D\delta_* + \delta_{\mathtt{opt}})^2}{v_2 + B_k \land \tau} \ge nD^2 \frac{\delta_{\mathtt{s}}^2}{v_2 + B_k \land \tau} \ge nD^2 \frac{ v_2 + (\tau \land \omega (v_p/V_n)^{1/p}) V_n}{v_2 + B_k \land \tau},
\end{align} when $\tau \le \omega D^2 (v_p/V_n)^{1/p}$, we have the r.h.s. of \eqref{eq:tail-critic-raduis} can be lower bounded by
\begin{align*}
 nD^2 \frac{ (v_2 + \tau) V_n}{v_2 + \tau} \ge nD^2 V_n,
\end{align*} which implies
\begin{align*}
    \mathsf{T}_1 \le \sum_{k=1}^\infty 2\exp\left(-C_2 nV_n D^2 2^{2k}\right) &\le \sum_{k=1}^\infty 2\exp\left(-C_2 nV_n D^2 k\right) \\
    &\le 2\left\{1-\exp(-C_2 nV_n D^2)\right\}^{-1} \exp(-C_2 nV_n D^2) \\
    & \le 2(1-e^{-C_2})^{-1} \exp\left(-C_2 nV_n D^2\right)
\end{align*} provided that $nV_n D^2 \ge 1$. Similarly, when $\tau > \omega D^2 (v_p/V_n)^{1/p}$, the r.h.s. of \eqref{eq:tail-critic-raduis} can be lowered bounded by $nV_n$, a similar argument gives
\begin{align*}
    \mathsf{T}_1 \le 2 (1-e^{-C_2})^{-1} \exp(-C_2 nV_n).
\end{align*} Moreover, the moment condition Condition \ref{cond2} implies
\begin{align*}
    \mathsf{T}_2 &\le \frac{v_p}{\omega^{p-1} v_p^{1-1/p} D^{2(p-1)}V_n^{-(p-1)/p} (\delta^\dagger)^2} \sum_{k=1}^\infty 4^{-k} 1\{\tau > \omega D^2 (v_p/V_n)^{1/p}\} \\
    &\le \frac{v_p^{1/p}}{\omega^{p-1} D^{2(p-1)} V_n^{-1+1/p} \delta_{\mathtt{s}}^2 D^2} 1\{\tau > \omega D^2 (v_p/V_n)^{1/p}\} \\
    &\le \frac{1}{\omega^{p-1} D^{2p}} 1\{\tau > \omega D^2(v_p/V_n)^{1/p}\}
\end{align*} where the last inequality follows from the fact that when $\tau > \omega D^2 (v_p/V_n)^{1/p} \ge \omega (v_p/V_n)^{1/p}$, one has $V_n^{-1+1/p} \delta_{\mathtt{s}}^2 \ge V_n^{-1+1/p} V_n^{1-1/p} v_p^{1/p} = v^{1/p}$. Putting these pieces together, we have
\begin{align*}
    \mathbb{P} \Bigg\{ \sup_{f\in \mathcal{S}_{n,\tau}(\delta_{\mathtt{opt}})} \|f-f_0\| \ge C_1 &\Big(D \delta_* + \delta_{\mathtt{opt}}\Big)\Bigg\} \\
    &\lesssim \begin{cases}
    \exp(-C_2 n V_n D^2)  & \tau \le \omega D^2 (v_p/V_n)^{1/p}\\
    \exp(-C_2 n V_n) + \omega^{1-p} D^{-2p} & \tau > \omega D^2 (v_p/V_n)^{1/p}
    \end{cases}.
\end{align*} which completes the proof.
\end{proof}

%\subsection{Proof of Corollary \ref{coro:generic-bound-lse}}

%It follows directly from Theorem \ref{thm:generic-bound} with $\tau = \infty$ and $\omega = 1$. \qed

\subsection{Proof of Theorem \ref{thm:rate-adaptive-huber-upper}}

The proof is based on Theorem \ref{thm:generic-bound} and Proposition \ref{prop:nn-approx} by plugging in our choice of $\tau$ and $NL$. 

We first specify each term in its statement to apply Theorem \ref{thm:generic-bound}. By assumption, the optimization error and the bias can be upper bounded by
\begin{align*}
    \delta_{\mathtt{opt}} \le \delta_{n,\mathtt{AH}} \le C_1 \left(\frac{\log^6 n}{n}\right)^{\frac{\gamma^*\nu^*}{2\gamma^*+\nu^*}} v_p^{\frac{1}{2p-1}}
\end{align*} and
\begin{align*}
    \delta_{\mathtt{b}} \lesssim v_p \tau^{1-p} \lesssim \left(\frac{\log^6 n}{n}\right)^{\frac{\gamma^*}{2\gamma^*+\nu^*} \frac{2(p-1)}{2p-1}} v_p^{1 - \frac{2(p-1)}{2p-1}} = \left(\frac{\log^6 n}{n}\right)^{\frac{\gamma^* \nu^*}{2\gamma^*+\nu^*}} v_p^{\frac{1}{2p-1}} 
\end{align*}respectively. Moreover, it follows from the condition of $\bar{N}$ and $\bar{L}$ in \eqref{eq:depth.width.order.n} and \eqref{depth.width.order} that
\begin{align*}
    V_n = \frac{(\bar{N}\bar{L})^2\log (\bar{N}\bar{L}) \log n}{n} &\lesssim \frac{(NL\log N \log L)^2 \log (N L) \log n}{n} \\
    & \lesssim \frac{(NL)^2 (\log n)^6}{n} \lesssim \left(\frac{\log^6 n}{n}\right)^{\frac{2\gamma^*}{2\gamma^*+\nu^*}},
\end{align*} and
\begin{align*}
    n V_n \gtrsim n \times \frac{(N L)^2  (\log n)^4}{n} \gtrsim (NL)^2 (\log n)^4 \gtrsim n^{\frac{\nu^*}{2\gamma^*+\nu^*}} (\log n)^{\frac{8\gamma^*-2\nu^*}{2\gamma^*+\nu^*}}.
\end{align*}
Our choice of $\tau$ and the order of $V_n$ above indicates that
\begin{align*}
    \frac{\tau}{v_p^{1/p} V_n^{-1/p}} \asymp \left(\frac{\log^6 n}{n}\right)^{\frac{2\gamma^*}{2\gamma^*+\nu^*} \{p^{-1} - (1 - \gamma^*) \}} v_p^{\frac{2}{2p-1} - \frac{1}{p}}\le \left(\frac{\log^6 n}{n}\right)^{\frac{2\gamma^*}{2\gamma^*+\nu^*} \left(\frac{1}{p} -\frac{1}{2p-1} \right)} v_p^{\frac{1}{(2p-1)p}} \ll 1,
\end{align*} which implies $\tau \le v_p^{1/p} V_n^{-1/p}$ for all sufficiently large $n$. The remaining proof proceeds under the condition that $\tau \le (v_p/V_n)^{1/p}$. Letting $\omega = 1$, we have
\begin{align*}
\delta_{\mathtt{s}} = \sqrt{V_n (\tau + v_2)} &\lesssim \left(\frac{\log^6 n}{n}\right)^{\frac{1}{2} \cdot \frac{2\gamma^*}{2\gamma^*+\nu^*} \{1 - (1 - \nu^*)\}} v_p^{\frac{1}{2p-1}} + \left(\frac{\log^6 n}{n}\right)^{\frac{\gamma^*}{2\gamma^*+\nu^*}} \sqrt{v_2} \\
&\lesssim \left(\frac{\log^6 n}{n}\right)^{\frac{\gamma^* \nu^*}{2\gamma^*+\nu^*}} v_p^{\frac{1}{2p-1}} + \left(\frac{\log^6 n}{n}\right)^{\frac{\gamma^*}{2\gamma^*+\nu^*}} \sqrt{v_2} .
\end{align*} At the same time, it follows from Proposition \ref{prop:nn-approx} and our choice of $\bar{L}$ and $\bar{N}$ that
\begin{align*}
    \delta_{\mathtt{a}} \lesssim (NL)^{-2\gamma^*} \lesssim \left(\frac{\log^6 n}{n}\right)^{\frac{\gamma^*\nu^*}{2\gamma^*+\nu^*}} ~~~~~~ \forall ~ f_0 \in \mathcal{H}(d,l,\mathcal{P}).
\end{align*}
Putting these pieces together gives
\begin{align*}
    D \delta_{n,\tau} + \delta_{\mathtt{opt}} \lesssim D \delta_{n,\mathtt{AH}}  ~~~~~~ \forall ~ f_0 \in \mathcal{H}(d,l,\mathcal{P}) ,
\end{align*}
and apply Theorem \ref{thm:generic-bound}, we have for any $f_0 \in \mathcal{H}(d,l,\mathcal{P})$, $(X,\varepsilon)$ satisfying Condition \ref{cond1} and \ref{cond2} that
\begin{align*}
    \mathbb{P} \Bigg\{ \sup_{f\in \mathcal{S}_{n,\tau_n}(\delta_{n,\mathtt{AH}})} &\|f-f_0\|_2 \ge C' D \left(\frac{\log^6 n}{n}\right)^{\frac{\gamma^*\nu^*}{2\gamma^*+\nu^*}} v_p^{\frac{1}{2p-1}}  \Bigg\} \\
    &\lesssim \exp(-D^2 nV_n / C') \lesssim \exp\left(-D^2 n^{\frac{\nu^*}{2\gamma^*+\nu^*}} (\log n)^{\frac{8\gamma^*-2\nu^*}{2\gamma^*+\nu^*}} / C''\right) .
\end{align*} 
This completes the proof.
\qed 

\subsection{Proof of Theorem \ref{thm:rate-lse-upper}}

The proof is almost identical to that of Theorem \ref{thm:rate-adaptive-huber-upper}. When $D\ge n$, the bound is trivial. Therefore, it suffices to prove the bound in the regime $D\in [1,n]$. Under the stated assumptions, one has
\begin{align*}
    V_n \lesssim \left(\frac{\log^6 n}{n}\right)^{\frac{2\gamma^*}{2\gamma^*+\nu^\dagger}}, 
\end{align*} 
which implies that for any $f_0 \in \mathcal{H}(d,l,\mathcal{P})$,
\begin{align*}
    \delta_{n,\infty} = \delta_{\mathtt{a}} \lor v_p^{1/p} V_n^{\frac{1}{2}(1-1/p)} \lesssim v_p^{1/p} \left(\frac{\log^6 n}{n}\right)^{\frac{\gamma^*\nu^\dagger}{2\gamma^*+\nu^\dagger}}.
\end{align*} 
Combining this with the fact that
\begin{align*}
    \exp\big\{ -C n^{1/p} (\bar{N}\bar{L})^{2(1-1/p)}\big\}  \ll \exp(-C n^{1/p}) \ll n^{-2p} \ll {D}^{-2p}
\end{align*} for sufficiently large $n$ and $D\in [1,n]$ completes the proof when $D\in [1,n]$.  \qed 

\subsection{Proof of Theorem \ref{thm:convergence-rate-nn-symmetric}}

We first present an generic bound under the scenario in which Condition \ref{cond1} and \ref{cond3} hold and $c_1 \le \tau\lesssim 1$.

\begin{lemma}
\label{lemma:generic-bound-symmetric}
    Assume Condition \ref{cond1} and \ref{cond3} hold, and $c_1 \le \tau \lesssim 1$. Then, there exists some universal constant $C'$ independent of $\bar{N}, \bar{L}, n, D$ and $f_0$ such that for any $\delta_{\mathtt{opt}} > 0$, 
    \begin{align}
    \label{eq:generic-bound-symmetric-claim}
        \mathbb{P} \left[ \sup_{f\in \mathcal{S}_{n,\tau}(\delta_{\mathtt{opt}})} \|f - f_0\|_2 \ge C'\left\{D\left(\delta_{a} \lor \sqrt{V_n}\right) + \delta_{\mathtt{opt}}\right\}\right] \lesssim e^{-nV_n D^2 / C'}.
    \end{align} where $\delta_{\mathtt{a}}$ and $V_n$ is defined in Theorem \ref{thm:generic-bound}.
\end{lemma}

\begin{proof}[Proof of Theorem~\ref{thm:convergence-rate-nn-symmetric} using Lemma \ref{lemma:generic-bound-symmetric}]
Similar to the proof of Theorem \ref{thm:rate-adaptive-huber-upper}, plugging our choice of $\bar{N}$ and $\bar{L}$ yields
\begin{align*}
    \delta_{\mathtt{a}} \lor \sqrt{V_n} \lesssim \left(\frac{\log^6 n}{n}\right)^{\frac{\gamma^*}{2\gamma^*+1}} ~~~~~~ \forall ~ f_0 \in \mathcal{H}(d,l,\mathcal{P}).
\end{align*} Combining with the fact that $\delta_{\mathtt{opt}} \lesssim \left(\frac{\log^6 n}{n}\right)^{\frac{\gamma^*}{2\gamma^*+1}}$ completes the proof.
\end{proof}

\begin{proof}[Proof of Lemma~\ref{lemma:generic-bound-symmetric}]
Similarly to the proof of Theorem \ref{thm:generic-bound} and \ref{thm:rate-adaptive-huber-upper}, again we rely on Lemma \ref{lemma:convergence-rate} but with a different choice of hyper-parameters. The main difference is that we no longer need a diverging $\tau$ to ensure a negligible bias as $n\to \infty$. 

To use Lemma \ref{lemma:convergence-rate}, we let $D$ be that in the statement of Lemma \ref{lemma:generic-bound-symmetric}, and
\begin{align*}
    \delta_* = 2\delta_{\mathtt{s}} \lor \delta_{\mathtt{a}} ~~~~~~ \text{and} ~~~~~~ B_k = 2\tau.
\end{align*} 
It follows from the definition of $\delta_{\mathtt{a}}$ that there exists some $\tilde{f} \in \mathcal{F}(d, \bar{L}, \bar{N}, M)$ such that
\begin{align*}
    \|\tilde{f} - f_0\|_2 \le 2\delta_{\mathtt{a}} \le 2\delta_{*},
\end{align*} then using Lemma \ref{lemma:prop-g} (1) with Condition \ref{cond3} gives
\begin{align*}
    \mathcal{R}_\tau(\tilde{f}) - \mathcal{R}_\tau(f_0) \le \|f - f_0\|_2^2 \le 4\delta_{*}^2,
\end{align*} which validates condition (1).
For condition (2), Proposition~\ref{prop:strong-convex} ensures that $\mathcal{R}_n(f) - \mathcal{R}_n(f_0) \ge \|f - f_0\|_2^2/4$ for all $f\in \Theta$. In particular, if $f\in \Theta^c(\delta/2)$, this further implies
\begin{align*}
    \mathcal{R}_n(f) - \mathcal{R}_n(f_0) \ge \frac{1}{4} \|f-f_0\|_2^2 \ge \frac{1}{16} \delta^2,
\end{align*} 
thus verifying condition (2) with $c_{26}=2^{-4}$.

Turning to condition (3), it follows from Lemma~\ref{lemma:empirical-process} that, for all $\delta\ge \delta_n \ge 1/n$,
\begin{align*}
	\EE \Bigg[ \sup_{f\in \cF_n \cap \Theta_0(\delta) } | \Delta_{n,B}(f) | \Bigg]  &\le C_1 \left(\delta \sqrt{V_{n,B}} + (B\land \tau) \sqrt{V_{n,B}} \right) \\
 &\le C_2 \left(\delta \sqrt{V_n} + V_n\right) := \phi_n(\delta, B),
\end{align*} provided $\tau \lesssim 1$, where $C_1 = c_{33} (1 + C_v) \le c_{33} (1 + \tau^2) \lesssim 1$. It is also easy to verify that $\phi_n(\alpha \delta, B) \le \alpha \phi_n(\delta, B)$ for any $\alpha \ge 1$. For condition (4), by the above choice of $\phi_n(\delta, B)$ and $\delta_*$, we find
\begin{align*}
    \phi_n(2^k \delta^\dagger, B_k) = C_2 \left( 2^k \delta^\dagger \sqrt{V_n}  + V_n\right) \lesssim 2^k \delta^\dagger \delta_*  + \delta_*^2 \lesssim 2^{2k} (\delta^\dagger)^2
\end{align*} where the first inequality follows from $\sqrt{V_n} \lesssim \delta_*$, and the second inequality follows from $\delta_* \le 2^k \delta^\dagger$. 

Putting these pieces together and applying Lemma \ref{lemma:convergence-rate}, we conclude
\begin{align*}
    \mathbb{P} \Big[\exists f\in \mathcal{S}_{n,\tau}(\delta_{\mathtt{opt}})& ~ \text{s.t.}~ \|f - f_0\|_2 \ge C_3(D\delta_* + \delta_{\mathtt{opt}}) \Big] \\
    &\le \sum_{k=1}^\infty 2\exp\left(- C_4 n(D\delta_*+\delta_{\mathtt{opt}})^2 2^{2k}\right) \\
    &\lesssim \exp\left(- C_4 n(D\delta_*)^2\right) \lesssim \exp\left(-C_5 n D^2 V_n\right)
\end{align*} by noting that $\tau < B_k = 2\tau$. This completes the proof of claim \eqref{eq:generic-bound-symmetric-claim}.

\end{proof}

\subsection{Proof of Lemma \ref{lemma:convergence-rate}}

We will use the following Talagrand's inequality for suprema of empirical processes.

\begin{lemma}[Talagrand inequality for suprema of empirical process]
\label{lemma:talagrand-ineq}
Let $Z_1, Z_2, \ldots, Z_n$ be i.i.d. random variables from some distribution $P$. Let $\mathcal{G}$ be a function class such that $\mathbb{E} \bar{g}(Z)=0$ for any $\bar{g}\in \mathcal{G}$. Suppose $\sup_{\bar{g}\in \mathcal{G}}\|\bar{g}\|_\infty \leq U <\infty$ and let $\sigma^2$ satisfy $\sigma^2 \ge \sup_{\bar{g}\in \mathcal{G}} \mathbb{E} \bar{g}^2(Z)$. Set 
\begin{align*}
	V = \sigma^2 + U \cdot  \mathbb{E}  \Bigg\{ \sup_{\bar{g}\in \mathcal{G}} \bigg| \frac{1}{n} \sum_{j=1}^n \bar{g}(Z_j)\bigg| \Bigg\} .
\end{align*}
Then, for any $x>0$,
\begin{align}
	\mathbb{P}\Bigg\{  \sup_{\bar{g}\in \mathcal{G}} \bigg| \frac{1}{n} \sum_{j=1}^n \bar{g}(Z_j) \bigg| \ge \mathbb{E} \Bigg\{ \sup_{\bar{g}\in \mathcal{G}} \bigg| \frac{1}{n} \sum_{j=1}^n \bar{g}(Z_j)\bigg|\Bigg\} + \sqrt{2V x} + \frac{U x}{3} \Bigg\} \le 2e^{-nx} .
\end{align}
\end{lemma}

\begin{proof}[Proof of Lemma~\ref{lemma:convergence-rate}]

The proof is based on the peeling argument complemented with Talagrand's inequality and truncation argument. 

\noindent {\sc Step 1. Application of Peeling and Truncation.} Denote $\delta = c_{28} \delta^\dagger$. For integers $k = 1, 2, \ldots$, define the donut-shaped sets
\begin{align*}
    \mathcal{A}_{k} = \{f\in\mathcal{F}_n: 2^{k-1} \delta  < \|f-f^*\|_2 \le 2^k \delta  \}, 
\end{align*} 
so that 
\begin{align}
\label{eq:peeling-decomposition}
    \mathbb{P}  \left[ \exists f \in \mathcal{S}_{n,\tau}(\delta_{\mathtt{opt}}) \text{ s.t. }\|f - f^*\|_2 > \delta \right] = \sum_{k=1}^\infty \mathbb{P}  \left[ \exists f \in \mathcal{S}_{n,\tau}(\delta_{\mathtt{opt}}) \cap \mathcal{A}_{k} \right].
\end{align}
It suffices to bound each probability $\mathbb{P}  \left[ \exists f \in \mathcal{S}_{n,\tau}(\delta_{\mathtt{opt}}) \cap \mathcal{A}_{k} \right]$ separately.  By definition, any $f \in \mathcal{A}_{k}$ satisfies $\|f - f^*\|_2 > 2^{k-1} \delta \ge \delta_*$.  It follows from condition (2) that
\begin{align}
\label{eq:proof-lemma-convergence-rate-convexity}
    \mathcal{R}_\tau(f) - \mathcal{R}_\tau(f^*) \ge c_{26} (2^{k} \delta)^2 \ge c_{26} 2^{2k} \delta^2.
\end{align}

Conditioning on $\{f \in \mathcal{S}_{n,\tau}(\delta_{\mathtt{opt}})\}$,  we have 
\begin{align*}
 &\mathcal{R}_\tau(f) - \mathcal{R}_\tau(f^*) \\
 %  &  =  \mathcal{R}_\tau(\hat{f}_n) - \hat{\mathcal{R}}_\tau(\hat{f}_n) + \hat{\mathcal{R}}_\tau(\hat{f}_n) - \mathcal{R}_\tau(f^*) \\
  & =  \mathcal{R}_\tau(f) - \hat{\mathcal{R}}_\tau(f) + \hat{\mathcal{R}}_\tau(f) - \hat{\mathcal{R}}_\tau(\tilde{f}) + \hat{\mathcal{R}}_\tau(\tilde{f}) - \mathcal{R}_\tau(\tilde{f}) + \mathcal{R}_\tau(\tilde{f}) - \mathcal{R}_\tau(f^*) \\
  &  \le  \mathcal{R}_\tau(f) - \hat{\mathcal{R}}_\tau(f) + \delta_{\mathtt{opt}}^2 + \hat{\mathcal{R}}_\tau(\tilde{f}) - \mathcal{R}_\tau(\tilde{f}) + \mathcal{R}_\tau(\tilde{f}) - \mathcal{R}_\tau(f^*).
\end{align*}
Recall the definition of $h_f(X,\varepsilon)$ in \eqref{eq:hf_def}, we have
\begin{align*}
    \mathcal{R}_\tau(f) - \hat{\mathcal{R}}_\tau(f) + &\hat{\mathcal{R}}_\tau(\tilde{f}) - \mathcal{R}_\tau(\tilde{f}) \\
    &= \frac{1}{n} \sum_{i=1}^n \{h_{\tilde{f}}(X_i,\varepsilon_i) - h_f(X_i,\varepsilon_i) \} - \mathbb{E} \left[h_{\tilde{f}}(X_i, \varepsilon_i) - h_f(X_i, \varepsilon_i)\right] \\
    &= \Delta_n(\tilde{f}) - \Delta_n(f)
\end{align*}
%Note that condition (1) ensures $\| f_n - f^* \|_2 \leq  3 \delta_* \leq 2 \delta^\dagger \leq 2^k D \delta_n$ for all $k\ge 1$, which further implies
%\r{Q:  Isn't it simpler to bound directly $\|\hat f - f_n\|$ and save a factor of 2?}
%\begin{align}
%\label{eq:lemma-convergence-rate-eq1}
%    \mathcal{R}_\tau(\hat{f}_n) - \hat{\mathcal{R}}_\tau(\hat{f}_n) + \hat{\mathcal{R}}_\tau(f_n) - \mathcal{R}_\tau(f_n) 
%    \le 2\sup_{f\in \mathcal{F}_n\cap \Theta_*(2^jD\delta_n)} | \Delta_n(f) | . \nn
%\end{align}
Combining the above calculations with the fact $\mathcal{R}_\tau(\tilde{f}) - \mathcal{R}_\tau(f^*) \le 9(\delta_*)^2$ in condition (1), we have
\begin{align*}
    \mathcal{R}_\tau (f) - \mathcal{R}_\tau(f^*) \le 9 (\delta_*)^2 + \delta_{\mathtt{opt}}^2 + \Delta_n(\tilde{f}) - \Delta_n(f). 
\end{align*}
This together with the lower bound \eqref{eq:proof-lemma-convergence-rate-convexity} yields 
\begin{align*}
    \left\{\exists f \in \mathcal{S}(\delta_{\mathtt{opt}}) \cap \mathcal{A}_k\right\} &\subset \left\{ \exists f \in \mathcal{A}_k ~~\text{s.t.} ~~\Delta_n(\tilde{f}) - \Delta(f) \ge c_{26} 2^{2k} \delta^2 - \delta_{\mathtt{opt}} - 9(\delta_*)^2\right\} \\
    &\overset{(a)}{\subset} \left\{ \exists f \in \mathcal{A}_k ~~\text{s.t.} ~~\Delta_n(\tilde{f}) - \Delta(f) \ge \frac{c_{26}}{2} 2^{2k} \delta^2\right\} \\
    &\overset{(b)}{\subset} \left\{ \sup_{f\in \mathcal{A}_k} |\Delta_n(f)| \ge \frac{c_{26}}{4} 2^{2k} \delta^2\right\}
\end{align*} where $(a)$ follows from the definition of $\delta$ that 
\begin{align*}
    c_{26} \frac{1}{2} 2^{2k} \delta^2 \ge c_{26} \times 2 \times \frac{9}{2 c_{26}} \left(D^2 \delta_*^2 + \delta_{\mathtt{opt}} \right)\ge 9\delta_{\mathtt{opt}}^2 + 9(\delta_*)^2 
\end{align*} provided $D\ge 1$, and $(b)$ follows from condition (1) that $\|\tilde{f} - f^*\|_2 \le 3\delta_* \le \delta$. So we have
\begin{align}
\label{eq:peeling-truncation}
\begin{split}
    \mathbb{P}  \left[ \exists f \in \mathcal{S}(\delta_{\mathtt{opt}}) \cap \mathcal{A}_{k} \right] &\le \mathbb{P} \left[ \sup_{f\in \mathcal{A}_k} |\Delta_n(f)| \ge \frac{c_{26}}{4} 2^{2k} \delta^2 \right] \\
    &\le \mathbb{P} \left[ \sup_{f\in \mathcal{A}_k} |\Delta_n(f) - \Delta_{n,B_k}(f)| \ge \frac{c_{26}}{8} 2^{2k} \delta^2 \right] \\    &~~~~~~~~~~+ \mathbb{P} \left[ \sup_{f\in \mathcal{A}_k} |\Delta_{n,B_k}(f)| \ge \frac{c_{26}}{8} 2^{2k} \delta^2 \right] \\
    &:= \mathsf{P}_{1} + \mathsf{P}_{2},
\end{split}
\end{align} where the last inequality follows from triangle inequality and the definition of supremum that
\begin{align*}
    \sup_{f\in \mathcal{A}_k} |\Delta_n(f)| \le \sup_{f\in \mathcal{A}_k} |\Delta_n(f) - \Delta_{n,B}(f)| + \sup_{f\in \mathcal{A}_k} |\Delta_{n,B}(f)|.
\end{align*} 

\noindent {\sc Step 2. Bound the Truncated Part $\mathsf{P}_{1}$. } We first have $\Delta_{n,B_k}(f) = \Delta_n(f)$ as long as $B_k \ge \tau$, which implies $\mathsf{P}_1 = 0$ when $\tau \le B_k$. Therefore, it remains to establish an upper bound on $\mathsf{P}_1$ when $\tau > B_k$. To this end, it follows from the uniform boundedness of $f^*$ and $f_0$ that
\begin{align*}
    \sup_{f \in \mathcal{A}_k} |h_f(X,\varepsilon) - h_f^B(X,\varepsilon)| 
    &= \Big|(f^*(X) - f(X)) \Big(\psi_\tau(\varepsilon + f_0(X) - f^*(X)) \\
    &~~~~~~~~~~~~~~~~~~~~~~~~~~~~~~~~~~~~~~ - \psi_{\tau\land B_k}(\varepsilon + f_0(X) - f^*(X))\Big)\Big| \\
    &\le 2M \left|\psi_\tau(\varepsilon + f_0 - f^*) - \psi_{\tau\land B_k}(\varepsilon + f_0 - f^*)\right| \\
    &\le 2M |\psi_\tau(\varepsilon + f_0 - f^*)| 1\{|\psi_\tau(\varepsilon + f_0 - f^*)| \ge B_k\}.
\end{align*} Then, using Markov inequality yields 
\begin{align*}
    \mathsf{P}_1 &= \mathbb{P} \left[ \sup_{f\in \mathcal{A}_k} |\Delta_n(f) - \Delta_{n,B_k}(f)| \ge \frac{c_{26}}{8} 2^{2k} \delta^2 \right] \\
    &\le \frac{8}{c_{26}} \frac{\mathbb{E} \left[\sup_{f\in \mathcal{A}_k} |\Delta_n(f) - \Delta_{n,B_k}(f)|\right]}{2^{2k} \delta^2}  \\
    &\le \frac{8}{c_{26}} \frac{\mathbb{E} \left[\sup_{f\in \mathcal{A}_k} \left|\frac{1}{n} \sum_{i=1}^n h_f(X_i,\varepsilon) - h_f^B(X_i,\varepsilon_i) - \mathbb{E}[h_f(X,\varepsilon) - h_f^B(X,\varepsilon)]\right|\right]}{2^{2k} \delta^2} \\
    &\le \frac{8}{c_{26}} \frac{\mathbb{E} \left[\sup_{f\in \mathcal{A}_k} \left|\frac{1}{n} \sum_{i=1}^n h_f(X_i,\varepsilon) - h_f^B(X_i,\varepsilon_i)\right| \right] + \mathbb{E} \left[\sup_{f\in \mathcal{A}_k} |h_f(X,\varepsilon) - h_f^B(X,\varepsilon)|\right]}{2^{2k} \delta^2} \\
    &\le \frac{16}{c_{26}} \frac{\mathbb{E}\left[2M |\psi_\tau(\varepsilon + f_0 - f^*)| 1\{|\varepsilon + f_0 - f^*| \ge B_k\}\right]}{2^{2k} \delta^2}
\end{align*} when $\tau>B_k$. At the same time, we also have
\begin{align*}
    \mathbb{E}\Big[2M &|\psi_\tau(\varepsilon + f_0 - f^*)| 1\{|\varepsilon + f_0 - f^*| \ge B_k\}\Big] \\
    &\le \mathbb{E}\left[2M |\psi_\tau(\varepsilon + f_0 - f^*)| \frac{|\psi_\tau(\varepsilon + f_0 - f^*)|^{p-1}}{B_k^{p-1}}\right] \le \frac{2M}{B_k^{p-1}} \mathbb{E} \left[|\psi_\tau(\varepsilon + f_0 - f^*)|^p\right].
\end{align*}Putting these pieces together, we can conclude that
\begin{align*}
    \mathsf{P}_1 \le \frac{32}{c_{26}} \frac{1\{B_k < \tau\}}{B_k^{p-1} 2^{2k} \delta^2} \mathbb{E} \left[|\psi_\tau(\varepsilon + f_0 - f^*)|^p\right]
\end{align*}

\noindent {\sc Step 3. Bound the Bounded Part $\mathsf{P}_2$. } In this step, we apply Talagrand's inequality to bound the supremum. Note that
\begin{align*}
    \Delta_{n,B}(f) = \frac{1}{n} \sum_{i=1}^n h_f^B(X_i,\varepsilon_i) - \mathbb{E} [h_f^B(X,\varepsilon)].
\end{align*} It follows from the definition of $h_f(X,\varepsilon)$ that
\begin{align*}
    \sup_{f\in \mathcal{A}_k} |h_f(X,\varepsilon)| \le \left|\psi_{\tau\land B_k} (\varepsilon + f_0(X) - f^*(X))\right| 2M + \frac{1}{2}(2M)^2 \le 2M(M + \tau \land B_k)
\end{align*} Let $Z_i = (X_i,\varepsilon_i)$ and $\bar{h}_f(Z_i) = h_f^B(X_i, \varepsilon_i) - \mathbb{E}[h_f^B(X,\varepsilon)]$, and note that
\begin{align*}
  \sup_{f\in \mathcal{A}_k} \|\bar{h}\|_\infty \le 4M(M+ \tau \land B_k) =: U .
\end{align*} 
Moreover,
\begin{align*}
    \mathbb{E} \left[ \{h_f^B(X,\varepsilon) \}^2 \right]   &\le 2 \mathbb{E} \Big[ \big\{ \psi^2_{\tau \land B_k}(\varepsilon + (f_0 - f^*)(X)) (f^*-f)(X) \big\} +  \big\{ |(f^* - f)(X)|^2 /2  \big\}^2 \Big] \\
    &\le  2 \mathbb{E} \Big[   \mathbb{E} \big\{ \psi^2_{\tau\land B_k}(\varepsilon + (f_0 - f^*)(X) ) |X \big\}  (f^* - f)(X)^2 +  M^2 (f^* - f)(X)^2 \Big].
\end{align*} 
We claim that \begin{align}
\label{eq:lemma:constant-2nd-moment}
    \mathbb{E} \big\{ \psi^2_\tau(\varepsilon + (f_0 - f^*)(X) ) |X=x \big\} \le C_v + 4M^2.
\end{align} Indeed, since $|\psi_{\tau\land B_k}(t)| = \min(|t|, \tau\land B_k) \le |t|$, if $v_2 < \infty$, then
\begin{align*}
    \mathbb{E} \big\{  \psi^2_\tau(\varepsilon + (f_0 - f^*)(X))   | X=x \big\} &\leq 
    \mathbb{E}(\varepsilon^2 | X=x) + (f_0 - f^*)(X)^2 \\
    &\leq   v_2 + 4M^2 \le C_v + 4M^2.
\end{align*} 
Meanwhile, since $|\psi_{\tau\land B_k}(t)| \le \tau$, then $\mathbb{E} \big\{  \psi^2_{\tau\land B_k}(\varepsilon + (f_0 - f^*)(X))   | X=x \big\} \le \tau^2 \le C_v$. Together the previous two displays imply
\begin{align*}
    \mathbb{E} \left[ \{h_f^B(X,\varepsilon) \}^2 \right] \le (2C_v + 2M^2) \mathbb{E} (f^*-f)(X)^2 = 2(C_v + 5M^2) \|f-f^*\|_2^2,
\end{align*} 
and hence
\begin{align*}
  \sup_{f\in \mathcal{A}_k} \mathbb{E} \left[\{ \bar{h}_f(X,\varepsilon) \}^2 \right] &\le \mathbb{E} \left[ \{h_f^B(X,\varepsilon) \}^2 \right] \\
  &\le \sup_{f\in \mathcal{A}_k} 2(C_v + 5M^2) \|f-f^*\|_2^2 \le 2(C_v + 5M^2) 2^{2k} \delta^2 =: \sigma^2.
\end{align*} 
Define $\mathsf{E}_k = \mathbb{E} [ \sup_{f\in \mathcal{A}_k} |\Delta_{n,B_k}(f)| ]$, it then follows from Talagrand's inequality that
\begin{align*}
    \mathbb{P} \left[\sup_{f\in \mathcal{A}_k} | \Delta_{n,B}(f) | \ge \mathsf{E}_k + \sqrt{2\big( \sigma^2 + U \mathsf{E}_k\big) x} + \frac{U}{3} x\right] \le 2e^{-nx} ~~~~ \forall x\ge 0.
\end{align*}
We will next specify a particular $x=x_k$. It follows from $\sqrt{a + b} \le \sqrt{a} + \sqrt{b}$ and $\sqrt{2ab} \le a + b$ that
\begin{align*}
    \mathsf{E}_k + \sqrt{2\big( \sigma^2 + U \mathsf{E}_k\big) x} + \frac{U}{3}x &\le \mathsf{E}_k + \sqrt{2\sigma^2 x} + \sqrt{2U\mathsf{E}_k x} + \frac{U}{3} x \\
    &\le 2\mathsf{E}_k + \sqrt{2\sigma^2 x} + \frac{4}{3} x
\end{align*}
Provided condition (3), (4) that
\begin{align*}
    \mathsf{E}_k = \mathbb{E} \left[ \sup_{f\in \mathcal{A}_k} |\Delta_{n,B_k}(f)| \right] \le \phi_n(2^k \delta, B_k) \le c_{28} \phi_n(2^k\delta^\dagger, B_k) &\le c_{28} c_{27} 2^{2k} (\delta^\dagger)^2 \\
    &= \frac{c_{27}}{c_{28}} 2^{2k} \delta^2 \le \frac{c_{26}}{48} 2^{2k} \delta^2
\end{align*} together with our choice of $x=x_k$ with
\begin{align*}
    \frac{x_k}{2^{2k}(\delta^\dagger)^2} = \frac{c_{29}}{C_v + M^2 + M(\tau\land B_k)} \le \frac{c_{26}^2 c_{28}^2}{4\times 24^2 \times (C_v + 5M^2)} \bigwedge \frac{c_{26} c_{28}^2}{32\times 4M(M+\tau \land B_k)}
\end{align*} such that $\sqrt{2\sigma^2 x_k} \lor \frac{4U}{3} x \le \frac{c_{26}}{24} 2^{2k}\delta^2$,  we have
\begin{align*}
    \mathsf{E}_k + \sqrt{2\big( \sigma^2 + U \mathsf{E}_k\big) x} + \frac{U}{3}x \le \frac{c_{26}}{8} 2^{2k} \delta^2.
\end{align*} Combining this with the above Talagrand's inequality gives
\begin{align*}
    \mathsf{P}_2 = \mathbb{P} \left[\sup_{f\in \mathcal{A}_k} | \Delta_{n,B}(f) | \ge \frac{c_{26}}{8} 2^{2k} \delta^2\right] \le \exp(-nx_k).
\end{align*}

\noindent {\sc Step 4. Combining the Pieces Together.} Combining the upper bounds on $\mathsf{P}_1$ and $\mathsf{P}_2$ with the peeling argument decomposition \eqref{eq:peeling-decomposition} and truncation argument inequality \eqref{eq:peeling-truncation}, we have
\begin{align*}
    \mathbb{P} \left[\exists f\in \mathcal{S}(\delta_{\mathtt{opt}}) \ge c_{28} \delta^\dagger \right] \le \sum_{k=1}^\infty \exp(-nx_k) + \frac{32}{c_{26} c_{28}^2} \frac{1\{B_k < \tau\}}{B_k^{p-1} 2^{2k} (\delta^\dagger)^2} \mathbb{E} \left[|\psi_\tau(\varepsilon + f_0 - f^*)|^p\right]
\end{align*} which completes the proof.
\end{proof}

\subsection{Proofs of Technical Lemmas}

\begin{proof}[Proof of Lemma~\ref{lemma:prop-g}]
\noindent
{\sc Part (1)}. It follows from the definition of $g$ \eqref{def:g} that 
\begin{align*}
\mathbb{E}\{ g(X,\varepsilon) \}  &= \mathbb{E} \bigg[ \psi_{\tau \land B}(\varepsilon) \Delta_f(X) + \int_{0}^{\Delta_f(X)} 1( |\varepsilon + t| \le \tau )  \{ \Delta_f(X)-t\}  {\rm d} t \bigg] \\
&\le \mathbb{E} \bigg\{   |\psi_{\tau \land B}(\varepsilon)  \Delta_f(X)| + \frac{1}{2} \Delta^2_f(X)   \bigg\}   \\
&\le v_p (\tau\land B)^{1-p}\|f-f_0\|_2 +  \frac{1}{2}\|f-f_0\|_2^2 . 
\end{align*} 
The constraint $f\in \Theta_0^{{\rm c}}(2v_p (\tau\land B)^{1-p})$ implies $v_p (\tau \land B)^{1-p}\|f-f_0\|_2 \le \frac{1}{2} \|f-f_0\|_2^2$,  and hence $\mathbb{E}\{ g(X,\varepsilon) \} \le \|f-f_0\|_2^2$,  as claimed.

Under Condition \ref{cond3},  it has been shown in Proposition \ref{prop:strong-convex} that $\mathbb{E}\{ \psi_\tau(\varepsilon)|X=x \}=0$, $\cP_X$-a.e.  By the tower rule,   now we have
\begin{align*}
    \mathbb{E} \{ g(X,\varepsilon) \} =  \mathbb{E} \bigg[ \int_{0}^{\Delta_f(X)} 1( |\varepsilon + t| \le \tau )  \{ \Delta_f(X)-t \}  {\rm d} t \bigg]  \le \frac{1}{2} \|f-f_0\|_2^2 ~\mbox{ for any } f .
\end{align*}

\noindent
{\sc Part (2)}.   Similarly,  
\begin{align*}
    \mathbb{E} \{ g^2(X,\varepsilon)  \}   &= \mathbb{E}   \bigg|  \psi_{\tau \land B}(\varepsilon) \Delta_f(X) + \int_{0}^{\Delta_f(X)} 1( |\varepsilon + t| \le \tau ) \{ \Delta_f(X)-t \} {\rm d} t \bigg|^2   \\
    &\le \mathbb{E} \bigg\{   2 \psi^2_{\tau \land B}(\varepsilon)   \Delta^2_f(X)  + \frac{1}{2}  \Delta^4_f(X)  \bigg\}  \\
    &\le2  \mathbb{E} \bigg(   \big[  M^2 + \mathbb{E} \big\{  \psi^2_{\tau \land B}(\varepsilon)   | X \big\} \big] \cdot  \Delta^2_f(X) \bigg) ,
\end{align*} 
where the last inequality uses the bound $\| f - f_0 \|_\infty  \leq 2M$. The first claim follow from $\mathbb{E}\big\{   \psi^2_{\tau \land B}(\varepsilon)  | X=x\big\} \le (\tau \land B)^2$. For the second claim,
note that $| \psi_{\tau\land B}(\varepsilon)| = \min( |\varepsilon|, \tau\land B) \leq |\varepsilon|$.
Under Condition~\ref{cond2} with $p\geq 2$, it holds for any $x$ that $\mathbb{E}\big\{   \psi^2_{\tau\land B }(\varepsilon)  | X=x\big\}  \leq \mathbb{E}\big\{ \varepsilon^2 | X=x \big\} \leq v_2$. Combining the pieces yields $ \mathbb{E} \{ g^2(X,\varepsilon)  \}   \le 2(M^2 + v_2  ) \|f-f_0\|_2^2$, as expected.
\end{proof}

\begin{proof}[Proof of Lemma~\ref{lemma:prop-envelop}]
It follows from the defintion of $g$ \eqref{def:g} and triangle inequality that
\begin{align*}
    |g(x,\epsilon)| \le |\psi_{\tau \land B}(\epsilon) \cdot \Delta_f(x)| + \frac{1}{2}  \Delta^2_f(x)  \le 2M \big\{  M +  |\psi_{\tau \land B}(\epsilon)|  \big\} = G(x,\epsilon) 
\end{align*}
holds for any $f \in \cF_n$.  Therefore $G: [0, 1]^d \times \RR \to [0, \infty)$ is an envelop of $\cG_n$.

Recall from the proof of Lemma \ref{lemma:prop-g} that $\EE \{   \psi^2_\tau(\varepsilon ) | X  \} \leq v_2 \land (\tau \land B)^2$.  This together with the Cauchy-Schwarz inequality implies $ \mathbb{E} \{  G^2(X , \varepsilon) \}   \le   8M^2( M^2 + v_2 \land (\tau \land B)^2)$.  The second bound follows from the fact that $|\psi_{\tau\land B}(\varepsilon)| \leq (\tau \land B)$ almost surely.
\end{proof}

\begin{proof}[Proof of Lemma \ref{lemma:cover-number-nn}]
By Theorem~12.2 of \cite{AB1999},   we have that for any $\epsilon>0$,
\$
	\log \mathcal{N}_\infty(\epsilon, \mathcal{F}_n, m) \leq {\rm Pdim}(\cF_n) \cdot  \log(emM/\epsilon) ,
\$
where ${\rm Pdim}(\cF_n)$ denotes the psuedo-dimension of $\cF_n$.  Applying further Theorem~7 of \cite{BHLM2019} yields the bound ${\rm Pdim}(\cF_n)  \lesssim W \bar{L} \log(W)$, where $W$ is the  number of parameters of the network $\cF_n$, satisfying $W= O(\bar{L} \cdot  \bar{N}^2)$.
\end{proof}

In order to prove Lemma \ref{lemma:empirical-process}, we need the following maximal inequality.

\begin{lemma}[A maximal inequality \citep{chernozhukov2014gaussian}]
\label{lemma:maximal-inequality}
Consider the function class $\cG$ of measurable functions $[0, 1]^d \times \RR \to \RR$, to which a measure envelop $G$ is attached.
 Suppose that $\|G\|_2 <\infty$,  and let $\sigma$ be any positive constant such that $\sup_{g\in \mathcal{G}} \mathbb{E} g^2(X, \varepsilon) \le \sigma^2 \le \|G\|_2^2$. 
 Moreover,  define $r=\sigma/\|G\|_2$ and $\overline{G} = \max_{1\leq i\leq n} G(X_i, \varepsilon_i)$. Then 
    \begin{align}
        \mathbb{E} \Bigg\{  \sup_{g\in \mathcal{G} }  \bigg|\frac{1}{\sqrt{n}} \sum_{i=1}^n g(X_i, \varepsilon_i) - \mathbb{E} g(X, \varepsilon )  \bigg|   \Bigg\} \lesssim\|G\|_2 \cdot  J(r, \mathcal{G}, G)  + \frac{ \| \overline{G} \|_2 \cdot   J^2(r, \mathcal{G}, G) }{r^2 \sqrt{n}} ,
    \end{align} 
   where
    \begin{align}
        J(r, \mathcal{G}, G) = \int_{0}^r \sup_{ Q_n } \sqrt{1+\log \mathcal{N}\big(\epsilon \|G\|_2, \mathcal{G}, \|\cdot\|_{2,Q_n} \big) }  \,  {\rm d} \epsilon
    \end{align}
   is the uniform entropy integral, and the supremum is taken over all $n$-discrete probability measures on $[0, 1]^d \times \RR$.
   % $\mathbb{P}(m)$ is all the empirical measure with $m$ points, i.e., $\mathbb{P}(m)=\{\mu = \frac{1}{m} \sum_{i=1}^m \delta_{x_i}(x)\}$.
\end{lemma}

\begin{proof}[Proof of Lemma~\ref{lemma:empirical-process}]

By Lemma \ref{lemma:cover-number-nn},  we have for any $\epsilon \in (0 , M)$ that
\begin{align*}
    \log \mathcal{N}_\infty(\epsilon, \mathcal{F}_n, n) \lesssim \log(M n/\epsilon) \cdot  (\bar{N} \cdot \bar{L})^2  \log (\bar{N} \cdot \bar{L}).
\end{align*} 
Recall the definition of $g$, we have 
\begin{align*}
    |g_f(X,\varepsilon) - g_{f'}(X,\varepsilon)| \le \{(B \land \tau) + M\} |f(X) - f'(X)|
\end{align*} which implies that any $(\epsilon/((B \land \tau) + M))$-net of $\mathcal{F}_n|_X$ under the $\|\cdot\|_\infty$-norm is also an $\epsilon$-net of $\mathcal{G}_n|_X$ under the  $\|\cdot\|_\infty$-norm. Therefore, for any $\epsilon \in (0,n M((B \land \tau) + M))$
\begin{align*}
    \log \mathcal{N}_\infty(\epsilon, \mathcal{G}_n, n) \lesssim  \log \left(\frac{Mn ((B \land \tau) + M)}{\epsilon}\right) (\bar{N}\bar{L})^2 \log(\bar{N} \bar{L}).
\end{align*}
Consequently,
\begin{align*}
	\sup_{Q_n} \sqrt{1+\log \mathcal{N}\big(\epsilon, \mathcal{G}_n, L_2(Q_n) \big) } &\le \sqrt{1 + \log \mathcal{N}_\infty(\epsilon, \mathcal{G}_n, n)} \\
	&\lesssim \left\{ \sqrt{\log \left(\frac{Mn ((B \land \tau) + M)}{\epsilon}\right)}   + 1\right\}  (\bar{N}\bar{L}) \sqrt{\log(\bar{N} \bar{L})}.
\end{align*}

By Lemma \ref{lemma:prop-g} and Lemma \ref{lemma:prop-envelop}, we have 
\begin{align*}
    \sigma^2 = 2(M^2 + C_v) \delta^2, ~~~~~~ \|G\|_2 \le 8 M^2( M^2 + C_v ), ~~~~~~ \| \overline{G} \|_2 \le \|G\|_\infty \leq 2M ( M + \tau \land B).
\end{align*}
Taking $r= \delta/2M$ yields
\begin{align*}
    J(r, \mathcal{G}_n, G) &\lesssim \int_{0}^{\delta/2M} \left\{ \sqrt{\log \left(\frac{Mn ((B \land \tau) + M)}{8 M^2( M^2 + C_v )\epsilon }\right)}   + 1\right\}  (\bar{N}\bar{L}) \sqrt{\log(\bar{N} \bar{L})} {\rm d}\epsilon \\
    &\le (\bar{N}\bar{L}) \sqrt{\log(\bar{N} \bar{L})} \Bigg\{  \frac{\delta}{2M} + \int_{0}^{\delta/2M} \Big(\sqrt{\log \{8n ((B \land \tau) + M)\}} \\
    &~~~~~~~~~~~~~~~~~~~~~~~~~~~~~~~~~~~~~~~~~~~~~~~~~~~~~~~~~~~~~~~~~~ +\sqrt{\log (1/\epsilon)} \lor 0   \Big) {\rm d}\epsilon\Bigg\}  \\
    &\le (\bar{N}\bar{L}) \sqrt{\log(\bar{N} \bar{L})} \frac{\delta}{2M} \left\{ 1+\sqrt{\log \{8n ((B \land \tau) + M)\}} + \sqrt{\log (2M/\delta)} \lor 1 \right\} ,
\end{align*} 
where the last inequality follows from the fact that $\int_{0}^x \sqrt{\log(1/y)} \lor 0  \, {\rm d} y \le x(\sqrt{\log(1/x)} \lor 1)$. We thus conclude that for any $\delta \ge 1/n$ and $\tau \land B\ge 1$,
\begin{align*}
    \frac{J(r, \mathcal{G}_n, G)}{\delta} \lesssim (\bar{N}\bar{L}) \sqrt{\log(\bar{N} \bar{L})} \sqrt{\log (2nM(B\land \tau))} \le \sqrt{n V_{n,B}}.
\end{align*} 

Putting these pieces together and applying Lemma \ref{lemma:maximal-inequality}, we have
\begin{align*}
     & \mathbb{E} \Bigg\{ \sup_{g\in \mathcal{G}_n}  \bigg|\frac{1}{\sqrt n} \sum_{i=1}^n g(X_i, \varepsilon_i) - \mathbb{E} g(X, \varepsilon) \bigg|\Bigg\}   \\
    & \le   \sqrt{8M^2(M^2+C_v)} J(r, \mathcal{G}_n, G) + \frac{2M ( M + \tau \land B) J(r, \mathcal{G}_n, G)^2 }{r^2 \sqrt{n}}\\
   & \lesssim \delta \sqrt{n (1+C_v) V_{n,B}} + \{1 + (\tau\land B)\} V_{n,B} \sqrt{n} \\
   & \lesssim \left\{\delta \sqrt{(1 + C_v)V_{n,B}} + (\tau \land B) V_{n,B} \right\} \sqrt{n}
\end{align*} 
for all $\delta \ge 1/n$ and $\tau \land B \ge 1$. This establishes the claim \eqref{eq:empircal-process}.
\end{proof}
\section{Proofs of Neural Network Approximation}

We include all the proofs of the neural network approximation result in this section. 
We first introduce some additional notations in Section \ref{sec:approx:notation}, which simplify the proofs of neural network approximation. Then, we prove Proposition \ref{prop:nn-approx} in Section~\ref{subsec:roc}. Sections~\ref{thm:approx-n2l2-dim-d} and \ref{thm:approx-noise-dim-d} provide proofs of Theorem \ref{thm:approx-n2l2-dim-d} and Theorem \ref{thm:approx-noise-dim-d}, respectively.

\subsection{Notations about the construction of neural network}

\label{sec:approx:notation}

In this subsection we introduce several notations and simple facts on the construction of neural networks that might be helpful if we want to make a constructive proof of the neural network approximation result. 

\medskip

\noindent \textbf{Representation of neural network.} In the section, we consider a more general representation of neural network, which has output dimension $o \in \mathbb{N}^+$ rather than $1$. In this case, the neural network can be viewed as a function $f:\mathbb{R}^{d} \to \mathbb{R}^{o}$. We write $f \in \mathcal{F}(d, L, N, o)$ if $f$ is a neural network with input dimension $d$, output dimension $o$, depth $\leq L$, and at most $N$ hidden nodes at each layer. Formally, suppose \begin{align}
    f(x) = \mathcal{L}_{L+1} \circ \sigma \circ \mathcal{L}_{L} \circ \sigma \circ \mathcal{L}_{L-1} \circ \sigma \circ \cdots \circ \mathcal{L}_2 \circ \sigma \circ \mathcal{L}_1 (x) \label{eq:nn-rep}
\end{align} 
where $\mathcal{L}_i(x)=W_i x + b_i$ is a linear transformation with $W_i\in \mathbb{R}^{d_i\times d_{i-1}}$, $b_i\in \mathbb{R}^{d_i}$ and $(d_0, d_1, \cdots, d_L, d_{L+1})$. Then $\mathcal{F}(d, L, N, o)$ can be written as 
\begin{align*}
    \mathcal{F}(d, L, N, o) = \big\{  f\text{ is of the form \eqref{eq:nn-rep} with } d_0=d, d_{L+1}=o \text{ and } d_i \le N \text{ for any } i \in \{1,\cdots, L\} \big\}.
\end{align*}

\medskip
\noindent \textbf{Neural network padding.} If $f$ is a neural network with depth between 1 and $L$, and at most $N$ hidden nodes at each layer, then there exists some neural network $g$ with depth $L$ and $N$ hidden nodes at each layer such that $f(x)=g(x)$ for all the input $x$. We refer to this construction as \emph{neural network padding}. The padding with respect to width is trivial. For the padding with respect to depth, assume that the neural network has $L'\ge 1$ hidden layers. We can apply the identity map together with the activation function $L-L'$ times between the first hidden layer and the layer next to it. This will not change $f(x)$, but will increase the number of layers by $L-L'$. Hence $\mathcal{F}(d, L, N, o)$ can also be seen as the set of all neural networks with input dimension $d$, output dimension $o$, depth $L$ and width $N$. From the above discussion, we also have that $\mathcal{F}(d, L, N, o) \subset \mathcal{F}(d, L', N', o)$ if $L'\ge L$ and $N'\ge N$.

\medskip
\noindent \textbf{Network composition.} Suppose $f \in \mathcal{F}(d_1, L_1, N_1, d_2)$ and $g \in \mathcal{F}(d_2, L_2, N_2, d_3)$, we use $h = g \circ f$ to denote the neural network which uses the input of $g$ as the output of $f$. It should be noted that $h$ is a neural network with width $N_1\lor N_2 \lor d_2$ and depth $L_1 + L_2$. This is because we can combine the weight connecting the final hidden layer and the output layer of $g$ and the weight connecting the input layer and the first hidden layer of $f$ as a single weight, i.e. $W_2 (W_1 x + b_1) + b_2 = W'_1 x + b_1'$. 

\medskip
\noindent \textbf{Network parallelization.} Suppose $f_i\in \mathcal{F}(d_i, N_i, L_i, o_i)$ for $i\in \{1,\ldots, k\}$. We use $h=(f_1, \ldots, f_k)$ to denote the neural network that takes $x\in \mathbb{R}^{\sum_{i=1}^k d_i}$ as the input, feeds the entries $x^{(i)}=\big(x_{\sum_{j=1}^{i-1} d_j + 1}, \cdots, x_{\sum_{j=1}^{i} d_j}\big)$ to the $i$-th sub-network $f_i$ that returns $y^{(i)}$, and combines these $y^{(i)}$ as the output. Such an $h$ is a neural network with input dimension $\sum_{i=1}^k d_i$, output dimension $\sum_{i=1}^k o_i$, depth at most $\max_{1\le i\le d} L_i$ and width at most $\sum_{i=1}^d N_i$. Suppose $d_i \le d$, we also use the notation $h=(f_1(x^{(1)}), \ldots, f_k(x^{(k)}))$ to denote the neural network that takes $x\in \mathbb{R}^d$ as the input, and feeds some of its entries $x^{(i)} = ((x)_{j_1},\ldots, (x)_{j_{d_i}})$ as input to the $i$-th subnetwork $f_i$, followed by the same procedure as above. Similarly, we conclude that $h$ is a neural network with input dimension $d$, output dimension $\sum_{i=1}^k o_i$, depth at most $\max_{1\le i\le d} L_i$ and width at most $\sum_{i=1}^d N_i$.

\medskip
\noindent \textbf{Simple functions.} At last, we introduce some simple functions that can be parameterized using ReLU neural networks:

\begin{lemma}[Identity, Absolute value, Min, Max]
\label{lemma:nn-simple}
For any $x, y\in \mathbb{R}$, the following properties hold:
\begin{itemize}
	\item[(1)] $x \in \mathcal{F}(1,1,2,1)$;
	\item[(2)] $|x| \in \mathcal{F}(1,1,2,1)$;
	\item[(3)] $\min(x,y) \in \mathcal{F}(2,1,4,1)$;
	\item[(4)] $\max(x,y) \in \mathcal{F}(2,1,4,1)$.
\end{itemize}
\end{lemma}
\begin{proof}[Proof of Lemma \ref{lemma:nn-simple}]
	For claims (1) and (2), recall that $\sigma(x) = (x)_+$, we thus have $x=\sigma(x)-\sigma(-x)$, $|x|=\sigma(x)+\sigma(-x)$. For claims (3) and (4), note that $\min(x,y)=\frac{1}{2}(x+y - |x-y|)$ and $\max(x,y)=\frac{1}{2}(x+y+|x-y|)$. It follows that \begin{align*}
		\min(x,y) &= \frac{1}{2} \big(x+y - |x-y|\big) \\
				  &= \frac{1}{2} \big(\sigma(x+y) - \sigma(-x-y) - \sigma(x-y) - \sigma(y-x)\big),
	\end{align*} 
	hence proving claim (3). Claim (4) can be similarly proved.
\end{proof}

\begin{lemma}[$\|\cdot\|_\infty$ norm in $\RR^d$]
\label{lemma:sup-norm-vec}
For any $d\geq 1$, the map $\|x\|_\infty : \RR^d \to \RR \in \mathcal{F}(d,\lceil \log_2 d\rceil+1, 2d, 1)$.
\end{lemma}

\begin{proof}[Proof of Lemma \ref{lemma:sup-norm-vec}]

The univariate case $d=1$ is trivial because  $\|x\|_\infty=|x|$ so that the result follows directly from Lemma~\ref{lemma:nn-simple}.

Assume $d\geq 2$ and let $x=(x_1,\ldots, x_d)$. For each $x_i$, by Lemma \ref{lemma:nn-simple} we can construct a neural network $\phi_i(x)\in \mathcal{F}(1,1,2,1)$ such that $\phi_i(x_i) = |x_i|$. Via a parallelization argument, we can then construct $f=(\phi_1, \ldots, \phi_d)$ such that $f(x) = (|x_1|, \ldots , |x_d|)$ and $f\in \mathcal{F}(d, 1, 2d, d)$. Since $\|x\|_\infty=\max_{1\le i\le d} |x_i|$, it suffices to build some $h^{(d)}(x)$ such that $h^{(d)}(x_1,\ldots, x_d) = \max_{1\le i\le d} x_i$. We construct such an $h^{(d)}$ by induction. We claim that $h^{(d)}\in \mathcal{F}(d,\lceil \log_2 d\rceil, 2d, 1)$.

Starting from $d=2$,  Lemma \ref{lemma:nn-simple} ensures that there exists some $h^{(2)}(x_1,x_2)\in \mathcal{F}(2,1,4,1)$ such that $h^{(2)}(x_1,x_2)=\max\{x_1,x_2\}$. Next we consider the general case of $d>2$. If $d=2k$ for some $k\geq 1$, by Lemma \ref{lemma:nn-simple} we can construct $\psi_i(x_{2i-1}, x_{2i}) \in \mathcal{F}(2,1,4,1)$ such that $\psi_i(x_{2i-1}, x_{2i})=x_{2i-1}\lor x_{2i}$ for each $1\leq i\leq k$. Again, using the parallelization argument we can construct $g=(\psi_1,\ldots, \psi_{k}) \in \mathcal{F}(d,1,2d,d)$ such that $g(x_1,\ldots, x_d) = (x_1\lor x_2,\ldots, x_{d-1}\lor x_d)$. This means that  $h^{(d)}$ can be constructed as a composition of $g$ and $h^{(d/2)}$.  Since $h^{(d/2)}\in \mathcal{F}(d/2, \lceil \log_2 (d/2)\rceil, d, 1)$ and $g \in \mathcal{F}(d,1,2d,d/2)$, it follows from a composition argument that $h^{(d)}=g \circ h^{(d/2)} \in \mathcal{F}(d,\lceil \log_2 d\rceil, 2d, 1)$. 

The case $d=2k+1$ for some $k\geq 1$ can be dealt with similarly. By Lemma \ref{lemma:nn-simple}, we can build $\psi_i(x_{2i-1}, x_{2i}) \in \mathcal{F}(2,1,4,1)$ such that $\psi_i(x_{2i-1}, x_{2i})$ for each $1 \leq i\leq k$, and $\psi_{k+1}(x_{2k+1})=x_{k+1} \in \mathcal{F}(1,1,2,1)$. Parallelization argument ensures that we can construct $g=(\psi_1,\ldots, \psi_{k+1}) \in \mathcal{F}(d,1,2d,(d+1)/2)$ such that $g(x_1,\ldots,  x_d)=(x_1\lor x_2,\ldots, x_{d-2}\lor x_{d-1}, x_d)$. By the composition argument, we further have $h^{(d)} = g\circ h^{((d+1)/2)} \in \mathcal{F}(d,\lceil \log_2 d\rceil, 2 d, 1)$.

Now we have  $f(x) = (|x_1|, \ldots , |x_d|) \in \mathcal{F}(d,1,2d,d)$ and $h^{(d)}(x) = \max_{1\le i\le d} x_i  \in \mathcal{F}(d,\lceil \log_2 d\rceil,2d,1)$. 
Finally, taking $f^\dagger=h^{(d)} \circ f$, we have $f^\dagger(x) = \|x\|_\infty$ and $f^\dagger \in \mathcal{F}(d,1+\lceil \log_2 d\rceil, 2d,1)$, as claimed.
\end{proof}

\subsection{Proof of Proposition \ref{prop:nn-approx}}

The proof is based on the following neural network approximation result for $d$-variate $(\beta,C)$ smooth functions.

\begin{lemma}
\label{lemma:nn-approx}
Let $f_0$ be a $d$-variate, $(\beta,C)$-smooth function with $\beta \in [1,\infty)$. For any $N, L\in \mathbb{N}$, there exists a function $\phi$ from a deep ReLU network with depth $c_{34} (L+2) \log_2(4L) + 2d$ and width $c_{35} (N+2)\log_2(8N)$ such that
\begin{align*}
	\|\phi - f\|_\infty	\le c_{36} (NL)^{-2\beta/d} ,
\end{align*} where $c_{34}$--$c_{36}$ are constants that only depend on $\beta$, $d$ and $C$.
\end{lemma}

Such lemma slightly generalize Theorem 1.1 in \cite{lu2020deep} that $\beta$ can be chosen as any positive real number instead of positive integer. The proof is almost identical to the proof of Theorem 1.1 in \cite{lu2020deep}. In the following, we only detailed the parts of the proof that are different.

\begin{proof}[Proof of Lemma~\ref{lemma:nn-approx}]
    We adopt same notations as that in \cite{lu2020deep}. Let $\beta=s+r$ where $s\in \mathbb{N}$ and $r\in (0,1]$. Denote $K=\lfloor N^{1/d}\rfloor^2 \lfloor L^{2/d} \rfloor$, $\bx_{\bbeta} = \bbeta/K$ for $\bbeta \in \{0,\ldots, K-1\}$. It suffices to show that the Taylor expansion of $f$ at these points $\{\bx_{\bbeta}\}$ can approximate $f$ with the error bound of $(NL)^{-2\beta/d}$, that is, to show 
    \begin{align*}
        \bar{f}(\bx) = \sum_{\balpha \in \mathbb{N}^d, \|\balpha\|_1\le s} \frac{\partial f}{\partial {\bx}^{\balpha}}(\bphi(\bx)) \frac{{\bh}^{\balpha}}{\balpha!} 
    \end{align*} with $\bphi(\bx)=\sum_{\bbeta} \bx_{\bbeta} 1\{\bx \in Q_{\bbeta}\}$ and $\bh = \bx - \bphi(\bx)$ satisfies $|\bar{f}(\bx) - f(\bx)| \lesssim (NL)^{-2\beta/d}$. Here $Q_{\bbeta}$ is defined as
    \begin{align*}
        Q_{\bbeta} := \left\{\bx = (x_1,\ldots, x_d), x_i \in \left[\frac{\beta_i}{K}, \frac{\beta_i+1}{K} - \delta 1\{\beta_i + 1 < K\}\right] ~\text{for}~ i=1,\ldots, d\right\}.
    \end{align*} 
    
    Apply Taylor's expansion at the point $\bphi(\bx)$ for $\bx$, we have for some $\xi \in (0,1)$ such that
    \begin{align*}
        f(\bx) = \sum_{\balpha \in \mathbb{N}^d, \|\balpha\|_1 \le s-1} \frac{\partial f}{\partial \bx^{\balpha}}(\bphi(\bx)) \frac{{\bh}^{\balpha}}{\balpha!}+\sum_{\balpha \in \mathbb{N}^d, \|\balpha\|_1 = s} \frac{\partial f}{\partial {\bx}^{\balpha}}(\bphi(\bx) + \xi \bh) \frac{{\bh}^{\balpha}}{\balpha!},
    \end{align*} then it follows from the definition of $(\beta, C)$-smooth function that
    \begin{align*}
        |\bar{f}(\bx) - f(\bx)| &= \sum_{\balpha \in \mathbb{N}^d, \|\balpha\|_1 = s} \frac{\bh^{\balpha}}{\balpha!} \left|\frac{\partial f}{\partial {\bx}^{\balpha}}(\bphi(\bx) + \xi \bh) - \frac{\partial f}{\partial \bx^{\balpha}}(\bphi(\bx))\right| \\
    &\le \sum_{\balpha \in \mathbb{N}^d, \|\balpha\|_1 = s} \frac{\bh^{\balpha}}{\balpha!} \|\xi \bh\|_2^r \le \sum_{\balpha \in \mathbb{N}^d, \|\balpha\|_1 = s} \frac{\|\bh\|_\infty^{\|\balpha\|_1}}{\balpha!} \sqrt{d}^r \|\bh\|_\infty^r \\
    &\lesssim K^{-(s+r)} \lesssim (NL)^{-2\beta/d}.
    \end{align*} The remaining of the proof, which use ReLU neural network to approximate $\bar{g}(\bx)$, is same as Theorem 2.2 and Theorem 1.1 in \cite{lu2020deep}.
\end{proof}

Now we are ready to prove Proposition \ref{prop:nn-approx}. The proof is based on a similar argument to that in {Section~4 of} \cite{kohler2021rate}. The key idea is that {neural network approximation is preserved under compositions}. To be specific, if $f$ and $g$ can be approximated by neural networks $\hat{f}$ and $\hat{g}$, each with an $\|\cdot\|_\infty$-error of $\epsilon$, and $g$ is an $L$-Lipschitz function, then $\hat{g} \circ \hat{f}$ approximates $g\circ f$ with an $\|\cdot\|_\infty$-error of $(L + 1)\epsilon$. The former `$\circ$' refers to the network composition introduced in Section \ref{sec:approx:notation}, and the latter `$\circ$' refers to function composition. Therefore, suppose the target $f_0$ is a composition of several low-dimensional smooth functions $g_1, \ldots, g_k$, then in order to approximate $f_0$ well, we only need to approximate each $g_i$ sufficiently well.

The current argument differs from that in \cite{kohler2021rate} in two aspects.
%There are two main differences compared with the proof in \cite{kohler2021rate}. 
First, we rely on Lemma \ref{lemma:nn-approx} in terms of approximating smooth functions.
Compared with the proof in \cite{kohler2021rate}, it has three benefits. 
(1) It leads to more flexible choices of the depth and width parameters, whereas their proof requires  $N=\mathcal{O}(1)$ or $L = \mathcal{O}(\log N)$.
(2) The sub-networks $\hat{g}_i$ that approximate the components $g_i$ now have the same order of depth and width, i.e., for each $\hat{g}_i$, it has depth $\asymp L \log L$ and width $\asymp N\log N$. (3) The connection between the approximation error of the target function $f_0$ and the approximation errors of its components $g_i$, and the connection between the approximation error of $f_0$ and the network architecture configuration (i.e., depth and width) are described more clearly in our proof. To be specific, assume each $g_i$ is a $d_i$-variate, $(\beta_i,C)$-smooth function. From Lemma \ref{lemma:nn-approx} we see that $\|\hat{g}_i-g_i\|_\infty \lesssim (NL)^{-2\beta_i/d_i}$. Then the approximation error for $f_0$ can be controlled by $\max_{i} (NL)^{-2\beta_i/d_i}=(NL)^{-2\min_{i}\beta_i/d_i}$. Secondly, in order to better evaluate the approximation error based on compositions, it is necessary to truncate each smooth function it approximates because the smooth function approximation result only holds in the region $[0,1]^d$.

\begin{proof}[Proof of Proposition \ref{prop:nn-approx}]

\revise{We define $\beta_{\max} = \sup_{(\beta,t) \in \mathcal{P}} \beta$ and $t_{\max} = \sup_{(\beta, t) \in \mathcal{P}} t$.} Let $h^{(l)}_1(x)=f_0$ for arbitrary $f_0$ that belongs to the function class $\mathcal{H}(d, l,\mathcal{P})$ with fixed integer $l>1$. To obtain \revise{$h^{(l)}_1(x) \in \mathcal{H}(d, l,\mathcal{P})$}, one needs to compute various hierarchical composition models at level $i\in \{ 1, \ldots, l-1\}$, the number of which is denoted by $M_i$. At level $i\in \{1, \ldots, l\}$, let $h^{(i)}_j: \RR^d \to \RR$ be the $j$-th ($j\in \{1, \ldots, M_i\}$) hierarchical composition model. The dependence of $h^{(i)}_j$ on $h^{(i-1)}_{\cdot}$ depends on a $(\beta^{(i)}_j, C)$-smooth function $g^{(i)}_j : \RR^{t^{(i)}_j} \to \RR$ for some $(\beta^{(i)}_j, t^{(i)}_j) \in \mathcal{P}$. Recursively, $h^{(l)}_1(\cdot)$ is defined as 
\begin{align*}
	h^{(i)}_j(x) = g^{(i)}_j \Bigg( h^{(i-1)}_{\sum_{\ell=1}^{j-1} t^{(i)}_\ell + 1} (x), \ldots , h^{(i-1)}_{\sum_{\ell=1}^{j} t^{(i)}_\ell} (x)\Bigg) 
\end{align*}
for $j\in \{ 1,\ldots, M_i\}$ and $i \in \{ 2, \ldots, l\}$, and 
\begin{align*}
	h^{(1)}_j(x) = g_j^{(1)}\Bigg( x_{\pi(\sum_{\ell=1}^{j-1} t^{(1)}_\ell + 1)}, \ldots , x_{\pi(\sum_{\ell=1}^{j} t^{(1)}_\ell)}\Bigg)	
\end{align*}
for some $\pi:\{1,\ldots, M_1\}\to \{1,\ldots, d\}$. The quantities $M_1, \ldots, M_l$ can be defined recursively as 
\begin{align*}
	M_i = \begin{cases}
 		1 & \qquad i=l ,  \\
 		\sum_{j=1}^{M_{i+1}} t^{(i+1)}_j & \qquad i\in \{1,\ldots, l-1\},
 \end{cases}
\end{align*} \revise{then it is easy to see that $M_i\le t_{\max}^{l-i}$ for any $i\in \{1, \ldots, l\}$.}

Moreover, define 
\begin{align*} 
	K_{f_0} = \max_{i\in \{1,\ldots, l\}, j\in \{1,\ldots, M_i\}} \|g^{(i)}_j\|_\infty \lor 1
\end{align*} and let $\mathcal{D}_{j}^{(i)}$ be the domain of function $g^{(i)}_{j}$ under the hierarchical composition model, i.e., 
\begin{align*}
	\mathcal{D}_{j}^{(i)} = \begin{cases}\Bigg\{ \bigg(h^{(i-1)}_{\sum_{\ell=1}^{j-1} t^{(\ell)}_\ell + 1} (x), \ldots , h^{(i-1)}_{\sum_{\ell=1}^{j} t^{(\ell)}_\ell} (x) \bigg): x\in [0,1]^d\Bigg\}& i \in \{2,\ldots, l\} \\
		[0,1]^{t^{(1)}_j} & i=1.
 \end{cases}
\end{align*} 
\revise{It is easy to see that $K_{f_0}$ can be upper bounded by a universal constant $K$ that only depends on $t_{\max}$, $\beta_{\max}$ and $C$. } We thus have $\mathcal{D}_j^{(i)} \subseteq [-K, K]^{t^{(i)}_j}$. Without loss of generality we may assume $\mathcal{D}_j^{(i)}=[-K, K]^{t^{(i)}_j}$; otherwise we can simply extend $g_j^{(i)}$ to the cube $[-K,K]^{t^{(i)}_j}$ and the following analysis remains valid.

\medskip
\noindent {\sc Step 1. Construction of neural network.} In the rest of the proof, for notational convenience we use $\mathcal{F}(N,L)$ to denote a deep ReLU neural network with width $N$ and depth $L$. 

Fix $i \in \{1,\ldots, l\}$ and $j\in \{1, \ldots, M_i\}$. Note that $g^{(i)}_j$ is a $t^{(i)}_j$-variate, $(\beta^{(i)}_j,C)$-smooth function defined on $[-K, K]^{t^{(i)}_j}$. Define the function
\begin{align*}
	\bar{g}^{(i)}_j(z) = g^{(i)}_j(2Kz-K) ~~\text{for}~~ z \in [0,1]^d,
\end{align*} 
so that  $\bar{g}^{(i)}_j$ is a $t^{(i)}_j$-variate, $(\beta^{(i)}_j,KC)$-smooth function defined on $[0,1]^{t^{(i)}_j}$, and satisfies
\begin{align*}
	g^{(i)}_j(z) = \bar{g}^{(i)}_j\bigg(\frac{z+K}{2K}\bigg) ~~\text{for}~~ z\in \mathcal{D}_j^{(i)}.
\end{align*}

For any given $N, L\in \mathbb{N}$, Lemma \ref{lemma:nn-approx} ensures that there exists a function $\tilde{g}_j^{(i)}$ from some deep ReLU neural network $\tilde{g}_j^{(i)}$ with depth \revise{$L'=C_1 L\log_2 L + 2t^{(i)}_{j} \le C_1 L\log_2 L + 2t_{\max}$} and width $N'=C_2 N\log_2 N$ such that 
\begin{align*}
	\Bigg\|\tilde{g}_j^{(i)}\bigg(\frac{z+K}{2K}\bigg) - \bar{g}^{(i)}_j\bigg(\frac{z+K}{2K}\bigg)\Bigg\|_\infty \le C_3 (NL)^{-2\beta_{j}^{(i)}/t_{j}^{(i)}} \le C_3 (NL)^{-2\gamma^*}\text{ for all } z \in \mathcal{D}_j^{(i)}.
\end{align*}
It should be noted that the constants $C_1$, $C_2$ and $C_3$ may depend on the parameters $(\beta^{(i)}_j, t_j^{(i)})$. Since there are only finitely many $g^{(i)}_j$, we can simply choose $(C_1 , C_2)$ to be the largest among all $(C_1 , C_2)$ depending on $(\beta^{(i)}_j, t_j^{(i)})$ because $ \mathcal{F}(N, L) \subseteq \mathcal{F}(N', L')$ if $N\le N'$ and $L\le L'$. Similarly, we also choose $C_3$ to be the largest among all $C_3$'s. \revise{Here $C_1$--$C_3$ are also universal constants that only depend on $\beta_{\max}$ and $t_{\max}$.}

Next, consider a `truncated' version of $\tilde{g}_j^{(i)}$, defined as
\begin{align*}
	\hat{g}_{j}^{(i)}(z) = \max\{\min\{\tilde{g}_{j}^{(i)}(z),K\}, -K\}=\sigma(2K-\sigma(K-\tilde{g}_{j}^{(i)}(z))) - K ,
\end{align*} 
where $\sigma(x) = \max\{x, 0\}$ is the ReLU activation function. 
Note that $\|T_K f-g\|_\infty \le \epsilon$ if $\|g\|_\infty \le K$ and $\|f-g\|_\infty \le \epsilon$. Therefore, we have $\hat{g}_{j}^{(i)} \in \mathcal{F}(N', L'+2)$ and 
\begin{align}
\label{eq:nn-approx-error-single}
	\Bigg\|\hat{g}_j^{(i)}\bigg(\frac{z+K}{2K}\bigg) - \bar{g}^{(i)}_j\bigg(\frac{z+K}{2K}\bigg)\Bigg\|_\infty \le C_3 (NL)^{-2\beta_{j}^{(i)}/t_{j}^{(i)}} \le C_3 (NL)^{-2\gamma^*} ~\text{ for all }~ z \in \mathcal{D}_j^{(i)}.
\end{align} 

Now we are ready to construct a neural network $f^\dagger$ to approximate $f_0=h^{(l)}_1$. To be specific, our construction proceeds recursively as 
\begin{align*}
 	\hat{h}^{(1)}_j(x) = \hat{g}^{(1)}_j\Bigg(\frac{x_{\pi(\sum_{\ell=1}^{j-1} t^{(1)}_\ell + 1)}+K}{2K}, \ldots , \frac{x_{\pi(\sum_{\ell=1}^{j} t^{(1)}_\ell)}+K}{2K}\Bigg)
 \end{align*} and
 \begin{align*}
  	\hat{h}^{(i)}_j(x) = \hat{g}^{(i)}_j\left( \frac{\hat{h}^{(i-1)}_{\sum_{\ell=1}^{j-1} t^{(i)}_\ell + 1} (x)+K}{2K}, \ldots , \frac{\hat{h}^{(i-1)}_{\sum_{\ell=1}^{j} t^{(i)}_\ell} (x)+K}{2K}\right).
 \end{align*} 
 The corresponding composited network, denoted by $\hat{f} = \hat{g}(\alpha_1 \hat{h}_1(x) + \beta_1, \ldots, \alpha_k \hat{h}_k(x) + \beta_k)$, is realized by first applying network composition $L_i \circ \hat{h}_i$ for each $i \in \{1,\ldots, k\}$, where $L_i(x) = \alpha_i x + \beta_i$, followed by network parallelization $(L_1 \circ \hat{h}_1(x), \ldots, L_k \circ \hat{h}_k(x))$, and then followed by network composition $\hat{g} \circ (L_1 \circ \hat{h}_1(x), \ldots, L_k \circ \hat{h}_k(x))$. For $i \in \{ 1, \ldots, k\}$, assume the deep ReLU neural network $\hat{h}_i: \mathbb{R}^d\to\mathbb{R}$ has depth $L_{h_i}$ and width $N_{h_i}$, and the deep ReLU neural network $\hat{g}$ has depth $L_g$ and width $N_g$. Following the discussions on the depth and width of network composition and network parallelization  in Section~\ref{sec:approx:notation}, we conclude that the composited network $\hat{f}$ has depth $(\max L_{h_i}) + L_g$ and width $(\sum_{i=1}^k N_{h_i}) \lor N_g$.

Based on the recursive construction of neural networks, {we set $f^\dagger$ to be $\hat{h}^{(l)}_1$. Now it suffices to calculate the width, depth and approximation error of $\hat{h}^{(l)}_1$.} These quantities will also be calculated recursively. 
 
\medskip
\noindent {\sc Step 2. Specifying width and depth.} The goal is to calculate the width and depth of each $\hat{h}^{(i)}_j$ from $i=1$ to $i=l$. Let ${N}^{(i)}_j$ and $L^{(i)}_j$ be the width and depth of the network $\hat{h}^{(i)}_j$. First, by Lemma \ref{lemma:nn-approx} and the discussion before, for each $j \in \{1, \ldots , M_i\}$, the depth and width satisfy
\begin{align*}
	 L^{(1)}_{j} = C_1 L\log_2 L + 2(t^{(1)}_j+1), \qquad N^{(1)}_{j} = C_2 N \log_2 N.
\end{align*} 

Now suppose we have already calculated the depth and width for all $\hat{h}^{(i-1)}_j$. Then, based on our discussion of the composited network before, for any given $j\in\{1,\ldots, M_i\}$, the depth and width of $\hat{h}^{i}$ satisfy
\begin{align*}
	L^{(i)}_j = \max_{j'\in P(i,j)}L^{(i-1)}_j + C_1 L\log_2 L + 2(\revise{t_{\max}}+1), \qquad N^{(i)}_j = \sum_{j'\in P(i,j)} N^{(i-1)}_{j'} ,
\end{align*} 
where $P(i,j)=\{\sum_{\ell=1}^{j-1} t_{\ell}^{(i)}+1, \ldots , \sum_{\ell=1}^j t_{\ell}^{(i)}\}$. Using the above recursive calculation, the depth of $f^\dagger=\hat{h}^{(l)}_1$ can be written as
\begin{align*}
	\bar{L} = {(l C_1)} L \log_2 L + \revise{(2l) t_{\max}} \le C_4 L \log_2 L + C_5 \le c_3 L \log L,
\end{align*} 
while the depth of $f^\dagger=\hat{h}^{(l)}_1$ can be written as \begin{align*}
	\bar{N} = N^{(l)}_1 \le \Big(\frac{M_1 C_2}{\log 2}\Big) N\log N \le \revise{\underbrace{\Big(\frac{t_{\max}^{l-1} C_2}{\log 2}\Big)}_{c_4}} N\log N.
\end{align*} 

\medskip
\noindent {\sc Step 3. Approximation error.} We claim that \begin{align}
\label{eq:nn-approx-error-est}
	\|\hat{h}^{(i)}_j - h^{(i)}_j\|_\infty \le C_3 ( C\sqrt{\revise{t_{\max}}} + 1)^{i-1} (NL)^{-2\gamma^*}.
\end{align} 
We prove inequality \eqref{eq:nn-approx-error-est} by mathematical induction, starting with the case of $i=1$. By our discussion in Step~1, let $z = \big( x_{\pi(\sum_{\ell=1}^{j-1} t^{(1)}_\ell + 1)}, \ldots, x_{\pi(\sum_{\ell=1}^{j} t^{(1)}_\ell)}\big)$, we have for all $x\in [0,1]^d$ that
\begin{align*}
	|\hat{h}^{(1)}_j(x) - h^{(1)}_j(x)| &= \Bigg|\hat{g}^{(1)}_j\bigg(\frac{z+K}{2K}\bigg) - g_j^{(1)}(z)\Bigg| \\
	&= \Bigg|\hat{g}^{(1)}_j\bigg(\frac{z+K}{2K}\bigg) - \bar{g}_j^{(1)} \bigg(\frac{z+K}{2K} \bigg)\Bigg| \\
	& \leq  C_3 (NL)^{-2\gamma^*} ,
\end{align*}
where the last step follows from \eqref{eq:nn-approx-error-single}.

Suppose \eqref{eq:nn-approx-error-est} holds for $i-1$ and $j\in \{1,\ldots, M_{i-1}\}$. Write $z=\big( h^{(i-1)}_{\sum_{\ell=1}^{j-1} t^{(i)}_\ell + 1} (x), \ldots, h^{(i-1)}_{\sum_{\ell=1}^{j} t^{(i)}_\ell} (x)\big)$ and $\hat{z} = \big( \hat{h}^{(i-1)}_{\sum_{\ell=1}^{j-1} t^{(i)}_\ell + 1} (x), \ldots, \hat{h}^{(i-1)}_{\sum_{\ell=1}^{j} t^{(i)}_\ell} (x)\big)$ for $x\in[0,1]^d$, we have
\begin{align*}
	|\hat{h}^{(i)}_j(x) - h^{(i)}_j(x)| &= \Bigg|\hat{g}^{(i)}_j\bigg(\frac{\hat{z}+K}{2K}\bigg)-g_j^{(i)}(z)\Bigg| \\
	&\leq \Bigg|\hat{g}^{(i)}_j\bigg(\frac{\hat{z}+K}{2K}\bigg)-g_j^{(i)}(\hat{z})\Bigg| + |g_j^{(i)}(\hat{z})-g_j^{(i)}(z)| .
\end{align*}
Together, \eqref{eq:nn-approx-error-single} and the fact that $\hat{z} \in [-K, K]^{t^{(i)}_j}$ imply
\begin{align}
\label{eq:nn-approx-est1}
\Bigg|\hat{g}^{(i)}_j\bigg(\frac{\hat{z}+K}{2K}\bigg)-g_j^{(i)}(\hat{z})\Bigg| = \Bigg|\hat{g}^{(i)}_j\bigg(\frac{\hat{z}+K}{2K}\bigg)-\bar{g}_j^{(i)}\bigg(\frac{\hat{z}+K}{2K}\bigg)\Bigg|	\le C_3 (NL)^{-2\gamma^*}.
\end{align}
Since $g^{(i)}_j$ is at least $C$-Lipschitz, we further have
\begin{align*}
|g_j^{(i)}(\hat{z})-g_j^{(i)}(z)| &\le C \|\hat{z} - z\|_2 \\
	&\le C\sqrt{\revise{t_{\max}}} \|\hat{z} - z\|_\infty \\
	& \leq  C\sqrt{\revise{t_{\max}}} (1 + C\sqrt{\revise{t_{\max}}})^{i-2} C_3 (NL)^{-2\gamma^*} ,
\end{align*} 
where the last inequality follows from the induction. Putting together the pieces, we obtain
\begin{align*}
	|\hat{h}^{(i)}_j(x) - h^{(i)}_j(x)|
	&\leq \Bigg|\hat{g}^{(i)}_j\bigg(\frac{\hat{z}+K}{2K}\bigg)-g_j^{(i)}(\hat{z})\Bigg| + |g_j^{(i)}(\hat{z})-g_j^{(i)}(z)| \\
	&\le C_3 (NL)^{-2\gamma^*} + C_3 C\sqrt{\revise{t_{\max}}} (1 + C\sqrt{\revise{t_{\max}}})^{i-2} (NL)^{-2\gamma^*} \\
	&\le C_3 (1 + C\sqrt{\revise{t_{\max}}})^{i-1} (NL)^{-2\gamma^*}.
\end{align*}
Finally, we conclude that
\begin{align*}
	\|f^\dagger - f_0\|_\infty = \|\hat{h}^{(l)}_1 - h^{(l)}_1\|_\infty \le \underbrace{C_3 ( C\sqrt{\revise{t_{\max}}} + 1)^{l-1}}_{c_5} (NL)^{-2\gamma^*} ,
\end{align*}
as claimed.
\end{proof}

\subsection{Proof of Theorem \ref{thm:approx-n2l2-dim-d}}

To prove theorem \ref{thm:approx-n2l2-dim-d}, we need several technical lemmas in \cite{lu2020deep} to build some basic modules via deep ReLU neural networks. Having these basic modules in hand, we can apply parallelization or composition to construct more complicated functions that we are interested in.

\begin{lemma}[Step function]
\label{lemma:stepfunc2}
	For any $N, L, d\in \mathbb{N}^+$, and $\Delta\in (0,1/3K]$ with $K=\lfloor L^{2/d}\rfloor \lfloor N^{1/d}\rfloor^2$, there exists an ReLU neural network $\phi$ with depth $4L+5$ and width $4\lfloor N^{1/d}\rfloor+3$ such that 
	\begin{align*}
		\phi(x) = k ~~\text{ if }~ x\in [k/K, (k+1)/K-1_{\{k+1<K\}}\Delta]	
	\end{align*} for $k= 0 ,1 , \ldots, K-1$.
\end{lemma}

\begin{lemma}[Point fitting]
\label{lemma:pointfitting}
	For any $N, L\in \mathbb{N}^+$, and $\theta_{i} \in \{0,1\}$ for $i\in\{0,\ldots, N^2L^2-1\}$, there exists a function $\phi: \RR \to \RR$ determined by an ReLU neural network with depth $5L+7$ and width $8N+6$ such that 
	\begin{align*}
		\phi(i) = \theta_i ~~\text{ for }~ i = 0, 1, \ldots, (NL)^2-1 .	
	\end{align*}
\end{lemma}

With the above step function and point fitting modules, we are ready to prove Theorem \ref{thm:approx-n2l2-dim-d}.

\begin{proof}[Proof of Theorem \ref{thm:approx-n2l2-dim-d}]
	Our target point fitting network $f^\dagger$ consists of two modules: the encoder module $f_e(\cdot)$ and the decoder module $f_d(\cdot)$. The encoder module $f_e(\cdot)$ takes $x \in [0, 1]^d$ as input and outputs an integer index $I(\alpha)=\sum_{i=1}^d (\alpha_i-1) K^{i-1}$ with $\alpha$ satisfying $x\in Q_\alpha(\Delta)$. The decoder module takes the index $I(\alpha) \in \{0,\ldots, K^d-1\}$ as input and outputs a value that approximates $y_\alpha$.
	
 \noindent {\sc Step 1. Construct Encoder $f_e(\cdot)$. }
Given any $N, L\in \mathbb{N}^+$, let $\tilde{L}=\lfloor {L}^{1/d}\rfloor^d$ so that $\tilde{L} \le L$ and $\lfloor \tilde{L}^{2/d}\rfloor = \lfloor L^{1/d}\rfloor^{2}$. For any $i\in \{1,\ldots, d\}$, applying Lemma \ref{lemma:stepfunc2} with $L=\tilde{L}$ and $K =\lfloor \tilde{L}^{2/d}\rfloor\lfloor N^{1/d}\rfloor^2=\lfloor L^{1/d}\rfloor^2 \lfloor N^{1/d}\rfloor^2$, 
%$N$ and $\tilde{L}$, for any $\Delta \in (0,1/3K]$ with $K=\lfloor \tilde{L}^{2/d}\rfloor\lfloor N^{1/d}\rfloor^2=\lfloor L^{1/d}\rfloor^2 \lfloor N^{1/d}\rfloor^2$, 
there exists an ReLU neural network $\phi_i$ with depth at most $4\tilde{L}+5 \le 4L+5$ and width $4N+3$ such that 
	\begin{align*}
		\phi_i(x_i) = k ~~\text{ if }~~ x_i \in[k/K, (k+1)/K - 1_{\{k+1<K\}}\Delta]	
	\end{align*} for $k\in \{0,\ldots,K-1\}$.
	Via parallelization, the function $g_1=(\phi_1,\ldots, \phi_d) \in \mathcal{F}(d, 4L+5, (4N+3)d, d) : \mathbb{R}^d\to\mathbb{R}^d$ satisfies
	\begin{align*}
		g_1(x) = (\alpha_1-1, \ldots, \alpha_d-1) ~~\text{ if }~~ x\in Q_\alpha(\delta).
	\end{align*} Moreover, let $g_2(x) = \sum_{i=1}^d x_i K^{i-1}$ so that $g_2 \in \mathcal{F}(d, 0, 0, 1)$. By the composition argument, we have $f = g_2 \circ g_1 \in \mathcal{F}(d, 4L+5, (4N+3)d, 1)$, satisfying
	\begin{align*}
		f_e(x) = I(\alpha) ~~\text{ if }~~ x \in Q_\alpha(\delta) 	
	\end{align*} for all $\alpha \in \{1,\ldots, K\}^d$.
	
\noindent {\sc Step 2. Construct Decoder $f_d(\cdot)$ in Two Ways.} First, note that every $u\in[0,1]$ can be written as $u=\sum_{i=0}^\infty 2^{-i} \theta_i$.
Set $u^{(r)} = \sum_{i=0}^{r} 2^{-i}\theta_i$ with $r=\lceil \log (1/\epsilon \rceil$ such that
	\begin{align*}
		| u - u^{(r)} |	\le \sum_{i=r}^\infty 2^{-i} \theta_i \le \sum_{i=r+1}^\infty 2^{-i} = 2^{-r}.
	\end{align*} Therefore, we have $|u^{(r)}-u| \le \epsilon$. Moreover, for any $u=\sum_{i=0}^s 2^{-i} \theta_i$ with $\theta_i\in\{0,1\}$ and some $s\le r$, we have $u-u^{(r)}=0$.
	
	By the above discussions, we only need to build a neural network to fit $y_\alpha^{(r)}$ for each $\alpha$. To be specific, in this part, we target to build a neural network $f_d(\cdot)$ such that 
	\begin{align}
		\label{eq:thm:approx-l2n2-fit-decoder}
		f_d(I(\alpha)) = y_\alpha^{(r)} ~~\text{ for all }~	\alpha \in \{1,\ldots, K\}^2.
	\end{align}
    Let $S=(\lfloor L^{1/d} \rfloor^d \lfloor N^{1/d}\rfloor^d)^{2}$ and note that $I(\cdot)$ is a bijective map from $\{1,\ldots, K\}^d$ to $\{0,\ldots, S-1\}$. Thus we let $y_\alpha^{(r)}=\sum_{i=0}^r 2^{-i} \theta_{i,I(\alpha)}$. Then for each $i \in \{0,\ldots, r\}$, it follows from Lemma \ref{lemma:pointfitting} that there exists some $\psi_i(\cdot) \in \mathcal{F}(1, 5\lfloor L^{1/d} \rfloor^d + 7, 8\lfloor N^{1/d}\rfloor^d + 6, 1)\subset \mathcal{F}(1,5L+7,8N+6,1)$ such that 
	\begin{align}
	\label{eq:thm:approx-l2n2-fit-psi}
		\psi_i(I(\alpha)) = \theta_{i,I(\alpha)} ~~\text{ for all }~ \alpha \in\{1,\ldots, K\}^d. 
	\end{align}

	Finally, we use $\psi_i(\cdot)$ to construct $f_d(\cdot)$ in two ways.
	
	\noindent
	\emph{Case 1 (Parallel). } We claim that there exists some $f_d \in \mathcal{F}(1, 5L+7, (8N+6)(r+1), 1)$ such that $f_d(I(\alpha))=y_\alpha^{(r)}$. Via the parallelization argument, we have $g = (\psi_0, \cdots, \psi_r) \in \mathcal{F}(1, 5L+7, (8N+6)(r+1), r+1)$. Moreover, we have $h(x)=\sum_{i=0}^r 2^{-i} x_i \in \mathcal{F}(r+1,0,0,1)$. Then using the composition argument we conclude that $f_d^{(1)}=h\circ g \in \mathcal{F}(1,5L+7, (8N+6)(r+1), 1)$ and 
	\begin{align*}
		f_d^{(1)}(x) = \sum_{i=0}^r 2^{-i} \psi_i(x) ~~\text{ for any }~ x.
	\end{align*} 
	Combining this with \eqref	{eq:thm:approx-l2n2-fit-psi}, it is easy to see that $f_d^{(1)}$ meets the requirements in \eqref{eq:thm:approx-l2n2-fit-decoder}.
	
    \noindent
	\emph{Case 2 (Series). } Here we consider a different construction. Since $\text{Id}(x)=x \in \mathcal{F}(1,1,2,1)$, and by the parallelization argument, we have $g_0(x)=(\psi_0(x), \text{Id}(x)) \in \mathcal{F}(1,5L+7, 8N+8, 2)$. 
	Moreover, define
	\begin{align*}
		g_i(x,y) = \begin{cases}
 		(2^{-i}\psi_i(y)+x, \text{Id}(y))& \qquad 1\le i < r , \\
 		(2^{-i}\psi_i(y)+x) & \qquad i=r .
 		\end{cases}
	\end{align*}
	Note that $g_i$ can be constructed by first applying parallelization to $(\psi_i, \text{Id})$, followed by a composition with a linear function.
%	We consider to construct the following components iteratively. To be specific, let
%it should be noted that the first function can be constructed by first applying parallelization of $(\psi_i, \text{Ident})$ and then doing a composition followed by a linear function.
	We thus have $g_i(x,y)\in \mathcal{F}(2, 5L+7, 8N+10, 2)$.
	Finally, by the composition argument and induction, we conclude that $f^{(2)}_d = g_d \circ g_{d-1} \circ \cdots \circ g_0 \in \mathcal{F}(1,(5L+7)(r+1), 8N+10, 1)$ and 
	\begin{align*}
		f^{(2)}_d(x) = \sum_{i=0}^r 2^{-i} \psi_i(x) .
	\end{align*} 
	Combined with the property of $\psi(\cdot)$ in \eqref{eq:thm:approx-l2n2-fit-psi}, our constructed $f^{(2)}_d$ satisfies the requirements in \eqref{eq:thm:approx-l2n2-fit-decoder}.
	
	\noindent
	\noindent {\sc Step 3. Composition of $f_e(\cdot)$ and $f_d(\cdot)$.} To conclude, let $f^\dagger_1 = f_d^{(1)} \circ f_e$ and $f_2^\dagger = f_d^{(2)} \circ f_e$. By  the composition argument, we have
	\begin{align*}
		f^\dagger_1 \in \mathcal{F}(d, 9L+12, (4N+3)d\lor (8N+6)(r+1), 1)	
	\end{align*} and
	\begin{align*}
		f^\dagger_2 \in \mathcal{F}(d, 4L+5+(5L+7)(r+1), (4N+3)d \lor (8N+10), 1) .
	\end{align*} 
	Moreover, for each $s\in \{1 ,2 \}$, 
	\begin{align*}
		f^\dagger_s(x) = f_d(f_e(x)) = f_d^{(s)}(I(\alpha)) = y_\alpha^{(r)} ~~\text{ if }~ x\in Q_\alpha(\Delta)  , \ \ \alpha = 1,\ldots, K. 
	\end{align*}  
	From the discussions in Step 2, the claimed approximation error of $f^\dagger_s$ holds.
\end{proof}

\subsection{Proof of Theorem \ref{thm:approx-noise-dim-d}}

We first prove a weaker version of Theorem \ref{thm:approx-noise-dim-d}.

\begin{proposition}
\label{prop:noise-pre}
	For any given $N, L\in \mathbb{N}^+$, let $K=\lfloor N^{1/d}\rfloor^2 \lfloor L^{1/d}\rfloor^2$. Then for any $\Delta_1 \in (0, 1/3K]$, $\Delta_2 >0$, suppose $(x_\alpha)_{\alpha \in \mathcal{A}}$ is an arbitrary set of  points indexed by $\mathcal{A}=\{1,\ldots, K\}^d$ satisfying $x_\alpha \in Q_\alpha(\Delta_1)$, where $Q_\alpha(\Delta)$ is defined in \eqref{eq:good-region-q-alpha}. Then there exist an ReLU neural network $f_1^\dagger$ with depth $c_{37} L \lceil \log_2 (1/\Delta_2) \rceil$ and width $c_{38} N$ and an ReLU neural network $f_2^\dagger$  with depth $c_{39} L$, $c_{40} N \lceil \log_2 (1/\Delta_2) \rceil$ satisfying 
	\begin{align*}
		f^\dagger_s(x_\alpha) = 1 ~\text{ for all }~ \alpha \in \mathcal{A},
	\end{align*} and \begin{align*}
 		f^\dagger_s(x) = 0 ~\text{ if }~ x\in Q_\alpha(\delta) ~\text{ for some }~ \alpha \in \mathcal{A} ~\text{ subject to }~ \|x - x_\alpha\|_\infty \ge \delta_2 , 	
 \end{align*} 
 where $s=1, 2$.
\end{proposition}

\begin{proof}[Proof of Proposition \ref{prop:noise-pre}]
	Suppose we have already constructed an ReLU neural network $g:\mathbb{R}^d\to\mathbb{R}^d$  with depth $\bar{L}_g$ and width $\bar{N}_g$, satisfying
	\begin{align*}
		g(x) = \hat{x}_\alpha ~~\text{ if }~ x\in \mathcal{Q}_\alpha ,
	\end{align*} 
	where $\|\hat{x}_\alpha - x_\alpha\|_\infty \le \Delta_2/3$ holds for all $\alpha \in \mathcal{A}$. 
	Consider the function
	\begin{align*}
		h(x, y) = \sigma\Big( 2 - (3/\Delta_2)\|x-y\|_\infty \Big) \land 1 , \quad  x, y\in \mathbb{R}^d .
	\end{align*}
	By the composition argument and Lemmas~\ref{lemma:nn-simple} and \ref{lemma:sup-norm-vec}, we have
	\begin{align*}
		h \in \mathcal{F}(2d, \lceil \log_2 d\rceil + 3, 4d, 1).
	\end{align*}
	Next we claim that $f^\dagger(x) = h \circ (g(x), \text{Id}(x))$ is the function of interest, where $\text{Id}(x): \mathbb{R}^d\to\mathbb{R}^d$ is the identify function. Combining the parallelization argument with Lemma \ref{lemma:nn-simple} yields $\text{Id}(x) \in \mathcal{F}(d, 1, 2d, d)$, which further implies
	\begin{align*}
		f^\dagger \in \mathcal{F}(d, \bar{L}_g + \lceil \log_2 d\rceil + 3, (\bar{N}_g + 2d) \lor 4d, 1).	
	\end{align*}

	To prove the claim, note that 
	\begin{align*}
		2 - (3/\Delta_2)\|x-\hat{x}_\alpha\|_\infty \in \begin{cases}
 		[1, 2] & \qquad \text{if } \|x - \hat{x}_\alpha\|_\infty \le \frac{\Delta_2}{3} \\
  		[0, 1) & \qquad \text{if } \frac{\Delta_2}{3} < \|x - \hat{x}_\alpha\|_\infty \le \frac{2\Delta_2}{3} \\
		(-\infty, 0) & \qquad \text{if } \frac{2\Delta_2}{3} < \|x - \hat{x}_\alpha\|_\infty
 	\end{cases}.
	\end{align*} This means $f(x_\alpha) = \sigma(2-(3/\Delta_2) \|x_\alpha - \hat{x}_\alpha \|_\infty) \land 1 = 1$. Moreover, if $x \in Q_\alpha(\Delta_1)$ but $\|x-x_\alpha\|_\infty \ge \Delta_2$, it follows that $\|x-\hat{x}_\alpha\|_\infty \ge \|x-x_\alpha\|_\infty - \|x_\alpha-\hat{x}_\alpha\|_\infty\ge \frac{2\Delta_2}{3}$. This implies $f(x) = \sigma(2-(3/\Delta_2) \|x_\alpha - \hat{x}_\alpha \|_\infty) = 0$.

It remains to show that we can implement $g$ using an ReLU neural network. To be specific, we will use the two different configurations in Theorem \ref{thm:approx-n2l2-dim-d} to construct $g$. Consider first the one that multiplies the $\log(1/\epsilon)$ factor to the depth. For any $N,L\in \mathbb{N}^+$ and $i\in \{1,\ldots, d\}$, applying Theorem \ref{thm:approx-n2l2-dim-d} with $\Delta=\Delta_1$ and $\epsilon=\Delta_2/3$ to the set $\{(x_\alpha)_i\}_{\{\alpha \in \mathcal{A}\}}$, we see that there exists a neural network $\phi_i \in \mathcal{F}(d, C_1 L \log (3/\Delta_2), C_2 N, 1)$ such that 
	\begin{align}
	\label{eq:prop:noise-pre:eq1}
		\forall \alpha \in \mathcal{A}, ~~~~|\phi_i(x) - (x_\alpha)_i| \le \Delta_2/3  ~~\text{ if }~ x\in Q_\alpha(\Delta_1).
	\end{align} 
     By the parallelization argument, $g_1 = (\phi_1,\cdots, \phi_d) \in \mathcal{F}(d, C_1 L \log (3/\Delta_2), C_2 d N, d)$ and $\|g_1(x)-x_\alpha\|_\infty \le \max_{1\le i\le d} |\phi_i(x) - (x_\alpha)_i| \le \Delta_2/3$ if $x\in Q_\alpha(\Delta_1)$. In this case, the final $f_1^\dagger = h \circ (g_1(x), \text{Id}(x))$ has depth at most $C_1 L \log_2 (3/\Delta_2) + \lceil \log_2 d\rceil + 3 \le c_{37} L \log_2 (3/\Delta_2)$ and width at most $(C_2 d N +2d)\lor 4d \le c_{38} N$.
	
	For the construction of $f^\dagger_2$, we follow the same arguments except use the neural network $\phi_i \in \mathcal{F}(d, C'_1 L, C'_2 N \log (3/\Delta_2), 1)$ such that \eqref{eq:prop:noise-pre:eq1} holds for each $i\in \{1,\ldots, d\}$. Therefore, $f_2^\dagger = h \circ (g_2(x), \text{Id}(x))$ satisfies all the requirements and has depth at most $C'_1 L + \log (3/\Delta_2) + 3\le c_{39} L$ and width at most $(C_2' d N \log_2 (3/\Delta_2)+2d)\lor 4d \le c_{40} N \log_2 (3/\Delta_2)$. This completes the proof.
\end{proof}

Now having the constructed module in Proposition \ref{prop:noise-pre}, we are ready to prove Theorem \ref{thm:approx-noise-dim-d}.

\begin{proof}[Proof of Theorem \ref{thm:approx-noise-dim-d}]
Let $N, L\in \mathbb{N}^+$, $\Delta_1 \in (0,1/3K]$, $\Delta_2>0$ be arbitrary, and write $\mathcal{A}=\{1,\ldots, K\}^d$ and \begin{align*}
 	\mathcal{A}^+ = \{\alpha \in \tilde{\mathcal{A}}: y_\alpha = 1\} \qquad \text{and} \qquad \mathcal{A}^- = \{\alpha \in \tilde{\mathcal{A}}: y_\alpha=-1\}.
 \end{align*}
 
We first construct $f_1^\dagger$. Let $X^+=\{x^+_\alpha\}_{\alpha \in \mathcal{A}}$ satisfy $x^+_\alpha = x_\alpha$ if $\alpha \in \mathcal{A}^+$ and we choose any $x^+_\alpha \in Q_\alpha(\Delta_1)$ if $\alpha \notin \mathcal{A}^+$. By Proposition \ref{prop:noise-pre} with given $N$, $L$, $\Delta_1$, $\Delta_2$, and our sample set $X^+$, there exists a neural network $f^+$ with depth at most $C_1 L \log (1/\Delta_2)$ and width at most $C_2 N$ such that
\begin{align*}
	f^+(x_\alpha)=f^+(x^+_\alpha) = 1 ~~\text{ for }~ \alpha \in \mathcal{A}^+
\end{align*} and
\begin{align*}
	f^+(x)=0  ~~\text{ if }~ x\in Q_\alpha(\Delta_1) \text{ with some } \alpha \in \mathcal{A^+} \text{ but } \|x - x_\alpha\|_\infty \ge \delta_2 .
\end{align*}
Now we try to apply Theorem \ref{thm:approx-n2l2-dim-d} with given $N$, $L$, $\Delta=\Delta_1$, $\epsilon=1/2$, and sample set $Y^+=\{y^+_\alpha\}_{\alpha}$ defined as $y^+_\alpha = 1_{\{\alpha \in \mathcal{A}^+\}}$. Because $y^+_\alpha$'s are all in $\{0,1\}$, there exists an ReLU neural network $m^+(x)$ with depth at most $C_3 L$ and width at most $C_4 N$ such that
\begin{align*}
	m^+(x) = 1_{\{\alpha \in \mathcal{A}^+\}} ~~\text{ if }~ x\in Q_\alpha(\Delta_1) 	
\end{align*} 
for all $\alpha \in \mathcal{A}$.

Here we consider the function $g^+(x) = f^+(x) \land m^+(x)$, then we have the following 
\begin{align*}
	f^+(x) \land m^+(x) = \begin{cases}
 	0 & ~~~ \text{ if } x\in Q_\alpha(\Delta_1) \text{ with some } \alpha \notin \mathcal{A}^+ 	\\
 	1 & ~~~ \text{ if } x\in Q_\alpha(\Delta_1) \text{ with some } \alpha \in \mathcal{A}^+\text{ and } x = x_\alpha \\
 	0 & ~~~ \text{ if } x\in Q_\alpha(\Delta_1) \text{ with some } \alpha \in \mathcal{A}^+\text{ but } \|x - x_\alpha\|_\infty \ge \Delta_2
 \end{cases}.
\end{align*}
Combining the parallelization and composition arguments with Lemma \ref{lemma:nn-simple}, we have $g^+(x) \in \mathcal{F}\big(d,(C_1 L\log(1/\Delta_2) \lor C_3 L) + 1, (C_2+C_4) N, 1\big)$.
Similarly, applying Proposition \ref{prop:noise-pre} to the sample set $X^-=\{x^-_\alpha\}_{\alpha \in \mathcal{A}}$ satisfying $x^-_\alpha = x_\alpha$ if $\alpha \in \mathcal{A}^-$, and applying Theorem \ref{thm:approx-n2l2-dim-d} to the sample set $Y^-=\{y^-_\alpha\}_{\alpha\in \mathcal{A}}$ with $y^-_\alpha=1_{\{\alpha \in \mathcal{A}^-\}}$, we can construct $f^-$ and $m^-$ satisfying
\begin{align*}
	g^-(x) = f^-(x) \land m^-(x) = \begin{cases}
 	0 & ~~~ \text{ if } x\in Q_\alpha(\Delta_1) \text{ with some } \alpha \notin \mathcal{A}^- 	\\
 	1 & ~~~ \text{ if } x\in Q_\alpha(\Delta_1) \text{ with some } \alpha \in \mathcal{A}^-\text{ and } x = x_\alpha \\
 	0 & ~~~ \text{ if } x\in Q_\alpha(\Delta_1) \text{ with some } \alpha \in \mathcal{A}^-\text{ but } \|x - x_\alpha\|_\infty \ge \Delta_2
 \end{cases},
\end{align*} and $g^-\in \mathcal{F}\big(d,(C'_1 L\log(1/\Delta_2) \lor C'_3 L) + 1, (C'_2+C'_4) N, 1\big)$ for some constants $C'_1-C'_4$. Now we are able to conclude that the function 
\begin{align*}
	f_1^\dagger(x) = \Big((2g^+(x)-2g^-(x) + u) \lor (-1)\Big) \land 1
\end{align*} satisfies the conditions stated in Theorem \ref{thm:approx-noise-dim-d}. If $x=x_\alpha$ with $\alpha \in \mathcal{A}^+$, we have $(2g^+(x)-2g^-(x) + u) = 2 + u - 0 \ge 1$, which implies $f^\dagger(x)=1$. Meanwhile, if $x = x_\alpha$ with $\alpha \in \mathcal{A}^-$, we have $(2g^+(x)-2g^-(x) + i) = u - 2 \le -1$, implying $f^\dagger(x)=-1$. 

At last, if $x\in Q_\alpha(\Delta_1)$ for some $\alpha \in \mathcal{A}$, and $\|x-x_{\tilde{\alpha}}\|_\infty \ge \Delta_2$ for all the $\tilde{\alpha} \in \tilde{\mathcal{A}}$, then we can divide it into two cases: for the first case, that $\alpha \notin \tilde{\mathcal{A}}$, by the condition that $g^+$ and $g^-$ satisfying, we have $g^+(x)=g^-(x)=0$, this means $(2g^+(x)-2g^-(x) + u)=u$ and $f_1^\dagger(x)=u$ because $u\in [-1,1]$; for the second case, that $\alpha \in \tilde{\mathcal{A}}$, then we have $\|x-x_\alpha\|_\infty\ge \Delta_2$ because it is hold for all the $\tilde{\alpha}\in \tilde{\mathcal{A}}$ by assumption, this gives $g^+(x)=0$ if $\alpha \in \mathcal{A}^+$ or $g^-(x)=0$ if $\alpha \in \mathcal{A}^-$, by using the same reason in the first case for another $g$ function, we can further conclude that $g^+(x)=g^-(x)=0$. So we can conclude $f_1^\dagger(x)=u$.

Finally, let us conclude by specifying the depth and width for $f^\dagger_1$, by using parallelization of $g^+$ and $g^-$ together with the composition of a given ReLU neural network with sum, min, max function implemented by ReLU neural network, we have $f_1^\dagger$ has depth at most
\begin{align*}
	\Big( (C_1 L\log(1/\Delta_2) \lor C_3 L) + 1 \Big) \lor 	\Big((C'_1 L\log(1/\Delta_2) \lor C'_3 L) + 1\Big) + 2 \le c_{22} L \log_2 (1/\Delta_2),
\end{align*} and width at most
\begin{align*}
	(C'_2+C'_4) N + (C_2+C_4) N \le c_{23} N.
\end{align*}

The construction of $f_2^\dagger$ is almost the same except that we choose different configurations to implement $g^+$ and $g^-$.
\end{proof}

\section{Proofs for Section \ref{sec:lb}}

\subsection{Proof of Lemma \ref{lemma:minimax-hcm}}

Let $\mathcal{C}(d,\beta, C)$ be the set of $d$-variate $(\beta,C)$-smooth functions. We will use the following minimax optimal rate of convergence.

\begin{lemma}[Theorem 3.2 in \cite{gyorfi2002distribution}]
\label{lemma:minimax-lower-bound-gaussian}
There exists a positive constant $c_{41}>0$ such that
\begin{align}
\label{eq:minimax-lower-bound-gaussian-noise}
    \liminf_{n\to\infty} \inf_{\hat{f}_n} \sup_{f_0 \in \mathcal{C}(d,\beta, C)} \frac{\mathbb{E}[\|\hat{f}_n - f_0\|_2^2]}{n^{-\frac{2\beta}{2\beta+d}}} \ge c_{41} ,
\end{align} 
where the infimum is taken over all possible estimators based on $n$ i.i.d. observations $(X_1,Y_1),\ldots, (X_n, Y_n)$, which follows the following data generating process 
\begin{align*}
    Y_i = f_0(X_i) + \varepsilon_i ~~~~\text{with}~~~~ \text{ independent } \varepsilon_i \sim \mathcal{N}(0,1)\text{ and } X_i\sim \mathrm{Unif}[0,1]^d.
\end{align*}
\end{lemma}

\begin{proof}[Proof of Lemma~\ref{lemma:minimax-hcm}]
    When $d^* \le d$, we have $\mathcal{C}(d^*,\beta^*, C) \subset \mathcal{H}(d,l,\mathcal{P})$, this implies
    \begin{align*}
        \inf_{\hat{f}_n} \sup_{f_0 \in \mathcal{H}(d, l, \mathcal{P})} \mathbb{E}[\|\hat{f}_n - f_0\|_2^2] \ge \inf_{\hat{f}_n} \sup_{f_0 \in \mathcal{C}(d^*, \beta^*, C)} \mathbb{E}[\|\hat{f}_n - f_0\|_2^2],
    \end{align*} which completes the proof.
\end{proof}

\subsection{Proof of Theorem \ref{thm:neural-network-approx-lower-bound}}

In this part, we consider the following nonparameteric regression task
\begin{align}
\label{eq:dgp-regression-gaussian-noise}
    Y = f_0(X) + \varepsilon, 
\end{align} where $\varepsilon$ is a \emph{standard normal} noise variable, $X\sim \text{Uniform}[0,1]^d$, and the regression function $f_0 \in \mathcal{F}_0$ for some function class with intrinsic dimension adjusted smoothness upper bounded by $\alpha$. The key idea to prove Theorem \ref{thm:neural-network-approx-lower-bound} is that if we can achieve a faster approximation rate, i.e., 
\begin{align*}
    \omega(L, N)\equiv \sup_{f_0\in \mathcal{C}(d,\beta)}  \inf_{f\in \mathcal{F}(d,L,N,1)} \|f-f_0\|_2 = o\left( (N^2L^2\log^{4}(NL))^{-\alpha}\right) ,
\end{align*} 
then we can obtain a better convergence rate than the minimax optimal rate, which is impossible. Before proving our main theorem, we need some results from the nonparameteric regression literature.

Lemma~\ref{lemma:expectation-l2-risk-deco} below provides an upper bound on $\mathbb{E}[\|\hat{f}_n - f_0\|_2^2]$ using a combination of $\omega(L, N)$ and a statistical error term.

\begin{lemma}[Lemma 18 in \cite{kohler2021rate}]
\label{lemma:expectation-l2-risk-deco}
Assume that the response variable $Y$ satisfies $\mathbb{E}\{\exp(c_{42} Y^2)\}<\infty$ for some constant $c_{42}>0$ and the regression function $f_0$ is uniformly bounded. Let $\tilde{f}_n$ be the least squares estimator 
\begin{align*}
    \tilde{f}_n = {\rm argmin}_{f\in \mathcal{F}_n} \frac{1}{n} \sum_{i=1}^n |Y_i - f(X_i)|^2
\end{align*}
constrained on some model class $\mathcal{F}_n$, and let $\hat{f}_n = T_{c_{43} \log n} \tilde{f}_n$ for some constant $c_{43}>0$, where $T_c$ is a truncation operator defined as $(T_c f)(x) = \text{sgn}(f(x)) (|f(x)|\land c)$. Then $\hat{f}_n$ satisfies
\begin{align*}
    \mathbb{E}[\|\hat{f}_n - f_0\|_2^2] \le \frac{c_{44} (\log n)^2 \Big(\log \mathcal{N}_\infty(\frac{1}{  c_{43} n \log(n)}, T_{c_{43} \log(n)}\mathcal{F}_n, n) + 1\Big)}{n} + 2 \inf_{f\in \mathcal{F}_n} \|f-f_0\|_2^2,
\end{align*} where $T_c \mathcal{F} = \{T_c f: f\in \mathcal{F}\}$. 
\end{lemma}

Now we are ready to present the following Proposition, which establishes the lower bound of the approximation error either as $NL\to \infty$ or when $NL\le M$ for some fixed $M>0$.

\begin{proposition}
\label{prop:approx-error-lb-limit}
Let $\mathcal{F}_0$ be a function class with instrinsic dimension adjusted smoothness upper bounded by $\alpha$, and $\omega(L, N) = \sup_{f_0\in \mathcal{F}_0} \inf_{f\in \mathcal{F}(d,L,N,1)} \|f-f_0\|_2$, then we have the following holds for $\omega(L, N)$:

\begin{itemize}
    \item [(i) ] For any $\epsilon>0$, it holds \begin{align}
\label{eq:approx-error-lb-limit}
    \liminf_{NL\to\infty} \frac{\omega(L, N)}{(N^2L^2 \log^{4+\epsilon} (NL))^{-\revise{\alpha}}} = c_{45} > 0 
\end{align} 
for some positive constant $c_{45}$ depending on $\epsilon$.
    \item [(ii) ] For any given $M > 0$, we have
    \begin{align}
    \label{eq:approx-error-lb-finite}
        \liminf_{NL \le M} \omega(L, N) = c_{46} > 0 ,
    \end{align} where $c_{46}$ is a constant that depends on $M$.
\end{itemize}

\end{proposition}
\begin{proof}

\noindent \emph{Part (i). } We prove \eqref{eq:prop:approx-error-lb-limit:omega-rela} by contradiction. %Let $\epsilon>0$ be arbitrary. 
Suppose the LHS of \eqref{eq:approx-error-lb-limit} is 0, then there exists a sequence of $\{L_k, N_k\}_{k=1}^\infty$ satisfying $N_k L_k \to \infty$ and
\begin{align}
\label{eq:prop:approx-error-lb-limit:omega-rela}
    \varrho_k = \frac{\omega(L_k, N_k)}{\big( (N_k L_k)^2 \log^{4+\epsilon} (N_k L_k) \big)^{-\revise{\alpha}}} \to 0.
\end{align} 
Now we choose a sequence of $n_k$ such that
\begin{align}
\label{eq:prop:approx-error-lb-limit:rela-NL-n}
    (N_k L_k)^2 \asymp \frac{n_k^{\revise{\frac{1}{2\alpha + 1}}}}{(\log n_k)^{4+\epsilon}} .
\end{align} 
Since $N_k L_k \to \infty$, we have $n_k \to \infty$.

Now we try to conduct nonparametric regression. In particular, consider the data generating process defined \eqref{eq:dgp-regression-gaussian-noise} with sample size $n_k$, and construct a truncated least squares estimator over the model class $\mathcal{F}_n(d, L_k, N_k, 1)$ as in Lemma \ref{lemma:expectation-l2-risk-deco}. Since $\varepsilon$ is Gaussian and $f_0 \in \revise{\mathcal{F}_0}$ is (uniformly) bounded, $Y=f_0(X) + \varepsilon$ is a sub-Gaussian random variable, i.e. $\mathbb{E} \exp(C_1 Y^2) < \infty$ for some constant $C_1>0$. Moreover, Lemma \ref{lemma:cover-number-nn} implies
\begin{align*}
    \log \mathcal{N}_\infty\Bigg(\frac{1}{n_k c_{43} \log(n_k)}, T_{c_{43} \log(n_k)}\mathcal{F}_n, n_k\Bigg) + 1 \lesssim \log \Big(n_k^2 c_{43} \log(n_k)\Big) (N_k L_k)^2 \log (N_k L_k).
\end{align*}
Then it follows from Lemma \ref{lemma:expectation-l2-risk-deco} that there exists an estimator $\hat{f}_{n_k}$ such that
\begin{align}
\label{eq:prop:approx-error-lb-limit:risk-upperbound}
    \mathbb{E}[\|\hat{f}_{n_k} - f_0\|_2^2] \lesssim \frac{(\log n_k)^3}{n_k} (N_k L_k)^2 \log N_k L_k + \omega^2(L_k, N_k).
\end{align}
Plugging \eqref{eq:prop:approx-error-lb-limit:omega-rela} and \eqref{eq:prop:approx-error-lb-limit:rela-NL-n} into \eqref{eq:prop:approx-error-lb-limit:risk-upperbound} yields
\begin{align*}
    \mathbb{E}[\|\hat{f}_{n_k} - f_0\|_2^2] &\lesssim \frac{(\log n_k)^3}{n_k} (N_k L_k)^2 \log N_k L_k + \varrho_k^2 \Big( (N_k L_k)^2 \log^{4+\epsilon} (N_k L_k) \Big)^{-2\revise{\alpha}} \\
    &\lesssim  \frac{(\log n_k)^4}{n_k} \frac{n_k^{\revise{\frac{1}{2\alpha + 1}}}}{(\log n_k)^{4+\epsilon}} + \varrho_k^2 \big(n_k^{\revise{\frac{1}{2\alpha+1}}} \big)^{-2\revise{\alpha}} \\
    &\lesssim n_k^{-\revise{\frac{2\alpha}{2\alpha+1}}} \Big(\varrho_k^2 + \frac{1}{(\log n_k)^\epsilon}\Big).
\end{align*} 
Hence, there exist some $n_k \to \infty$ and estimator $\hat{f}_{n_k}$ based on i.i.d. samples $\{(X_i, Y_i)\}_{i=1}^{n_k}$ such that 
\begin{align*}
    \lim_{k\to \infty} \frac{\mathbb{E}[\|\hat{f}_{n_k} - f_0\|_2^2]}{n_k^{-\revise{\frac{2\alpha}{2\alpha+1}}}} = \lim_{k\to \infty} \Big(\varrho_k^2 + \frac{1}{(\log n_k)^\epsilon}\Big) = 0.
\end{align*}
This contradicts the minimax lower bound \eqref{eq:minimax-lower-bound-gaussian-noise} stated in Lemma \ref{lemma:minimax-lower-bound-gaussian}, implying that the assumption 
\begin{align*}
    \liminf_{NL\to\infty} \frac{\omega(L, N)}{(N^2L^2 \log^{4+\epsilon} (NL))^{-\revise{\alpha}}} = 0
\end{align*} cannot be true.  This concludes the proof.

\noindent \emph{Part (ii). } The proof of (ii) proceeds in a similar way % is almost the same with the proof of (i).
via a contradiction argument. Suppose the LHS of \eqref{eq:approx-error-lb-finite} is $0$, then there exists a sequence of $\{(L_n, N_n)\}_{n=1}^\infty$ such that $L_n N_n \le M$ and $\omega(L_n, N_n) \to 0$ as $n\to\infty$. Without loss of generality, we assume that $\omega(L_n, N_n) \le \frac{1}{\sqrt{n}}$; otherwise we can choose a sub-sequence of $\{(L_n, N_n)\}_{n=1}^\infty$. Again, we consider the data generating process specified in \eqref{eq:dgp-regression-gaussian-noise}  and the same truncated least squares estimator over model class $\mathcal{F}_n(d, L_n, N_n, 1)$ as in Lemma \ref{lemma:expectation-l2-risk-deco}. The boundedness assumption $L_n N_n \le M$ implies
\begin{align*}
    \log \mathcal{N}_\infty\Bigg(\frac{1}{n c_{43} \log(n)}, T_{c_{43} \log(n)}\mathcal{F}_n, n\Bigg) + 1 &\lesssim \log \Big(n^2 c_{43} \log(n)\Big) (N_n L_n)^2 \log (N_n L_n) \\
    &\lesssim \log n.
\end{align*} 
By Lemma \ref{lemma:expectation-l2-risk-deco}, there exists an estimator $\hat{f}_n$ such that
\begin{align*}
    \mathbb{E}[\|\hat{f}_{n} - f_0\|_2^2] \lesssim \frac{(\log n)^3}{n} +  \omega^2(L_n, N_n) = \frac{(\log n)^3}{n} + \frac{1}{n} \lesssim \frac{(\log n)^3}{n}.
\end{align*} Therefore, we have
\begin{align*}
    \lim_{n\to \infty} \frac{\mathbb{E}[\|\hat{f}_{n} - f_0\|_2^2]}{n^{-\revise{\frac{2\alpha}{2\alpha+1}}}} = \lim_{n\to \infty} (\log n)^3 n^{-\revise{\frac{1}{2\alpha+1}}} = 0,
\end{align*} 
which contradicts the result in Lemma \ref{lemma:minimax-lower-bound-gaussian}. This completes the proof.
\end{proof}

With Proposition \ref{prop:approx-error-lb-limit} in hand, we are ready to prove Theorem \ref{thm:neural-network-approx-lower-bound}.

\begin{proof}[Proof of Theorem \ref{thm:neural-network-approx-lower-bound}]
The goal is to show that there exists some constant $c_{13}>0$ such that
\begin{align*}
    \omega(L, N) \ge c_{13} {(N^2L^2 \log^{5} (NL))^{-\revise{\alpha}}} .
\end{align*}
Applying part (i) of Proposition \ref{prop:approx-error-lb-limit} with $\epsilon=1$, we have
\begin{align*}
    \liminf_{NL\to \infty} \frac{\omega(L, N)}{(N^2L^2 \log^5 NL)^{-\revise{\alpha}}} = C_1 > 0.
\end{align*} 
If $C_1=\infty$, we can choose any $C_1>0$ so that $\liminf_{NL\to \infty} \frac{\omega(L, N)}{(N^2L^2 \log^5 NL)^{-\revise{\alpha}}} \ge C_1$. Then there exists some $M>0$ such that 
\begin{align*}
    \frac{\omega(L, N)}{(N^2L^2 \log^5 NL)^{-\revise{\alpha}}} \ge \frac{C_1}{2}
\end{align*} 
as long as $NL \ge M$. At the same time, part (ii) of Proposition \ref{prop:approx-error-lb-limit}  shows that there exists a positive constant $C_2$ such that
\begin{align*}
    \omega(N, L) \ge C_2 \ge C_2 \frac{(N^2L^2 \log^5 NL)^{-\revise{\alpha}}}{\sup_{NL\le M}(N^2L^2 \log^5 NL)^{-\revise{\alpha}}} = \frac{C_2}{e^{-2\beta/d}} (N^2L^2 \log^5 NL)^{-\revise{\alpha}}
\end{align*} 
provided $e \le NL \le M$. Putting together the pieces, we conclude that
\begin{align*}
    \omega(N, L) \ge \underbrace{\Big(\frac{C_1}{2}\Big)\land \Big(\frac{C_2}{e^{-2\revise{\alpha}}}\Big)}_{c_{13}} (N^2L^2 \log^{5} (NL))^{-\revise{\alpha}}
\end{align*} for all   $N, L$ satisfying $NL \ge e$.
\end{proof}

\subsection{Proof of Proposition \ref{prop:lowerbound-variance}}

We need the following facts from probability theory to prove Proposition \ref{prop:lowerbound-variance}. 

\begin{lemma}[Concentration for Binomial distribution]
\label{lemma:concentration-binomial}
Suppose $X$ is a binomial random variable with parameter $n\in \mathbb{N}^+$ and $p\in (0,1)$. For any $\epsilon\in (0,1)$,
\begin{align*}
	\mathbb{P}\big\{   (1-\epsilon)np\le X \le (1+\epsilon)np \big\}   \ge 1-\frac{3}{np\epsilon^2}.
\end{align*}
\end{lemma}
\begin{proof}[Proof of Lemma \ref{lemma:concentration-binomial}]
	By (3.5) of \cite{feller2008introduction} (page 151), for any $r\ge np$, we have
	\begin{align}
	\label{eq:binomail-tail-bound}
		\mathbb{P}(X \ge r) \le \frac{r(1-p)}{(r-np)^2}	.
	\end{align}
	Choosing $r=(1+\epsilon)np$, this implies
	\begin{align*}
		\mathbb{P}(X \ge (1+\epsilon)np) \le \frac{(1+\epsilon)np (1-p)}{(\epsilon np)^2} = \frac{2(1-p)}{np} \le \frac{1}{np} \frac{1+\epsilon}{\epsilon^2}.
	\end{align*}
	Moreover, note that
	\begin{align*}
		\mathbb{P}\big( X \le (1-\epsilon) np \big) = \mathbb{P}\big( n - X \ge n-(1-\epsilon)np \big)	.
	\end{align*} 
	Here $n-X$ is a Bernoulli random variable with parameter $n$ and $(1-p)$. Therefore, using \eqref{eq:binomail-tail-bound} with $r=n(1-(1-\epsilon) p) \ge n(1-p)$ gives
	\begin{align*}
			\mathbb{P}\big( n - X \ge n-(1-\epsilon)np \big) \le \frac{n(1-(1-\epsilon) p)p}{[n(1-(1-\epsilon) p)-n(1-p)]^2} = \frac{np(1-(1-\epsilon)p)}{(\epsilon np)^2} \le \frac{1}{np} \frac{1}{\epsilon^2}.
	\end{align*}
	Putting these pieces together, we have
	\begin{align*}
		\mathbb{P}\big( (1-\epsilon)np\le X \le (1+\epsilon)np \big) &= 1 - \mathbb{P}(X > (1+\epsilon)np) - \mathbb{P}\big(X < (1-\epsilon)np \big)\\
			&\ge 1 - \frac{1}{np} \frac{\epsilon+2}{\epsilon^2} \ge  1 - \frac{3}{np\epsilon^2},
	\end{align*}
	as claimed.
\end{proof}

\begin{lemma}[Concentration for number of boxes containing balls]
\label{lemma:concentration-ball}
Suppose we throw the $m$ balls into the $n$ boxes one by one independently with equal probability, and let $Z$ be the number of boxes that have at least one ball. Then we have $\mathbb{E}[Z]=n \{ 1-(1-1/n)^m \}$ and
\begin{align*}
	\mathbb{P}(|Z - \mathbb{E} Z | \ge t) \le \exp(-2t^2/m)
\end{align*}
\end{lemma}

\begin{proof}[Proof of Lemma \ref{lemma:concentration-ball}]
We first calculate $\mathbb{E}[Z]$. Let $Z_k$ be the number of boxes that contains at least one ball after $k$ balls have been thrown. By convention, $Z_0=0$. For any $k\geq 1$, because the balls are thrown into boxes independently with equal probability, given $Z_k$, the next ball has a probability of $\frac{n-Z_k}{n}$ being thrown into an empty box, leading to $Z_{k+1}=Z_{k}+1$, and has a probability of $\frac{Z_k}{n}$ being thrown into a box that already contains at least one ball, leading to $Z_{k+1}=Z_k$. We thus have
\begin{align}
\label{eq:lemma:recursive}
\mathbb{E}[Z_{k+1}]=\mathbb{E}\Bigg[Z_k + \frac{n-Z_k}{n} \Bigg]= (1-1/n)\mathbb{E}[Z_k] + 1.
\end{align} 
We prove by induction that $\mathbb{E}[Z_k] = n[1-(1-1/n)^k]$. For $k=0$, $\mathbb{E}[Z_0]=0=n \{ 1-(1/n)^0 \}$. If this holds for a general $k\geq 1$, the recursive equation \eqref{eq:lemma:recursive} yields
\begin{align*}
	\mathbb{E}[Z_{k+1}] &= (1-1/n) n \{ 1-(1-1/n)^{k} \}  +1 \\
                            &= n \{ 1-1/n - (1-1/n)^{k+1} + 1/n \} = n \{ 1-(1-1/n)^{k+1} \} .
\end{align*} This proves the first claim that $\mathbb{E}[Z] = n\{ 1-(1-1/n)^m \}$.

Let $X_i \in \{1,\ldots, n\}$ denote the index of the box where the $i$-th ball is thrown, and write $Z=f(X_1,\ldots, X_n)$ for some function $f$. Note that for any $i\in \{1,\ldots, m\}$, $|f(X_1,\ldots, X_i, \ldots, X_n) - f(X_1,\ldots, X_i', \ldots, X_n)| \le 1$. By McDiarmid's inequality, we have
\begin{align*}
\mathbb{P}(|f(X_1,\ldots, X_m) - \mathbb{E}f(X_1,\ldots, X_m)| \ge t) \le \exp(-2t^2/m).
\end{align*}
This completes the proof.
\end{proof}

Now we are ready to prove Proposition \ref{prop:lowerbound-variance}. The key is the neural network approximation ability stated in Theorem \ref{thm:approx-noise-dim-d} that neural network is able to fit arbitrary values 
%tune its value arbitrarily
at ``uniformly located" $\Theta(N^2L^2)$ points while staying as a constant in most areas.

\begin{proof}[Proof of Proposition \ref{prop:lowerbound-variance}]

\noindent {\sc Step 1. Construct Noise Distribution.} We first construct the distribution of the noise $\varepsilon$. Let $S=(\tilde{N}\tilde{L})^{2d}$ for some $\tilde{N}, \tilde{L} \in \mathbb{N}^+$ to be determined, and assume $S \le n$. Then let $\varepsilon$ be a discrete random variable independent of $X$, satisfying
%Consider that the i.i.d. noise $\varepsilon_i \sim \varepsilon$ satisfies the distribution
	\begin{align*}
		\varepsilon	= \begin{cases}
			\big(\frac{n}{S}\big)^{1/p} & \text{with probability } \frac{S}{2n}, \\
		-\big(\frac{n}{S}\big)^{1/p} & \text{with probability } \frac{S}{2n}, \\
		0 & \text{with probability}~ 1 - \frac{S}{n}.
 		\end{cases}	
	\end{align*} 
	%and $\varepsilon$ is independent of the covariate vector $X$. 
	It is easy to see that
	\begin{align*}
		\mathbb{E}[\varepsilon|X=x] = \mathbb{E}[\varepsilon] = 0	\qquad \text{ and } \qquad \mathbb{E}[|\varepsilon|^p|X=x] = 2\cdot \frac{S}{2n} \cdot \frac{n}{S} = 1.
	\end{align*} 
	Moreover, the above $\varepsilon$ is symmetric, thus implying $f_0 = f_{0,\tau}$.

\noindent {\sc Step 2. Construct ``Good" Event $\mathcal{E}$.} Let $X\sim \text{Uniform}([0,1]^d)$ and $M_1$ be the number of non-zero $\varepsilon_i$'s. Without loss of generality,   assume  $\varepsilon_1,\ldots,\varepsilon_{M_1}$ are non-zero and $\varepsilon_{M_1+1}=\varepsilon_{M_1+2}=\cdots= \varepsilon_n=0$. Then $X_1,\ldots, X_{M_1}$ and $X_{M_1+1}, \ldots, X_{n}$ are the corresponding covariate vectors. Let $K=(\tilde{N} \tilde{L})^2$, and
	\begin{align*}
		Q_\alpha(\Delta) = \{x=(x_1,\ldots, x_d): (\alpha_i-1)/K \le x_i \le \alpha_i/K - \Delta\}.
	\end{align*} 
	We also define the set of the indexes $\alpha$ such that each hypercube $Q_\alpha(\Delta)$ contains at least one $X_i$ with $i \in \{1,\ldots, M_1\}$, that is,
	\begin{align*}
		\mathcal{A} = \Big\{\alpha \in \{1,\ldots,K\}^d: \exists i \in \{1,\ldots, M_1\} \text{ s.t. } X_i \in Q_\alpha(1/(n^2K)) \Big\}.
	\end{align*} 
	Set $M_2 = |\mathcal{A}|$.

    Define the event $\mathcal{E}=\bigcap_{i=1}^4 \mathcal{E}_i$, where 
	\begin{align*}
		\mathcal{E}_1 &= \big\{0.5 S \le M_1 \le 1.5 S\big\}, \\
		\mathcal{E}_2 &= \Big\{ \forall i\in \{1,\ldots, n\}, X_i \in Q_\alpha(1/(n^2K)) \text{ for some } \alpha \in \{1,\ldots, K\}^d \Big\},\\
		\mathcal{E}_3 &= \big\{M_2 \ge S/8\big\}, \\
		\mathcal{E}_4 &= \big\{\|X_i-X_j\|_\infty \ge 1/(2n^3)\text{ for all } i\neq j\big\}.
	\end{align*} 
	In the following we aim to show that $\mathcal{E}$ occurs with high probability if $S$ and $n$ are sufficiently large.
	
	We first consider   event $\mathcal{E}_1$. Note that $M_1$ is a Binomial random variable with parameter $n$ and $S/n$. Applying  Lemma \ref{lemma:concentration-binomial} with $\epsilon=1/2$, we have
	\begin{align*}
		\mathbb{P}(\mathcal{E}_1) = \mathbb{P}( 0.5 S \le M_1 \le 1.5 S) \ge 1-\frac{12}{S}.
	\end{align*}
    	For $\mathcal{E}_2$, since $X_1,\ldots, X_n$ are i.i.d. from $ \text{Uniform}([0,1]^d)$, it follows from the union bound that
	\begin{align*}
		\mathbb{P}(\mathcal{E}_2) &\ge 1 - \sum_{i=1}^n \mathbb{P}\Big\{ X_i \notin Q_\alpha(1/(n^3K)) \text{ for all } \alpha \in \{1,\ldots, K\}^d \Big\} \\
		&\ge 1 - n \mathbb{P}\Big\{X_1 \notin Q_{(1,\ldots,1)}(1/(n^2K)) \Big| X_1 \in [0,1/K]^d\Big\} \\
		&\ge 1 - n \cdot \Bigg(1-\frac{\big(\frac{1}{K} - \frac{1}{n^2K} \big)^d}{(1/K)^d} \Bigg)\\
		&=1-n\Bigg(1-\bigg(1-\frac{1}{n^2}\bigg)^d\Bigg) \\
		&\ge 1-\frac{d}{n} ,
	\end{align*} where the last inequality follows from the fact that $(1-x)^d \ge 1-dx$ for $x\in [0,1]$ with $x=1/n^2$.

	For $\mathcal{E}_3$, we need to bound the probability $\mathbb{P}(\mathcal{E}_3|\mathcal{E}_1\cap \mathcal{E}_2)$ from below. Conditioned on $\mathcal{E}_1\cap \mathcal{E}_2$,  $X_1, X_2, \ldots, X_{M_1}$ are independent and uniformly distributed on 
	$\bigcup_{\alpha \in \{1,\cdots, K\}^d} Q_\alpha(1/(n^2K))$, and the probability that $X_1$ lies in $\mathcal{Q}_\alpha(1/(n^2K))$ is the same for each $\alpha$. Then we can apply Lemma \ref{lemma:concentration-ball} to provide an lower bound for $M_2$. To this end, note that $X_1,\ldots, X_{M_1}$ can be viewed as $M_1$ balls, and $Q_\alpha(1/(n^2K))$ with $\alpha \in \{1,\ldots, K\}^d$ can be treated as $K^d=S$ boxes. Hence, Lemma \ref{lemma:concentration-ball} directly implies a concentration result for $M_2$. Specifically, we have
	\begin{align*}
		\mathbb{E}[M_2] = S \big\{ 1 - (1-1/S)^{M_1}\big\} \ge S \big\{ 1 - (1-1/S)^{S/2}\big\}.
	\end{align*} 
	Because $\lim_{x\to \infty} (1-1/x)^x=1/e$,   there exists some constant $C_1$ such that for any $x\ge C_1$, $|(1-1/x)^x - 1/e| \le (9/16-1/e)$, which implies $(1-1/x)^x \le 9/16$ for $x\ge C_1$. Consequently,
	\begin{align*} 
		\mathbb{E}[M_2] = S \big\{ 1-(1-1/S)^{S/2} \big\} \ge S(1-\sqrt{9/16}) = S/4
	\end{align*} 
	as long as $S \ge C_1$. This together with the tail probability in Lemma \ref{lemma:concentration-ball} with $t=S/8$ yields  
	\begin{align*}
		\mathbb{P}(M_2 \le  S/8  |\mathcal{E}_1\cap\mathcal{E}_2) \le \mathbb{P}(M_2 \le \mathbb{E}[M_2] -  S/8  | \mathcal{E}_1\cap\mathcal{E}_2) \le \exp\bigg(-\frac{S^2}{32M_1} \bigg) \le e^{-S/48},
	\end{align*} where the last inequality follows from the fact that  $M_1 \le 1.5S$ conditioned on $\mathcal{E}_1$. We thus conclude that $
	\mathbb{P}(\mathcal{E}_3 |\mathcal{E}_1\cap \mathcal{E}_2) \ge 1-e^{-S/48}$. 
	
	Turning to $\mathcal{E}_4$,  applying the union bound yields
	\begin{align*}
		\mathbb{P}(\mathcal{E}_4^c) &\le \sum_{i\neq j} \mathbb{P} \big\{\|X_i-X_j\|_\infty \le 1/(2n^3)\big\} \\
		&\le \frac{n(n-1)}{2} \mathbb{P}\big\{ \|X_1-X_2\|_\infty \le 1/(2n^3) \big\} \\
		&\le \frac{n(n-1)}{2} \bigg(\frac{2}{2n^3}\bigg)^d \le \frac{1}{n} .
	\end{align*}
   Putting together the pieces we obtain
	\begin{align*}
		\mathbb{P}(\mathcal{E}) &\ge \mathbb{P}(\mathcal{E}_1\cap \mathcal{E}_2 \cap \mathcal{E}_3) + \mathbb{P}(\mathcal{E}_4) - 1 \\
		&= \mathbb{P}(\mathcal{E}_3|\mathcal{E}_1\cap \mathcal{E}_2) \mathbb{P}(\mathcal{E}_1\cap \mathcal{E}_2) + \mathbb{P}(\mathcal{E}_4) - 1\\
		&\ge (1-e^{-S/48})(1- d/n -6/S) - \frac{1}{n} \\
		&\ge 1-e^{-S/48} - \frac{d+1}{n} - \frac{6}{S}.
	\end{align*} Moreover, the following properties hold conditioned on $\mathcal{E}$:

\begin{itemize}
	\item[1.] $S/8 \le M_2 \le M_1 \le 1.5 S$;
	\item[2.] $\inf_{i\neq j} \| X_i - X_j \|_\infty \geq 1/(2n^3)$.
	%The pairwise $\|\cdot\|_\infty$ distance between all the $X_i$ are lower-bounded by $1/(2n^3)$.	
\end{itemize}

\noindent {\sc Step 3. Construct Neural Network $\tilde{f}_n$.} Here we use our neural network approximation result Theorem \ref{thm:approx-noise-dim-d} to construct $\tilde{f}_n$. For each $\alpha \in \mathcal{A}$, we can choose arbitrary $X_i \in Q_\alpha(1/(n^2K))$ with corresponding $|\varepsilon_i|=(n/S)^{1/p}$ and construct our point $(x_\alpha, y_\alpha)$ to be $x_\alpha=X_i$ and $y_\alpha = \text{sgn}(\varepsilon_i)$, where $\text{sgn}(x)=1$ if $x > 0$ and $\text{sgn}(x)=-1$ if $x<0$. Let $I_\mathcal{A}$ be the index set that contains all the index $i$ we selected for all the $\alpha\in \mathcal{A}$. Now we apply Theorem \ref{thm:approx-noise-dim-d} with $\Delta_1=1/(n^2K)$,  $\Delta_2=1/(2n^3)$ and $u$ that we will specify later, then there exist some $\tilde{f}_n\in \mathcal{F}_n(d, C_2 \tilde{L}^d \log_2 n, C_3 \tilde{N}^d, 1)$, or $\tilde{f}_n \in \mathcal{F}_n(d, C'_2 \tilde{L}^d, C'_3 \tilde{N}^d \log_2 n)$, such that 
\begin{align*}
	\tilde{f}_n(X_i) = \text{sgn}(\varepsilon_i)  ~\text{ for all }~ i \in I_{\mathcal{A}},
\end{align*} and for any $x \in Q=\bigcup_{\alpha \in \{1,\ldots, K\}^d} Q_\alpha(1/(n^2K))$, we have
\begin{align}
	\tilde{f}_n(x) = u ~~\text{ if }~ \|x-x_\alpha\|_\infty \ge 1/(2n^3) ~\text{ for all }~ \alpha \in \mathcal{A}.
\label{eq:def-tilde-f-u}
\end{align}

This means that if $\mathcal{E}$ occurs, 
%that all the $X_i$ are in $Q$, and all the pairwise $\|\cdot\|_\infty$ distance between all the $X_i$ are lower bounded by $1/(2n^3)$, we have
\begin{align*}
	\tilde{f}_n(X_i) = u ~\text{ for all }~ i \notin I_{\mathcal{A}} .
\end{align*}
The remaining proof of Step 3 proceeds conditioned on $\mathcal{E}$. Using a second-order Taylor expansion of $\ell_\tau(\cdot)$, for any $f$ we have 
\begin{align}
\label{eq:prop-lb-eq1}
\begin{split}
	\hat{\mathcal{R}}_\tau(f) - \hat{\mathcal{R}}_\tau(f_0) &= \frac{1}{n} \sum_{i=1}^n \{ \ell_\tau(\varepsilon_i - f(X_i)) - \ell_\tau(\varepsilon_i) \} \\
 &\le \frac{1}{n} \sum_{i=1}^n \bigg\{  -\psi_\tau(\varepsilon_i) f(X_i) + \frac{1}{2}  f^2(X_i) \bigg\} .
 \end{split}
\end{align} 
By the definition of $M_1$ and $I_\mathcal{A}$, we have
\begin{align*}
 -\sum_{i=1}^n \psi_\tau(\varepsilon_i) \tilde{f}_n(X_i) 
&= \bigg\{ \sum_{i\in I_\mathcal{A}} -\psi_\tau(\varepsilon_i) \tilde{f}_n(X_i) + \sum_{i\notin I_\mathcal{A}, i \le M_1} -\psi_\tau(\varepsilon_i) \tilde{f}_n(X_i) \\
&~~~~~~~~~~~~~~~~ + \sum_{i > M_1} -\psi_\tau(\varepsilon_i) \tilde{f}_n(X_i)\bigg\}   \\
&\le -M_2 \big\{  \tau \land (n/S)^{1/p}\big\}   + (M_1-M_2) u \big\{ \tau \land (n/S)^{1/p} \big\} + (n-M_1) \cdot 0\\
&\le \big\{  - S/ 8 + (2-1/8) S u\big\}  \big\{ \tau \land (n/S)^{1/p} \big\} \\
&= \{ (2-1/8) u - 1/8\} S \big\{ \tau \land (n/S)^{1/p} \big\}.
\end{align*}
On the other hand,  
\begin{align*}
	\sum_{i=1}^n \tilde{f}_n(X_i)^2 = \sum_{i\in I_\mathcal{A}} \tilde{f}_n(X_i)^2 + \sum_{i\notin I_\mathcal{A}} \tilde{f}_n(X_i)^2 = M_2 + (n-M_2) u \le S + nu^2.
\end{align*} 
Taking $f=\tilde{f}_n$ in \eqref{eq:prop-lb-eq1}, it follows from the above  inequalities that
\begin{align*}
	\hat{\mathcal{R}}_\tau(\tilde{f}_n) - \hat{\mathcal{R}}_\tau(f_0) \le \frac{S}{n} \big\{ \tau \land (n/S)^{1/p} \big\} \big\{ (2-1/8) u - 1/8 \big\} + \frac{S}{n} + u^2.
\end{align*} 
Suppose $S/n \le (4/225)^2$, because $p\ge 2$, we   choose
\begin{align*}
u = \sqrt{\frac{1}{32} \frac{S}{n}  \big\{ \tau \land (n/S)^{1/p} \big\}	 } \le \sqrt{\frac{1}{32} } (S/n)^{1/4} \le \frac{1}{30}.
\end{align*} 
Under the assumption $\tau\ge c_{12}=32$, and if $S/n \le 32^{-p}$,  it follows that
\begin{align*}
	\hat{\mathcal{R}}_\tau(\tilde{f}_n) - \hat{\mathcal{R}}_\tau(f_0) &\le \frac{S}{n} \big\{ \tau \land (n/S)^{1/p} \big\} \bigg(\frac{15}{8} u - 1/8\bigg) + \frac{S}{n} + u^2 \\
	&\le \Big(\frac{1}{16} - \frac{1}{8} \Big) \frac{S}{n} \Bigg\{ \big\{ \tau \land (n/S)^{1/p} \big\} -16\Bigg\} + u^2 \\
	&\le -\frac{1}{32}  \frac{S}{n}  \big\{ \tau \land (n/S)^{1/p} \big\}  + u^2 \\
	&\le 0.
 \end{align*}
On the other hand, by the convexity of $\ell_\tau(\cdot)$ we have for any $f$ that
\begin{align*}
	\hat{\mathcal{R}}_\tau(f) - \hat{\mathcal{R}}_\tau(f_0) = \frac{1}{n} \sum_{i=1}^n \ell_\tau(\varepsilon_i + f(X_i)) - \ell_\tau(\varepsilon_i) \ge \frac{1}{n} \sum_{i=1}^n   \psi_\tau(\varepsilon_i) f(X_i)  .
\end{align*} 
This implies
\begin{align*}
	\inf_{\|f\|_\infty \le 1} \hat{\mathcal{R}}_\tau(f) - \hat{\mathcal{R}}_\tau(f_0) \ge - \frac{M_1}{n} \big\{ \tau \land (n/S)^{1/p} \big\} \ge -1.5 \big\{ \tau \land (n/S)^{1/p} \big\} \frac{S}{n} = -48u^2, 
\end{align*} 
and hence \begin{align*}
	\hat{\mathcal{R}}_\tau(\tilde{f}_n) \le \hat{\mathcal{R}}_\tau(f_0) \le \inf_{\|f\|_\infty\le 1} \hat{\mathcal{R}}_\tau(f) + 48 u^2.
\end{align*}

Now we are ready to provide a lower bound on $\|\tilde{f}_n - f_0\|_2=\|\tilde{f}_n\|_2$. Let $\mu(\cdot)$ denote the uniform measure on $[0,1]^d$. By  our construction of $\tilde{f}_n$ satisfying \eqref{eq:def-tilde-f-u}, it follows that 
\begin{align*}
	\mu(\{\tilde{f}_n = u\}) &= \sum_{\alpha\in \{1,\cdots,K\}^d} \mu \Big(\{x: \|x-x_\alpha\|_\infty \ge 1/(2n^3)\} \cap Q_\alpha(1/(n^2K)) \Big) \\
	&= K^d \Bigg\{  \bigg(\frac{1}{K} - \frac{1}{n^2K} \bigg)^d - \frac{1}{n^{3d}} \Bigg\}  \\
	&\ge \bigg(1 - \frac{1}{n^2} \bigg)^d - \frac{1}{n^{3d-1}} \\
	&\ge 1 - \frac{d}{n^2} - \frac{1}{n^{3d-1}}.
\end{align*} 
Under the condition $n \ge \sqrt{2(d+1)}$, it holds
\begin{align*}
	\|\tilde{f}_n\|_2 \ge u \cdot \mu(\{\tilde{f}_n = u\}) \ge 0.5 u.
\end{align*}

\noindent {\sc Step 4. Conclude by Choosing Different $S$.} From the previous analysis, we have $S=(\tilde{N}\tilde{L})^{2d}$ for $\tilde{N}, \tilde{L}\in \mathbb{N}^+$. If $C_1 \le S \le \{ 32^{-p} \land (4/255)^2 \} n = C_{4} n$ for some positive constant $C_1 $ and $n \ge \sqrt{2(d+1)}$, then we can find some $\tilde{f}_n^{(1)} \in \mathcal{F}_n(d, C_2 \tilde{L}^d \log_2 n, C_3 \tilde{N}^d, 1)$ and $\tilde{f}_{n}^{(2)} \in \mathcal{F}_n(d, C'_2 \tilde{L}^d, C'_3 \tilde{N}^d \log_2 n, 1)$ such that conditioned on event  $\mathcal{E}$ with $u=\sqrt{ \{ \tau \land (n/S)^{1/p} \} \cdot S/(32n) }$,
\begin{align*}
	\|\tilde{f}_n^{(s)}\|_2 \ge 0.5 u \qquad  \text{ and } \qquad 	\hat{\mathcal{R}}_\tau(\tilde{f}_n^{(s)}) \le \hat{\mathcal{R}}_\tau(f_0) \le \inf_{\|f\|_\infty \le 1} \hat{\mathcal{R}}_\tau(f) + 48 u^2 
\end{align*} 
for $s\in \{1,2\}$. Moreover, event $\mathcal{E}$ occurs with probability at least  $1-e^{-S/128}-\frac{d+1}{n}-6/S$.

Note that  $(\lfloor L^{1/d} \rfloor \lfloor N^{1/d} \rfloor)^{2d}\ge C_5 N^2L^2$ for some constant $C_5 > 0$. We prove the final conclusion by considering the following two cases, in which $\delta_n$ is chosen as $0.5 u$ but with different $S$.

\noindent \emph{Case 1. $c_{14}=C_1/C_5 \le (NL)^2 \le C_4 n$.} In this case, let $\tilde{N}=\lfloor N^{1/d} \rfloor$ and $\tilde{L}=\lfloor L^{1/d} \rfloor$  so that
\begin{align*}
	C_1 \le C_5 N^2L^2 \le S= (\lfloor N^{1/d} \rfloor \lfloor L^{1/d} \rfloor)^{2d} \le N^2L^2 \le C_4 n.
\end{align*} 
Then, there exist some $\tilde{f}_n^{(1)} \in \mathcal{F}_n(d, C_2 \lfloor L^{1/d}\rfloor^d \log_2 n, C_3 \lfloor N^{1/d}\rfloor^d, 1) \subseteq \mathcal{F}_n(d, C_2 L \log_2 n, C_3 N, 1)$, and some $\tilde{f}_{n}^{(2)} \in \mathcal{F}_n(d, C'_2 \lfloor L^{1/d}\rfloor^d, C'_3 \lfloor N^{1/d}\rfloor^d \log_2 n, 1)\subseteq \mathcal{F}_n(d, C'_2 L, C'_3 N\log_2 n, 1)$, such that 
\begin{align*}
	\|\tilde{f}_n^{(s)} \|_2 \ge \delta_n \qquad \text{and} \qquad 	\hat{\mathcal{R}}_\tau(\tilde{f}_n^{(s)}) \le \hat{\mathcal{R}}_\tau(f_0) \le \inf_{\|f\|_\infty \le 1} \hat{\mathcal{R}}_\tau(f) + 192 \delta_n^2 
\end{align*} 
hold with probability at least 
\begin{align*}
  1 - \exp(-N^2L^2/(48/C_5)) - \frac{6/C_5}{(NL)^2} - \frac{d+1}{n} .
\end{align*}
The prescribed $\delta_n$ satisfies
\begin{align*}
	C_5 \cdot \frac{0.5}{\sqrt{32}} \frac{N L}{\sqrt{n}} \sqrt{  \tau \land \bigg(\frac{n}{N^2L^2}\bigg)^{1/p}  }  \le \delta_n \le \frac{0.5}{\sqrt{32}}  \frac{N L}{\sqrt{n}} \sqrt{ \tau \land \bigg(\frac{n}{N^2L^2}\bigg)^{1/p}  } 
\end{align*}

\noindent \emph{Case 2. $(NL)^2 \ge C_4 n$.} In this case, let $\tilde{N}$ and $\tilde{L}$ be the maximum integers satisfying $(\tilde{N}\tilde{L})^{2d} \le C_4 n$, $\tilde{N}^d \le N$ and $\tilde{L}^d \le L$. Similarly, we have $(\tilde{N}\tilde{L})^{2d} \ge C_6 n$ for some constant $C_6>0$. Then with $S=(\tilde{N}\tilde{L})^{2d}$, we have
\begin{align*}
	\frac{1}{8\sqrt{2}} \sqrt{C_6 \tau \land C_6^{1-1/p}} \le \delta_n  = 0.5 \sqrt{\frac{1}{32} \Bigg\{\frac{\tau S}{n}   \land \Big(\frac{S}{n}\Big)^{1-1/p} \Bigg\}} \le \frac{1}{8\sqrt{2}} \sqrt{C_4 \tau \land C_4^{1-1/p}}.
\end{align*} 

Because the conditions on $S$ automatically hold by our choice of $\tilde{N}$ and $\tilde{L}$,  there exist some $\tilde{f}_n^{(1)} \in \mathcal{F}_n(d, C_2 \tilde{L}^d \log_2 n, C_3 \tilde{N}^d, 1)\subseteq \mathcal{F}_n(d, C_2 L \log_2 n, C_3 N, 1)$ and some $\tilde{f}_{n}^{(2)} \in \mathcal{F}_n(d, C'_2 \tilde{L}^d, C'_3 \tilde{N}^d \log_2 n, 1) \subseteq \mathcal{F}_n(d, C'_2 L, C'_3 N \log_2 n, 1)$ such that
\begin{align*}
		\|\tilde{f}_n^{(s)} \|_2 \ge \delta_n \qquad \text{and} \qquad 	\hat{\mathcal{R}}_\tau(\tilde{f}_n^{(s)}) \le \hat{\mathcal{R}}_\tau(f_0) \le \inf_{\|f\|_\infty \le 1} \hat{\mathcal{R}}_\tau(f) + 192 \delta_n^2 ~~~~~(s \in \{ 1, 2\})
\end{align*} 
hold with probability at least
\begin{align*}
  1 - \exp(-C_6n/48) - \frac{1}{C_6 n} - \frac{d+1}{n} .
\end{align*} 

Finally, combining the above two cases we choose 
$C_7=\min\{C_5 \cdot \frac{0.5}{\sqrt{32}}, \frac{1}{8\sqrt{2}} \sqrt{C_6 \tau \land C_6^{1-1/p}}\}$, $C_8=\max\{\frac{1}{8\sqrt{2}}, \frac{1}{8\sqrt{2}}\sqrt{C_4\tau \land C_4^{1-1/p}}\}$ so that $\delta_n = 0.5 u$ satisfies
\begin{align*}
	 C_7 \Bigg\{ \frac{N L}{\sqrt{n}} \sqrt{ \tau \land \bigg(\frac{n}{N^2L^2}\bigg)^{1/p}  } \bigwedge 1\Bigg\} \le \delta_n \le C_8 \Bigg\{ \frac{N L}{\sqrt{n}} \sqrt{ \tau \land \bigg(\frac{n}{N^2L^2}\bigg)^{1/p}  }   \bigwedge 1\Bigg\}.
\end{align*} Moreover, setting $c_{15} = 2\max\{128/C_6, 1/C_6+d+1, 6/C_5+d+1, 128/C_5\}$, we see that
\begin{align*}
	\mathbb{P}(\mathcal{E}) \ge 1 - \exp\{-2((NL)^2\land n)/ c_{15}\} - \frac{c_{15}}{2\left((NL)^2 \land n\right)} \ge 1 - \frac{c_{15}}{(NL)^2 \land n} 
\end{align*} 
due to the fact $e^{-x} \le \frac{1}{x}$. The constants $c_{16}$--$c_{20}$ are set to be $c_{16}=192, c_{17}=C_2/\log 2, c_{18}=C_3, c_{19}=C'_2, c_{20}=C'_3/\log 2$. 
\end{proof}

\subsection{Proof of Theorem \ref{thm:generic-lower-bound}}

Before proving Theorem \ref{thm:generic-lower-bound}, we need a result stating the rate of convergence for $\hat{f}_n$ to $f_{0,\tau}$ under special cases.

\begin{lemma}[Convergence rate of $\hat{f}_n$ to $f_{0,\tau}$]
\label{lemma:convergence-f0tau}
   Let $n, \bar{N}, \bar{L}\in \mathbb{N}^+\setminus\{1,2\}$, $p\ge 2$ be arbitrary, $X$ be uniformly distributed on $[0,1]^d$, $f_0 = 0$. Let $M=1$. Suppose the noise $\varepsilon$ is independent of $X$, and satisfies $\mathbb{E}[|\varepsilon|^p|X=x] \le 1$ for all $x\in [0,1]^d$. Then for any $\tau$ satisfying $\tau\ge 8$, $\omega \ge 1$ and $D \ge 1$, we have
   \begin{align*}
       \mathbb{P}\Bigg(\sup_{f\in \mathcal{S}_{n,\tau}(n^{-50})}\|f - f_{0,\tau}\|_2 \ge D &c_{47} \sqrt{V_n (\tau \land \omega V_n^{-1/p})}\Bigg) \\
       &\le c_{48} \left\{\exp\Big(- c_{49} n V_n\Big) + \frac{1\{\tau > \omega D^2 V_n^{-1/p}\}}{\omega^{p-1} D^{2p}}\right\},
   \end{align*} where $c_{47}$--$c_{49}$ are constants independent of $\tau$, $n$, $\bar{N}$ and $\bar{L}$, and $c_{47} > 1$.
\end{lemma}

Now we are ready to prove Theorem \ref{thm:generic-lower-bound}.

\subsubsection{Setup}

We first introduce some notations and constants that will be used throughout the proof. Let $\bar{N}, \bar{L} \in \mathbb{N}^+$ be such that $\bar{N}, \bar{L} \ge \max\{c_{18}, c_{19}\}=c_{11}$ and $\tau \ge c_{12}$, where the constants $c_{12}, c_{18},c_{19}$ are from Proposition \ref{prop:lowerbound-variance} and Theorem \ref{thm:approx-noise-dim-d}. Moreover, let $\mathcal{S}_1=\{\tilde{f}\in \mathcal{F}_n(d, \bar{L}, \bar{N}, 1): \hat{\mathcal{R}}_{\tau} (\tilde{f}) \le \inf_{f\in \mathcal{F}_n(d,\bar{L}, \bar{N}, 1)}\hat{\mathcal{R}}_{\tau}(f) + n^{-100}\}$ and $\mathcal{S}_2(\delta)=\{\tilde{f}\in \mathcal{F}_n(d, \bar{L}, \bar{N}, 1) : \hat{\mathcal{R}}_{\tau} (\tilde{f}) \le \hat{\mathcal{R}}_{\tau} (f_{0,\tau}) \land (\inf_{f\in \mathcal{F}_n(d,\bar{L}, \bar{N}, 1)}\hat{\mathcal{R}}_{\tau}(f) + C'\delta^2)\}$ and, where $C'$ is a constant to be specified.

\subsubsection{Proof of Theorem \ref{thm:generic-lower-bound} Claim (1)}

\noindent {\sc Step 1. Approximation Error Lower Bound.} We claim that \begin{align}
\label{eq:proof-lb-claim-approx}
    \sup_{(X,f_0,\epsilon) \in \mathcal{U}(d, p, \mathcal{F}_0)} \mathbb{P}\Big(\forall \hat{f}_n \in \mathcal{S}_1, \|\hat{f}_n - f_0\|_2 \ge \delta_1\Big)=1,
\end{align} where $\delta_1 = C_1 \big(\bar{N}^2\bar{L}^2 \log ^5(\bar{N}\bar{L})\big)^{-\alpha}$ for some constant $C_1 > 0$. This is a direct consequence of Theorem \ref{thm:neural-network-approx-lower-bound} by noting that
\begin{align*}
    \|\hat{f}_n - f_0\|_2 \ge \inf_{f\in \mathcal{F}_n(d, \bar{L}, \bar{N}, 1)} \|f-f_0\|_2.
\end{align*}

\noindent {\sc Step 2. Stochastic Error Lower Bound. } We claim that if $\bar{N} \bar{L} \ge (C_2 \log n)^2$ with $C_2 = \max\{c_{17}, c_{18}, c_{19}, c_{20}\} \cdot c_{14}$, it holds
\begin{align}
\label{eq:thm:adaptive-huber-lb:delta2}
	\sup_{(X,f_0, \varepsilon) \in \mathcal{U}(d, p, \mathcal{F}_0)} \mathbb{P}\Big(\exists \hat{f}_n \in \mathcal{S}_2(\delta_2),\|\hat{f}_n - f_0\|_2 \ge \delta_2\Big) \ge 1 - \frac{C_3}{(\log n)^2} ~\text{ for all }~ n\ge \sqrt{2(d+1)}
\end{align} and some constant $C_3>0$, where $\delta_2$ satisfies
\begin{align*}
    \delta_2 = \frac{1}{C_4} \delta_{2,*} ~~~~~~\text{with}~~~~~~ \delta_{2,*} = \frac{\bar{N} \bar{L}}{\sqrt{n}\log n} \sqrt{\tau \land \bigg(\frac{n(\log n)^2}{\bar{N}^2 \bar{L}^2}\bigg)^{1/p}} \bigwedge
    1.
\end{align*} and some universal constant $C_4>1$.

%The proof in this part is very similar to the proof in Corollary \ref{cor:lse-lb} step 2. 
Because $\bar{N} \bar{L} \ge (C_2 \log n)^2$ implies $\max\{\bar{L}, \bar{N}\} \ge C_2 \log n$, we can divide the discussion into two cases. We first consider the case where $\bar{L} \ge C_2 \log n$. Let $L=\lfloor \bar{L} / (c_{17} \log n) \rfloor$ and $N = \lfloor\bar{N} / c_{18}\rfloor$. Combined with the fact that $\bar{N} \ge c_{18}$,  this implies $L, N\in \mathbb{N}^+$, $L \ge c_{14}$ and \begin{align*}
c_{14} \le NL =\Big\lfloor \frac{\bar{L}}{c_{17} \log_2 n} \Big\rfloor \Big\lfloor\frac{\bar{N}}{c_{18}}\Big\rfloor\asymp \frac{\bar{N}\bar{L}}{\log n}.
\end{align*} 
Proposition \ref{prop:lowerbound-variance} implies that there exists some pair $(X,f_0=0,\varepsilon) \in \mathcal{U}(d, p, \mathcal{F}_0)$ with symmetric noise $\varepsilon$ such that there exists some $\tilde{f}_n \in \mathcal{F}_n(d, c_{17} L\log n, c_{18}N, 1) \subseteq \mathcal{F}_n(d, \bar{L}, \bar{N}, 1)$ satisfying
\begin{align}
\label{eq:thm:adaptive-huber-lb-eq1}
    \|\tilde{f}_n - f_0\|_2 \ge \delta' \qquad \text{and} \qquad \hat{\mathcal{R}}_\tau(\tilde{f}_n) \le \min\bigg\{\hat{\mathcal{R}}_\tau(f_{0,\tau}), \inf_{\|f\|_\infty\le 1} \hat{\mathcal{R}}_\tau(f)+c_{16} \delta'^2 \bigg\} 
\end{align} with probability at least
\begin{align*}
    1 - \frac{c_{15}}{ (NL)^2 \land n}   \ge 1 - \frac{C_5}{(\log n)^2}
\end{align*} for some constant $C_5>0$, and $\delta'$ satisfies
\begin{align}
\label{eq:thm:adaptive-huber-lb-eq2}
    \delta' \asymp  \frac{NL}{\sqrt{n}}\sqrt{\tau \land \bigg(\frac{n}{N^2L^2}\bigg)^{1/p}} \bigwedge 1 \asymp  \frac{\bar{N} \bar{L}}{\sqrt{n}\log n} \sqrt{\tau \land \bigg(\frac{n(\log n)^2}{\bar{N}^2 \bar{L}^2}\bigg)^{1/p}} \bigwedge 
    1 = \delta_{2,*}.
\end{align}
For the case $\bar{N} \ge C_2 \log n$, we follow a similar argument by letting $L=\lfloor \bar{L}/c_{19} \rfloor$ and $N=\lfloor \bar{N}/(c_{20}\log n) \rfloor$. This time, we also have $NL\asymp \bar{N}\bar{L} / \log n$. By Proposition \ref{prop:lowerbound-variance}, there exists some pair $(X,f_0=0,\varepsilon) \in \mathcal{U}(d, p,\mathcal{F}_0)$ such that with probability at least 
$1 - \frac{C_6}{(\log n)^2}$, there exists some $\tilde{f}_n\in \mathcal{F}_n(d, c_{19}L, c_{20} N \log n, 1) \subseteq \mathcal{F}_n(d, \bar{L}, \bar{N}, 1)$ satisfying \eqref{eq:thm:adaptive-huber-lb-eq1} with $\delta_2$ satisfying \eqref{eq:thm:adaptive-huber-lb-eq2}. Therefore, the claim  \eqref{eq:thm:adaptive-huber-lb:delta2} follows immediately by taking $C_3 = C_5 \lor C_6$ and choose some large $C_4$ and $C'$.

\medskip
\noindent {\sc Step 3. Bias Lower Bound.} Denote
\begin{align*}
    \delta_3 = \frac{C_7}{\tau^{p-1} \log^2 n} ~~~~~~ \text{for} ~~~~~~ C_7 = 2^{-p-1} c_{47}^{-1} C_4^{-1}
\end{align*} In this part, we assert that if $\delta_{3} \ge \delta_{2}$, then one has
\begin{align}
\label{eq:proof-lb-claim-step3}
    \sup_{(X,f_0,\varepsilon) \in \mathcal{U}(d,p,\mathcal{F}_0)}\mathbb{P}\Big(\forall \hat{f}_n \in \mathcal{S}_1, \|\hat{f}_n - f_0\|_2 \ge \delta_3\Big) \ge 1- \frac{C_8}{\log n},
\end{align} for some universal positive constants $C_8$ and all the $n \ge 3$. In this case, we only need to consider the regime where $\bar{N} \bar{L} \le \sqrt{n}$, otherwise we have $\delta_2 \ge \frac{1}{C_4} \frac{1}{\log n} > \delta_3$. 

We construct the tuple $(X,f_0,\varepsilon)$ as follows. Let $X \sim \mathrm{Uniform}([0,1]^d)$ and $f_0 = 0$. Further let the noise $\varepsilon$ be a discrete random variable independent of $X$, satisfying
\begin{align*}
    \varepsilon = \begin{cases}
        -1 & \qquad \text{ with probability } \vartheta \\
        2\tau & \qquad \text{ with probability } \frac{\vartheta}{2\tau} \\
        0 & \qquad \text{ with probability } 1 - \vartheta(1+\frac{1}{2\tau})
    \end{cases},
\end{align*} where $\vartheta=2^{-p} \tau^{1-p}$. It is easy to show that for any $\tau\geq 1$ given,
\begin{align*}
    \mathbb{E} [\varepsilon ] = 0 \qquad \text{ and } \qquad \mathbb{E}[|\varepsilon|^p] = \vartheta (1 + (2\tau)^{p-1}) \le \vartheta 2^p\tau^{p-1} \le 1 .
\end{align*} 
Moreover, because $X$ and $\varepsilon$ are independent,  we have $f_0(x) - f_{0,\tau}(x) \equiv \Delta_\tau$, where $\Delta_\tau$ only depends on $\tau$ and satisfies $\mathbb{E}[\psi_\tau(\varepsilon + \Delta_\tau)]=0$. This implies
\begin{align*}
    (-1 + \Delta_\tau) \vartheta + (0 + \Delta_\tau) \big(1 - \vartheta(1+1/(2\tau))\big) = 0 .
\end{align*} 
Since $\tau \geq 1 > \vartheta$, we further have
\begin{align*}
    \Delta_\tau = \frac{\vartheta}{1-\vartheta/(2\tau)} \ge \vartheta = \frac{1}{2^p \tau^{p-1}} .
\end{align*}

It follows from Lemma \ref{lemma:convergence-f0tau} with $D=1$ and $\omega=(\log n)^{p/3}$ that with probability at least
\begin{align*}
    1 - c_{48} \left\{\exp\left( -c_{49} (\bar{N}\bar{L})^2 \log n \log (\bar{N}\bar{L}) \right) + \frac{1}{(\log n)^{\frac{3(p-1)}{p}}} \right\} \ge 1 - \frac{C'}{\log n},
\end{align*} we have the following event $\mathcal{E}_3$ holds
\begin{align*}
    \forall \hat{f}_n\in \mathcal{S}_1 ~~~~~~
    \|\hat{f}_n - f_{0,\tau}\|_2 &\le c_{47} \sqrt{V_n \{\tau \land (V_n^{-1} \omega)\}} \\
    &\overset{(a)}{\le} c_{47} \sqrt{\frac{(\bar{N}\bar{L})^2 \log^2 n}{n} \left\{\tau \land \left(\frac{n\log^2 n}{(\bar{N}\bar{L})^2}\right)^{1/p}\right\}} \\
    &\overset{(b)}{\le} c_{47} (\log n)^2 \delta_{2,*} = c_{47} C_4 (\log n)^2 \delta_2,
\end{align*}
where $(a)$ follows from the fact that 
\begin{align*}
    V_n = \frac{(\bar{N}\bar{L})^2 \log n \log (\bar{N} \bar{L})}{n} \le \frac{(\bar{N}\bar{L})^2 \log^2 n}{n}, ~ V_n^{-1} \omega^p \le \frac{n \log^3 n}{(\bar{N}\bar{L})^2 \log n \log (\bar{N} \bar{L})} \le \frac{n \log^2 n}{(\bar{N}\bar{L})^2},
\end{align*} and $(b)$ follows from the fact that $\delta_{2,*} = \frac{\bar{N}\bar{L}}{\sqrt{n}\log n} \{\sqrt{\tau} \land (\frac{\sqrt{n} \log n}{\bar{N}\bar{L}})^{1/p}\}$ provided $\bar{N}\bar{L} \le \sqrt{n}$.
At the same time, by the triangle inequality,
\begin{align*}
    \|f_{0,\tau} - f_0\|_2 \le \|f_{0,\tau} - \hat{f}_n\|_2 + \|\hat{f}_n - f_0\|_2, 
\end{align*} 
so that conditioned on the event $\mathcal{E}_3$, the following holds
\begin{align*}
    \forall \hat{f}_n\in \mathcal{S}_1 ~~~~~~ \|\hat{f}_n - f_0\|_2 &\ge \|f_0 - f_{0,\tau}\|_2 - \|\hat{f}_n - f_{0,\tau}\|_2 \\
    &\ge |\Delta_\tau| - c_{47} C_4 \delta_2 (\log n)^2  \\
    &\ge c_{47} C_4 (\log n)^2 \left(\frac{c_{47}^{-1} C_4^{-1}}{2^p \tau^{p-1} \log^2 n} - \delta_2\right) \\
    &= c_{47} C_4 (\log n)^2 (2 \delta_3 - \delta_2) \ge \delta_3
\end{align*} provided $c_{47} \land C_4 \land (\log n) \ge 1$. This completes the proof of the claim \eqref{eq:proof-lb-claim-step3}.

\medskip
\noindent {\sc Step 4. Combining the Separate Lower Bounds. } From {\sc Step 1} and {\sc Step 2} we have
\begin{align}
\label{eq:thm:adaptive-huber-lb:step1}
    \sup_{(X,f_0,\epsilon) \in \mathcal{U}(d, p, \mathcal{F}_0)} \mathbb{P}\Big(\forall \hat{f}_n \in \mathcal{S}_1, \|\hat{f}_n - f_0\|_2 \ge \delta_1\Big)=1 
\end{align} and
\begin{align*}
	\sup_{(X,f_0, \varepsilon) \in \mathcal{U}(d, p, \mathcal{F}_0)} \mathbb{P}\Big(\exists \hat{f}_n \in \mathcal{S}_2(\delta_2),\|\hat{f}_n - f_0\|_2 \ge \delta_2\Big) \ge 1-\frac{C_3 \lor \log \sqrt{2(d+1)}}{\log n}
\end{align*} for all  $\bar{L}\bar{N} \ge (C_2\log n)^2$.

We first combine the results in {\sc Step 2} and {\sc Step 3}. We argue that if $\bar{L}\bar{N} \ge (C_2\log n)^2$ or $\delta_3 \ge \delta_2$, then
\begin{align}
\label{eq:thm:adaptive-huber-lb:step23}
    \forall n \ge 3, ~~~~~~ \sup_{(X,f_0, \varepsilon) \in \mathcal{U}(d,p, \mathcal{F}_0)} \mathbb{P}\Big(\exists \hat{f}_n \in \mathcal{S}_2(\delta_{2}) \cup \mathcal{S}_1, \|\hat{f}_n-f_0\|_2 \ge \delta_2 \lor \delta_3 \Big) \ge 1 - \frac{C_9}{\log n}
\end{align} for $C_9 = C_3 \lor C_8 \lor \log\sqrt{2(d+1)}$.

We prove the claim \eqref{eq:thm:adaptive-huber-lb:step23} by considering the two separate cases. On one hand, when $\delta_3 \ge \delta_2$, the claim is a direct consequence of the claim \eqref{eq:proof-lb-claim-step3} in {\sc Step 3} and the fact that $\mathcal{S}_1$ is not an empty set. And we do not need to impose $\bar{N}\bar{L} \ge (C_2 \log n)^2$. On the other hand, when $\delta_3 < \delta_2$, it follows from the claim \eqref{eq:thm:adaptive-huber-lb:delta2} in {\sc Step 2} that the above claim \eqref{eq:thm:adaptive-huber-lb:step23} holds provided $\bar{N}\bar{L} \ge (C_2 \log n)^2$.

Finally, we combine \eqref{eq:thm:adaptive-huber-lb:step1} and \eqref{eq:thm:adaptive-huber-lb:step23} to prove the main statement. Specifically, we claim that for all $n\ge 3$, the following holds
\begin{align}
\label{eq:thm:adaptive-huber-lb:step4}
	\sup_{(X,f_0, \varepsilon) \in \mathcal{U}(d, p, \mathcal{F}_0)} \mathbb{P}\Big(\exists \hat{f}_n \in \mathcal{S}_2(\delta^*)\cup \mathcal{S}_1,\|\hat{f}_n - f_0\|_2 \ge \delta^* \Big) \ge 1 - \frac{C_{10}}{\log n} ~\text{with}~ \delta^* = \delta_1 \lor \delta_2 \lor \delta_3
\end{align} for some large enough $C_{10}$.
To see this, we first consider the case where $\bar{N}\bar{L} \ge (C_2 \log n)^2$ or $\delta_3 \ge \delta_2$. In this case, if $\delta_1 \ge \delta_2 \lor \delta_3$, then for all $\hat{f}_n \in \mathcal{S}_1$,  $\|\hat{f}_n - f_0\|_2 \ge \delta_1 = \delta^*$. Combined with the fact that $\mathcal{S}_1$ is not an empty set, there exists some $\hat{f}_n \in \mathcal{S}_1 \cup \mathcal{S}_2(\delta_*)$ such that $\|\hat{f}_n - f_0\|_2 \ge \delta^*$. If $\delta_1 \le \delta_2 \lor \delta_3$,   \eqref{eq:thm:adaptive-huber-lb:step23} ensures that with probability at least $1-C_9 / (\log n)$, there exists some $\hat{f}_n \in \mathcal{S}_1 \cup \mathcal{S}_2(\delta_2) \subseteq \mathcal{S}_1 \cup \mathcal{S}_2(\delta_*)$ such that $\|\hat{f}_n - f_0\|_2 \ge \delta_2 \lor \delta_3 = \delta^*$. Hence \eqref{eq:thm:adaptive-huber-lb:step4} holds in this case.

It remains to prove \eqref{eq:thm:adaptive-huber-lb:step4} if $\bar{N}\bar{L} \le (C_2 \log n)^2$ and $\delta_2 > \delta_3$. In this case, we have
\begin{align*}
\frac{\delta_1}{\delta_2} \ge C_{11} \frac{\big(\bar{N}^2 \bar{L}^2 \log^5 (\bar{N}\bar{L})\big)^{-\alpha}}{\sqrt{\big(\frac{\bar{N}^2\bar{L}^2}{n}\big)^{1-1/p}}} &= C_{11} n^{\frac{1}{2}(1-1/p)} (\bar{N}^2\bar{L}^2)^{-\alpha-\frac{1}{2}(1-\frac{1}{p})} \log (\bar{N} \bar{L})^{-5\alpha} \\
&\ge C_{12} n^{\frac{1}{2}(1-\frac{1}{p})} (\log n)^{-C_{13}}.
\end{align*} 
Then there exists some universal constant $C_{14}$ such that $\delta_1 \ge \delta_3$ for all the $n \ge C_{14}$. Consequently, there exists some $\hat{f}_n \in \mathcal{S}_1\subseteq \mathcal{S}_1 \cup \mathcal{S}_2(\delta^*)$, the inequality $\|\hat{f}_n - f_0\|_2 \ge \delta_1 = \delta^*$ holds. 

Putting these pieces together, we can conclude that the claim \eqref{eq:thm:adaptive-huber-lb:step4} holds for $C_{10} = \log (C_{14}) \lor C_9$ and hence completes the proof of Theorem~\ref{thm:generic-lower-bound} Part (1). \qed

\subsubsection{Proof of Theorem \ref{thm:generic-lower-bound} Claim (2)}

The proof is very similar to that of Claim (1). For each $n \ge \sqrt{2(d+1)} \lor C_{14}$, let
\begin{align*}
    \delta_{n,*} = \min_{\tau \ge c_{12}, \bar{L}, \bar{N} \ge c_{11}} \delta_1 \lor \delta_2 \lor \delta_3.
\end{align*} Then we have
\begin{align*}
    \delta_{n,*} &\asymp \min_{\tau \ge c_{12}, \bar{L}, \bar{N} \ge c_{11}} \frac{1}{3} \left(\delta_1 + \delta_2 + \delta_3\right) \\
    &\asymp \min_{\tau \ge c_{12}, \bar{L}, \bar{N} \ge c_{11}} \frac{1}{(\bar{N}\bar{L})^{2\alpha} (\log \bar{N}\bar{L})^{5\alpha}} + \frac{\bar{N} \bar{L}}{\sqrt{n}\log n} \sqrt{\tau \land \bigg(\frac{n(\log n)^2}{\bar{N}^2 \bar{L}^2}\bigg)^{1/p}} + \frac{1}{\tau^{p-1} \log^2 n} \\
    &=\min_{\tau \ge c_{12}, \bar{L}, \bar{N} \ge c_{11}} \mathcal{L}_n(\tau, \bar{N},\bar{L})
\end{align*} It is easy to verify that $\mathcal{L}_n(\tau, \bar{N}, \bar{L})$ can attain optimal value when
\begin{align*}
    \bar{N}\bar{L} \asymp n^{\frac{\nu^*}{2(2\alpha+\nu^*)}} (\log n)^{\frac{2-\nu^*-5\alpha}{2\alpha+\nu^*}} ~~~~~~ \text{and} ~~~~~~ \tau \asymp \left(n^{\frac{\alpha}{2\alpha+\nu^*}} \log n^{\frac{3\alpha - 2}{2\alpha+\nu^*}}\right)^{2(1-\nu^*)}
\end{align*} and thus
\begin{align*}
    \delta_{n,*} \asymp \min_{\tau \ge c_{12}, \bar{L}, \bar{N} \ge c_{11}} \mathcal{L}_n(\tau, \bar{N},\bar{L}) \asymp n^{-\frac{\alpha \nu^*}{2\alpha+\nu^*}} (\log n)^{-\frac{3\alpha\mu^*+4\alpha}{2\alpha+\nu^*}}.
\end{align*} 

We will show that, for large enough $n$,
\begin{align}
    \inf_{\bar{N},\bar{L} \ge c_{11}, \tau \ge c_{12}} \sup_{(X,f_0,\varepsilon) \in \mathcal{U}(d,p,\mathcal{F}_0)} \mathbb{P}\left[\exists \hat{f}_n\in \mathcal{S}_1\cup \mathcal{S}_2(\delta_{n,*}) ~\text{s.t.}~ \|\hat{f}_n - f_0\|_2 \ge \delta_{n,*}\right] \ge 1 - \frac{C_{15}}{\log n} \label{eq:target-part2}
\end{align}

We prove claim \eqref{eq:target-part2} by considering the three cases with regard the choice of $\bar{L}$, $\bar{N}$ and $\tau$.

\noindent \emph{Case 1. $\delta_1 \ge \delta_2 \lor \delta_3$. } In this case, we must have $\delta_1 \ge \delta_{n,*}$. Hence it follows from the claim \eqref{eq:proof-lb-claim-approx} in the proof of Claim (1) {\sc Step 1} that, for $n \ge 3$,
\begin{align}
    \inf_{\tau \ge c_{12}, \bar{N},\bar{L} \ge c_{11}, \delta_1 \ge \delta_2 \lor \delta_3} \sup_{(X,f_0,\varepsilon)\in \mathcal{U}(d,p,\mathcal{F}_0)} \mathbb{P} \left( \exists \hat{f}_n \in \mathcal{S}_1, ~~ \|\hat{f}_n - f_0\|_2 \ge \delta_{n,*}\right) = 1
\label{eq:part2-approx-error}
\end{align}

\noindent \emph{Case 2. $\delta_3 \ge \delta_2 \lor \delta_1$. } In this case, we have $\delta_{3} \ge \delta_2$ and $\delta_3 \ge \delta_{n,*}$, then it follows directly from of claim \eqref{eq:proof-lb-claim-step3} in the proof of Claim (1) {\sc Step 3} that, for $n \ge 3$,
\begin{align}
    \inf_{\tau \ge c_{12}, \bar{N},\bar{L} \ge c_{11}, \delta_3 \ge \delta_2 \lor \delta_1} \sup_{(X,f_0,\varepsilon)\in \mathcal{U}(d,p,\mathcal{F}_0)} \mathbb{P} \left( \exists \hat{f}_n \in \mathcal{S}_1, ~~ \|\hat{f}_n - f_0\|_2 \ge \delta_{n,*}\right) \ge 1 - \frac{C_8}{\log n}
\label{eq:part2-bias}
\end{align}

\noindent \emph{Case 3. $\delta_{2}\ge \delta_1 \lor \delta_3$. } Because $\delta_2 \ge \delta_1$, we thus let $n$ be large enough ($n \ge n'$) such $\bar{N}\bar{L}$ satisfies
\begin{align*}
    (\bar{N}\bar{L}) \ge C_1 C_4 (n\log^2 n)^{\frac{1-1/p}{2(1-1/p+2\alpha)}} \log(\bar{N}\bar{L})^{-\frac{5\alpha}{1-1/p+2\alpha}} \ge (C_2 \log n)^2
\end{align*} When $\bar{L} \ge C_2 \log n$, we can apply Proposition \ref{prop:lowerbound-variance} with $N$, $L$ satisfying (1) $c_{17} L \log n \le \bar{L}$, (2) $c_{18} N \le \bar{N}$, (3) $(NL)^2 \ge c_{14}$ and (4)
\begin{align*}
    \delta_{n,*} \overset{(a)}{\gtrsim} \delta_n \asymp \frac{N L}{\sqrt{n}} \sqrt{  \Bigg\{\tau \land \bigg(\frac{n}{(NL)^2}\bigg)^{1/p}\Bigg\}} \bigwedge 1\overset{(b)}{\ge} \delta_{n,*}.
\end{align*} here $\delta_n$ \eqref{eq:prop-lb-delta-n} is the notation used in Proposition \ref{prop:lowerbound-variance}. We can choose some $N,L$ satisfying (1)-(3) such that the inequality (b) in (4) holds because our construction in the proof of Claim (1) {\sc Step 2} asserts that there exists some $N',L'$ such that $N'L'\ge c_{14}$, $c_{18} N' \le \bar{N}$, $c_{17} (\log n) L'\le \bar{L}$ satisfying
\begin{align*}
    \delta_n \ge \frac{\delta_{2,*}}{C_{4}} = \delta_2 \ge \delta_{n,*}.
\end{align*} Different from the choice of $N$ and $L$ in {\sc Step 2} which matches $\delta_n \asymp \delta_{2,*}$, now we choose small $N$ and $L$ instead such that $\delta_n \le C_{16} \delta_{n,*}$. Therefore, by Proposition \ref{prop:lowerbound-variance}, when $X\sim\mathrm{Uniform}([0,1]^d)$, $f_0=0$, $\varepsilon$ follows some symmetric distribution independent of $X$, with probability at least $1-C'/(\log n)^2$, there exists some $\tilde{f}_n \in \mathcal{F}_n(d,c_{17} (\log n)L, c_{18} N, 1) \subset \mathcal{F}_n(d,\bar{L},\bar{N}, 1)$ satisfying $\|\tilde{f}_n - f_0\|_2 \ge \delta_n \ge \delta_2 \ge \delta_{n,*}$ and
\begin{align*}
    \hat{\mathcal{R}}_\tau(\tilde{f}_n) &\le \min\left\{\hat{\mathcal{R}}_\tau(f_0), \inf_{\|f\|_\infty \le 1} \hat{\mathcal{R}}_\tau(f) + c_{16} \delta_n\right\} \\
    &\le \min\left\{\hat{\mathcal{R}}_\tau(f_{0,\tau}), \inf_{f\in \mathcal{F}(d,\bar{L},\bar{N},1)} \hat{\mathcal{R}}_\tau(f) + c_{16} C_{16}  \delta_{n,*}\right\},
\end{align*} where the inequality follows from the fact that $f_0=f_{0,\tau}$ when $\varepsilon$ is symmetric. The discussion when $\bar{N} \ge C_2 \log n$ is similar. In this case, we can find some $\tilde{f}_n \in \mathcal{F}_n(d,c_{19} L, c_{20} (\log n) N, 1) \subset \mathcal{F}_n(d,\bar{L},\bar{N}, 1)$ satisfying $\hat{\mathcal{R}}_\tau(\tilde{f}_n) \le \min\left\{\hat{\mathcal{R}}_\tau(f_{0,\tau}), \inf_{\|f\|_\infty \le 1} \hat{\mathcal{R}}_\tau(f) + c_{16} C_{17}  \delta_{n,*}\right\}$ and $\|\tilde{f}_n - f_0\|_2 \ge \delta_{n,*}$. Putting these pieces together, we can conclude that, for $n \ge n'$,
\begin{align}
\label{eq:part2-stochastic-error}
    \inf_{\tau \ge c_{12}, \bar{N},\bar{L} \ge c_{11}, \delta_2 \ge \delta_1 \lor \delta_3} \sup_{(X,f_0,\varepsilon)\in \mathcal{U}(d,p,\mathcal{F}_0)} \mathbb{P} \left( \exists \hat{f}_n \in \mathcal{S}_2(\delta_{n,*}), ~~ \|\hat{f}_n - f_0\|_2 \ge \delta_{n,*}\right) \ge 1-\frac{C_{18}}{\log n}
\end{align} by assigning $C'=c_{16} (C_{16} \lor C_{17})$.

Putting the pieces \eqref{eq:part2-approx-error}, \eqref{eq:part2-bias}, \eqref{eq:part2-stochastic-error} together completes the proof of Claim (2). \qed

\subsection{Proof of Theorem \ref{thm:generic-lower-bound-lse}}

\noindent The proof is almost identical to that of Theorem \ref{thm:generic-lower-bound} except we choose $\tau=\infty$, so we only provide a sketch here to highlight the difference. Let $\mathcal{S}_1$ and $\mathcal{S}_2(\delta)$ be the same as in the proof of Theorem \ref{thm:generic-lower-bound} but with $\tau=\infty$. Moreover, let $\bar{N}, \bar{L} \in \mathbb{N}^+$ be such that $\bar{N}, \bar{L} \ge \max\{c_{18}, c_{19}\}=c_{11}$.

\begin{proof}[Proof of Theorem \ref{thm:generic-lower-bound-lse} Claim (1)]
Similar to the proof of Theorem \ref{thm:generic-lower-bound} Claim (1), we have
\begin{align*}
    \sup_{(X,f_0,\epsilon) \in \mathcal{U}(d, p, \mathcal{F}_0)} \mathbb{P}\Big(\forall \hat{f}_n \in \mathcal{S}_1, \|\hat{f}_n - f_0\|_2 \ge \delta_1\Big)=1,
\end{align*} with $\delta_1 = C_1 \big(\bar{N}^2\bar{L}^2 \log ^5(\bar{N}\bar{L})\big)^{-\alpha}$, and for any $\bar{N}\bar{L} \ge (C_2 \log n)^2$\begin{align*}
	\sup_{(X,f_0, \varepsilon) \in \mathcal{U}(d, p, \mathcal{F}_0)} \mathbb{P}\Big(\exists \hat{f}_n \in \mathcal{S}_2(\delta_2),\|\hat{f}_n - f_0\|_2 \ge \delta_2\Big) \ge 1 - \frac{C_3}{(\log n)^2} ~\text{ for all }~ n\ge \sqrt{2(d+1)},
\end{align*} where $\delta_2$ satisfies
\begin{align*}
    \delta_2 = \frac{1}{C_4} \delta_{2,*} ~~~~~~\text{with}~~~~~~ \delta_{2,*} = \left(\frac{\bar{N} \bar{L}}{\sqrt{n}\log n} \right)^{1-1/p} \bigwedge
    1.
\end{align*} Here $C_1$--$C_4$ are same universal constants as those in Theorem \ref{thm:generic-lower-bound}. So it remains to combine the two lower bounds together. Similar to the discussion before, there exists some large $n_0$ such that for any $n \ge n_0$, we have $\delta_{1} \ge \delta_2$ for any $\bar{N}\bar{L} \le (C_2\log n)^2$. Therefore, we can conclude that
\begin{align*}
    \sup_{(X,f_0, \varepsilon) \in \mathcal{U}(d, p, \mathcal{F}_0)} \mathbb{P} \Big(\exists \hat{f}_n \in \mathcal{S}_1 \cup \mathcal{S}_2(\delta_2),& \|\hat{f}_n - f_0\| \ge \delta_1 \lor \delta_2 \Big) \\
    &\ge 1-\frac{C_3 \lor \log (\sqrt{2(d+1)}) \lor \log n_0}{\log n}
\end{align*} this completes the proof of Claim (1).
\end{proof}

\begin{proof}[Proof of Theorem \ref{thm:generic-lower-bound-lse} Claim (2)]
    Similarly, for any $n\ge \sqrt{2(d+1)} \lor n_0$, let $\delta_{n,*}=\min_{\bar{N},\bar{L} \ge c_{11}} \delta_1 \lor \delta_2$. Then
    \begin{align*}
        \delta_{n,*} \asymp \min_{\bar{N}, \bar{L} \ge c_{11}} \mathcal{L}_n(\bar{N},\bar{L}) ~~~~\text{with}~~~~ \mathcal{L}_n(\bar{N},\bar{L}) = \frac{1}{(\bar{N}\bar{L})^{2\alpha} (\log \bar{N}\bar{L})^{5\alpha}} + \left(\frac{\bar{N} \bar{L}}{\sqrt{n}\log n}\right)^{1-1/p}.
    \end{align*} It is easy to verify that $L_n(\bar{N}, \bar{L})$ attains optimal value at $\bar{N}\bar{L} \asymp n^{\frac{\nu^\dagger}{2(2\alpha+\nu^\dagger)}} (\log n)^{\frac{\nu^\dagger - 5\alpha}{2\alpha+\nu^\dagger}}$, thus
    \begin{align*}
        \delta_{n,*} \asymp n^{-\frac{\alpha \nu^\dagger}{2\alpha + \nu^\dagger}} (\log n)^{-\frac{7\alpha \nu^\dagger}{2\alpha + \nu^\dagger}}.
    \end{align*}
    The following discussion is similar. We consider the two cases (1) $\delta_1 \ge \delta_2$ and (2) $\delta_2 \ge \delta_1$ which is determined by the choice of $\bar{N}$ and $\bar{L}$, and use a similar argument to show that
    \begin{align*}
        \forall n \ge 3, ~~~~\inf_{\bar{N},\bar{L} \ge c_{11}, \delta_1 \ge \delta_2} \sup_{(X,f_0,\varepsilon)\in \mathcal{U}(d,p,\mathcal{F}_0)} \mathbb{P} \left( \exists \hat{f}_n \in \mathcal{S}_1, ~~ \|\hat{f}_n - f_0\|_2 \ge \delta_{n,*}\right) = 1
    \end{align*} and
    \begin{align*}
        \inf_{\bar{N},\bar{L} \ge c_{11}, \delta_2 \ge \delta_1} \sup_{(X,f_0,\varepsilon)\in \mathcal{U}(d,p,\mathcal{F}_0)} \mathbb{P} \left( \exists \hat{f}_n \in \mathcal{S}_2(\delta_{n,*}), ~~ \|\hat{f}_n - f_0\|_2 \ge \delta_{n,*}\right) \ge 1-\frac{C_{18}}{\log n}
    \end{align*} for any $n \ge n'$, where $n'$ is a large constant such that $(\bar{N}\bar{L}) \ge (C_2 \log n)^2$ whenever $\delta_2 \ge \delta_1$. Putting these pieces together completes the proof.
\end{proof}

\subsection{Proof of Lemma~\ref{lemma:convergence-f0tau}}

The proof relies on the tail probability given in Lemma \ref{lemma:convergence-rate} with $f^*=f_{0,\tau}$. Without loss of generality, we assume that $\omega \lor D^2 \le n$, otherwise the bound is trivial. Since $X$ is independent of $\varepsilon$, we claim that there exists some constant $C_1 \in [-1/2,1/2]$ depending only on $\tau$ such that $f_{0,\tau}=f_0 - C_1$. 
By the first-order condition and with a sufficiently large $\tau$, $f_{0,\tau} = \text{argmin}_{f\in \Theta} \mathcal{R}(f)$ satisfies
\begin{align*}
    \mathbb{E}[\psi_\tau(\varepsilon + f_0(X) - f_{0,\tau}(X) )|X=x] = 0 .
\end{align*} 
Because $\varepsilon$ and $X$ are independent, we must have $f_0 - f_{0,\tau} \equiv C_1$, where $C_1$ is such that $ \mathbb{E}[\psi_\tau(\epsilon + C_1)]=0$. By Proposition \ref{prop:bias}, we have $\|f_0-f_{0,\tau}\|_2 \le \frac{4}{\tau^{p-1}}$ so that $|C_1|=\|f_0 - f_{0,\tau}\|_2 \le \frac{1}{2}$ if $\tau \ge 8$. We thus choose $\tilde{f} = -C_1$, 
\begin{align*}
    \delta_* = 2 \sqrt{V_n (\tau \land \omega V_n^{-1/p})} ~~~~~~ \text{and} ~~~~~~ B_k = \omega D^2 V_n^{-1/p}.
\end{align*}
and it remains to verify conditions (1)-(4) of Lemma~\ref{lemma:convergence-rate}.

For condition (1), note that $f_n=f_{0,\tau}$, we have $\|f_n - f_{0,\tau}\|_2 = 0$ and $\mathcal{R}_\tau(f_n) - \mathcal{R}_\tau(f_{0,\tau})=0$.

For condition (2), similar to the proof of Proposition \ref{prop:strong-convex}, by Taylor's expansion, if $f\in \Theta$ with $M=1$, we have
\begin{align}
\label{eq:lemma:convergence-f0tau:eq1}
    \mathcal{R}_\tau(f) - \mathcal{R}_\tau(f_{0,\tau}) = \mathbb{E}\Big[ \psi_\tau(\varepsilon + C_1) &\Delta_{f,\tau}(X) \nn \\+ &\int_{0}^{\Delta_{f,\tau}(X)} 1\{|\varepsilon +C_1 + t| \le \tau\}(\Delta_{f,\tau}(X)-t)\mathrm{d}t\Big],
\end{align} where $\Delta_{f,\tau}(X)=(f_{0,\tau}-f)(X)$. By  the definition of $C_1$ and tower rule,
\begin{align}
\label{eq:lemma:convergence-f0tau:eq2}
    \mathbb{E}\Big[ \psi_\tau(\varepsilon + C_1) (f - f_{0,\tau})(X) \Big] = \mathbb{E}\Big[ \mathbb{E}[\psi_\tau(\varepsilon + C_1)|X] (f - f_{0,\tau})(X) \Big] = 0.
\end{align} Moreover, 
\begin{align*}
    &\mathbb{E} \Big[ \int_{0}^{\Delta_{f,\tau}(X)} 1\{|\varepsilon +C_1 + t| \le \tau\}(\Delta_{f,\tau}(X)-t)\mathrm{d}t \Big| X=x\Big] \\
    =& \mathbb{E} \Big[ \int_{0}^{\Delta_{f,\tau}(X)} \big(1-1\{|\varepsilon +C_1 + t| > \tau\}\big)(\Delta_{f,\tau}(X)-t)\mathrm{d}t \Big| X=x\Big] \\
    \ge& \frac{1}{2} (\Delta_{f,\tau}(x))^2 - \mathbb{E} \Big[ \int_{0}^{\Delta_{f,\tau}(X)} \big(1\{|\varepsilon| > \tau/2\} + 1\{|C_1 + t| > \tau/2\}\big)(\Delta_{f,\tau}(X)-t)\mathrm{d}t \Big| X=x\Big] \\
    \ge& \frac{1}{2} (\Delta_{f,\tau}(x))^2 \{1-\mathbb{P}(|\epsilon|\ge \tau/2|X=x)\},
\end{align*} where the last inequality follows from $|C_1 + t| \le 2.5 < 4 \le \tau/2$. Together with the fact that
\begin{align*}
    \mathbb{P}(|\varepsilon|>\tau/2|X=x) \le \frac{\mathbb{E}(|\epsilon|^p|X=x)}{(\tau/2)^p} \le \frac{1}{4^p} \le \frac{1}{2}, 
\end{align*} this yields
\begin{align}
\label{eq:lemma:convergence-f0tau:eq3}
    \mathbb{E} \Big[ \int_{0}^{\Delta_{f,\tau}(X)} 1\{|\varepsilon +C_1 + t| \le \tau\}(\Delta_{f,\tau}(X)-t)\mathrm{d}t\Big] \ge \frac{1}{4} \|f-f_0\|_2^2.
\end{align} 
Substituting \eqref{eq:lemma:convergence-f0tau:eq2} and \eqref{eq:lemma:convergence-f0tau:eq3} into \eqref{eq:lemma:convergence-f0tau:eq1}, we  conclude that
\begin{align*}
    \mathcal{R}_\tau(f) - \mathcal{R}_\tau(f_{0,\tau}) \ge \frac{1}{4} \|f-f_{0,\tau}\|_2^2
\end{align*} 
holds for all the $f\in \Theta$ with $M=1$. This validates condition (2).

For condition (3), by letting $\tilde{\varepsilon} = \varepsilon + C_1$ it suffices to bound 
\begin{align}
\label{eq:lemma:convergence-f0tau:eq4}
    \mathbb{E}\Bigg[\sup_{h\in \mathcal{H}_{n,\tau, B}(\delta)} \Big|\frac{1}{n} \sum_{i=1}^n h(X_i, \tilde{\varepsilon}_i) - \mathbb{E} h(X,\tilde{\varepsilon}) \Big|\Bigg]
\end{align} for
\begin{align*}
    h_f(X,\tilde{\varepsilon}) = \psi_{\tau\land B}(\tilde{\varepsilon}) \Delta_{f,\tau}(X) + \int_{0}^{\Delta_{f,\tau}(X)} 1\{|\tilde{\varepsilon} + t| \le \tau\}(\Delta_{f,\tau}(X)-t)\mathrm{d}t,
\end{align*} and $\mathcal{H}_{n,\tau, B}(\delta)=\{h_f: f\in \mathcal{F}_n  \cap \Theta_*(\delta)\}$ whose envelop function $H$ can be chosen as
\begin{align*}
    \sup_{f\in \mathcal{F}_n \cap \Theta_*(\delta)} h(X,\tilde{\varepsilon}) \le H(X,\epsilon) = 2|\psi_{\tau \land B}(\tilde{\varepsilon})| + 2 .
\end{align*}

Similar to the proof of Lemma \ref{lemma:convergence-rate}, we have for any $f\in \Theta$ that
\begin{align*}
    \mathbb{E} |h(X, \tilde{\varepsilon})|^2 &\le 2 \mathbb{E} \{\psi_\tau(\tilde{\varepsilon}) \Delta_{f,\tau}(X)\}^2 + 2 \mathbb{E} \{\frac{1}{2} |\Delta_{f,\tau}|^2\}^2 \\
    &\le 2 \mathbb{E} \{ \mathbb{E}[\psi_\tau(\tilde{\varepsilon})^2|X]\Delta_{f,\tau}(X)^2\} + 2\mathbb{E} \Delta_{f,\tau}(X)^2 \\
    &\le 2 (\mathbb{E} |\tilde{\varepsilon}|^2 + 1) \|f-f_{0,\tau}\|_2^2 \\
    &\le 2 (2v_2 + 2C_1^2 + 1) \|f-f_{0,\tau}\|_2^2 = C_2 \|f - f_{0,\tau}\|_2^2 , 
\end{align*}where $C_2 > 0$ is a constant.
Next we use Lemma \ref{lemma:maximal-inequality} to bound \eqref{eq:lemma:convergence-f0tau:eq4}. Similar to the proof of Lemma \ref{lemma:empirical-process}, because any $(\varepsilon/((B\land \tau)+1))$-net of $\mathcal{F}_n|_X$ under the $\|\cdot\|_\infty$ norm is also an $\varepsilon$-net of $\mathcal{H}_n|_X$ under the $\|\cdot\|_\infty$ norm, it follows from Lemma \ref{lemma:cover-number-nn} that
\begin{align*}
    \log \mathcal{N}_\infty (\epsilon, \mathcal{H}_n, n) \le \log \mathcal{N}_\infty (\epsilon/(\tau \land B), \mathcal{F}_n, n) \le C_3 \log(en(\tau\land B)/\epsilon) (\bar{N} \bar{L})^2 \log (\bar{N}\bar{L})
\end{align*} for all $\epsilon \in (0, n((\tau \land B)+1))$.
From previous discussions, we have $\|H\|_\infty \le 2\tau+2 \le 4\tau$, $\|H\|_2 \le \sqrt{8\mathbb{E}|\tilde{\varepsilon}|^2 + 8} = C_4$, and
\begin{align*}
    \sigma^2 = \sup_{h\in \mathcal{H}_n(\delta)} \mathbb{E} h(X,\varepsilon)^2 \le C_2 \sup_{f\in \mathcal{F}_n \cap \Theta_*(\delta)}\|f-f_{0,\tau}\|_2 \le C_2 \delta^2.
\end{align*} Therefore, letting $r=\sigma / \|H\|_2 = C_5 \delta$, we have
\begin{align*}
    J(r, \mathcal{H}_n, H) &= \int_{0}^{r} \sup_{Q\in \mathbb{P}(n)} \sqrt{1 + \log\mathcal{N}(\epsilon\|H\|_2, \mathcal{H}_n, \|\cdot\|_{2,Q})} \mathrm{d}\epsilon \\
    &\le \int_{0}^{C_5 \delta} \sqrt{1 + \log \mathcal{N}_\infty(\epsilon \|H\|_2, \mathcal{H}_n, n)} \mathrm{d}\epsilon \\
    &\le \bar{N}\bar{L} \sqrt{\log(\bar{N} \bar{L})}  \int_{0}^{C_5 \delta} \sqrt{1 + \log\Big(\frac{en(\tau \land B)}{\epsilon C_4}\Big) \lor 0} \mathrm{d}\epsilon \\
    &\le \bar{N}\bar{L} \sqrt{\log(\bar{N} \bar{L})} \Bigg(C_5 \delta + \sqrt{\log(n(\tau \land B))} C_5 \delta + C_5 \delta \big(\sqrt{\log(1/C_6\delta) \lor 1}\big) \Bigg).
\end{align*} 
When $\delta\ge \frac{1}{n}$, we conclude that
\begin{align}
    J(r, \mathcal{H}_n, H) \leq C_7 \delta \bar{N}\bar{L} \sqrt{\log(\bar{N} \bar{L}) \log (n(\tau \land B))} 
\end{align} 
for some constant $C_7>0$. Denote $V_{n,B}= n^{-1} (\bar{N}\bar{L})^2 {\log(\bar{N} \bar{L}) \log (n(\tau \land B))} $. By the maximal inequality in Lemma \ref{lemma:maximal-inequality},
\begin{align*}
    \mathbb{E}\Bigg[ \sup_{h\in \mathcal{H}_n(\delta)} \Big|\frac{1}{\sqrt{n}} \sum_{i=1}^n h(X_i, \tilde{\varepsilon}_i) - \mathbb{E} h(X,\tilde{\varepsilon}) \Big|\Bigg] &\le \|H\|_2 J(r, \mathcal{H}_n, H) + \frac{\|H\|_\infty J(r, \mathcal{H}_n, H)^2}{r^2 \sqrt{n}} \\
    &\le C_8 \Bigg( \delta \sqrt{n V_n} + (\tau \land B) \frac{nV_n}{\sqrt{n}}\Bigg) .
\end{align*}
This further implies that for any $\delta \geq \delta_n \geq 1/n$,
\begin{align*}
    \mathbb{E}\Bigg[\sup_{h\in \mathcal{H}_n(\delta)} \Big|\frac{1}{n} \sum_{i=1}^n h(X_i, \tilde{\varepsilon}_i) - \mathbb{E} h(X,\tilde{\varepsilon}) \Big|\Bigg] \le \phi_n(\delta,B) = C_8 \Bigg( \delta \sqrt{V_n} + (\tau \land B) V_n\Bigg)  .
\end{align*} 
It is easy to see that $\phi_n(\alpha \delta,B) \le \alpha \phi_n(\delta,B)$ for all $\alpha\ge 1$ and $B>1$ because $\phi_n(\delta,B)$ is linear in $\delta$ and $\phi_n(0,B)>0$.

For condition (4), our choice of $B_k$ satisfies $V_{n,B_k} \le 4V_n$, this yields
\begin{align*}
    (B_k \land \tau) V_{n,B_k} \le 4 (B_k \land \tau) V_n \le 4 \left\{(V_n^{1-1/p} \omega D^2) \land (V_n \tau) \right\} \le D^2 \delta_*^2
\end{align*} provided $D\ge 1$, combined with the fact that $V_n \le D^2 V_n \le \delta_*^2$, we find 
\begin{align*}
    \phi_n(2^k \delta^\dagger, B_k) &\le C_9 \left(2^k \delta^\dagger \sqrt{V_{n, B_k}} + (B_k \land \tau) V_{n, B_k}\right) \\&\le C_9\left(2^k \delta^\dagger D \delta_* + D^2 \delta_*^2\right) \le C_9 2^{2k} (\delta^\dagger)^2.
\end{align*} for $\delta^\dagger = D\delta_* + n^{-50}$. Similar to the proof of Theorem \ref{thm:generic-bound}, Lemma \ref{lemma:convergence-rate} implies
\begin{align*}
    \mathbb{P}\Bigg[\sup_{f\in \mathcal{S}_{n,\tau}(n^{-50})} &\|f-f_{0,\tau}\|_2 \ge C_{10} (D\delta_* + n^{-50})\Bigg] \\
    &\le \sum_{k=1}^\infty 2\exp\left(-C_{11} n V_n 2^{2k}\right) + \frac{1\{\tau > \omega D^2 V_n^{-1/p}\}}{\omega^{p-1} D^{2p}} 4^{-k} \\
    &\lesssim \exp\left( - C_{11} n V_n \right) + \frac{1\{\tau > \omega D^2 V_n^{-1/p}\}}{\omega^{p-1} D^{2p}},
\end{align*} 
which yields the conclusion by observing that $D \delta_* \ge n^{-50}$. \qed

\section{A Detailed Comparison to Related Works}

This section presents a detailed comparison to previous works by highlighting the major differences and our improvements.

\subsection{Upper bound analysis}

In Table~\ref{table:convergence-rate-upper} below, we briefly summarize the obtained upper bounds for the class of Huber ReLU-DNN estimators under heavy-tailed errors. Logarithmic factors are omitted for a clearer presentation.
%followed by some comparisons with prior work. 
%A detailed comparison with prior work is postponed to Section \ref{subsec:upper-bound-discussion}.
%Our results on the class of Huber ReLU-DNN estimators are summarized in Table~\ref{table:convergence-rate-upper}. 

\begin{table}[htb!]
    \centering \scriptsize
    \begin{tabular}{llllll}
      \hline
      \hline
      Theorem & $\varepsilon$ & $f_0$ & Choice of $\bar{N}\bar{L}$, $\tau$  & Error Bound & Tail\\
      \hline
      Theorem~\ref{thm:generic-bound} & Cond. \ref{cond2} &  any & $\tau \in [c_1, \infty]$, any $\bar{N}$, $\bar{L}$  & \begin{tabular}{@{}c@{}}$\frac{\bar{N}\bar{L}\sqrt{\tau}}{\sqrt{n}} \land (\frac{\bar{N}\bar{L}}{\sqrt{n}})^{1-1/p}$ \\ $+\tau^{1-p} + \delta_{\mathtt{a}}$\end{tabular}  & --- \\
      
      Remark~\ref{coro:generic-bound-lse} & Cond. \ref{cond2} &  any & $\tau=\infty$, any $\bar{N}$, $\bar{L}$  & $ \delta_{\mathtt{a}} + \big(\frac{\bar{N}\bar{L}}{\sqrt{n}}\big)^{1-1/p}$ & $D^{-2p}$\\

      Theorem~\ref{thm:rate-adaptive-huber-upper} & Cond. \ref{cond2} & $\mathcal{H}(d,l,\mathcal{P})$ & $\tau \asymp n^{\frac{2\gamma^*(1-\nu^*)}{2\gamma^*+\nu^*}}$, $\bar{N}\bar{L} \asymp n^{\frac{\nu^*}{4\gamma^*+2\nu^*}}$  & $n^{-\frac{\nu^*\gamma^*}{2\gamma^*+\nu^*}}$ & $e^{-D^2}$\\

      Theorem \ref{thm:rate-lse-upper} & Cond. \ref{cond2} &  $\mathcal{H}(d,l,\mathcal{P})$ & $\tau = \infty$, $\bar{N}\bar{L} \asymp n^{\frac{\nu^\dagger}{4\gamma^*+2\nu^\dagger}}$  & $n^{-\frac{\nu^\dagger\gamma^*}{2\gamma^*+\nu^\dagger}}$ & $D^{-2p}$\\
      \hline
      
      Theorem \ref{thm:convergence-rate-nn-symmetric} & Cond. \ref{cond3} & $\mathcal{H}(d,l,\mathcal{P})$ & $c_1 \le \tau \lesssim 1$, $\bar{N}\bar{L}\asymp n^{\frac{1}{4\gamma^*+2}}$ & $n^{-\frac{\gamma^*}{2\gamma^*+1}}$ & $e^{-D^2}$ \\
      \hline
    \end{tabular}
	\caption{A summary of error bounds on $\|\hat{f} - f_0\|_2$ derived in Section \ref{sec3} under Condition \ref{cond1}. Here $c_1$ is defined in Proposition \ref{prop:strong-convex}, $\delta_\mathtt{a}$ is the neural network approximation error to $f_0$, $\gamma^*=\inf_{(\beta,t)\in \mathcal{P}} (\beta/t)$, $\nu^*=1-1/(2p-1)$ and $\nu^\dagger = 1-1/p$.}
    \label{table:convergence-rate-upper}
\end{table}

\label{subsec:upper-bound-discussion}

\noindent \textbf{Non-asymptotic $L_2$ error bound for regression with Huber loss.} \cite{farrell2020deep} and \cite{farrell2021deep} established upper bounds on the $L_2$ error $\|\hat{f} - f_0\|_2$ for general loss functions under a bounded noise/response assumption. To be specific, \cite{farrell2021deep} assumed that the loss function $\ell(\cdot, Y)$ is $C_{\ell}$-Lipschitz and satisfies
\begin{align}
\label{eq:cond-liang}
    \|f - f^*\|_2^2 \lesssim \mathbb{E} [\ell(f(X),Y) - \ell(f^*(X),Y)] \lesssim \|f - f^*\|_2^2 
\end{align} 
for any $f \in \mathcal{F}(d,\bar{L}, \bar{N}, 1)$, where $f^* = \text{argmin}_f \mathbb{E}[\ell(f(X),Y)]$ is the population risk minimizer. When $Y$ is (almost surely) bounded, they showed that the empirical risk minimizer $\hat{f}$ satisfies with probability at least $1-e^{-t}$ that 
\begin{align*}
    \|\hat{f} - f^*\|_2 \lor \|\hat{f} - f^*\|_n \lesssim \delta_{\mathtt{a}} + C_{\ell}^{1/2} \sqrt{\frac{(\bar{N}\bar{L})^2 \log(\bar{N}\bar{L}) \log n}{n}} + \sqrt{\frac{t}{n}} .
\end{align*}
Our setting is similar to theirs to some extent. For example, we also derive a non-asymptotic $L_2$ error bound without requiring the network weights to be uniformly bounded. According to Proposition \ref{prop:strong-convex}, the Huber loss satisfies condition \eqref{eq:cond-liang} with $f^* = f_{0,\tau}$ and $C_\ell = \tau$. When the noise variable $\varepsilon$ is bounded, applying their results directly yields
\begin{align}
\label{eq:liang-error-bound}
    \|\hat{f} - f_{0,\tau}\|_2 \lesssim \inf_{f\in \mathcal{F}(d,\bar{L},\bar{N}, M)} \|f - f_{0,\tau} \|_\infty + \sqrt{\tau \frac{(\bar{N}\bar{L})^2 \log(\bar{N}\bar{L}) \log n}{n}} .
\end{align}
The technical proofs when $\varepsilon$ only has bounded $p$-moment can be drastically different. We thus rely on a more refined argument that combines peeling with truncation, plus a bias analysis. Consequently, our results are sharper under certain important settings. To be more precise, when $\tau  \lesssim \{n/[(\bar{N}\bar{L})^2 \log(\bar{N}\bar{L}) \log n]\}^{1/p}$, our upper bound is essentially the same as that in \eqref{eq:liang-error-bound} but is derived under a much weaker moment assumption. When $\tau  \gtrsim  \{n/[(\bar{N}\bar{L})^2 \log(\bar{N}\bar{L}) \log n]\}^{1/p}$, our upper bound becomes strictly sharper.

%Moreover, our upper bound of the stochastic error \eqref{eq:generic-upper-bound-three-terms} in Theorem \ref{thm:generic-bound} shows that we can even do better than \eqref{eq:liang-error-bound} for large $\{n/[(\bar{N}\bar{L})^2 \log(\bar{N}\bar{L}) \log n]\}^{1/p}=o(\tau)$, in which increasing $\tau$ will not increase the stochastic error. 

Under the same heavy-tailed noise setting,  \cite{shen2021robust} established the following bound on expected $L_2$ error
\begin{align}
\label{eq:shen-error-bound}
    \EE \|\hat{f} - f_{0,\tau}\|_2 \lesssim  \inf_{f\in \mathcal{F}(d,\bar{L}, \bar{N}, M)} \|f - f_{0,\tau}\|_\infty + \sqrt{\tau \frac{(\bar{N}\bar{L})^2 \log(\bar{N}\bar{L}) \log n}{n^{1-1/p}}}.
\end{align} 
The second (stochastic) term on the right-hand side of \eqref{eq:shen-error-bound} turns out to be sub-optimal, even compared to the least squares ReLU-DNN estimator. To see this, together Theorem \ref{thm:generic-bound} and Corollary \ref{coro:generic-bound-lse} show that the stochastic error for the least squares ReLU-DNN estimator is of order
\begin{align*}
    \sqrt{\frac{\{ (\bar{N}\bar{L})^2 \log(\bar{N}\bar{L}) \log n \}^{1-1/p}}{n^{1-1/p}}} , 
    %=o\left( \sqrt{\tau \frac{(\bar{N}\bar{L})^2 \log(\bar{N}\bar{L}) \log n}{n^{1-1/p}}} \, \right)  , 
\end{align*} 
which is strictly smaller than theirs because $\bar{N}\bar{L}$ also increases with $n$ to achieve optimal rate in nonparametric settings.

\medskip
\noindent \textbf{$L_2$ bound for least squares ReLU-DNN estimator under heavy-tailed noise.} As a special case ($\tau=\infty$) in the class of Huber estimators, our Corollary \ref{coro:generic-bound-lse} provides an $L_2$ error bound for the LSE under heavy-tailed errors. 
For a general nonparametric LSE over some class $\cF$, \cite{kuchibhotla2019least} showed that if $p=2$, $\mathcal{F}$ is a uniformly bounded function class with finite Pseudo-dimension, and the local envelope function,  defined as $
    F_\delta(x) = \sup_{f\in \mathcal{F}, \|f - f_0\|_2 \le \delta} |f(x) - f_0(x)|$, satisfies $\|F(x)\|_2 \lesssim \delta^s$, then it holds
\begin{align*}
    \mathbb{P}\big\{ \| \hat{f} - f_0 \|_2 \ge D n^{-\frac{1}{2(2-s)}}\big\} \le C D^{-4(2-s)/3} ,
\end{align*} 
where $\hat{f} = \argmin_{f\in \mathcal{F}}   \sum_{i=1}^n \{ Y_i - f(X_i)\}^2$ and $f_0 \in \mathcal{F}$. It is easy to verify that the local envelope function for the ReLU neural network class satisfies the above condition with $s=0$. In this case, their result leads to a stochastic error of order $n^{-1/4}$ and tail probability $D^{-8/3}$.
%This implies one can apply their result to get a stochastic error of $n^{-\frac{1}{4}}$ with tail $D^{-8/3}$. 
It remains unclear whether a faster rate can be obtained when $p>2$ by extending their proof techniques. 
Our Corollary \ref{coro:generic-bound-lse}, on the other hand, leads to sharp error bound $n^{-\frac{1}{2}(1-1/p)}$ and smaller tail probability $D^{-2p}$ for any $p\geq 2$.  
%is able to recover the sharp error bound for arbitrary $p\ge 2$ with faster-decayed tail probability.

\end{document}